\documentclass[a4paper,10pt,twoside,openright]{book}
\usepackage{amsmath,amssymb,amsfonts}
\usepackage{epsfig,color}
\usepackage{mathrsfs}
\usepackage[hyperindex]{hyperref}
\usepackage[utf8]{inputenc}
\usepackage{dsfont}
\usepackage[english]{babel}
\usepackage{amsthm}
\usepackage{glossaries}
\usepackage{makeidx}
\usepackage{fancyhdr}
\def    \bs     {\boldsymbol}

\usepackage[a4paper,top= 4cm, bottom=3.5cm,left=3cm,right=3cm,bindingoffset=5mm]{geometry}
\makeindex
\makeglossary

\theoremstyle{plain}
\newtheorem{thm}{Theorem}[section]
\newtheorem{prop}[thm]{Proposition}
\newtheorem{defn}[thm]{Definition}
\newtheorem{cor}[thm]{Corollary}
\newtheorem{lemma}[thm]{Lemma}
\usepackage{tikz}

\theoremstyle{definition}
\newtheorem{Assumption}{Assumption}[chapter]
\newtheorem{exm}[thm]{Example}

\theoremstyle{remark}
\newtheorem{rmk}[thm]{Remark}

\newcommand{\X}{\mathbf{X}}
\newcommand{\Y}{\mathbf{Y}}
\newcommand{\TX}{\mathbf{TX}}
\newcommand{\NN}{\mathbf{N}}
\newcommand{\E}{\mathbb{E}}

\newcommand{\SSw}{\mathbf{S}}
\newcommand{\ph}{\varphi}
\newcommand{\Z}{\mathbf{Z}}
\newcommand{\U}{\boldsymbol{U}}
\newcommand{\N}{\mathbb{N}}
\newcommand{\R}{\mathbb{R}}
\newcommand{\C}{\mathbb{C}}
\newcommand\numberthis{\addtocounter{equation}{1}\tag{\theequation}}

\def    \EE      {\boldsymbol{E}}
\def    \E      {\mathbb{E}}
\def    \sst        {\scriptscriptstyle}

\def\P{{\rm I\kern-0.16em P}}
\numberwithin{equation}{section}
\font\eka=cmex10
\usepackage{ae}
\def\ind{\mathrel{\hbox{\rlap{%
\hbox to 7.5pt{\hrulefill}}\raise6.6pt\hbox{\eka\char'167}}}}
\makeatletter
\def\cleardoublepage{\clearpage\if@twoside
\ifodd \c@page
\else \hbox {}\thispagestyle{empty}\newpage
\if@twocolumn\hbox{}\newpage\fi\fi\fi}
\makeatother
\allowdisplaybreaks

\date{}
\begin{document}
\frontmatter

\thispagestyle{empty}
\frontmatter
\thispagestyle{empty}
\begin{figure}[t]
\centering
\includegraphics[scale = 1]{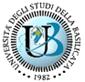}
\end{figure}

\begin{center}

	\textbf{\large{UNIVERSIT\`{A} DEGLI STUDI DELLA BASILICATA}}\\
	Dipartimento di Matematica, Informatica ed Economia\\
	Potenza - Italy
	\vskip 0.5cm

	\textbf{\large{International Doctoral Seminar in Mathematics ``J\'{a}nos Bolyai" }}\\ Cycle XXVII  \\
	Sector Codes: MAT/06, MAT/02\\
	\end{center}

\vskip 1cm
\begin{center}
\huge{Universality and Fourth Moment Theorem for homogeneous sums. \\ Orthogonal polynomials and apolarity}
\end{center}

	\vskip 2cm
	\begin{tabular}{lccccr}
	
	\textbf{\large{Coordinator:}} & 	& &	 &	&  \qquad \qquad \qquad \qquad \qquad \textbf{\large{Ph.D. Candidate:}}\\
	\large{Prof. Antonio Cossidente}& 	& & 	&& \qquad\qquad \qquad\qquad \qquad	\large{Rosaria Simone} \\
	 	& & 	& &	&  \\
		\textbf{\large{Advisors:}}& 	&& 	& 	& \\
		\large{Prof. Giovanni Peccati} 		&&  	& & &\\
		\large{Prof. Domenico Senato}&    & && & \\
		\end{tabular}

		\vskip 2.8cm
			\begin{picture}(0,10)
	\put(0,0){\line(1,0){410}}
	\end{picture}
\begin{center}		A.A. 2013/2014 \end{center}

%

\newpage \thispagestyle{empty}

$ $
\newpage


\thispagestyle{empty}

\begin{center}
\textbf{\large{Acknowledgements}}
\end{center}

First and foremost, my deepest thanks go to my advisors, Professor Domenico Senato and Professor Giovanni Peccati, for guiding me with constant enthusiasm towards such an important achievement. Their willingness in sharing with me ideas and their experience in the world of research, has been the greatest opportunity I could have ever asked for moving my first steps as a researcher.

In particular, I would like to express my sincere gratitude to Professor Giovanni Peccati, for allowing me to conduct part of my research under his supervision, visiting the University of Luxembourg. Being a member of his research group has meant a lot to me: thank you for providing me with an amazing working environment, and for introducing me to extremely compelling topics and intriguing problems I am pleased to have learnt about. Most importantly, I acknowledge his caring assistance and confidence in my work. 

My great appreciation goes to Professor Domenico Senato,  for having constantly encouraged me in pursuing my research activity, and for all the consideration he has always addressed to my work. I would  like also to acknowledge Pasquale Petrullo, whose perseverance has been a model for me since the very beginning, and Elvira Di Nardo, for all the good advices.

I wish also to thank Professor Ivan Nourdin, that I had the pleasure to meet and work with during my stay in Luxembourg, as well as all the members of the mathematical research group of Prof. Peccati: Yvik Swan, Ehsan Azmoodeeh and Guillaume Poly, who deserves special acknowledgements for his friendly support. 

I cannot free myself from thanking the ones that I will never thank enough: first of all, my mother, my father and my beloved sister. Then, my whole-hearted thanks go to my colleagues,  Sara, Emanuela and Pietro, for being such incomparable mates, as well as to my friend Rocco, that is always beyond compare.

And finally, to the best that is yet to come: this thesis is dedicated to us, Marco, and to our future together. Thank you.\\
 
The work and the studies that led to the present dissertation have been supported by a  M.I.U.R. grant.


\newpage

\thispagestyle{empty}

\tableofcontents

\thispagestyle{empty}

\printglossary

\pagestyle{plain}

\chapter*{Introduction}
\addcontentsline{toc}{chapter}{Introduction}

The aim of the present essay is threefold: the first two parts are devoted to fully explore the Fourth Moment Theorem and the universality phenomenon in the framework of homogeneous sums, both in the classical and in the free probability setting, while the last part approaches the classical theory of orthogonal polynomials via the invariant theory of binary forms, focusing on apolarity. \\

A \textit{universality} result (or \textit{invariance principle}) is a mathematical statement implying that the asymptotic behaviour of a given random system does not depend on the distribution of its components: the most celebrated instance of such a phenomenon is undoubtedly the Central Limit Theorem (CLT), or its functional version for random walks, known as Donsker's Theorem. If $S(\mu)$ denotes a random system depending on some probability distribution $\mu$, we shall say that $\mu$ is \textit{universal} if the fact that $S(\mu)$ verifies a limit theorem, implies that $S(\nu)$ (that is, the random system obtained by replacing $\mu$ with another probability law $\nu$) displays the same asymptotic behaviour, for any choice of $\nu$ in a large class of probability laws.

A whole line of research about invariance principles has recently emerged from a new version of the celebrated Lindberg method,  established in \cite{Mossel} and relying on \textit{low-influence functions}. This technique dates back to \cite{Rotar} and will be the trailhead of the first two parts of the dissertation. \\

Part \ref{Invariance} deals with the Lindberg method of influence functions in the setting of free probability spaces: the first and main step that will be accomplished in this direction is a general multidimensional invariance principle for vectors of homogeneous sums in freely independent random variables. The work expands the ideas first developed in \cite{NourdinDeya}, which constitute the free counterpart, in dimension 1, to the universality of the Gaussian Wiener Chaos proved in \cite{NourdinPeccatiReinert}. Secondly, a class of \textit{universal laws} for semicircular and free Poisson approximations of homogeneous sums in freely independent random variables will be derived by combining the Lindberg method of influence functions and the so-called \textit{Fourth Moment Theorems}.\\

\textit{Fourth Moment Theorems} are limit theorems for non-linear functionals of a random field, holding under the only assumption of the convergence of the corresponding sequences of the second and fourth moments. Both in the commutative and non-commutative setting, Fourth Moment Theorems (for central and non-central convergence), determine an elegant simplification of the method of moments and cumulants for several classes of random fields, starting from the landmark examples of the Gaussian Wiener Chaos \cite{NualartPeccati} and of the Wigner Chaos \cite{KempNourdinPeccatiSpeicher}. \\ 
The goal of  Part \ref{FMT} is the characterization of those random variables $X$ such that homogeneous sums based on i.i.d. copies of $X$ verify the Fourth Moment Theorem, both in the classical and in the free probability scenarios. So far, it is known that the fourth moment phenomenon for homogeneous sums applies, for instance, when $X$ is Gaussian, Poisson, semicircular or free Poisson distributed. Is there a characterizing property enabling this phenomenon, or does it occur accidentally? The main results will be the determination of a condition on the fourth cumulant of $X$ that is sufficient for the Fourth Moment Theorem to hold for homogeneous sums  (in both the probability settings): discussions on the existence of an optimal necessary and sufficient condition are also provided.\\

The choice of focusing on random variables having the form of multilinear homogeneous polynomials is motivated by the fact that these random objects have been gaining more and more interest in modern probability theory, both in the classical and in the free setting. As to classical probability theory, homogeneous sums in independent variables are instances of \textit{degenerate $U$-statistics} (which are the most appropriate non-linear extension of random sums), and they represent the seminal examples of those random variables determining the chaotic decomposition for square-integrable functionals of a Brownian motion (see \cite{NourdinPeccatilibro}), or other random fields. Analogously, in free probability theory, homogeneous polynomials in non-commutative semicircular random variables represent the base building blocks of multiple integrals with respect to a free Brownian motion, up to a density argument \cite{KempNourdinPeccatiSpeicher}. 

For algebraists and physicists, the ring of homogeneous polynomials in independent \linebreak non-commutative variates $x_1,\dots,x_n$, provides a representation for the tensor algebra of a given $\C$-vector space $V$. According to Hermann Weyl's philosophy, equations in the tensor algebra suffice for the description of \textit{geometric facts}, that is of the mathematical properties of the space $V$ that do not depend on the choice of the coordinate system. Quoting Gian-Carlo Rota \cite{Rota98}, ``[...]\textit{The program of invariant theory, from Boole to our day, is precisely the translation of geometric facts into invariant algebraic equations expressed in terms of tensors.}'' To accomplish this task, an effective choice of notation is paramount, not only for its own sake: for instance, the choice of dealing with the invariant theory (of binary forms) via a suitable symbolic method, allows one to revisit some aspects of the classical theory of orthogonal polynomials within the theory of apolarity: this will be the core of Part \ref{Orth}. \\

Every part will be further introduced by brief synopsis, where we present its contents and give more details about the results to be shown: at the same time, some bibliographic comments about related literature are given, to supplement all the discussion provided  throughout the chapters. 

Here is a short overview of the contents: 

\begin{enumerate}
\item the main result of Part \ref{Invariance} is Theorem \ref{Multiinvariance2}. The proximity in law between vectors of homogeneous sums in freely independent random variables is assessed in terms of the maximum of the influence functions. As a consequence, by 
combining the invariance principle with the Fourth Moment Theorem for Wigner integrals \cite{KempNourdinPeccatiSpeicher}, it can be shown that the law of the $n$-th Chebyshev polynomial $U_n(S)$ is universal for semicircular and free Poisson approximations of vectors of homogeneous sums for every $n\geq 1$, where $S$ denotes a standard semicircular random variable on a fixed free probability space. These results perfectly match the findings in \cite{NourdinDeya}, corresponding to the unidimensional setting as to the invariance principle, and to $n=1$ for the universality of $U_1(S) =S$ for semicircular approximations of homogeneous sums.  The same approach is then sketched to fit the commutative setting, yielding new universal laws for normal and Gamma approximations of homogeneous sums in independent random variables.\\

\item Part \ref{FMT} is devoted to the analysis of the Fourth Moment phenomenon for central and non-central convergence of random variables having the form of multilinear homogeneous polynomials (both in the commutative and non-commutative setting).
Theorems \ref{superTeo1} and \ref{superTeo2} extend the findings in \cite{NualartPeccati, NourdinPeccati2,KempNourdinPeccatiSpeicher, NourdinPeccati1} to homogeneous sums in independent copies of a random variable having non-negative kurtosis (also called \textit{leptokurtic}). Some minimal additional hypothesis will be needed in the classical setting. Moreover, the same condition is shown to be sufficient for the universality phenomenon to occur, in turn. In both settings, the starting point is a new combinatorial formula for the fourth moment of a homogeneous sum. Beyond their intrinsic interest, these results allow one to examine further properties of leptokurtic random variables: for instance, the equivalence between componentwise and joint central convergence  for vectors of homogeneous sums based on such laws, that extend the findings of \cite{PeccatiTudor} and \cite{NouSpeiPec}. This equivalence, in turn, can be applied to provide a general multidimensional \textit{transfer principle} for Fourth Moment Theorems, that supplements the correspondence between Wiener and Wigner Chaos established in \cite{NouSpeiPec}.\\

\item  Part \ref{Orth} has a purely algebraic flavour, and deals with orthogonality of polynomials and invariant theory. As a matter of fact, orthogonal polynomials constitute a recurrent theme in the whole theory of stochastic analysis \cite{Schoutens}, and, more generally, in probability theory. But what are orthogonal polynomials, really? Through the so called \textit{symbolic method of invariant theory} \cite{KungRota}, an algebraic and universal representation for (generalized) sequences of orthogonal polynomials (in any number of variables) is achieved, via a general Heine integral formula and a determinantal formula. Recasting the theory of orthogonality within the invariant theory of binary forms,  some applications to probability theory are derived, as explicit formulae for the moments of the so called \textit{random discriminants} \cite{Lu}. Finally, a brief focus on  the most prominent example of semi-invariants in probability and statistics, that is, the cumulants, will be presented via a combinatorial technique: more precisely, the analysis will be run through the tools deriving from the combinatorial approach to stochastic integration developed in \cite{Rota1}. The starting point is the representation of cumulants  as expected value of the so-called \textit{diagonal measures}. This setting turns out to be particularly suitable to manage cumulants of the process of variations of a L\'{e}vy process, as well as to describe  $\kappa$-statistics and polykays for positive random measures, in both the classical and the free setting.
\end{enumerate}

It is worth to stress that the first two parts share a common thread: the universality phenomenon, though they rely on different motivations and techniques. In particular, the discussion presented in Part \ref{Invariance} is primarily oriented at establishing  the free counterpart to the multidimensional invariance principle provided in \cite{NourdinPeccatiReinert}. The derivation of other universal laws for semicircular and free Poisson approximation of homogeneous sums follows as a consequence of the strategy proposed in \cite{NourdinDeya}. On the other hand, the core of Part \ref{FMT} is the search of a possible characterization of the laws verifying a Fourth Moment Theorem. It is the strategy of proof that highlights the interactions with the universality phenomenon. Beyond all that, the results of Part \ref{Invariance} will be referred to in Part \ref{FMT} to derive the general \textit{transfer principle} for Fourth Moment Theorems between the two probability settings.
\newpage

\pagestyle{fancy}

\chapter{Preliminaries}

Before dwelling on the contents developed in the present dissertation, some preliminary facts need to be recalled: to this aim, Section \ref{App1} is devoted to the combinatorics underlying the probabilistic aspects that will be addressed, specially focusing on the moment-cumulant formulae and on the properties of cumulants. Parts \ref{Invariance} and \ref{FMT} are strongly connected to the theory of multiple stochastic integration under several aspects. For the reader's convenience, some basic definitions and standard notations are preliminarily introduced in Sections \ref{Pre_Classic} and \ref{AppendixfreeProb}, for the commutative and non-commutative setting respectively.  Finally, Section \ref{Pre_Orth} contains some background material concerning classical orthogonal polynomials, and it will be particularly useful for the discussion faced in Part \ref{Orth}. None of these sections is  meant to be exhaustive: more references are quoted therein. In particular, the main references concerning free probability theory are \cite{Speicher,Voiculescu} (see also \cite{Tao} for a survey on random matrix theory), while for analysis on Gaussian spaces, the main references are \cite{NualartBook, Janson_GaussianSpaces, NourdinPeccatilibro}.
%

\section{The lattice of partitions: moment-cumulant formulae}\label{App1}

Let $n \in \mathbb{N}$\glossary{name={$\N$},description={Set of positive  integers}}. A \textit{partition} $\pi$ of the set $[n]:=\{1,\dots,n\}$  \glossary{name={$[n]$}, description={Set of the first $n$ positive integers}}is a collection of non-empty and pairwise disjoint subsets of $[n]$ (called \textit{blocks}) whose union is $[n]$. The set $\mathcal{P}([n])$ of all partitions of $[n]$ \index{Partition of a set}is a lattice with respect to the \textit{refinement} order: $\pi \leq \sigma$ if each block of $\pi$ is contained in a block of $\sigma$, in which case $\pi$ is said to be \textit{finer} than $\sigma$, or $\sigma$ \textit{coarser} than $\pi$. The maximum and the minimum are usually denoted by $\hat{1} = \{\{1,2,\dots, n\}\}$ and $\hat{0} = \{\{1\},\{2\},\dots,\{n\}\}$,  respectively. The greatest lower bound is the meet partition, denoted by $\pi \wedge \sigma$, whose blocks are the non-empty intersections between a block of $\pi$ and a block of $\sigma$, while the least upper bound is the join partition denoted by $\pi \vee \sigma$, whose blocks are obtained by joining a block of $\pi$ and a block of $\sigma$ if they share at least one element, and repeating this operation with the block so obtained, until it is possible. For example, if $\pi = 12|345|6|78$, and $\sigma = 123|678|45$, then $\sigma \wedge \pi = 12|3|45|6|78$, and $\sigma \vee \pi = 12345|678$. For $i, j \in [n]$ and $\sigma \in \mathcal{P}([n])$, the notation $i \sim_{\sigma} j$ will denote that $i$ and $j$ belong to the same block of $\sigma$.

In the sequel, $\mathcal{P}_2([n])$ will denote the set of the \textit{pairing partitions} \index{Partition of a set!Pairing Partition}(also called \textit{perfect matchings}), that is, the partitions whose blocks all have  cardinality equal to $2$: trivially, $\mathcal{P}_2([n])$ is non-empty if and only if $n$ is even, in which case $| \mathcal{P}_2([n])| = (n-1)!! = (n-1)(n-3)\cdots 5\cdot 3 \cdot 1$.

For $n,m \geq 1$, a crucial role will be played  by the interval partition in $\mathcal{P}([mn])$: 
$$\pi^{\star} := m^{\otimes n} = \{\{1,\dots m\}, \{ m+1,\dots, 2m\}, \{2m+1,\dots, 3m\}, \dots, \{(n-1)m+1,\dots, nm\}\},$$
which is usually denoted by $\pi^{\star} = 1\cdots m | (m+1)\cdots 2m| (2m+1)\cdots 3m|\cdots | ((n-1)m+1)\cdots nm.$
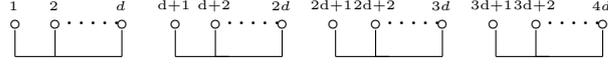
\begin{figure}[h]
\begin{picture}(250, 20)
\put(85,7){\makebox(-1,0){\tiny{$1$}}}
\put(85,0){\circle{3}}
\put(100,7){\makebox(-1,0){\tiny{$2$}}}
\put(100,0){\circle{3}}
\put(85,-12){\line(0,0){10}}
\put(85,-12){\line(1,0){15}}
\put(100,-12){\line(0,0){10}}
\put(100,-12){\line(1,0){25}}
\put(104,0){\dots .\,.}
\put(125,7){\makebox(-1,0){\tiny{$d$}}}
\put(125,0){\circle{3}}
\put(125,-12){\line(0,0){10}}

\put(145,7){\makebox(-1,0){\tiny{d+1}}}
\put(145,0){\circle{3}}
\put(160,7){\makebox(-1,0){\tiny{d+2}}}
\put(160,0){\circle{3}}
\put(145,-12){\line(0,0){10}}
\put(145,-12){\line(1,0){20}}
\put(160,-12){\line(0,0){10}}
\put(160,-12){\line(1,0){25}}
\put(164,0){\dots .\,.}
\put(185,7){\makebox(-1,0){\tiny{$2d$}}}
\put(185,0){\circle{3}}
\put(185,-12){\line(0,0){10}}

\put(204,7){\makebox(-1,0){\tiny{2d+1}}}
\put(204,0){\circle{3}}
\put(220,7){\makebox(-1,0){\tiny{2d+2}}}
\put(220,0){\circle{3}}
\put(205,-12){\line(0,0){10}}
\put(205,-12){\line(1,0){20}}
\put(220,-12){\line(0,0){10}}
\put(220,-12){\line(1,0){25}}
\put(224,0){\dots .\,.}
\put(245,7){\makebox(-1,0){\tiny{$3d$}}}
\put(245,0){\circle{3}}
\put(245,-12){\line(0,0){10}}

\put(264,7){\makebox(-1,0){\tiny{3d+1}}}
\put(264,0){\circle{3}}
\put(280,7){\makebox(-1,0){\tiny{3d+2}}}
\put(280,0){\circle{3}}
\put(265,-12){\line(0,0){10}}
\put(265,-12){\line(1,0){20}}
\put(280,-12){\line(0,0){10}}
\put(280,-12){\line(1,0){25}}
\put(284,0){\dots .\,.}
\put(305,7){\makebox(-1,0){\tiny{$4d$}}}
\put(305,0){\circle{3}}
\put(305,-12){\line(0,0){10}}
\end{picture}
\vspace{0.5cm}
\caption{Diagram of $\pi^{\star} = d^{\otimes 4} \in \mathcal{P}([4d])$.}
\end{figure}

Observe that, if $\sigma \in \mathcal{P}([nm])$ is such that $\sigma \wedge \pi^{\star}= \hat{0}$, then in each block of $\sigma$ there is at most one element from every block of $\pi^\star$: in this case, we say that $\sigma$ \textit{respects} $\pi^\star$\index{Partition of a set!Respectful}.

For $d\geq 1$, if $\mathbf{i} =(i_1,\dots,i_d) \in [n]^{d}$\glossary{name={$[n]^d$},description={$d$-fold cartesian product of $[n]$}}, the \textit{kernel} $\text{Ker}(\mathbf{i})$ \index{Partition of a set!Kernel}is the partition in $\mathcal{P}([d])$ determined by the rule $ h \sim j$ if and only if $i_h = i_j$. Moreover, for every $f:[n]^d \rightarrow \mathbb{R}$\glossary{name={$\mathbb{R}$},description={Real numbers}}, for any $\rho \in \mathcal{P}([d])$, $f_{\rho}(\mathbf{i})$ is obtained from $f(i_1,\dots,i_d) \in [n]^d$ by setting $i_l = i_s$ if and only if $l \sim_{\rho} s$. 
If $\sigma \in \mathcal{P}([n])$, its \textit{class} \index{Partition of a set!Class} is the partition of the integer $n$, say $\lambda = (1^{r_1} \; 2^{r_2}\, \dots \;n^{r_n})\vdash n$, where $r_j$ is the number of blocks of $\sigma$ having cardinality $j$ (for a synthetic account of the lattice of partitions, as well as some other examples, see \cite[Chapter 2]{PeccatiTaqqu}).\\

A partition $\pi$ is said to be \textit{non-crossing} \index{Non-crossing partition}if, whenever there exist integers $i < j < k < l$, with $i\sim_{\pi} k$, $j \sim_{\pi} l$, then $j \sim_{\pi} k$. The lattice of non-crossing partitions, denoted by $\mathcal{NC}([n])$, is the combinatorial structure underlying the free probability setting. In the sequel, $\mathcal{NC}_2([n])$ will denote the set of the \textit{non-crossing pairings} \index{Non-crossing partition!Non-crossing pairing} of $[n] = \{1,2,\dots,n\}$, that is, the set of all non-crossing partitions of the set $[n]$ where each block has exactly two elements. Of course,  $\mathcal{NC}_2([n])$ is empty if $n$ is odd, while it has $C_{\frac{n}{2}}$ elements if $n$ is even, where $C_k = \frac{1}{k+1}\binom{2k}{k}$ denotes the $k$-th \textit{Catalan number}\index{Catalan numbers} (see \cite[Lecture 2]{Speicher}, or \cite{Stanley}).

Finally, $\mathcal{P}^{\star}(m^{\otimes n})$, $\mathcal{P}_2^{\star}(m^{\otimes n})$, and $\mathcal{NC}^{\star}(m^{\otimes n})$, $\mathcal{NC}_2^{\star}(m^{\otimes n})$ will stand respectively for the set of  partitions in $\mathcal{P}([mn])$, $\mathcal{P}_2([mn])$, and $\mathcal{NC}([mn])$, $\mathcal{NC}_2([mn])$ that  respect $\pi^{\star} = 1\cdots m | (m+1)\cdots 2m| (2m+1)\cdots 3m|\cdots | ((n-1)m+1)\cdots nm.$

\section{Elements of Classical Probability Theory}\label{Pre_Classic}

All the random objects (in the classical sense) that will be encountered in the following are assumed to be defined on a suitable classical probability space $(\Omega, \mathcal{F},\mathbb{P})$,\glossary{name={$(\Omega,\mathcal{F},\mathbb{P})$},description={Classical probability space}} and $\mathbb{E}$ will denote the expectation on it. Standard references for classical probability theory include \cite{Billingsley,Billingsley2,Dudley,Chung}.\\

Given a sequence of (real) random variables $Z_n$,  and a random variable $Z$, it is said that $Z_n$ \textit{converges in law} \index{Convergence in law} (or \textit{in distribution}, or \textit{weakly}) to $Z$ (for short, $Z_n \xrightarrow{\text{\rm Law}} Z$) if and only if, for every continuous bounded function $g: \mathbb{R} \to \mathbb{R}$, it holds true that:
$$ \mathbb{E}[g(Z_n)]  \underset{n \rightarrow \infty}{\longrightarrow} \mathbb{E}[g(Z)].$$
From the definition, it follows  that convergence in law does not always imply the convergence of the corresponding moments: this implication holds true, for instance, when the random variables $Z_n$ have a density with compact support, by virtue of the Stone-Weierstrass approximation Theorem (see \cite{Dudley}). Another situation when the convergence of the moments is fulfilled under the assumption of convergence in law occurs when the random variables $g_k(Z_n)$, for $g_k(x)=x^k$, $k\geq 1$, are uniformly integrable. For instance, if $Z_n \xrightarrow{\text{\rm Law}} Z$, and there exists $\delta >0$ such that $\sup\limits_{n\geq 1}\mathbb{E}[|Z_n|^{k+\delta}] < \infty$, then the sequence $|Z_n|^k$ is uniformly integrable, and hence $\E[Z_n^r] \rightarrow \E[Z^r]$ for every $r=1,\dots,k$ (see \cite[Chapter 6]{DasGupta} for a concise overview about moment convergence and uniform integrability).\\

For a probability measure $\mu$, the notation $X \sim \mu$ will indicate that the random variable $X$, with finite moments of every order, is distributed according to $\mu$. Recall that  a probability measure $\mu$, all of whose moments are finite,  is said to be \textit{determined by its moments} $\E[X^n]$, with $X \sim \mu$, if and only if, given a random variable $Y$ such that $\E[X^k]=\E[Y^k]$ for all $k\geq 1$,  necessarily $Y \sim \mu$. For a sequence of random variables $Z_n$ and a random variable $Z$, whose moments are finite (or, more weakly, such that the random variables $|Z_n|^k$ are uniformly integrable for every $k\geq 1$),  $\E[Z_n^k] \rightarrow \E[Z^k]$ for every $k\geq 1$ implies that $Z_n \xrightarrow{\text{\rm Law}} Z$  whenever the law of $Z$ is determined by its moments (see, for instance, \cite[Theorem A.3.1]{NourdinPeccatilibro}). Such statement encodes the so-called \textit{method of moments and cumulants}.\\

Finally, recall that the \textit{Wasserstein distance} between two  random variables $T, F$  taking values in $\mathbb{R}^d$, is defined as:
$$ d_{\mathcal{W}}(T,F)  = \sup_{h \in H}\big\{ | \E[h(T)]  - \E[h(F)] |    \big\}  ,$$
where $H$ is the class of the Lipschitz  functions  $h:\mathbb{R}^d \rightarrow \mathbb{R}$, with Lipschitz constant less or equal than $1$. The Wasserstein distance is particularly useful within approximations of probability laws because the induced topology is strictly stronger than that of convergence in distribution: this means that, given a sequence of random variables $\{F_n\}_{n\geq 1}$, if $d_{\mathcal{W}}(T,F_n) \rightarrow 0$ as $n \rightarrow \infty$, then $F_n \xrightarrow{\text{\rm Law}} T$.

\subsection{Moment-cumulant formula for classical random variables}

Let $\chi_j(X)$, for $j \in \mathbb{N}$, denote the (classical) $j$-th \textit{cumulant} of a random variable $X$ having moments of all orders, namely 
$$ \chi_j(X) = (-\imath)^j \dfrac{\partial^{n}}{\partial{t}^n}\ln \E[\exp(\imath tX)],$$
with $\E[\exp(\imath tX)]$ denoting the moment generating function of $X$ (see, for instance, ~\cite[Chapter 3]{PeccatiTaqqu}).

If $X_1,\dots,X_n$ are random variables on a fixed classical probability space, moments and cumulants are related by the formula:
\begin{equation}
\label{MomCum}
\E[X_1 X_2 \cdots X_n ] = \sum_{\sigma \in \mathcal{P}([n])} \prod_{B \in \sigma } \chi(X_j: j \in B)
\end{equation}
with $\chi(X_j:j \in B)$ denoting the \textit{multidimensional cumulant}, given by \textit{M\"obius inversion}:
\begin{equation}
\label{MomCum2}
\chi(X_1,\dots,X_n) = \sum_{\sigma \in \mathcal{P}([n])} \mu(\sigma, \hat{1} ) \prod_{B \in \sigma} \E\big[\prod_{j \in B}X_j\big],
\end{equation}
and $\mu(\sigma,\hat{1})$ denoting the \textit{M\"{o}bius function} on the interval $[\sigma, \hat{1}]$ \cite{Rota2}. Note that, originally, cumulants are defined via:
$$ \chi(X_1,\dots,X_n) = (-\imath)^n \dfrac{\partial^{n}}{\partial{t_1}\cdots \partial{t_n}}\ln \E\bigg[\exp\bigg(\imath \sum_{j=1}^n t_j X_j    \bigg)\bigg].$$
In particular, setting $X_j=X$ for all $j$, the \textit{generalized cumulant} $\chi_{\pi}(X)$ of $X$ is the multiplicative function defined for $\pi \in \mathcal{P}([n])$, and satisfying the formula:
$$ \E[X^n] = \sum\limits_{\pi \in \mathcal{P}([n])} \chi_{\pi}(X),$$ 
with $\chi_\pi(X)= \prod\limits_{B\in \pi} \chi_{|B|}(X)$, or equivalently, \index{Cumulant}
$$\chi_n(X) = \sum_{\pi \in \mathcal{P}([n])}  \mu(\pi, \hat{1}) \prod_{B \in \pi} \E[X^{|B|}],$$
where $|B|$ denotes the cardinality of the block $B$. In particular, $\chi_1(X) = \E[X]$, $\chi_2(X) = \E[X^2]- \E[X]^2$ and, if $X$ is centered, the \textit{kurtosis}\index{Cumulant!Kurtosis} $\chi_4(X)$ is given by $\chi_4(X)= \E[X^4]-3\E[X^2]^2$.\\
The most important feature of cumulants is that they characterize independence better than moments, in the sense that $\chi(X_{i_1},\dots,X_{i_n}) =0$ whenever there exists $b \subset [n]$ such that $\{X_{i_j}:j \in b\}$ is independent of $\{X_{i_j}:j \in [n]\setminus b\}$. This property entails in turn the \textit{additivity} of cumulants, that is, if $X$ and $Y$ are independent, then $\chi_n(X+Y) = \chi_n(X) + \chi_n(Y)$ for every $n\geq 1$.\\

\paragraph*{Main distributions}
The main distributions that will be encountered in the sequel are listed below, together with some of their principal properties. \\

In the sequel, $\mathcal{N}(0,1)$ will denote the standard Gaussian distribution \index{Gaussian distribution} with density $x\mapsto \frac{1}{\sqrt{2\pi}}e^{-x^2/2}$ on $\mathbb{R}$, so that the notation $N\sim \mathcal{N}(0,1)$ will indicate that the random variable $N$ is distributed according to the standard Gaussian law; similarly, $\mathcal{N}(m,\sigma^2)$ denotes the Gaussian distribution with mean $m$ and variance $\sigma^2$. The Gaussian law $\mathcal{N}(m, \sigma^2)$ is determined by its moments: if $N\sim \mathcal{N}(m, \sigma^2)$, then $\E[N]=m$ and, for every $k\geq 2$,
$$
\E[N^{k}]=
\begin{cases} 
0 & \text{ if } k \text{ is odd;}\\
\sigma^{k}(k-1)!! & \text{ if } k \text{ is even,}
\end{cases}
$$
where $(k-1)!! = (k-1)(k-3)\cdots 5 \cdot 3 \cdot 1$. Equivalently, $\chi_1(N)=m$, $\chi_2(N)= \sigma^2$, $\chi_k(N)=0$ for all $k\geq 3$.\\

For $\alpha, \beta >0$, the Gamma distribution $\Gamma(\alpha,\beta)$ of shape parameter $\alpha$ and rate $\beta$, is the probability measure on the real interval $(0,\infty)$, with density $\dfrac{1}{\Gamma(\alpha)}\beta^{\alpha}x^{\alpha-1}e^{-\beta x}dx$, where \linebreak$\Gamma(\alpha)= \int_{0}^{\infty}x^{\alpha-1}e^{-x}dx$ denotes the Gamma function. The Gamma distribution is also determined by its moments: if $X \sim \Gamma(\alpha,\beta)$, its moments are given by the formula:
$$ \E[X^{k}]= \dfrac{1}{\beta^k}\prod\limits_{j=1}^k (\alpha + j-1).$$
Therefore, if $X \sim \Gamma(\alpha,\beta)$,  for any real number $c>0$, $cX \sim \Gamma(\alpha, \frac{\beta}{c})$. In particular, for $k \in \mathbb{N}$ the Gamma distribution $\Gamma(\frac{k}{2},\frac{1}{2})$ concides with the $\chi^2(k)$ distribution with $k$ degree of freedom; moreover, $X \sim \chi^2(k)$ if $X \stackrel{\text{Law}}{=} N_1^2 + \cdots + N_k^2$, where $N_1,\dots, N_k$ are i.i.d. standard normal random variables. \\

For $\lambda >0$, a discrete random variable $X(\lambda)$ with  probability mass determined by \linebreak$\mathbb{P}(X(\lambda)=n) = \dfrac{e^{-\lambda}\lambda^n}{n!}$, is called a \textit{Poisson} random variable of parameter (or rate) $\lambda$. The Poisson law is characterized by having all the cumulants equal to the rate $\lambda$, namely $\chi_m(X(\lambda))= \lambda$ for every $m\geq 1$.\\

For Gaussian systems, the moment-cumulant formula \eqref{MomCum} simplifies due to the vanishing of the cumulants of order greater or equal than $3$. Indeed, every Gaussian system of centered random variables $\{N_i\}_{i\geq 1}$ is completely determined by its covariance structure, and \eqref{MomCum} reduces to the so-called \textit{Wick formula}\index{Wick Formula}: 
\begin{equation}
\label{Wick}
\mathbb{E}[N_{i_1}\cdots N_{i_k}]  = \sum_{\sigma \in \mathcal{P}_2([k])} \prod_{\{r,s\}\in \sigma}\E[N_{i}N_j] .
\end{equation}

\subsection{Gaussian Wiener Chaos}

The main references concerning Gaussian spaces and the theory of stochastic integration with respect to an isonormal Gaussian processes are \cite{NourdinPeccatilibro,NualartBook}, where the accent is put on the Malliavin Calculus and the Stein's method for normal approximations; for a survey on the Brownian motion, see \cite{RevuzYor}.\\

Given a (possible separable) real Hilbert space $\mathcal{H}$, an \textit{isonormal Gaussian process} \linebreak$G=\{G(h): h \in \mathcal{H}\}$ is a Gaussian field indexed on $\mathcal{H}$, with covariance given by $\E[G(h_1)G(h_2)] = \langle h_1,h_2\rangle_{\mathcal{H}}$. \\

When $\mathcal{H}= \mathrm{L}^2(\mathbb{R}_+)$\glossary{name={$\mathbb{R}_+$},description={Set of the non-negative real numbers}}, it can be easily shown that the stochastic process $t \mapsto W_t := G(1_{[0,t]})$ defines a \textit{Wiener process}\index{Wiener Process}  (also known as \textit{Brownian motion}\index{Brownian Motion}), up to a continuity property. Recall that a Wiener process is a stochastic process $\{W_t\}_{t\geq 0}$ such that:
\begin{enumerate}
\item $W_t \sim \mathcal{N}(0,t)$, with $W_0 = 0$ a.s.;
\item if $0 \leq t_1 < t_2 < \cdots < t_n$, the increments $W_{t_1}, W_{t_2}-W_{t_1},\dots, W_{t_n}-W_{t_{n-1}}$ are independent;
\item the increments have stationary distribution, that is $W_{t-s} \stackrel{\text{Law}}{=}W_{t}-W_{s}$ for $s < t$;
\item continuity: the path $t \mapsto W_t$ is continuous with probability $1$.
\end{enumerate}
In this case, stochastic integration with respect to $\{W_t\}_{t \geq 0}$ reduces to the classical theory of \textit{multiple Wiener-It\^{o} integrals}, that will be denoted by $I_{d}^{W}(g)$ and will be introduced in the next definition (for a general isonormal process $G$,  multiple stochastic integrals $I_d^{G}(f)$ with respect to $G$ are defined accordingly, and enjoy the same properties \cite{NourdinPeccatilibro}).

\begin{defn}
Let $g$ be a simple function in $\mathrm{L}^{2}(\mathbb{R}_{+}^{q})$,  \textit{vanishing on diagonals} (namely,\linebreak $g(i_1,\dots,i_q)=0$ whenever $i_j=i_k$ for $j\neq k$), say $g= \prod\limits_{j=1}^{q}\mathbf{1}_{(s_j,t_j)}$, with $(s_j,t_j)$ pairwise disjoint intervals of the positive real line. The \textbf{multiple Wiener integral} of order $q$ of $g$ is defined by:
$$I_q^{W}(g) = \big(W_{t_1}-W_{s_1}\big)\cdots (W_{t_q}-W_{s_q}).$$
\end{defn}
By linearity, the last definition can be extended to every function that is a finite linear combination of simple functions vanishing on diagonals, and then, by a density argument, to every symmetric function in $\mathrm{L}^{2}(\mathbb{R}_{+}^{q})$. Moreover, the density in $\mathrm{L}^2(\mathbb{R}_+^{q})$ of the set of  simple functions (also called elementary functions), vanishing on diagonals, implies that every multiple integral can be approximated (in $\mathrm{L}^{2}$-norm) by simple integrals having the form of multilinear homogeneous  polynomials in independent Gaussian random variables:
$$ Q_{\mathbf{N}}(f) = \sum_{i_1,\dots,i_q=1}^{n} f(i_1,\dots,i_q)N_{i_1} \cdots N_{i_q},$$
with $f(i_1,\dots,i_q)=0$ whenever $i_j=i_k$ for $j\neq k$. 

\begin{rmk}
Similarly, random variables living in the \textit{Poisson Wiener Chaos} (resp. \linebreak\textit{Rademacher Chaos}) of order $d\geq 1$ can be represented as homogeneous sums of degree $d$ in independent random variables having the (compensated) Poisson (resp. symmetric Bernoulli) distribution, and with symmetric, vanishing on diagonals coefficients. See \cite{PeccatiTaqqu,PeccatiSole,Peccati_Lachieze}  for stochastic analysis on the Poisson Chaos, and \cite{NourdinPeccatiReinert2} and the references therein for the Rademacher Chaos. \\
\end{rmk}

The product of multiple Gaussian Wiener-It\^{o} integrals of symmetric functions $f \in \mathrm{L}^2(\mathbb{R}_+^{d})$, $g \in \mathrm{L}^2(\mathbb{R}_+^{q})$ linearizes in the sum of integrals of \textit{contraction kernels}, defined as:
\begin{equation}\label{GenContraction}
f \otimes_r g (t_1,\dots,t_{d+q-2r}) = \int_{\mathbb{R}^r}f(t_1,\dots,t_{d-r},s_1,\dots,s_r)g(s_r,\dots,s_1, t_{d-r+1},\dots,t_{d+q-2r})ds_1\cdots ds_r \,
\end{equation}
in the sense that:
$$ I_d^{W}(f) I_q^{W}(g) = \sum_{r=0}^{min\{d,q\}} \binom{d}{r}\binom{q}{r}r! I_{d+q-2r}^{W}(f \widetilde{ \otimes_r} g),$$
where $\widetilde{g}$ denotes the standard symmetrization of the function $g$\glossary{name={$\mathfrak{S}_d$}, description={The symmetric group on $d$ elements}}: 
$$ \widetilde{g}(t_1, \dots,t_q) = \dfrac{1}{q!} \sum_{\sigma \in \mathfrak{S}_q} g(t_{\sigma(1)}, \dots,t_{\sigma(q)}). $$

\begin{rmk}
Note that the notation $f \otimes_r g$ will be used also in the non-commutative probability setting, where it will replace the standard notation $\stackrel{r}{\smallfrown}$ for contractions, introduced in the seminal paper \cite{BianeSpeicher}. In order to facilitate the connection with the commutative setting, we will always use  the notation $\otimes_r$ for contractions of symmetric functions in $\mathrm{L}^2$-spaces, typical of the classical probability literature, whereas  the notation $\stackrel{r}{\smallfrown}$ will be reserved to discrete kernels, that will play a prominent role for the whole discussion.\\
\end{rmk}

The isometry property:
$$\E[I_{d}^{W}(f)I_{q}^{W}(g)] = \sqrt{d!}\sqrt{q!}\, \langle f,g \rangle_{\mathrm{L}^2(\mathbb{R}_+^{d})} \, \delta_{d,q},$$ yields the following useful formula for the fourth moment of $I_d^{W}(f)$:
\begin{align*}
\E[I_d^{W}(f)^4] &= \sum_{r=0}^d \binom{d}{r}^{4}r!^2 \, \|f \widetilde{ \otimes_r} f  \|^2,
\end{align*}
from which it can be deduced that, for $d \geq 2$ and $f \neq 0$:
\begin{equation}\label{Pos4cum}
\chi_4(I_d^{W}(f)) = \E[I_d^{W}(f)^4] - 3\E[I_d^{W}(f)^2]^2 = \sum_{r=1}^{d-1} \binom{d}{r}^{4}r!^2 \, \|f \widetilde{ \otimes_r} f  \|^2 \,>0 
\end{equation}
(see, for instance, \cite[Lemma 5.2.4]{NourdinPeccatilibro}).\\

\textit{Hermite polynomials} $H_n(x)$ are determined by the recurrence relation $H_0=1, H_1(x)=1$, $H_{n+1}(x) = xH_n(x) -nH_{n-1}(x)$ for every $n\geq1$, and correspond to the orthogonal polynomial sequence of the (standard) Gaussian law. 
Sums of finite products of Hermite polynomials in independent Gaussian random variables form the so-called \textit{Wiener homogeneous chaos}, by virtue of the formula:
\begin{equation}
\label{IntegralClassic}
H_{m_1}(N_{i_1})H_{m_{2}}(N_{i_2})\cdots H_{m_k}(N_{i_k}) = I_{m}^{W}\big(e_{i_1}^{\otimes m_1}\otimes e_{i_2}^{\otimes m_2}\otimes \cdots \otimes e_{i_k}^{\otimes m_k}\big),
\end{equation}
provided that $i_1 \neq i_2 \neq \cdots \neq i_k$, $m=m_1+\cdots +m_k$, with $\{e_j\}_{j\geq 1}$ orthonormal basis of $\mathrm{L}^2(\mathbb{R}_+)$ and $\{N_j\}_{j\geq 1}$ the associated Gaussian isonormal process (namely, $\{N_j\}_{j\geq 1}$ is the sequence of  independent standard normal random variables determined by $N_j = I_1^{W}(e_j)$). \\
Moreover, if $N \sim \mathcal{N}(0,1)$, then:
$$ 
\E[H_d(N)^{m}] = \E[\big(I_d^{W}(e^{\otimes d})\big)^{m}] =
\begin{cases}
0 &\text{ if } md \text{ is odd;}\\
|\mathcal{P}_2^{\star}(d^{\otimes m})|& \text{ if } md \text{ is even,}
\end{cases}
$$
where $\mathcal{P}_2^{\star}(d^{\otimes m})$ denotes the set of partitions in $\mathcal{P}_2([dm])$ that respect $\pi^{\star} = d^{\otimes m}. $

\section{Elements of Free Probability Theory}\label{AppendixfreeProb}

This section aims at giving a brief overview of free probability theory, summarizing the tools and the results that will be used throughout the whole discussion. The reader is referred to the fundamental references \cite{Speicher,Voiculescu} for a more detailed presentation.\\

A \textit{free probability space} (also called \textit{non-commutative probability space})  is a pair $(\mathcal{A},\varphi)$\index{Free Probability space}, where $\mathcal{A}$ is a unital algebra over $\mathbb{C}$\glossary{name={$\C$}, description={Field of complex numbers}}, and $\varphi: \mathcal{A} \rightarrow \mathbb{C}$ is a unital linear functional, that is $\varphi(1_{\mathcal{A}})=1$ if $1_{\mathcal{A}}$ denotes the unity of $\mathcal{A}$. When it is not otherwise specified, it will be always assumed that $\mathcal{A}$ is a \textit{$\ast$-algebra}, and that the state $\varphi$  satisfies the following properties:
\begin{enumerate}
\item $\varphi$ is a \textit{trace}: $\varphi(a b)= \varphi(b a)$ for every $a, b \in \mathcal{A}$;
\item $\varphi$ is \textit{positive}: if $a^{\ast}$ denotes the adjoint of an element $a \in \mathcal{A}$, then $\varphi(a a^{\ast}) \geq 0$;
\item $\varphi$ is \textit{faithful}: $\varphi(a a^{\ast}) = 0$ implies that $a=0$.
\end{enumerate}
A \textit{$W^{\ast}$-probability space} is a free probability space $(\mathcal{A},\varphi)$, where $\mathcal{A}$ is a \textit{von Neumann algebra} of operators (that is, an algebra of bounded operators on a Hilbert space that is closed under adjoint and convergence in weak operator topology) and $\varphi$ is a positive, faithful trace.\\

An element $a \in \mathcal{A}$ is \textit{self-adjoint} if $a=a^{\ast}$. If $a$ is self-adjoint, its \textit{spectral radius} is defined as $\rho(a)=\lim\limits_{k \rightarrow \infty}|\varphi(a^{2k}) |^{\frac{1}{2k}}$;  if $\rho(a)$ is finite, then $a$ is called a (bounded) \textit{random variable}. If $a$ is a random variable, the elements of the sequence $\{\varphi(a^{m}): m \in \mathbb{N}\}$ are called the \textit{moments} of $a$: in particular, a random variable $a$ with zero mean ($\varphi(a) = 0$) will be called centered (if a random variable $b$ is not centered, the element $b- \varphi(b)1$ is called the centering of $b$). It is always preferable to work with self-adjoint elements since, for every bounded random variable $a$, there exists a unique real measure $\mu_{a}$, with compact support contained in $[-\rho(a), \rho(a)]$ (called the \textit{law}, or the \textit{distribution} of $a$), that establishes the following integral representation for the moments of $a$:
$$ \varphi(a^{k}) = \int_{\mathbb{R}}x^{k}\mu_{a}(dx).$$
The proofs of the existence and the uniqueness of such measure can be found in \cite[Proposition 3.13]{Speicher} or also \cite[Theorem 2.5.8]{Tao}.  
Thanks to the positivity of the state $\varphi$, the following Cauchy-Schwarz \index{Cauchy-Schwarz inequality} inequality applies: for every $a, b \in \mathcal{A}$, 
$$ |\varphi(a b^{\ast}) |^{2} \leq \varphi(a a^{\ast}) \varphi(b b^{\ast}).$$

The unital subalgebras $\mathcal{A}_1, \dots, \mathcal{A}_n$ of $\mathcal{A}$ are said to be \textit{freely independent} \index{Free Probability space!Free independence} if, for every $k \geq 1$, for every choice of positive integers $i_1,\dots, i_k$ with $i_j \neq i_{j+1}$, and centered random variables $a_{i_j} \in \mathcal{A}_j$,  $\varphi(a_{i_1}a_{i_2}\cdots a_{i_k}) = 0$. Centered random variables $a_1,\dots, a_n$ are said to be \textit{freely independent} if the (unital) subalgebras they generate are freely independent.\\

\subsection{Moment-cumulant formula for non-commutative random variables}

For a non-commutative random variable $Y$, write $\kappa_j(Y)$, $j \in \mathbb{N}$, to indicate the $j$-th {\it free cumulant} of $Y$ (see \cite[Lecture 11]{Speicher}).\\

Given a vector of random variables of the type $(Y_{i_1},...,Y_{i_{n}})$ (with possible repetitions), for every partition $\sigma = \{b_1,...,b_k\} \in \mathcal{NC}([n])$,  the \textit{generalized free joint cumulant} is defined by:
$$
\kappa_\sigma(Y_{i_1},...,Y_{i_{n}}) = \prod_{j=1}^k \kappa(Y_{i_a} : a\in b_j),
$$
where $\kappa(Y_{i_a}: a\in b_j)$ is the \textit{free joint cumulant} of the random variables composing the vector $(Y_{i_a} : a\in b_j)$, ordered according to the order of the elements in $b_j$. Generalized free cumulants satisfy the moment-cumulant formula: 
\begin{equation}\label{MomCumFree}
 \varphi(Y_{i_1}\cdots Y_{i_n}) = \sum_{\sigma \in \mathcal{NC}([n])}\kappa_\sigma(Y_{i_1},...,Y_{i_{n}}),
 \end{equation}
or its inversion on the lattice $\mathcal{NC}([n])$:
$$\kappa(Y_{i_1},\dots, Y_{i_n}) = \sum_{\pi \in \mathcal{NC}([n])}  \mu(\pi, \hat{1}) \prod_{b \in \pi} \varphi(\prod_{j \in b} Y_{i_j})$$
with $\mu(\pi,\hat{1})$ denoting the \textit{M\"{o}bius function} on the interval $[\pi, \hat{1}]$ (see \cite[Lecture 11]{Speicher} for more details). 
In particular, when $Y_{i}=Y$ for all $i$, the generalized \textit{free cumulant} \index{Free cumulant} $\kappa_{\pi}(Y)$ is the mapping $Y\mapsto \kappa_\pi(Y)= \prod\limits_{b\in \pi} \kappa_{|b|}(Y)$ defined for $\pi \in \mathcal{NC}([n])$, verifying the formula:
$$ \varphi(Y^n) = \sum_{\pi \in \mathcal{NC}([n])} \kappa_{\pi}(Y),$$ 
or, equivalently, its inversion:
$$\kappa_n(Y) = \sum_{\pi \in \mathcal{NC}([n])}  \mu(\pi, \hat{1}) \prod_{b \in \pi} \varphi(Y^{|b|}).$$

For instance, the first four cumulants of a random variable $a$ are given by:
\begin{enumerate}
\item  $\kappa_1(a) = \varphi(a)$, called the mean;
\item  $\kappa_2(a) = \varphi(a^2) - \varphi(a)^2$, called the variance;
\item  $\kappa_3(a) = 2\varphi(a)^3 + \varphi(a^3) - 3 \varphi(a^2) \varphi(a)$;
\item $\kappa_4(a) = \varphi(a^4) -2 \varphi(a^2)^2 +10 \varphi(a^2) \varphi(a)^2  - 4 \varphi(a) \varphi(a^3) - 5\varphi(a)^4$. The fourth cumulant $\kappa_4(a)$ is called \textit{free kurtosis}\index{Free cumulant!Free kurtosis}. When $a$ is centered, $\kappa_4(a) = \varphi(a^4)-2\varphi(a^2)^2$.
\end{enumerate}

The most important feature about free cumulants is their behaviour on freely independent arguments:  the \textit{additivity property}, that is, $\kappa_n(X+Y) = \kappa_n(X) + \kappa_n(Y)$ for every $n\geq 1$, whenever $X$ and $Y$ are freely independent. This is a consequence of the vanishing property of mixed cumulants: if $\mathcal{A}_1,\dots,\mathcal{A}_n$ are freely independent unital subalgebras, and $a_i \in \mathcal{A}_i$ is a centered random variable, then $\kappa(a_{i_1},\dots, a_{i_k})=0$ whenever there exist $l\neq j$ with $i_l \neq i_j$ (see \cite[Theorem 11.16]{Speicher} for the precise statement). \\

For a probability measure $\mu$ with compact support and a random variable $X$, the notation $X \sim \mu$ will stand for $X$ having $\mu$ as distribution. Given a sequence $\{Y_n\}_{n\geq 1}$ of random variables on $(\mathcal{A},\varphi)$, it will be said that $Y_n$ \textit{converges in law} (or \textit{in distribution}) to  a random variable $Y$ defined on $(\mathcal{A},\varphi)$ if, for every $m\geq 1$:
$$\lim_{n\rightarrow \infty}\varphi(Y_n^{m}) = \varphi(Y^m). $$

\paragraph*{Main distributions}$ $

Some of the most important distributions that will be encountered in the sequel are listed below.

\begin{itemize}
\item[-] A centered random variable $s \in \mathcal{A}$ is called a \textit{semicircular element} \index{Semicircular random variable} of parameter $\sigma^2 > 0$ (for short, $s \sim \mathcal{S}(0,\sigma^2)$) if its distribution is the \textit{Wigner semicircle law}\index{Semicircular random variable!Wigner Semicircle law} on the interval $[-2\sigma,2\sigma]$ given by:
$$ \mathcal{S}(0,\sigma^2)(dx) = \dfrac{1}{2\pi \sigma^2}\sqrt{4\sigma^2 - x^{2}} dx. $$
If $\sigma=1$, $s$ is called a \textit{standard semicircular random variable}. 

The even moments of a semicircular element $s$ of parameter $\sigma^2$ are given by:
$$ \int_{-2\sigma}^{2\sigma}x^{2m}\mathcal{S}(0,\sigma^2)(dx) = C_m \sigma^{2m},$$
with $\{C_m\}_{m\geq 1}$ denoting the sequence of the \textit{Catalan numbers}\index{Semicircular random variable!Catalan numbers}, namely $C_m = \frac{1}{m+1}\binom{2m}{m}$, while all its  odd moments are equal to zero. Equivalently, $\kappa_1(s)=0$, $\kappa_2(s) = \sigma^{2}$ and $\kappa_n(s) = 0$ for all $n\geq 3$.
\item[-]
A random variable $X(\lambda) \in \mathcal{A}$ is called a \textit{free Poisson element} of parameter $\lambda > 0$ if its distribution has the form:
$$p(\lambda)(dx) = (1-\lambda)\delta_0 + \lambda \tilde{\nu} \qquad  \text{ for } \lambda \leq 1, $$
$$ p(\lambda)(dx) = \dfrac{1}{2\pi x}\sqrt{4\lambda -(x-\lambda-1)^{2}}\; 1_{((1-\sqrt{\lambda})^{2},(1+\sqrt{\lambda})^{2})}(dx), \qquad \text{ for } \lambda > 1$$
(where $\delta_0$ is the Dirac's mass at $0$). 
Denote by $Z(\lambda)$ a centered \textit{free Poisson random variable} \index{Free Poisson distribution}of parameter $\lambda >0$, namely $Z(\lambda) = X(\lambda) - \lambda 1$. As shown in \cite[Proposition 2.4]{NourdinPeccati1}, the moments of $Z(\lambda)$ are given by:
$$ \varphi\big(Z(\lambda)^{m}\big) = \sum_{j=1}^{m}\lambda^{j}R_{m,j},$$
with $R_{m,j}$ counting the number of non-crossing partitions in $\mathcal{NC}([m])$ having no singleton and having exactly $j$ blocks. In particular, if $\lambda =1$, $ \varphi\big(Z(1)^{m}\big) =  R_m$, the $m$-th \textit{Riordan number}\index{Free Poisson distribution!Riordan numbers}, counting the number of non-crossing partitions in $\mathcal{NC}([m])$ having no singletons. Equivalently, $\kappa_1(Z(\lambda))=0$ and $\kappa_n(Z(\lambda)) = \lambda$ for all $n\geq 2$.

\item[-] The \textit{free Poisson distribution} with integer parameter $p$ and the standard semicircle law correspond each other via the second Chebyshev polynomial. \index{Chebyshev polynomials}Indeed, if $S \sim \mathcal{S}(0,1)$, then $U_2(S) \stackrel{\text{Law}}{=} Z(1)$, and more generally, $Z(p) \stackrel{\text{Law}}{=} \sum\limits_{j=1}^{p}(S_j^{2}-1)$, with $S_1,\dots,S_p$ freely independent standard semicircular elements (see \cite{NourdinPeccati1}). 

\item[-] The \textit{free symmetric Bernoulli law} (or \textit{free Rademacher law}) \index{Rademacher law, free} is the probability measure $ \mu = \dfrac{1}{2}\delta_1 + \dfrac{1}{2}\delta_{-1}$, with $\delta_{x}$ denoting the Dirac's delta at the point $x$. Then, if $X \sim \mu$, $\varphi(X^{m}) = 1$ for every even integer $m$, and $\varphi(X^{n}) = 0$ for every odd integer $n$.

\item[-] A random variable $\mathcal{T}$ is said to be \textit{Tetilla distributed} \index{Tetilla law}if $\mathcal{T} \stackrel{\text{Law}}{=} \frac{1}{\sqrt{2}}(s_1 s_2 + s_2 s_1)$, the standardized commutator between two freely independent standard semicircular random variables $s_1, s_2$. It can be shown that, for every $m \geq 1$:
$$ \kappa_m(\mathcal{T})=
\begin{cases}
2^{1-\frac{m}{2}} &\text{ if } m \text{ is even}, \\
0 & \text{ otherwise},
\end{cases}$$
or, equivalently,
$$ \varphi(\mathcal{T}^m)=
\begin{cases}
\dfrac{1}{n\;2^n}\sum\limits_{k=1}^n 2^k\binom{2n}{k-1}\binom{n}{k} &\text{ if } m = 2n, \\
0 &\text{ otherwise}
\end{cases}
$$
(see \cite[Lemma 2.6 and Proposition 2.8]{NourdinDeya2}).
\end{itemize}

Moreover, if $s_1,\dots, s_n$ are standard semicircular elements, with covariance $\varphi(s_i s_j) = C_{i,j}$ such that the matrix $C = (C_{i,j})$ is symmetric and positive definite, the joint moments of $s_1,\dots,s_n$ are completely determined by $C$ according to the following Wick-type formula\index{Wick Formula}: for every $m\geq 1$ and every choice of positive integers $i_1,\dots,i_m \in [n]$,
$$ \varphi(s_{i_1}s_{i_2}\cdots s_{i_m}) = \sum_{\pi \in \mathcal{NC}_2([m])}\prod_{\{r,p\}\in \pi}\varphi(s_{i_r}s_{i_p}).$$

\subsection{Wigner Chaos}

\textit{Wigner Chaos} \index{Wigner Chaos} is the non-commutative counterpart to the Gaussian Wiener chaos, and corresponds to the theory of stochastic integration with respect to a free Brownian motion, that has been first developed in \cite{BianeSpeicher} and then further investigated in \cite{KempNourdinPeccatiSpeicher}.  Note that the notation used in the sequel is largely inspired from the set up of \cite{BianeSpeicher}.\\

For every $p: 1 \leq p < \infty$, denote by $\mathrm{L}^{p}(\mathcal{A},\varphi)$ the space obtained by completion of $\mathcal{A}$ with respect to the norm $\| a\|_p = \varphi(|a|^{p})^{\frac{1}{p}}$, with $|a|$ such that $|a|^{2} = a^{\ast}a$. 

If $\{\mathcal{A}_t\}_{t \geq 0}$ denotes a filtration of  unital subalgebras of $\mathcal{A}$ (namely, $\{\mathcal{A}_t\}_{t \geq 0}$ is an increasing sequence of subalgebras: $\mathcal{A}_s \subset \mathcal{A}_t$ for $s \leq t$), a \textit{free Brownian motion} \index{Free Brownian motion} is as a collection $S = \{S(t)\}_{t \geq 0}$ of self-adjoint operators in $(\mathcal{A},\varphi)$ such that:
\begin{enumerate}
\item for every $t \geq 0$, $S(t)\sim \mathcal{S}(0,t)$ and $S(t) \in \mathcal{A}_t$;
\item (\textit{stationary increments}) for every $0 \leq t_1 < t_2$, the increment $S(t_2) - S(t_1)$ has the same distribution as $S(t_2-t_1)$;
\item (\textit{freely independent increments}) for every $0 \leq t_1 < t_2$, the increment $S(t_2) - S(t_1)$ is freely independent of $\mathcal{A}_{t_1}$.
\end{enumerate}

Let $q \geq 2$ be an integer. A function $f \in \mathrm{L}^{2}(\mathbb{R}_{+}^{q})$ is said to be \textit{mirror symmetric} if 
$$f(t_1, t_{2},\dots,t_q) = f(t_q,\dots,t_2,t_1).$$
More generally, a complex valued kernel $f$ is mirror symmetric if  $f(t_1, t_{2},\dots,t_q) = \overline{f(t_q,\dots,t_2,t_1)}$, for every $t_1,\dots,t_q \in \mathbb{R}_+$, where $\overline{f(t_q,\dots,t_2,t_1)}$ denotes the complex conjugate of $f(t_q,\dots,t_2,t_1)$. \\

Given a free Brownian motion $S$ on $(\mathcal{A},\varphi)$, the construction of the Wigner stochastic integral of a function $f \in \mathrm{L}^{2}(\mathbb{R}_{+}^{q})$, denoted by $I_q^{S}(f)$ (that is, the stochastic integral with respect to a free Brownian motion) requires exactly the same steps as those included in the definition of the classic Wiener-It\^{o} integrals with respect to a (classical) Brownian motion. 
\begin{defn}
Let $g$ be a simple function in $\mathrm{L}^{2}(\mathbb{R}_{+}^{q})$, vanishing on diagonals, namely \linebreak $g= \prod\limits_{j=1}^{q}\mathbf{1}_{(a_j,b_j)}$, with $(a_j,b_j)$ pairwise disjoint intervals of the positive real line. The \textbf{multiple Wigner integral} of $g$, of order $q$, is defined as:
$$I_q^{S}(g) = \big(S(b_1)-S(a_1)\big)\cdots (S(b_q)-S(a_q)).$$
\end{defn}
By linearity, the last definition can be extended to every function that is a finite linear combination of simple functions vanishing on diagonals. As for the Wiener stochastic integration, for such functions the following isometric relation holds:
$$ \langle I_q^{S}(f), I_q^{S}(g)\rangle_{\mathrm{L}^{2}(\mathcal{A},\varphi)} = \langle f,g\rangle_{L^{2}(\mathbb{R}_+^{q})},$$
that leads to the definition of the Wigner integral of any $f \in \mathrm{L}^{2}(\mathbb{R}_{+}^{q})$ by a density argument. Therefore, a special role is played by Wigner integrals of simple functions, having the form of multilinear homogeneous polynomials in freely independent standard semicircular random variables:
$$ Q_{\mathbf{S}}(f)= \sum_{i_1,\dots,i_q=1}^n f(i_1,\dots,i_q)S_{i_1}\cdots S_{i_q}, $$
with mirror symmetric coefficients such that $f(i_1,\dots,i_q) = 0$ if $i_j = i_l$ for $i\neq l$. Moreover, $I_q^{S}(g)$ is self-adjoint if and only if $f$ is mirror symmetric.\\

\textit{Chebyshev polynomials}\index{Chebyshev polynomials} (of the second kind) are defined via the recurrence relation $U_{0}(x) = 1$, $U_1(x) = x$, and $U_{m+1}(x) = xU_m(x) - U_{m-1}(x) \; \text{ for every } m \geq 1$, and they constitute  the unique family of polynomials  that are orthogonal  with respect to the standard semicircle Wigner law $s(dx) = \dfrac{1}{2\pi}\sqrt{4-x^2}(dx)$ on the interval $[-2,2]$ (for more details, see \cite{Ans3, Chihara}). In the framework of the Wigner stochastic integration, this family of polynomials plays the same role as the Hermite polynomials for the multiple integrals of Wiener-It\^{o} type (see e.g. \cite[Chapter 2]{NourdinPeccatilibro}).

In particular,  since the free Brownian motion admits a representation in terms of operators on the Fock space associated with a Hilbert space $\mathcal{H}$ (for instance, $\mathcal{H} = L^{2}(\mathbb{R}_{+})$), for every $k\geq 1$ and for every choice of positive integers $m_1, \dots,m_k$, it can be shown that (see \cite{Ans3},\cite{BianeSpeicher}):
\begin{equation}
\label{Integral}
U_{m_1}(S_{i_1})U_{m_{2}}(S_{i_2})\cdots U_{m_k}(S_{i_k}) = I_{m}^{S}\big(e_{i_1}^{\otimes m_1}\otimes e_{i_2}^{\otimes m_2}\otimes \cdots \otimes e_{i_k}^{\otimes m_k}\big),
\end{equation}
provided that $i_1 \neq i_2 \neq \cdots \neq i_k$, $m=m_1+\cdots +m_k$. Here $\{e_j\}_{j\geq 1}$ is an orthonormal basis of $\mathcal{H}$ and $\{S_j\}_{j\geq 1}$ denotes the associated free Brownian motion, that is, $\{S_j\}_{j\geq 1}$ is the sequence of freely independent standard semicircular elements determined by $S_j = I_1^{S}(e_j)$. Moreover, for every $m \geq1$:
$$ 
\varphi\big(U_d(S)^{m}\big) = \varphi\big((I_d^{S}(e^{\otimes d}))^{m}\big) =
\begin{cases}
0 &\text{ if } md \text{ is odd;}\\
|\mathcal{NC}_2^{\star}(d^{\otimes m})|& \text{ if } md \text{ is even,}
\end{cases}
$$
where $S \sim \mathcal{S}(0,1)$ and $\mathcal{NC}_2^{\star}(d^{\otimes m})$ denotes the set of the partitions in $\mathcal{NC}_2([dm])$ that respect the partition $\pi^{\star} = d^{\otimes m}$.\\

If $f \in \mathrm{L}^{2}(\mathbb{R}^{d}), g \in \mathrm{L}^{2}(\mathbb{R}^{q})$, the product of the corresponding multiple Wigner integrals can be computed via the \textit{multiplication formula}:
\begin{equation}\label{MultFormulaFree}
I_d^{S}(f) I_q^{S}(g) = \sum_{r=0}^{\min(d,q)} I_{d+q-2r}^{S}(f \otimes_r g),
\end{equation}
where the contraction $f\otimes_r g \in \mathrm{L}^2(\mathbb{R}^{d+q-2r})$ of mirror symmetric kernels $f \in \mathrm{L}^2(\mathbb{R}^{d}), g \in \mathrm{L}^2(\mathbb{R}^{q})$ is defined as in \eqref{GenContraction}. 
Note that the notation $f\otimes_r g$ here adopted  for contractions,  corresponds to the notation $f \stackrel{r}{\smallfrown} g$ used in \cite{BianeSpeicher}. Here the notation $\stackrel{r}{\smallfrown}$ will be used exclusively for contractions of discrete kernels $f:[n]^d \rightarrow \mathbb{R}$, as introduced in Definition \ref{starcontraction} in Part \ref{Invariance}, via \eqref{contraction}.\\


As for Gaussian Wiener integrals, if $f \neq 0$, the isometry property entails the following formula for the fourth moment of $I_d^{S}(f)$:
$$ \varphi(I_d^{S}(f)^4) = \sum_{r=0}^d \| f\otimes_r f \|^{r},$$
from which the positiveness of the fourth cumulant can be deduced (see \cite[Corollary 1.7]{KempNourdinPeccatiSpeicher}):
\begin{equation}\label{Pos4cumFree}
\kappa_4(I_d^{S}(f)) = \varphi(I_d^{S}(f)^4) - 2\varphi(I_d^{S}(f)^2)^2 = \sum_{r=1}^{d-1}\| f \otimes_r f \|^2 >0\,. 
\end{equation}

$ $\\
\textit{Free Charlier polynomials} \index{Free Charlier polynomials} of parameter $t >0$ have been introduced in \cite{Ans3} as those polynomials satisfying the recurrence relation: $C_{0,0}(x,t) \in \mathbb{R}$,
$$ C_{0,1}(x,t) = x C_{0,0}(x,t), \qquad C_{0,m+1}(x,t) = (x-1)C_{0,m}(x,t) - t C_{0,m-1}(x,t) \quad \forall \, m \geq 1,$$
which make the sequence $\{C_{0,m}(x,t)\}_{m\geq 1}$ orthogonal with respect to the probability distribution of the free Poisson law $Z(t)$ of parameter $t \in \mathbb{R}_{+}$.

Setting $C_{0,0}=1$, the following formula relating Free Charlier and Chebyshev polynomials  can be proved by a simple induction argument:
\begin{equation}
\label{CharlierToCheby1}
C_{0,k}\big(U_{2}(x),1\big) = U_{2k}(x,1) \qquad \forall k \geq 1.
\end{equation}

Identity \eqref{CharlierToCheby1} provides the free analogue of the following correspondence between (generalized) Laguerre polynomials $L_n^{(\alpha)}(x)$ and Hermite polynomials $H_n(x)$:
$$ H_{2n}(x) = (-1)^{n}2^{2n}n! L_{n}^{(-\frac{1}{2})}(x^2)$$
(see, for instance, \cite{Abramovitz_Stegun}). Moreover, identity (\ref{CharlierToCheby1}) entails that $U_{2k}(S) \stackrel{\text{Law}}{=} C_{0,k}\big(Z(1),1\big)$,  where $S$ denotes a random variable with the standard semicircular distribution and $Z(1)$ a random variable with the free Poisson distribution of parameter $1$.

%

\section{Orthogonal polynomials}\label{Pre_Orth}

It is out of the scope of this section to provide a self-contained short preface about orthogonal polynomials, as the involved mathematics is very rich and constantly updated. This is why here only some basic facts will be recorded: any other result hereafter quoted can be traced to the fundamental references  \cite{Chihara,Szego}. \\

Consider a linear functional $\EE:\C[x]\rightarrow \C$, with $\EE[1] \neq 0$. A sequence of polynomials $\{p_n(x)\}_{n\geq 0}$ in $\C[x]$, with $p_n(x)$ of degree $n$, is called a sequence of orthogonal polynomials (for short, OPs) for $\EE$ if $ \EE[p_{k}(x)p_n(x)] = 0$ for all $k\neq n$, and $\EE[p_n(x)^2]\neq 0$. Given the sequence of its moments $\{\EE[x^n]\}_{n\geq 0}$, there are plenty of results concerning the so-called \textit{Hamburger moment problem} for $\EE$, consisting in determining the existence (and the uniqueness) of a real measure $\mu$ such that the following integral representation holds:
$$ \EE[x^n] = \int_{\mathbb{R}}t^n \mu(t)dt. $$
When seeking for a measure with support included in the positive half-line $(0,\infty)$, the problem of moments is referred to as \textit{Stieltjes' moment problem}, while for compactly supported measures one speaks about the \textit{Hausdorff's moment problem}. 

If a sequence of orthogonal polynomials $p_n(x)$ exists for $\EE$, then it is uniquely determined up to a non-zero multiplicative factor, in the sense that if $\{q_n(x)\}_{n\geq 0}$ is another orthogonal sequence for $\EE$, then there exists non-zero constants $c_n \in \C$ such that $q_n(x) =c_n p_n(x)$ for every $n\geq 0$. 

A necessary and sufficient condition for the existence of an OPs for $\EE$ is given by the non-vanishing of all the Hankel determinants $\det( \EE[x^{i+j}])_{i,j=0,\dots,n}$ (see, for instance, \cite[Theorem 3.1]{Chihara}). Orthogonal polynomials $p_n(x)$'s are a basis of the vector space $\C[x]$, in the sense that every polynomial $\pi(x) \in \C[x]$ of degree $n$ can be written as
$$ \pi(x)= \sum_{k=0}^n d_k p_{k}(x), \text{ with } d_k = \dfrac{\EE[\pi(x)p_k(x)]}{\EE[p_k(x)^2]}.$$
This property entails that every sequence of orthogonal polynomials can be characterized by a $3$-terms recurrence relation: given $p_0(x) \in \C$ and $p_1(x)= a x +b$,  there exist two sequences $\{\alpha_n\}_{n\geq 1}, \{\beta_n\}_{n\geq 1}$ (usually called Jacobi-Szego parameters), such that
$$p_{n+1}(x) = (x-\alpha_{n+1})p_n(x) - \beta_{n+1}p_{n-1}(x)$$
for every $n\geq 1$: this result  usually goes under the name of Favard's Theorem. Another nice feature of orthogonal polynomials is that $p_n(x)$ has simple real roots for every $n\geq 1$.

\mainmatter
\thispagestyle{empty}
\part{A multidimensional invariance principle in the free probability setting}\label{Invariance}

\chapter*{Synopsis}
\addcontentsline{toc}{chapter}{Synopsis}

The findings exposed in the present part are taken from \cite{Simone}. \\

Let $d\geq 2$ be an integer, and let $$Q_{\bs{x}}(f_n) = \sum\limits_{i_1,\dots,i_d=1}^{n}f_n(i_1,\dots,i_d)x_{i_1}\cdots x_{i_d}$$ be a homogeneous polynomial of degree $d$ in non-commuting variables $\bs{x} = \{x_i\}_{i\geq 1}$;  suitable assumptions on the coefficient $f_n:[n]^{d}\rightarrow \mathbb{R}$ will be required (see Definition \ref{admissible_Univ}). \\

The goal of this part is to develop a new collection of techniques allowing one to compare the distribution of vectors of homogeneous sums in freely independent random variables on a fixed non-commutative probability space $(\mathcal{A},\varphi)$. Moreover, new universality results for $U$-statistics in free probability spaces will be derived, with explicit comparisons with analogous phenomena in the classical setting. In order to accomplish this task, we shall focus on a family of $U$-statistics based on Chebyshev polynomials, that we shall name \textit{Chebyshev sums} (see Definition \ref{SumCheby2}). In its simplest form, a Chebyshev sum is a polynomial of the type: $$Q_{\bs{x}}^{(h)}(f) = \sum_{i_1,\dots,i_d=1}^{n} f(i_1,\dots,i_d) U_h(x_{i_1})\cdots U_h(x_{i_d}),$$ where $U_h(x)$ denotes the $h$-th Chebyshev polynomial (of the second kind) on the interval $[-2,2]$. \\

The strategy of the proof  expands ideas introduced in \cite{NourdinPeccatiReinert}. Let $\mathbf{W}_1, \mathbf{W}_2$ be freely independent sequences of freely independent random variables. For a given $m\geq 1$, in order to estimate the discrepancy between the laws of $(Q_{\mathbf{W}_1}(f_n^{(1)}),\dots,Q_{\mathbf{W}_1}(f_n^{(m)}))$ and $(Q_{\mathbf{W}_2}(f_n^{(1)}),\dots,Q_{\mathbf{W}_2}(f_n^{(m)}))$, first assess the difference between the joint moments of  $(Q_{\mathbf{W}_i}(f_n^{(1)}),\dots,Q_{\mathbf{W}_i}(f_n^{(m)}))$ and those of a vector of Chebyshev sums $(Q_{\bs{S}}^{(\bs{h})}(f_n^{(1)}),\dots,Q_{\bs{S}}^{(\bs{h})}(f_n^{(m)}))$ for $i=1,2$ (where $\bs{S}=\{S_i\}_{i\geq 1}$ is a sequence of freely independent standard semicircular random variables). Finally,  apply the triangle inequality. In this way, it is sufficient to focus on the proximity in law between a vector of Chebyshev sums and a vector of homogeneous sums $(Q_{\mathbf{W}_1}(f_n^{(1)}),\dots,Q_{\mathbf{W}_1}(f_n^{(m)}))$.\\
The technique here adopted is a generalized Lindberg method relying on \textit{influence functions}, which has been developed in \cite{Mossel} and then successfully applied in \cite{NourdinPeccatiReinert} to derive the (multidimensional) universality of the homogeneous Gaussian Wiener chaos. Afterwards, this version of the Lindberg method has been adapted in \cite{NourdinDeya} to fit the non-commutative setting in the unidimensional case. As a consequence, it has been established that homogeneous sums in semicircular entries and with symmetric coefficients, enjoy the following property: for $d\geq 2$, $Q_{\bs{S}}(f_n) \stackrel{\text{ Law }}{\longrightarrow} \mathcal{S}(0,1)$ implies that $Q_{\bs{X}}(f_n) \stackrel{\text{ Law }}{\longrightarrow} \mathcal{S}(0,1)$ for any other sequence $\bs{X}=\{X_i\}_{i\geq 1}$ of freely independent centered random variables having unit variance. For short, this property is customarily referred to by saying that the semicircle law is \textit{universal} for semicircular approximations of homogeneous sums.\\

The questions that will be tackled in the sequel can be summarized as follows:
\begin{enumerate}
\item are there other ``universal laws''  for semicircular approximations of homogeneous sums? In other words, is it possible to find another sequence of freely independent r.v.'s $\bs{Y}=\{Y_i\}_{i\geq 1}$ such that $Q_{\bs{Y}}(f_n)  \stackrel{\text{ Law }}{\longrightarrow} \mathcal{S}(0,1)$  implies that $Q_{\bs{X}}(f_n)$ has the same asymptotic behaviour for any other sequence $\bs{X}=\{X_i\}_{i\geq 1}$ of freely independent random variables?
\item Is it possible to prove a similar universality result if the target limit law is the free Poisson distribution (or other laws)?
\item If the answers to the previous questions are positive, is it possible to extend these results to a general multidimensional setting, as done in the classical case in \cite{NourdinPeccatiReinert}?
\end{enumerate}

The invariance principle achieved via Theorem \ref{Multiinvariance2} provides a positive answer to all the three questions in a unified way, supplied with some technical results. Therefore, Theorem \ref{Multiinvariance2} represents the first multidimensional universality principle for homogeneous sums proved in a free setting.\\

As to non-central convergence, it is worth to remark that so far, in the classical setting, the only law that is known to be universal for Gamma approximations of homogeneous sums is the Gaussian distribution \cite{NourdinPeccatiReinert}. For the non-commutative counterpart to the Gamma law, that is, the free Poisson law, the results presented in the present part show a new infinite collection of universal distributions with respect to  free Poisson approximations. The same technique can be transferred to the commutative setting, in order to provide other instances of universal laws for Gamma approximations of homogeneous sums. \\

The structure of the present part can be summarized as follows:
\begin{enumerate}
\item in Chapter \ref{AgenInvPrinc}, the general framework is described. In particular, the notation is fixed while the basic definitions and preliminaries are given. Then, in Section \ref{MainResult}, our main result, that is, a multidimensional version of the Lindberg principle in a free setting, is stated via Theorem \ref{Multiinvariance2};
\item Chapter \ref{Chapter_Universality} exploits the results presented in Chapter \ref{AgenInvPrinc} to directly answer the above questions; this presentation is supplied with several remarks and examples. Finally,  the commutative counterpart is discussed: in particular, Proposition \ref{HermiteSum} extends the universality of Gaussian homogeneous sums to the universality of \textit{Hermite sums}. 
\end{enumerate}

At the beginning of each chapter, an additional overview  about the contents therein discussed will be provided.

\section*{Bibliographic comments}
The Lindberg method for the Central Limit Theorem has been established in \cite{Lindberg}. 
Influence functions were first employed to describe universal asymptotic behaviour of multilinear polynomials in \cite{Rotar}, and they have gained renewed interest thanks to the paper \cite{Mossel}, where a general invariance principle for multilinear homogeneous polynomials (based on the Lindberg method) is provided with explicit bounds depending on the maximum of the influence functions. Thanks to this technique, several applications have been  developed  in terms of influence functions: in particular, the authors solved the so called ``Major is stablest'' conjecture, from theoretical computer science, and the ``It ain't over until it's over'' conjecture arising in the economic theory of social choice. A companion paper was later provided (\cite{Mossel2}), where the multidimensional version of the invariance principle can be found in the case one of the sequences under consideration lives in a discrete probability space: afterwards, it has been extended to the case where one of the sequences is a Gaussian system in \cite{NourdinPeccatiReinert}.\\

The invariance principle in \cite{Mossel} has then been fruitfully combined with the Fourth Moment Theorem from \cite{NualartPeccati} to prove that the Gaussian distribution satisfies a universality phenomenon for homogeneous sums with respect to  Gaussian and Gamma approximation (see \cite{NourdinPeccatiReinert} for both the unidimensional and the multidimensional frameworks). Similar results for central convergence have been established for the discrete Poisson Chaos in \cite{PeccatiZheng2} and \cite{PeccatiZheng1}. 

In \cite{Peccati2}, as an application of the universality of the Gaussian Wiener Chaos \cite{NourdinPeccatiReinert}, the authors provide a multidimensional CLT for spectral moments of non-Hermitian random matrices with real-valued i.i.d. entries. See also \cite[Chapter 11]{NourdinPeccatilibro} and the bibliographic comments therein for a survey of the existing literature on the topic. \\

In \cite{NourdinDeya}, the aforementioned invariance principle based on influence functions was adapted to fit the framework of homogeneous polynomials in freely independent random variables living in a non-commutative probability space: as a consequence, the authors established the free counterpart, in dimension 1, to \cite{NourdinPeccatiReinert}, namely, the universality of the Wigner chaos for semicircular approximations.\\

Several other generalizations of the Lindberg method have been developed to provide, for instance, the universality of the circular law for i.i.d. random matrices and for their least singular value (\cite{TaoPaper}), or an invariance principle for smooth functionals of independent and weakly dependent random variables (\cite{ChatterjeeLindeberg}). See also \cite{Kargin} for a different proof of the Central Limit Theorem for non-commutative random variables, based on the Lindberg method and holding under a weaker assumption than the usual free independence of the summands.


\chapter{A general invariance principle}\label{AgenInvPrinc}

\section{Overview, notation and preliminaries}

In the sequel, $\bs{x}=\{x_i\}_{i \geq 1}$ will denote a sequence of non-commutative variables. As anticipated in the introduction, a crucial role will  be played by the Chebyshev polynomials.

\begin{defn}\label{ChebyPoly}
The polynomials $\{U_n(x)\}_{n\geq 0}$ defined via the recurrence relation $U_{0}(x) = 1$, $U_1(x) = x$, and $U_{m+1}(x) = xU_m(x) - U_{m-1}(x)  \text{ for every } m \geq 1$, are called \textbf{Chebyshev polynomials}\index{Chebyshev polynomials} (of the second kind): they constitute the unique  family of polynomials that is orthogonal with respect to the Wigner semicircle law \index{Wigner semicircle law} $$s(dx) = \dfrac{1}{2\pi}\sqrt{4-x^2}(dx)$$ on the interval $[-2,2]$, where uniqueness is meant up to multiplicative coefficients. 
\end{defn} 
For instance, $U_1(x)=x, U_2(x) = x^2 -1$, $U_3(x) = x^3 - 2x$ (for more details, see \cite{Ans3, Chihara}).

\begin{defn}\label{admissible_Univ}
Let $d \geq 1$ be an integer. For every $n \in \mathbb{N}$, a function $f:[n]^{d}\rightarrow \mathbb{R}$ is called an \textbf{admissible kernel} if it verifies the following properties: 
\begin{itemize}
\item[(i)] \textrm{mirror symmetry}\index{Mirror symmetry}: $f(i_1, i_2,\dots,i_d) = f(i_d, \dots,i_2, i_1)$ for every $i_1,\dots, i_d \in [n]$;
\item[(ii)] \textrm{vanishing on diagonals}: $f(i_1,\dots,i_d) = 0$ whenever $i_j = i_k$ for $j\neq k$;
\item[(iii)] $f$ has unit variance: \begin{equation}
\label{variance}
 \|f\|^{2}:= \sum\limits_{i_1,\dots,i_d=1}^{n}f(i_1,\dots,i_d)^2 = 1.
\end{equation}
\end{itemize}
\end{defn}

\begin{defn}
Let $\bs{h}=(h_1,\dots,h_d)$ be a vector of positive integers such that $h_i = h_{d-i+1}$ for every $i=1,\dots, \lfloor \frac{d}{2}\rfloor$ (if $d\geq 2$). If $f$ is an admissible kernel, the \textbf{Chebyshev sum} \index{Chebyshev polynomials!Chebyshev sum} of orders $\bs{h}= (h_1,\dots,h_d)$ and kernel $f$ is defined by the formula:
\begin{equation}
\label{SumCheby2}
Q_{\bs{x}}^{(\bs{h})}(f) = \sum_{i_1,\dots,i_d=1}^{n}f(i_1,\dots,i_d)U_{h_1}(x_{i_1})\cdots U_{h_{d}}(x_{i_d}).\\
\end{equation}
\end{defn}

The simplest example of Chebyshev sums are multilinear homogeneous  polynomials of degree $d$, occurring when $h_i=1$ for every $i=1,\dots,d$: 
\begin{equation}
\label{QN}
Q_{\bs{x}}(f) = \sum_{i_1,\dots, i_d=1}^{n}f(i_1,\dots,i_d)x_{i_1}\cdots x_{i_d}.
\end{equation}

Henceforth, $(\mathcal{A},\varphi)$ will denote a fixed  $W^{\star}$-probability space, that is, $\mathcal{A}$ is a von-Neumann algebra of operators, and $\varphi$ is a tracial positive faithful state on it. We shall say that a random variable  $Y$ satisfies Assumption {\bf (1)} if it is centered and has unit variance,  namely if $\varphi(Y)=0$ and $\varphi(Y^2)=1$. \\

Note that, if $\bs{X} = \{X_i\}_{i\geq 1}$ denotes a sequence of freely independent random variables, the conditions $h_i = h_{d-i+1}$ for $i=1,\dots, \lfloor \frac{d}{2}\rfloor$ if $d\geq 2$, ensure that $Q_{\bs{X}}^{(\bs{h})}(f)$ is a self-adjoint element in $\mathcal{A}$, and hence a properly defined random variable whose law is uniquely determined by the sequence of its moments. Indeed, compactly supported measures are uniquely determined by the sequence of their moments by Weierstrass's Theorem. \\

Contraction operators between kernels of multiple stochastic integrals play an important role in  \textit{fourth moment}-type statements and multiplication formulae (see \cite[Proposition 1.25]{KempNourdinPeccatiSpeicher}). In the next definition, we will introduce contractions of discrete kernels. As shown in the subsequent discussion, discrete contractions may be used to describe the contractions operators $\otimes_r$ introduced via formula \eqref{GenContraction}. It is worth to stress again that here the notation $\stackrel{r}{\smallfrown}$ is used only for contractions of discrete kernels, while in \cite{BianeSpeicher} it corresponds to the contractions here denoted with $\otimes_r$, and defined in \eqref{GenContraction}.

\begin{defn}
\label{starcontraction}
For $n, d, p \in \mathbb{N}$, consider the functions $f:[n]^{d} \rightarrow \mathbb{R}$ and $g:[n]^{p} \rightarrow \mathbb{R}$. For every $r=1,\dots, \min(d,p)$, the (discrete) \textbf{star contraction} \index{Contraction} $f \star_{r}^{r-1} g : [n]^{d + p - 2r +1} \rightarrow \mathbb{R}$ is given by:
\begin{align*}
f \star_{r}^{r-1} g (t_1&,\dots, t_{d-r},\gamma, s_1,\dots, s_{p-r}) = \\
&=\sum\limits_{i_1,\dots, i_{r-1}=1}^{n}f(t_1,\dots, t_{d-r},\gamma, i_1,\dots, i_{r-1})g(i_{r-1},\dots, i_{1}, \gamma,s_1,\dots,s_{p-r}).
\end{align*}
For every $q=0,\dots, \min(d,p)$,  the \textbf{contraction of order $q$} is the function $f\stackrel{q}{\smallfrown}g:[n]^{d+p-2q}\rightarrow \mathbb{R}$, defined as:
\begin{align*}
f \stackrel{q}{\smallfrown}&\, g \;(t_1,\dots, t_{d-q},s_1,\dots,s_{p-q}) = \numberthis \label{contraction} \\
&= \sum\limits_{i_{1},\dots,i_q=1}^{n}f(t_1,\dots,t_{d-q},i_1,\dots,i_q) g(i_q,\dots,i_1,s_1,\dots,s_{p-q}) \,.\notag 
\end{align*}
\end{defn}

Contractions can be defined over tensor powers $\mathcal{H}^{\otimes k}$ of any (possibly separable) real Hilbert space $\mathcal{H}$, extending by linearity the following definition: for every $r=1,\dots, \min(d,p)$,
\begin{equation}
\big(e_{i_1}\otimes \cdots \otimes e_{i_d} \big) \otimes_r \big(e_{j_1}\otimes \cdots \otimes e_{j_p}\big) = \prod_{l=0}^{r-1}\langle e_{i_{d-l}}, e_{j_{l+1}}\rangle_{\mathcal{H}} e_{i_1}\otimes \cdots \otimes e_{i_{d-r}}\otimes e_{j_{r+1}}\otimes \cdots \otimes e_{j_{p}},
\end{equation}
where $\langle \cdot, \cdot \rangle_{\mathcal{H}}$ denotes the inner product on $\mathcal{H}$ (see, for instance, \cite[Appendix B]{NourdinPeccatilibro}). In particular:
\begin{align*}
\big(e_{i_1}\otimes \cdots \otimes e_{i_d} \big) \otimes_d \big(e_{j_1}\otimes \cdots \otimes e_{j_d}\big) &= \langle e_{i_1}\otimes \cdots \otimes e_{i_d}, e_{j_d}\otimes \cdots \otimes e_{j_1}\rangle_{\mathcal{H}^{\otimes d}} \\
&= \prod_{l=1}^{d}\langle e_{i_l}, e_{j_{d-l+1}} \rangle_{\mathcal{H}} \;,
\end{align*}
where $\langle \cdot,\cdot\rangle_{\mathcal{H}^{\otimes d}}$ denotes the inner product on $\mathcal{H}^{\otimes d}$ induced by $\langle \cdot, \cdot\rangle_{\mathcal{H}}$. Therefore, if $f \in \mathcal{H}^{\otimes p}$ and $g \in \mathcal{H}^{\otimes d}$, then $f \otimes_r g \in \mathcal{H}^{\otimes p+d-2r}$. 

Discrete contractions as introduced in Definition \ref{starcontraction} are related to the contractions defined via formula \eqref{GenContraction} as follows.
Given two discrete kernels $f_1:[n]^d \rightarrow \mathbb{R}, f_2:[n]^p \rightarrow \mathbb{R}$, set:  
$$k_1:= \sum\limits_{i_1,\dots,i_d=1}^n f_1(i_1,\dots,i_d) e_{i_1}\otimes \cdots \otimes e_{i_d} \in \mathcal{H}^{\otimes d},$$
$$k_2:= \sum\limits_{j_1,\dots,j_p=1}^n f_2(j_1,\dots,j_p) e_{j_1}\otimes \cdots \otimes e_{j_p} \in \mathcal{H}^{\otimes p}. $$
Then, for every $r= 0, \dots, \min(d,p)$:
$$ k_1 \otimes_r \, k_2  = \sum_{\substack{i_1,\dots,i_{d-r}\in [n] \\ j_1,\dots,j_{p-r} \in [n]}} f_1 \stackrel{r}{\smallfrown} f_2(i_1,\dots,i_{d-r},j_1,\dots,j_{p-r}) e_{i_1}\otimes \cdots \otimes e_{i_{d-r}}\otimes e_{j_1}\otimes \cdots \otimes e_{j_{p-r}}.  $$ 

\begin{exm}
If $\{e_i\}_{i\geq 1}$ is an orthonormal sequence of $\mathcal{H}$, then:
\begin{enumerate}
\item $\big(e_1 \otimes e_2 \otimes e_{3} \big) \otimes_2 \big(e_3 \otimes e_2 \otimes e_1 \big) = \langle e_3,e_3\rangle_{\mathcal{H}} \langle  e_2,e_2 \rangle_{\mathcal{H}} e_1\otimes e_{1} = e_1\otimes e_1;$
\item $\big(e_1 \otimes e_2 \otimes e_3\big)  \otimes_1 \big(e_4 \otimes e_2 \otimes e_5\big) = \langle e_3,e_4\rangle_{\mathcal{H}} e_1 \otimes e_2 \otimes e_2 \otimes e_5 =0$.
\item For $n > 2$, consider $f:[n]^2 \rightarrow \mathbb{R}$ defined via $f(i,j)= \dfrac{1}{\sqrt{n-2}}$ for $i \neq j$, and $f(i,i) = 0$. Then
$$f \stackrel{1}{\smallfrown} f(h,k) = 
\begin{cases}
1 & \text{ if } h \neq k \\
\dfrac{n-1}{n-2} & \text{ if } h=k.
\end{cases}
$$
\end{enumerate}
\end{exm}

\begin{rmk}
The symbol of the norm $\|\cdot \|$ will be used for both the (square root) of the variance of a discrete kernel (as in (\ref{variance})) and for vectors in the fixed Hilbert space: the use of the symbol will be clear from the context. Moreover, in order to simplify the notation, the subscripts for the norms $\|f \|_{\mathcal{H}^{\otimes r}},  f \in \mathcal{H}^{\otimes r}$, will be omitted.\\
\end{rmk}

From now on, let $\bs{h}= (h_1,\dots, h_d)$ denote a fixed vector of orders for Chebyshev sums, with $h_i = h_{d-i+1}$ for all $i=1,\dots,\lfloor \frac{d}{2} \rfloor$, and consider fixed a real separable Hilbert space $\mathcal{H}$, with orthonormal basis $\{e_i\}_{i \geq 1}$ (in general, $\mathcal{H}=\mathrm{L}^2(\mathbb{R}^{q})$ for a certain $q\geq 2$). If $m=h_1 + \cdots + h_d$, every admissible kernel $f:[n]^d \rightarrow \mathbb{R}$ can be uniquely associated with the element $k(f)$ in $\mathcal{H}^{\otimes m}$ defined by:
\begin{equation}
\label{kN1}
 k(f) = \sum_{i_1,\dots,i_d=1}^{n}f(i_1,\dots,i_d)e_{i_1}^{\otimes h_1}\otimes e_{i_2}^{\otimes h_2}\otimes  \cdots \otimes e_{i_d}^{\otimes h_d}.
 \end{equation}

In view of the constraints on $h_1,\dots,h_d$, $k(f)$ is mirror symmetric (as a function of $m$ variables) if and only if $f$ is mirror symmetric (as a function of $d$ variables). \\

Contractions of the kernel $f$ and of the kernel $k=k(f)$ are related via:
\begin{align*}
k &\otimes_r k = \\
& \quad \sum_{\substack{ i_1,\dots, i_{d-q} \in [n] \\ j_{q+1},\dots, j_{d} \in [n] }} f \stackrel{q}{\smallfrown} f (i_1,\dots, i_{d-q},j_{q+1},\dots, j_{d}) e_{i_1}^{\otimes h_1} \otimes \cdots \otimes e_{i_{d-q}}^{\otimes h_{d-q}}\otimes  e_{j_{q+1}}^{\otimes h_{j_{q+1}}} \otimes \cdots \otimes e_{j_d}^{\otimes h_{d}}                
\end{align*}
if $r= h_1 + \cdots + h_q$, for $q=1,\dots, d-1$, while
\small{
\begin{align*}
&k \otimes_r k = \\
&\; \sum_{\substack{ i_1,\dots, i_{d-q} \in [n] \\ j_q, j_{q+1},\dots, j_{d} \in [n] }} f \star_q^{q-1} f (i_1,\dots, i_{d-q},j_q, \dots, j_{d}) e_{i_1}^{\otimes h_1}  \cdots \otimes e_{i_{d-q}}^{\otimes h_{d-q}}\otimes e_{j_q}^{\otimes 2(h_q - t)}\otimes e_{j_{q+1}}^{\otimes h_{j_{q+1}}} \cdots \otimes e_{j_d}^{\otimes h_{d}}    
   \end{align*}} \normalsize  if $r= \sum\limits_{j=1}^{q-1}h_j + t$, for some $t=1,\dots, h_q -1$ and $q=1,\dots, d$.

As a consequence, contractions of the kernel $f$ and of the kernel $k:=k(f)$ enjoy the following properties, whose proofs follow via straightforward computations (see also \cite[Lemma 3.4]{NourdinPeccatiReinert}).

\begin{prop}
\label{contraction5}
Let $f:[n]^d \rightarrow \mathbb{R}$ be an admissible kernel, and  For the fixed $\bs{h}=(h_1,\dots,h_d)$, consider the kernel $k:=k(f)$ as in \eqref{kN1}. For every $r=1,\dots, m - 1$,
\begin{itemize}
\item[(i)] if $r= h_1 + \cdots + h_q$, for $q=1,\dots, d-1$, then: $$ \|k \otimes_r k \| = \|f\stackrel{q}{\smallfrown} f \|;$$
\item[(ii)] if $r= \sum\limits_{j=1}^{q-1}h_j + t$, for some $t=1,\dots, h_q -1$ and $q=1,\dots, d$, then: $$ \|k \otimes_r k \|  = \|f \star_{q}^{q-1} f \|. $$
\end{itemize}
\end{prop}

\begin{prop}
\label{contraction6}
Let $f:[n]^d \rightarrow \mathbb{R}$ be an admissible kernel, and consider $k:=k(f)$ as in \eqref{kN1}. For the fixed $\bs{h}=(h_1,\dots,h_d)$, assume that $m:= h_1+\cdots + h_d $ is even.  
\begin{itemize}
\item[(i)] If $d$ is even (and so $h_1 + \cdots +h_d = 2(h_1 + \cdots + h_{\frac{d}{2}})$), then:
$$ \|k \otimes_{\frac{m}{2} }k - k \| = \|f \stackrel{\frac{d}{2}}{\smallfrown}f - f \|;$$
\item[(ii)] if $d$ is odd (and therefore $h_1 + \cdots +h_d = 2(h_1 + \cdots + h_{\frac{d-1}{2}}) + h_{\frac{d+1}{2}}$ is even whenever $h_{\frac{d+1}{2}}$ is even), then:
$$ \|k \otimes_{\frac{m}{2}} k - k \| = \|f\star_{\frac{d+1}{2}}^{\frac{d-1}{2}}f - f \|.$$
\end{itemize}
\end{prop}
\medskip

\subsection{The Lindberg method via influence functions}
The celebrated \textit{Lindberg replacement trick}\index{Lindberg method} is a technique for proving Central Limit Theorems  for random sums and, more generally, for bounding the distance of their probability measures, consisting in successive replacements of the involved summands. The origin of this method dates back to Lindberg's proof of the Central Limit Theorem for the normalized sum of centered and scaled i.i.d. random variables (see \cite{Lindberg} or \cite[Theorem 11.1.1 and Proposition 11.1.3]{NourdinPeccatilibro}, as well as the references therein).\\

The Lindberg-type method that will be proved in the sequel has been inspired by the strategy worked out in \cite{Mossel}, and relies on the concept of \textit{influence functions}.

\begin{defn}
If $f:[n]^{d}\rightarrow \mathbb{R}$ is an admissible kernel, for every $i=1,\dots,n$, the $i$-th \textbf{influence function} \index{Influence function} of $f$ is defined as:
\begin{equation}
\label{Influence}
\mathrm{Inf}_{i}(f) = \sum_{l=1}^{d}\sum_{j_1,\dots,j_{d-1}=1}^{n}f(j_{1},\dots,j_{l-1},i,j_{l},\dots,j_{d-1})^{2}  \;.
\end{equation}
\end{defn}

Note that, for an admissible kernel $f$, $\sum\limits_{i=1}^n \mathrm{Inf}_{i}(f) =  \|f \|^2 =1$. If one requires that  $f$ is fully symmetric, then the $i$-th influence function reduces to:
$$\mathrm{Inf}_{i}(f) = d\sum_{j_1,\dots,j_{d-1}=1}^{n}f(i,j_1,\dots,j_{d-1})^{2}.$$
In this case, since $\| f\|^2 =1$, then $\sum\limits_{i=1}^{n}\mathrm{Inf}_i(f) = d$ (more generally, $\sum\limits_{i=1}^{n}\mathrm{Inf}_i(f) = d \|f \|^2$ if $f$ has a different normalization).

Theorem \ref{invMossel} records the invariance principle stated in \cite[Theorem 3.18]{Mossel} for multilinear polynomials, in a simplified version that is sufficient for the present purposes.\\

\begin{thm}\label{invMossel}
Let $(\Omega, \mathcal{F}, \mathbb{P})$ be a classical probability space, and $\mathbf{Y}=\{Y_i\}_{i\geq 1}$ a sequence of independent centered random variables on $\Omega$, with unit variance. For $d\geq 1$, consider a sequence of ensembles $\bs{\mathcal{X}}^{(n)} = (\bs{\mathcal{X}}_1,\dots,\bs{\mathcal{X}}_n)$, with $\bs{\mathcal{X}}_i = \{X_{i,1},\dots,X_{i,d}\}$, and where $(X_{i,j})_{i \in \mathbb{N}, j=1,\dots,d}$ is a double-indexed sequence of independent random variables, and set:
$$ Q_{\bs{\mathcal{X}}^{(n)}}(f_n) := \sum_{i_1,\dots,i_d=1}^n f_n(i_1,\dots,i_d) X_{i_1,1}\cdots X_{i_d,d}. $$
Assume that there exists $r\geq 3$ such that the ensembles $\bs{\mathcal{X}}^{(n)}$ are $(2,r,\eta)$-hypercontractive, that is, that there exists a positive real number $\eta$ such that:
$$ \E[|Q_{\bs{\mathcal{X}}^{(n)}}(f_n)|^r ] \leq \eta^{-d}\E[Q_{\bs{\mathcal{X}}^{(n)}}(f_n)^2 ] .$$
Then, for every smooth function $\psi$ with uniformly bounded $r$-th derivative, and for every sequence of symmetric admissible kernels $f_n:[n]^d \rightarrow \mathbb{R}$, 
\begin{equation*}
\big|\mathbb{E}[\psi\big(Q_{\bs{\mathcal{X}}^{(n)}}(f_n)\big)] - \mathbb{E}[\psi\big(Q_{\bs{Y}}(f_n)\big)]\big| = \mathcal{O}\big( \sqrt{\tau_n}\big)\, ,
\end{equation*}
where $\tau_n := \tau(f_n)\;= \max\limits_{i=1,\dots,n}\mathrm{Inf}_{i}(f_n)$.
In particular, for $\bs{\mathcal{X}}_i = \{X_i\}$, that is, if  $ Q_{\bs{\mathcal{X}}^{(n)}}(f_n) = Q_{\bs{X}}(f_n)$ is a  homogeneous sum in a sequence $\bs{X}=\{X_i\}_{i\geq 1}$ of independent centered random variables on $\Omega$, with unit variance and $(2,3,\eta)$-hypercontractive, then:
\begin{equation*}
\big|\mathbb{E}[\big(\psi\big(Q_{\bs{X}}(f_n)\big)] - \mathbb{E}[\psi\big(Q_{\bs{Y}}(f_n)\big)] \big|= \mathcal{O}\big( \sqrt{\tau_n}\big)\, .\\
\end{equation*}
\end{thm}

Next theorem recalls \cite[Theorem 1.3]{NourdinDeya}, where the authors extended Theorem \ref{invMossel} in the free probability setting, for homogeneous polynomials in freely independent variables.\\

\begin{thm}
\label{teoNourdin}\label{InvarianceNourdin}
Let $(\mathcal{A},\varphi)$ be a $W^{\ast}$-probability space. Let $\mathbf{X}=\{X_i\}_{i\geq 1}$ and $\mathbf{Y}= \{Y_i\}_{i\geq 1}$ be two sequences of centered freely independent random variables with unit variance, such that $\mathbf{X}$ and $\mathbf{Y}$ are freely independent. Assume, moreover, that the elements of $\mathbf{X}$ (respectively $\mathbf{Y}$) have uniformly bounded moments, that is, for every $r \geq 1$:
$$\sup_{i\geq 1}\varphi(|X_i|^{r}) < \infty \quad (\text{resp. } \sup_{i\geq 1}\varphi(|Y_i|^{r}) < \infty).$$
Set $d\geq 1$ and let $Q_{\bs{x}}(f_n)$ denote a homogeneous sum of degree $d$ as in (\ref{QN}), with admissible coefficient $f_n:[n]^d \rightarrow \mathbb{R}$ as in Definition \ref{admissible_Univ}. Then, for any integer $m\geq 1$:
\begin{equation}
\varphi\big(Q_{\bs{X}}(f_n)^{m}\big) - \varphi\big(Q_{\bs{Y}}(f_n)^{m}\big) = \mathcal{O}\big(\sqrt{\tau_n}\big), 
\end{equation}
where $\tau_n := \tau(f_n)\;= \max\limits_{i=1,\dots,n}\mathrm{Inf}_{i}(f_n)$.\\
\end{thm}

In particular, Theorem \ref{InvarianceNourdin} applies when $\bs{X}$ and $\bs{Y}$ are composed of identically distributed random variables: roughly speaking, Theorem \ref{teoNourdin} implies that whenever the kernels $f_n$ have low-influences as $n\rightarrow \infty$ (meaning that $\tau_n = o(1)$), the asymptotic behaviour of $Q_{\bs{X}}(f_n)$ is basically insensitive of the distribution of its entries $\bs{X}$.

\begin{exm}
For $d=1$, set $f_n(i) = \dfrac{1}{\sqrt{n}}$ for all $i=1,\dots,n$, so that: 
$$Q_{\bs{X}}(f_n) = \dfrac{1}{\sqrt{n}}\sum\limits_{i=1}^n X_i$$
and $\mathrm{Inf}_i(f_n)= f_n(i)^2$, giving $\tau_n = o(1)$. Then the free CLT (see \cite[Theorem 8.10]{Speicher}) follows  from $Q_{\bs{X}}(f_n) \sim \mathcal{S}(0,1)$ for every $n$, when $X \sim \mathcal{S}(0,1)$.\\
\end{exm}

Apart from providing an explicit nice bound for the proximity in law of homogeneous sums, the main consequence of Theorem \ref{InvarianceNourdin} has been stated  in \cite[Theorem 1.4]{NourdinDeya}, and consists in the universality of the semicircular law for semicircular approximations of homogeneous sums with symmetric coefficients, in the sense of Theorem \ref{invnoncom}.

\begin{thm}\label{invnoncom}
For $d\geq 2$, let $f_n:[n]^d\to\mathbb{R}$ be a sequence of symmetric admissible kernels. The following statements are equivalent as $n \rightarrow \infty$:
\begin{itemize}
\item[(i)] $Q_{\bs{S}}(f_n) \stackrel{\text{ Law }}{\longrightarrow} \mathcal{S}(0,1)$;
\item[(ii)] $Q_{\bs{X}}(f_n) \stackrel{\text{ Law }}{\longrightarrow} \mathcal{S}(0,1)$ for any other sequence $\bs{X}=\{X_i\}_{i\geq 1}$ of freely independent and identically distributed random variables, satisfying Assumption {\bf (1)}.
\end{itemize}
\end{thm}

Theorem \ref{teoNourdin} is the starting point of the analysis developed in Section \ref{MainResult}, where it is extended, together with its consequences, in a general multidimensional setting via Chebyshev sums.

\medskip
\subsection{Auxiliary statements}

For the sake of clarity, it is convenient to recall some technical statements that will be used in the proofs of our main results. These are, in order, the \textit{non-commutative binomial expansion}, the \textit{free H\"{o}lder inequality}\index{Free H\"{o}lder inequality} and the {hypercontractivity} \index{Hypercontractivity}of homogeneous sums in freely independent variables (which is the free counterpart of \cite[Lemma 4.2]{NourdinPeccatiReinert}). All these properties are meant to hold in the fixed $W^\star$-probability space $(\mathcal{A},\varphi)$.

\begin{lemma}[\cite{NourdinDeya}]
\label{freebin}
Let $X$ and $Y$ be random variables in $(\mathcal{A},\varphi)$. Then, for every positive integer $m$:
\begin{equation*}
(X + Y)^{m} = X^{m} + \sum_{n=1}^{m}\sum_{(r, \mathbf{i}_{r},\mathbf{j}_{r})\in D_{m,n}} X^{i_1}Y^{j_1}X^{i_2}Y^{j_2}\cdots X^{i_r}Y^{j_r},
\end{equation*}
where
$$ D_{m,n} = \{(r, \mathbf{i}_{r},\mathbf{j}_{r}) \in [m]\times \mathbb{N}^{r} \times \mathbb{N}^{r}: \sum_{l=1}^{r}i_l= m-n, \sum_{l=1}^{r}j_l = n\}.$$
\end{lemma}
\begin{lemma}[\textrm{\cite[Lemma 12]{Kargin}}] 
\label{lemma0}
Let $X$ and $Y$ be random variables in $(\mathcal{A},\varphi)$. For every $r\in \mathbb{N}$ and every choice of non-negative integers $m_1,n_1,\dots, m_r,n_r$, the following H\"older type inequality holds:
$$ |\varphi\big(X^{m_1}Y^{n_1}\cdots X^{m_r}Y^{n_r}\big)| \leq \big[\varphi\big(X^{2^{r}m_1}\big)\big]^{2^{-r}} \big[\varphi\big(Y^{2^{r}n_1}\big)\big]^{2^{-r}}\cdots \big[\varphi\big(X^{2^{r}m_r}\big)\big]^{2^{-r}}\big[\varphi\big(Y^{2^{r}n_r}\big)\big]^{2^{-r}}.
$$\end{lemma}

Let $\bs{X}= \{X_1,\dots, X_n\}$ be a set of centered freely independent random variables in $(\mathcal{A},\varphi)$, having unit variance (not necessarily with the same distribution), and denote by $\{\mu_k^{\bs{X}}\}_{k\geq 1}$ the corresponding sequence of the \textit{largest even moments}, that is: 
$$\mu_k^{\bs{X}} = \sup\limits_{\substack{i=1,\dots,n\\l=1,\dots,k}}\varphi(X_i^{2l}). $$

\begin{prop}[\textrm{\cite[Proposition 3.3]{NourdinDeya}}]
\label{prop02}
For $d\geq 1$, let $g:[n]^{d}\rightarrow \mathbb{R}$ be a mirror symmetric kernel, vanishing on diagonals. For the homogeneous sum
$$Q_{\bs{X}}(g) = \sum_{i_1,\dots,i_d=1}^n g(i_1,\dots,i_d)X_{i_1}\cdots X_{i_d},$$ the following hypercontractivity estimate applies: for every integer $r \geq 1$, there exists a constant $C_{r,d}$ such that:
$$ \varphi\big(Q_{\bs{X}}(g)^{2r}\big) \leq C_{r,d}\;\mu_{2^{rd-1}}^{\bs{X}} \big(\varphi(Q_{\bs{X}}(g)^2)\big)^2\, $$
or, equivalently,
$$ \varphi\big(Q_{\bs{X}}(g)^{2r}\big) \leq C_{r,d}\;\mu_{2^{rd-1}}^{\bs{X}}\bigg(\sum_{j_1,\dots,j_d=1}^{n}g(j_1,\dots,j_d)^{2}\bigg)^{r}.$$
\end{prop}

\begin{lemma}[\textrm{\cite[Lemma 3.2]{NourdinDeya}}]
\label{lemma01}
For every integer $r \geq 1$, and every sequence $\bs{X} = \{X_i\}_{i \geq 1}$ of random variables in $(\mathcal{A},\varphi)$, the following estimate holds:
$$ |\varphi(X_{i_1}\cdots X_{i_{2r}})| \leq \mu_{2^{r-1}}^{\bs{X}},$$
for every choice of positive integers $i_1,\dots, i_{2r}$.
\end{lemma}

\medskip
\section{Main result: free Lindberg principle}\label{MainResult}


\begin{Assumption}\label{Assumption_Cheby}
Throughout this section, let $d\geq 2$ be a fixed integer and $\bs{h} =(h_1,\dots,h_d)$ be a fixed vector  of orders for a Chebyshev sum, with $h_j \geq 1$ and $h_j = h_{d-j+1}$ for $j=1,\dots,\lfloor \frac{d}{2}\rfloor$. For these orders, let $\bs{X}=\{X_i\}_{i\geq 1}$ be a sequence of freely independent random variables in $(\mathcal{A},\varphi)$ such that $U_{h_j}(X_i)$ is centered and has unit variance, for every $i$ and every $j=1,\dots,d$. \\
\end{Assumption}

The set of random variables for which Assumption \ref{Assumption_Cheby} holds obviously includes the standard semicircle law: indeed, if $S\sim \mathcal{S}(0,1)$, $\varphi(U_h(S))=0$ and $\varphi(U_h(S)^2)= 1$ for all $h\geq 1$, being $U_h(S) = I_h^{S}(e^{\otimes h})$ (see Section \ref{AppendixfreeProb}).  Other non trivial examples are the following:
\begin{itemize}
\item[(i)] let $d=2$ and choose $h_1 =h_2 =2$. For a random variable $X$, the constraints $\varphi(U_2(X))= 0$ and $\varphi(U_2(X)^2)=1$, give $\varphi(X^2)=1$ and $\varphi(X^4) = 2$, so $X$ can be any centered random variable with second moment equal to 1 and zero free fourth cumulant $\kappa_4(X)$. For instance, let $Z(1)$ be a centred random variable, with the free Poisson distribution of parameter one, and $Y$ be a symmetric free Bernoulli variable, say $Y \sim \frac{1}{2}(\delta_1 + \delta_{-1})$, freely independent of $Z(1)$. Since $\kappa_4(Z(1))= 1$ and $\kappa_4(Y)= -1$, the random variable $X:= \dfrac{1}{\sqrt{2}}(Z(1) + Y)$ is  centered and satisfies the desired hypotheses.
\item[(ii)] More generally, for the same choice of parameters, the scaled sum $X$ of two centered freely independent random variables $Y$ and $Z$, with unit variance and with $\kappa_4(Y) = 1$ and $\kappa_4(Z) = -1$, satisfies Assumption \ref{Assumption_Cheby}.
\item[(iii)] Let $d=3$ and choose $h_1=h_3=1$ and $h_2=3$. Since $U_3(x) = x^3-2x$, $\varphi(U_3(X)) =0$ is satisfied whenever $\varphi(X^3)=2\varphi(X)$, while $\varphi(U_3(X)^2) = 1$ is verified if $\varphi(X^6) - 4\varphi(X^4) + 4\varphi(X^2) = 1$. Without loss of generality, assume that $\varphi(X)=0$ and $\varphi(X^2)=1$, so that the desired $X$ should satisfy $\varphi(X^3)=0$ and $\varphi(X^6) = 4\varphi(X^4)-3$ (for instance, $X$ can have the free symmetric  Bernoulli distribution $X \sim \frac{1}{2}(\delta_1 + \delta_{-1})$). 
More generally, for the existence of a solution, the problem of moments requires that the Hankel matrix $\big(\varphi(X^{i+j})\big)_{i,j=0,\dots,3}$ should be positive definite (see \cite[Theorem 6.1]{Chihara}). By virtue of the so-called Sylvester's criterion, this condition is satisfied if all its upper-left minors are strictly positive.  For instance,  under the extra assumption $\varphi(X^5)=0$, few calculations yield that $X$ has to satisfy $\varphi(X^4)( \varphi(X^6) - \varphi(X^4)^2) > 0,$ which is always satisfied (indeed, by the Cauchy-Schwarz inequality, if $\varphi(X^2)=1$, then $\varphi(X^4)=\varphi(X^3 X) \leq \varphi(X^6)^{\frac{1}{2}}$). Under the constraint $\varphi(X^6)= 4 \varphi(X^4) - 3$, $X$ satisfies Assumption \ref{Assumption_Cheby} if its fourth moment satisfies $\varphi(X^4) \in [1,3]$.\\
\end{itemize}


Following the strategy proposed in \cite{Mossel}, we introduce some further notation for Chebyshev sums, which will simplify the discussion contained in the proofs and ease the connection with the findings in \cite{Mossel}, where the authors deal with homogeneous sums in sequences of ensembles. 

More precisely, if  $f:[n]^d \rightarrow \mathbb{R}$ is an admissible kernel, we will introduce objects of the type $Q_{\bs{\mathcal{Y}}^{(n)}}(f)$, where $\bs{\mathcal{Y}}^{(n)}$ is no longer a sequence of random variables, but an ensemble, that is:
$$\bs{\mathcal{Y}}^{(n)} = \big(\bs{\mathcal{Y}}_1,\dots, \bs{\mathcal{Y}}_n\big)\, \text{ with  } \bs{\mathcal{Y}}_i = (\mathcal{Y}_{i,1},\dots, \mathcal{Y}_{i,d}), $$
and each $\mathcal{Y}_{i,j}$ is a random variable on the fixed space. With this notation, we set:
\begin{equation}\label{HomInEnsemble}
Q_{\bs{\mathcal{Y}}^{(n)}}(f) := \sum_{i_1,\dots,i_d=1}^{n}f(i_1,\dots,i_d)\mathcal{Y}_{i_1,1}\mathcal{Y}_{i_2,2}\cdots \mathcal{Y}_{i_d,d} 
\end{equation}
(namely the $j$-th factor in each summand is the $j$-th element in $\bs{\mathcal{Y}}_{i_j}$).\\

For a fixed vector of orders $\mathbf{h}=(h_1,\dots,h_d)$, Chebyshev sums correspond to a particular choice of ensemble, that is:
$$\bs{\mathcal{X}}^{(n)} = \big(\bs{\mathcal{X}}_1,\dots, \bs{\mathcal{X}}_n\big)\, \text{ with  } \bs{\mathcal{X}}_i = (\mathcal{X}_{i,1},\dots, \mathcal{X}_{i,d}), \, \mathcal{X}_{i,j} = U_{h_j}(X_i) ,$$
namely:
\begin{equation}
\label{ensemble}
\bs{\mathcal{X}}_i = \big(U_{h_1}(X_i),\dots, U_{h_d}(X_i)\big).
\end{equation}

In this case, we have:
\begin{align*}
Q_{\bs{\mathcal{X}}^{(n)}}(f) &:= \sum_{i_1,\dots,i_d=1}^{n}f(i_1,\dots,i_d)\mathcal{X}_{i_1,1}\mathcal{X}_{i_2,2}\cdots \mathcal{X}_{i_d,d} \\
&=  \sum_{i_1,\dots,i_d=1}^{n}f(i_1,\dots,i_d)U_{h_1}(X_{i_1})\cdots U_{h_d}(X_{i_d})\\
&= Q_{\bs{X}}^{(\bs{h})}(f)
\end{align*}

Observe that the notation $Q_{\bs{X}}^{(\bs{h})}(f)$ emphasizes the dependence of the random variable $Q_{\bs{X}}^{(\bs{h})}(f)$ on the sequence $\bs{X}= \{X_i\}_{i \geq 1}$ and on the orders $\bf{h}$, while the notation $Q_{\bs{\mathcal{X}}^{(n)}}(f)$ is particularly useful to apply the Lindberg replacement trick. Indeed, 
as already remarked, the Lindberg method basically consists in progressive replacements of the summands of the random functional under consideration. To accomplish such a goal, further sequences of ensembles are needed. Let $\bs{Y}=\{Y_i\}_{i\geq 1}$ be a sequence of freely independent  random variables, satisfying Assumption {\bf (1)}, freely independent of $\bs{X}=\{X_i\}_{i\geq 1}$. Then, the notation introduced in \eqref{HomInEnsemble} will be extended in a canonical way to the auxiliary ensembles:
\begin{equation}
\label{auxiliary}
 \bs{\mathcal{Z}}^{(i)} :=  (\bs{\mathcal{Z}}_{1}^{(i)},\dots, \bs{\mathcal{Z}}_{n}^{(i)}) = (\bs{Y}_1,\dots,\bs{Y}_{i-1}, \bs{\mathcal{X}}_{i},\dots, \bs{\mathcal{X}}_{n}),
 \end{equation}
for $i=1,\dots,n$, where $\bs{Y}_j = \underbrace{(Y_j,\dots,Y_j)}_{d \text{ times}}$, so that the random variable $Q_{ \bs{\mathcal{Z}}^{(i)}}(f)$ is obtained from $Q_{\bs{\mathcal{X}}^{(n)}}(f)$ by replacing $\mathcal{X}_{l,j}$ with $Y_l$, for $l=1,\dots,i-1$ and for every $j=1,\dots,d$. In particular, $\bs{\mathcal{Z}}^{(1)} = \bs{\mathcal{X}}^{(n)}$ and $\bs{\mathcal{Z}}^{(n)} = (\bs{Y}_1,\dots,\bs{Y}_n)$.

\subsection{Main Statement}

For $p\geq 1$, for any integer $n$ and for every $j=1,\dots,p$, let $f_n^{(j)}:[n]^{d}\rightarrow \mathbb{R}$ be an admissible kernel (according to Definition \ref{admissible_Univ}), and consider the associated  homogeneous polynomial in the non-commuting variables $x_1,\dots,x_n$:
\begin{equation}
\label{QNj}
Q_{\bs{x}}^{(j)}:=Q_{\bs{x}}(f_n^{(j)}) = \sum_{i_1,\dots,i_{d}=1}^{n}f_n^{(j)}(i_1,\dots,i_{d})x_{i_1}\cdots x_{i_{d}}.\\
\end{equation}

The forthcoming Theorem \ref{Multiinvariance2} provides an estimate of the proximity in law (expressed in terms of joint moments) between $\big(Q_{\bs{\mathcal{X}}^{(n)}}(f_n^{(1)}),\dots, Q_{\bs{\mathcal{X}}^{(n)}}(f_n^{(p)})\big)$ and $(Q_{\Y}^{(1)},\dots,Q_{\Y}^{(p)})$, where $\bs{\mathcal{X}}^{(n)}$  is an ensemble defined as in (\ref{ensemble}) for a sequence $\bs{X}$ of freely independent random variables satisfying Assumption \ref{Assumption_Cheby}, yielding the  generalization of the invariance principle given in Theorem \ref{teoNourdin}.\\

The differences of the joint moments will be controlled by means of the quantities $\tau_n^{(j)} = \max\limits_{i=1,\dots,n}\mathrm{Inf}_{i}(f_n^{(j)})$, for $j=1,\dots,p$ in such a way that the resulting bound perfectly matches with the bound given in \cite[Theorem 7.1]{NourdinPeccatiReinert}. 

\begin{thm}
\label{Multiinvariance2} 

If $d\geq 1$, let $\bs{h}= (h_1,\dots,h_d)$ be a vector of  positive integers with $h_i = h_{d-i+1}$ for $i=1,\dots,\lfloor \frac{d}{2}\rfloor$ \glossary{name={$\lfloor r\rfloor$},description={Integer part of $r \in \mathbb{R}$}}(if $d \geq 2$).
Let $\bs{X} = \{X_i\}_{i \geq 1}$ be a sequence of freely independent random variables satisfying Assumption \ref{Assumption_Cheby}, and $\bs{Y}=\{Y_j\}_{j\geq 1}$ be a sequence of freely independent centered random variables with unit variance, freely independent of $\bs{X}$. Assume further that $\bs{X}$ and $\bs{Y}$ are composed of random variables with uniformly bounded moments\footnote{If $\bs{X}$ has uniformly bounded moments, so have the elements of the ensemble $\bs{\mathcal{X}}^{(n)}$ for every $n$.}, that is, for every integer $r\geq 1$,
$$ \sup_{i\geq 1} \varphi(X_i^{r}) < \infty \qquad (\text{resp. } \sup_{i\geq 1} \varphi(Y_i^{r}) < \infty).$$
Then, for every integer $k\geq 1$, and for every choice of non-negative integers $m_{1,s},\dots,m_{p,s}$, for $s=1,\dots,k$, if $f_n^{(j)}: [n]^{d}\rightarrow \mathbb{R}$ is an admissible kernel for every $j=1,\dots,p$, 
\begin{align*}
\varphi\bigg(\prod_{s=1}^{k}\big(Q^{(1)}_{\bs{\mathcal{X}}^{(n)}}\big)^{m_{1,s}}\cdots \big(Q^{(p)}_{\bs{\mathcal{X}}^{(n)}}\big)^{m_{p,s}}\bigg) &-\varphi\bigg(\prod_{s=1}^{k}\big(Q_{\bs{Y}}(f_n^{(1)})\big)^{m_{1,s}}\cdots \big(Q_{\bs{Y}}(f_n^{(p)}\big)^{m_{p,s}}\bigg)\\
&= \mathcal{O}\bigg(\max\limits_{j=1,\dots,p}\sqrt{\tau_{n}^{(j)}}\bigg) \numberthis \label{bound}
\end{align*}
\end{thm}

\begin{rmk}
In the classical case, the invariance principle provided in \cite{NourdinPeccatiReinert} is somewhat stronger, since it is possible to consider vectors of homogeneous sums with possibly different degrees. As outlined from the proofs, here the choice of taking homogeneous sums of different degrees, is admissible only when considering vectors of Chebyshev sums of  order $\bs{h}=(h,h,\dots,h)$, namely, vectors of the type $\big(Q_{\bs{X}}^{(h)}(f_n^{1}), \dots,Q_{\bs{X}}^{(h)}(f_n^{m})\big)$, with $f_n^{(j)}:[n]^{d_j}\rightarrow \mathbb{R}$ and
$$ Q_{\bs{X}}^{(h)}(f_n^{j}) = \sum_{i_1,\dots,i_{d_j}=1}^n  f_n^{(j)}(i_1,\dots,i_{d_j})U_h(X_{i_1}) U_h(X_{i_2})\cdots U_{h}(X_{i_{d_j}}). $$
\end{rmk}

\begin{exm}
Here we are going to shortly discuss two explicit cases  where Theorem \ref{Multiinvariance2} entails or not  the universality phenomenon. 
 For  $p=d=2$, consider the kernels:
\begin{enumerate}
\item 
$$ f_n^{(1)}(i,j) = \begin{cases}
\dfrac{1}{\sqrt{2n-2}} & \text{ if } i \neq j, i = 1 \text{ or }j=1,\\
0 & \text{otherwise};
\end{cases}
$$
\item
$$ 
f_n^{(2)}(i,j)=
\begin{cases}
0 & \text{ if } i =j \\
\dfrac{1}{\sqrt{n(n-1)}} & \text{ if } i \neq j;
\end{cases}
$$
\item 
$$ 
f_n^{(3)}(i,j) =
\begin{cases}
\dfrac{1}{\sqrt{(n-1)(n-2)}} & \text{ if } i\neq j \text{ and } i, j \neq 1, \\
0 & \text{ otherwise}.
\end{cases}
$$
\end{enumerate}
Note that $\|f_n^{(j)}\|^2 = 1$ for all $j=1,2,3$. Simple computations yield that:
\begin{enumerate}
\item $\mathrm{Inf}_1(f_n^{(1)}) = 1$ and $\mathrm{Inf}_j(f_n^{(1)}) = \dfrac{1}{n-1} $ for $j=2,\dots,n$;
\item $\mathrm{Inf}_i(f_n^{(2)}) = \dfrac{2}{n}$ for every $i=1,\dots,n$;
\item $\mathrm{Inf}_1(f_n^{(3)}) = 0$, and $\mathrm{Inf}_j(f_n^{(3)}) = \dfrac{2}{n-1}$ for all $j=2,\dots,n$,
\end{enumerate}
which in turn imply that $\tau_n^{(1)} = 1$, $\tau_n^{(2)} = \dfrac{2}{n}$ and $\tau_n^{(3)} = \dfrac{2}{n-1}$. Therefore, for Chebyshev sums  with kernels $f_n^{(1)},f_n^{(2)}, f_n^{(3)}$ respectively, 
\begin{align*}
\varphi\bigg(\prod_{s=1}^{k}\big(Q_{\bs{\mathcal{X}}^{(n)}}(f_n^{(2)})\big)^{m_{1,s}} \big(Q_{\bs{\mathcal{X}}^{(n)}}(f_n^{(3)})\big)^{m_{3,s}}\bigg) &-\varphi\bigg(\prod_{s=1}^{k}\big(Q_{\bs{Y}}(f_n^{(2)})\big)^{m_{1,s}}\big(Q_{\bs{Y}}(f_n^{(3)})\big)^{m_{3,s}}\bigg) \\
&= \mathcal{O}\bigg(\dfrac{1}{\sqrt{n-1}}\bigg),
\end{align*}
while
\begin{equation*}
\varphi\bigg(\prod_{s=1}^{k}\big(Q_{\bs{\mathcal{X}}^{(n)}}(f_n^{(1)})\big)^{m_{1,s}} \big(Q_{\bs{\mathcal{X}}^{(n)}}(f_n^{(2)})\big)^{m_{2,s}}\bigg) -\varphi\bigg(\prod_{s=1}^{k}\big(Q_{\bs{Y}}(f_n^{(1)})\big)^{m_{1,s}}\big(Q_{\bs{Y}}(f_n^{(2)})\big)^{m_{2,s}}\bigg) 
= \mathcal{O}\big(1\big),
\end{equation*}
and thus  no universal behaviour can be detected from Theorem \ref{Multiinvariance2}.
\end{exm}

\subsection{Sketch of the proof}

\smallskip

Before detailing the complete proof of Theorem \ref{Multiinvariance2} (to which the next section is entirely dedicated), here is a brief sketch of the general strategy. To simplify the notation, the dependence on $n$ will be dropped when there is no risk of confusion. \\

Consider the auxiliary ensembles introduced in equation \eqref{auxiliary}.  For every $j=1,\dots,p$, set:
 $$Q_{\bs{\mathcal{Z}}^{(i)}}(f_n^j) = W_j^{(i)} + V_j^{(i)}(\bs{\mathcal{X}}_i)\, ,$$ with $W_j^{(i)}$, $V_j^{(i)}(\bs{\mathcal{X}}_i)$ self-adjoint sums defined by:
\begin{equation}
\label{WN}
W_j^{(i)} := \sum_{i_1,\dots,i_d \in [n]\setminus \{i\}}f_n^{(j)}(i_1,\dots,i_d) \mathcal{Z}_{i_1,1}^{(i)}\cdots \mathcal{Z}_{i_d,d}^{(i)}
\end{equation}
(that is, $W_j^{(i)}$ is obtained by gathering together the summands where no $U_{h_l}(X_i)$'s appear), and
\small{
\begin{equation}
\label{VN}
V_j^{(i)}(\bs{\mathcal{X}}_i) = \sum_{l=1}^{d}\sum_{\substack{i_1,\dots,i_{d-1}\in\\ [n]\setminus \{i\}}}f_n^{(j)}(i_1,\dots,i_{l-1},i,i_{l},\dots,i_{d-1})\mathcal{Z}_{i_1,1}^{(i)}\cdots \mathcal{Z}_{i_{l-1},l-1}^{(i)}\; U_{h_l}(X_i) \;\mathcal{Z}_{i_{l}, l+1}^{(i)}\cdots \mathcal{Z}_{i_{d-1},d}^{(i)}\;,
\end{equation}
}
\normalsize
with $$\mathcal{Z}_{i_j,j}^{(i)}= 
\begin{cases}
Y_{i_j} & \text{ if } i_j \leq i-1 ,\\
U_{h_j}(X_{i_j}) & \text{ if } i_j > i.
\end{cases}
$$
Similarly, set:
\begin{equation}
V_j^{(i)}(\bs{Y}_i) = \sum_{l=1}^{d}\sum_{\substack{i_1,\dots,i_{d-1}\\ \in [n]\setminus \{i\}}}f_n^{(j)}(i_1,\dots,i_{l-1},i,i_{l},\dots,i_{d-1})\mathcal{Z}_{i_1,1}^{(i)}\cdots \mathcal{Z}_{i_{l-1},l-1}^{(i)}\; Y_i\; \mathcal{Z}_{i_{l}, l+1}^{(i)}\cdots \mathcal{Z}_{i_{d-1},d}^{(i)}.
\end{equation}

Therefore,
\small{
\begin{align*}
\bigg|&\varphi\bigg(\prod_{s=1}^{k}\big(Q_{\bs{\mathcal{X}}^{(n)}}^{(1)}\big)^{m_{1,s}}\cdots \big(Q_{\bs{\mathcal{X}}^{(n)}}^{(p)})\big)^{m_{p,s}}\bigg) -\varphi\bigg(\prod_{s=1}^{k}\big(Q_{\bs{Y}}^{(1)}\big)^{m_{1,s}}\cdots \big(Q_{\bs{Y}}^{(p)}\big)^{m_{p,s}}\bigg) \bigg|\\ 
&= \bigg|\sum_{i=1}^{n}
\varphi\bigg(\prod_{s=1}^{k}\big(Q^{(1)}_{\bs{\mathcal{Z}}^{(i)}}\big)^{m_{1,s}}\cdots \big(Q^{(p)}_{\bs{\mathcal{Z}}^{(i)}}\big)^{m_{p,s}}\bigg) \\
& \qquad \qquad \qquad-\varphi\bigg(\prod_{s=1}^{k}\big(Q^{(1)}_{\bs{\mathcal{Z}}^{(i+1)}}\big)^{m_{1,s}}\cdots \big(Q^{(p)}_{\bs{\mathcal{Z}}^{(i+1)}}\big)^{m_{p,s}}\bigg)\bigg|  \\ 
&= \bigg|\sum_{i=1}^{n}
\varphi\bigg(\prod_{s=1}^{k}\big(W_1^{(i)}+ V_1^{(i)}(\bs{\mathcal{X}}_i)\big)^{m_{1,s}}\cdots \big(W_p^{(i)}+ V_p^{(i)}(\bs{\mathcal{X}}_i)\big)^{m_{p,s}}\bigg) \\ 
&\qquad \qquad \qquad-\varphi\bigg(\prod_{s=1}^{k}\big(W_1^{(i)}+ V_1^{(i)}(\bs{Y}_i))\big)^{m_{1,s}}\cdots \big(W_p^{(i)}+ V_p^{(i)}(\bs{Y}_i)\big)^{m_{p,s}}\bigg)\bigg|  \numberthis\label{differenceBIS}. 
\end{align*}
}
\normalsize

The conclusion is then obtained by showing that the non-zero summands in (\ref{differenceBIS}) either cancel out between each other, or are of the order of $\max\limits_{j=1,\dots,p}\sqrt{\tau_n^{(j)}}$, as outlined in the examples below. 

\begin{exm}
This example illustrates the sketch of the proof for a particular choice of parameters.  Consider $d=3, k=1, p=2, m_{1,1}=2, m_{2,1}=1$. Then, for every fixed $i=1,\dots,n$, the expansion of
$$ \varphi\big( (W_1^{(i)} + V_{1}^{(i)}(\bs{\mathcal{X}}^{(n)}))^2  (W_2^{(i)} + V_{2}^{(i)}(\bs{\mathcal{X}}^{(n)}))\big) $$ gives the following 8 summands:
\begin{enumerate}
\item $\varphi\big((W_1^{(i)})^2 W_2^{(i)}\big)$, that will be cancelled out in the difference (\ref{differenceBIS}) with the same expectation coming from $\varphi\big( (W_1^{(i)} + V_{1}^{(i)}(\bs{Y}))^2  (W_2^{(i)} + V_{2}^{(i)}(\bs{Y}))\big)$;
\item $\varphi \big((W_1^{(i)})^2 V_2^{(i)}(\bs{\mathcal{X}}^{(n)})\big)$;
\item $\varphi \big(W_1^{(i)} V_1^{(i)}(\bs{\mathcal{X}}^{(n)}) W_2^{(i)}\big) $;
\item $\varphi \big(W_1^{(i)} V_1^{(i)}(\bs{\mathcal{X}}^{(n)}) V_2^{(i)}(\bs{\mathcal{X}}^{(n)})\big) $;
\item $\varphi \big( V_1^{(i)}(\bs{\mathcal{X}}^{(n)}) W_1^{(i)} W_2^{(i)}\big) $;
\item $\varphi \big( V_1^{(i)}(\bs{\mathcal{X}}^{(n)}) W_1^{(i)}  V_2^{(i)}(\bs{\mathcal{X}}^{(n)})\big) $;
\item $\varphi \big( V_1^{(i)}(\bs{\mathcal{X}}^{(n)})^2 W_2^{(i)}  \big) $;
\item $ \varphi \big( V_1^{(i)}(\bs{\mathcal{X}}^{(n)})^2 V_2^{(i)}(\bs{\mathcal{X}}^{(n)}) \big).$
\end{enumerate}
It is easily seen by direct calculations that the items $2,3$, and $5$ are always zero, because the first item of Lemma \ref{3.1bis} applies. The items 4,6, and 7, are sums of terms that are either zero or cancel with the corresponding terms in 
$\varphi\big( (W_1^{(i)} + V_{1}^{(i)}(\bs{Y}))^2  (W_2^{(i)} + V_{2}^{(i)}(\bs{Y}))\big)$. In order to ease the notation, for the fixed $i$  set $Z_{j_l}:= \mathcal{Z}_{j_l,l}$ for every $l$. Then, for the fourth item in the above list, among other summands that equal zero, there is a sum of terms of the type:
$$ \sum_{\substack{i_1,i_2,i_3 \neq i \\ k_1,k_2 \neq i, l_1,l_2 \neq i} }f_n^{(1)}(i_1,i_2,i_3)f_n^{(1)}(k_1,k_2,i)f_n^{(2)}(i,l_1,l_2)\varphi\big( Z_{i_1}Z_{i_2}Z_{i_3} Z_{k_1}Z_{k_2}U_{h_3}(X_i)U_{h_1}(X_i)Z_{l_1}Z_{l_2}   \big),$$
which becomes (since $h_1=h_3$):
\begin{align*}
 \sum_{\substack{i_1,i_2,i_3 \neq i \\ k_1,k_2 \neq i, l_1,l_2 \neq i}}&f_n ^{(1)}(i_1,i_2,i_3)f_n^{(1)}(k_1,k_2,i)f_n^{(2)}(i,l_1,l_2)\varphi\big( Z_{i_1}Z_{i_2}Z_{i_3} Z_{k_1}Z_{k_2}U_{h_1}(X_i)^2 Z_{l_1}Z_{l_2}   \big) \\
&= \sum_{i_1,i_2,i_3 \neq i}f_n^{(1)}(i_1,i_2,i_3)f_n^{(1)}(i_3,i_2,i)f_n^{(2)}(i,i_2,i_1)\varphi\big( Z_{i_2}^3 \big). \numberthis \label{es1}
\end{align*}
On the other hand, the same computations for the corresponding terms in  $$\varphi\big( (W_1^{(i)} + V_{1}^{(i)}(\bs{\mathcal{X}}^{(n)}))^2  (W_2^{(i)} + V_{2}^{(i)}(\bs{\mathcal{X}}^{(n)}))\big)$$ yield:
\begin{align*}
 \sum_{\substack{i_1,i_2,i_3 \neq i \\ k_1,k_2 \neq i, l_1,l_2 \neq i}}& f_n^{(1)}(i_1,i_2,i_3)f_n^{(1)}(k_1,k_2,i)f_n^{(2)}(i,l_1,l_2)\varphi\big( Z_{i_1}Z_{i_2}Z_{i_3} Z_{k_1}Z_{k_2}Y_i^2 Z_{l_1}Z_{l_2}   \big)\\
 &= \sum_{i_1,i_2,i_3 \neq i}f_n^{(1)}(i_1,i_2,i_3)f_n^{(1)}(i_3,i_2,i)f_n^{(2)}(i,i_2,i_1)\varphi\big( Z_{i_2}^3 \big), \numberthis \label{es2}
\end{align*}
so that (\ref{es1}) and (\ref{es2}) cancel each other in (\ref{differenceBIS}). 
Note that, since the state $\varphi$ is a trace, the computations required for the items 4,6, and 7, proceed similarly, the only difference being in the occurring kernels.

The case to pay more attention to is that in item 8 of the above list. In this case, a priori, nothing can be said about its value, because it might depend on the distribution of $U_{h_j}(X_i)$. 
Indeed, by linearity, being $\varphi$ a trace and the rule of free independence, the only non trivial case to be considered is:
$$ \sum_{\substack{i_1,i_2 \neq i \\ l_1,l_2 \neq i \\ k_1,k_2 \neq i} } f_n^{(1)}(i_1,i,i_2)f_n^{(1)}(l_1,i,l_2)f_n^{(2)}(k_1,i,k_2)\varphi\big( Z_{i_1}U_{h_2}(X_i)Z_{i_2}Z_{l_1}U_{h_2}(X_i)Z_{l_2} Z_{k_1}U_{h_2}(X_i) Z_{k_2}\big)$$
when $i_2=l_1, l_2=k_1, k_2=i_1$. Indeed, in this case,
$$\varphi\big( Z_{i_1}U_{h_2}(X_i)Z_{i_2}Z_{l_1}U_{h_2}(X_i)Z_{l_2} Z_{k_1}U_{h_2}(X_i) Z_{k_2}\big) = \varphi\big(U_{h_2}(X_i)^3\big). $$
Similarly, replacing $\bs{\mathcal{X}}^{(n)}$ with $\bs{Y}$, one would obtain:
 $$\varphi\big( Z_{i_1}\,Y_i\,Z_{i_2}Z_{l_1}Y_i Z_{l_2} Z_{k_1} Y_i Z_{k_2}\big) = \varphi\big(Y_i^3\big).$$

In order to give another instance of this ``cancelling'' phenomenon, consider the simpler case $d=p=2, h_1=h_2=h, k=1,m_{1,1}=m_{2,1}=1$.
By linearity, in order to compute the expectation \linebreak$\varphi\big( (W_1^{(i)} + V_{1}^{(i)}(\bs{\mathcal{X}}^{(n)}))(W_2^{(i)} + V_{2}^{(i)}(\bs{\mathcal{X}}^{(n)})) \big)$, one has to compute:
\begin{enumerate}
\item $\varphi\big( W_1^{(i)}V_{2}^{(i)}(\bs{\mathcal{X}}^{(n)})\big), \varphi\big( W_2^{(i)}V_{1}^{(i)}(\bs{\mathcal{X}}^{(n)})\big)$, both of which are zero, since the first item in Lemma \ref{3.1bis} applies for each of the summands of its expansion;
\item $\varphi\big((W_1^{(i)} W_{2}^{(i)}) \big)$, which simplifies with the same expectation appearing in the expansion of $\varphi\big( (W_1^{(i)} + V_{1}^{(i)}
(\bs{Y}))(W_2^{(i)} + V_{2}^{(i)}(\bs{Y})) \big)$;
\item $\varphi\big(V_{1}^{(i)}(\bs{\mathcal{X}}^{(n)})V_{2}^{(i)}(\bs{\mathcal{X}}^{(n)})\big)$.
\end{enumerate}
As to the last item, in its expansion, there will appear non-zero summands of the type:
$$\sum_{i_1,j_1 \neq i} f_n^{(1)}(i_1,i)f_n^{(2)}(i,j_1)\varphi\big( Z_{i_1}U_h(X_i)^2 Z_{j_1}\big).$$
When summing over $j_1=i_1$, the corresponding terms in the difference (\ref{differenceBIS}), will be cancelled by the corresponding ones in the expansion of:
$$\sum_{i_1,j_1 \neq i} f_n^{(1)}(i_1,i)f_n^{(2)}(i,j_1)\varphi\big( Z_{i_1}Y_i^2 Z_{j_1}\big),$$
for $i_1=j_1$.\\
\end{exm}

\section{Proof of Theorem \ref{Multiinvariance2}}
The proof of Theorem \ref{Multiinvariance2} is meant to generalize the proof of \cite[Theorem 1.3]{NourdinDeya} in the multidimensional setting. Albeit it follows the same strategy,   some additional difficulties arise here: indeed, the non-commutativity of the variables makes the computations of the joint moments more difficult, and therefore, in order to apply the hypercontractivity argument (which is a fundamental step), one needs to appeal to an iterated Cauchy-Schwarz inequality to bound  an expectation of a product with a certain product of expectations. This leads to deal with some technicalities and parity arguments. Moreover, since we are dealing with Chebyshev sums, an extended version of some auxiliary statements involved in the proof of \cite[Theorem 1.3]{NourdinDeya} is also required.\\

Since $\sum\limits_{i=1}^n \mathrm{Inf}_i(f_n^{(j)}) =1$, we shall assume that $\mathrm{Inf}_i(f_n^{(h)}) \leq 1$ for every $i=1,\dots,n$ and every $h=1,\dots,p$. For the reader's convenience, the proofs of the technical results hereafter quoted are presented in the following separate subsection.

\subsection{Auxiliary statements}

The next lemma (whose proof follows straightforwardly) is meant to generalize \cite[Lemma 3.1]{NourdinDeya}.
\begin{lemma}
\label{3.1bis}
Let $(\mathcal{A},\varphi)$ be a fixed $W^\star$-probability space. Let $\{\mathcal{A}_i\}_{i \geq 1}$ be a sequence of freely independent unital subalgebras of $\mathcal{A}$, and let $\mathcal{B}$ be a unital subalgebra of $\mathcal{A}$, freely independent of $\{\mathcal{A}_i\}_{i \geq 1}$. For random variables $B_1,B_2  \in \mathcal{B}$, and $C_p \in \mathcal{A}_p$, centered and with unit variance, it is:
\begin{itemize}
\item[(i)] $\varphi(C_{p_1}\cdots C_{p_r} B_i C_{p_{r+1}}\cdots C_{p_s}) = 0$,  for every $r, s \geq 0$, and every $p_1,\dots,p_s \in \mathbb{N}$;
\item[(ii)]  if $\mathcal{D}$ is  another unital subalgebra freely independent of $\{\mathcal{A}_i\}_{i \geq 1}$, 
for every $0\leq r < s \leq k$, and $m_1,\dots,m_k \in \mathbb{N}$, such that there exists at least one $j=r+1,\dots,s$ with $m_{j} = 1$, and any centered random variable $Z $ in $\mathcal{D}$ with unit variance, then:
\begin{align*}
\varphi(C_{p_1}^{m_1}\cdots C_{p_r}^{m_r}B_1 C_{p_{r+1}}^{m_{r+1}}\cdots &C_{p_s}^{m_s} B_2 C_{p_{s+1}}^{m_{s+1}}\cdots C_{p_k}^{m_k}) = \\ &\qquad =  \varphi(C_{p_1}^{m_1}\cdots C_{p_r}^{m_r} Z C_{p_{r+1}}^{m_{r+1}}\cdots C_{p_s}^{m_s} Z C_{p_{s+1}}^{m_{s+1}}\cdots C_{p_k}^{m_k}), 
\end{align*}
for every choice of integers $p_1 \neq p_2 \neq \cdots \neq p_{r}$, $p_{r+1}\neq p_{r+2}\neq \cdots \neq p_s$, $p_{s+1}\neq p_{s+2}\neq \cdots \neq p_{k}$;
\item[(iii)] if $B :=B_1 = B_2$, for every $0 \leq r \leq s\leq k, m_j = 0 $ or $m_j \geq 2$ for all $j=r+1,\dots,s$, then:
\small{
\begin{align*}
\varphi(C_{p_1}^{m_1}\cdots C_{p_r}^{m_r}B C_{p_{r+1}}^{m_{r+1}}\cdots & C_{p_s}^{m_s} B C_{p_{s+1}}^{m_{s+1}}\cdots C_{p_k}^{m_k}) = \\
&\qquad =  \varphi(C_{p_1}^{m_1}\cdots C_{p_r}^{m_r} Z C_{p_{r+1}}^{m_{r+1}}\cdots C_{p_s}^{m_s} Z C_{p_{s+1}}^{m_{s+1}}\cdots C_{p_k}^{m_k}). 
\end{align*}}
\normalsize
\end{itemize}
\end{lemma}

For the proof of the Theorem \ref{Multiinvariance2}, the following iterated Cauchy-Schwarz inequality \index{Cauchy-Schwarz inequality}will play a fundamental role.

\begin{lemma}
\label{algo}
Let $c_1,\dots,c_n$ be non-trivial elements in $\mathcal{A}$. Then, setting $\bs{c}=c_1\cdots c_n$:
\begin{enumerate}
\item if $n$ is even:
$$ |\varphi\big(c_1\cdots c_n\big)| \leq \prod_{l=1}^{n}\prod_{s_{j} \in I_l(\bs{c})}\varphi\big((c_l c_l^{\ast})^{2^{s_{j}}}\big)^{2^{-\frac{n}{2}}}\; ,$$
where, for every $l=1,\dots,n$, $I_l(\bs{c})$ is a multiset of integers\footnote{Multisets arise because repetitions may occur.} $s_{j}$ such that $\sum\limits_{j}2^{s_{j}} = 2^{\frac{n}{2}-1}$;
\item if $n \geq 3$ is odd:
$$ |\varphi\big(c_1\cdots c_n\big)| \leq \prod_{l=1}^{\frac{n-1}{2}}\prod_{s_j \in I_l(\bs{c})}\varphi\big((c_l c_l^{\ast})^{2^{s_{j}}}\big)^{2^{-\frac{n-1}{2}}} \cdot \prod_{l= \frac{n+1}{2}}^{n}\prod_{s_j \in I_l(\bs{c})}\varphi\big((c_l c_l^{\ast})^{2^{s_{j}}}\big)^{2^{-\frac{n+1}{2}}} , $$
where, for every $l=1,\dots,n$, $I_l(\bs{c})$ is a multiset of integers $s_{j}\geq 0$ such that $\sum\limits_{j}2^{s_{j}} = 2^{\frac{n-3}{2}}$ for $l=1,\dots,\frac{n-1}{2}$, and  $\sum\limits_{j}2^{s_{j}} = 2^{\frac{n-1}{2}}$ for $l=\frac{n+1}{2},\dots, n$.
\end{enumerate}
\end{lemma}

\begin{rmk}
\label{Osserv}
As made clear in the proof, the multiset $I_l(\bs{c})$ is determined by the rule of association chosen in order to iteratively apply the Cauchy-Schwarz inequality. For the purposes of the present discussion (i.e. the proof of Theorem \ref{Multiinvariance2}), there is no need to further specify the structure of $I_l(\bs{c})$.
\end{rmk}

\begin{exm}
For the sake of clarity, in this example it is shown how the technique of Lemma \ref{algo} applies in the simplest cases $n=2,3,4,5$.

\begin{enumerate}
\item[$(n=2)$] The claim reduces to the standard Cauchy-Schwarz inequality:
$$ |\varphi(c_1 c_2)| \leq \varphi(c_1 c_1^{\ast})^{\frac{1}{2}}\varphi(c_2 c_2^{\ast})^{\frac{1}{2}}.$$
\item[$(n=3)$] The Cauchy-Schwarz inequality, together with the trace property of the state $\varphi$, yields that:
\begin{align*}
|\varphi(c_1 (c_2 c_3))| &\leq \varphi(c_1 c_1^{\ast})^{\frac{1}{2}} \varphi( c_2 c_3 c_3^{\ast} c_2^{\ast})^{\frac{1}{2}} = \varphi(c_1 c_1^{\ast})^{\frac{1}{2}} \varphi((c_2^{\ast} c_2) (c_3 c_3^{\ast}))^{\frac{1}{2}} \\
&\leq \varphi(c_1 c_1^{\ast})^{\frac{1}{2}} \varphi((c_2 c_2^{\ast})^2)^{\frac{1}{4}} \varphi((c_3 c_3^{\ast})^2)^{\frac{1}{4}},
\end{align*}
so that the conclusion of the lemma is achieved by setting $I_1(\bs{c}) = \{0\}$, $I_2(\bs{c})=I_3(\bs{c})=\{1\}$, in such a way that $2^{0}= 2^{\frac{n-3}{2}}$, and $2 = 2^{\frac{n-1}{2}}$. Moreover, $\frac{1}{4} = 2^{-\frac{n+1}{2}}$, and $\frac{1}{2} = 2^{-\frac{n-1}{2}}$.

Note that,  associating the argument of $\varphi$ as $\varphi( (c_1 c_2)c_3)$ would yield:
$$ 
|\varphi(c_1 c_2 c_3)| \leq \varphi((c_1 c_1^{\ast})^2)^{\frac{1}{4}} \varphi((c_2 c_2^{\ast})^2)^{\frac{1}{4}} \varphi(c_3 c_3^{\ast})^{\frac{1}{2}},$$
yielding as multiset $I_1(\bs{c}) = \{1\} = I_2(\bs{c}), I_3(\bs{c}) = \{0\}$ (see Remark \ref{Osserv}).
\item[$(n=4)$]
\begin{align*}
|\varphi((c_1 c_2) (c_3 c_4))| & \leq \varphi\big( (c_1^{\ast}c_1)(c_2 c_2^{\ast})\big)^{\frac{1}{2}}\varphi\big( (c_3^{\ast}c_3)(c_4 c_4^{\ast})\big)^{\frac{1}{2}} \\
&\leq \varphi\big( (c_1^{\ast}c_1)^{2}\big)^{\frac{1}{4}}\varphi\big( (c_2^{\ast}c_2)^{2}\big)^{\frac{1}{4}}\varphi\big( (c_3^{\ast}c_3)^{2}\big)^{\frac{1}{4}}\varphi\big( (c_4^{\ast}c_4)^{2}\big)^{\frac{1}{4}}
\end{align*}
so that the conclusion of the lemma is achieved by setting $I_l(\bs{c}) = \{1\}$ for $l=1,\dots,4$, with $2= 2^{\frac{n}{2}-1}$, and $\frac{1}{4} = 2^{-\frac{n}{2}}$.

\item[$(n=5)$]
\begin{align*}
|\varphi((c_1 c_2) (c_3 c_4 &c_5))| \leq \varphi\big( (c_1^{\ast}c_1)(c_2 c_2^{\ast})\big)^{\frac{1}{2}} \varphi\big(((c_3^{\ast}c_3)c_4)((c_5 c_5^{\ast})c_4^{\ast})\big)^{\frac{1}{2}} \\
&\leq \varphi\big( (c_1^{\ast}c_1)^{2}\big)^{\frac{1}{4}}\varphi\big( (c_2^{\ast}c_2)^{2}\big)^{\frac{1}{4}} \varphi\big( (c_3^{\ast}c_3)^{2}(c_4 c_4^{\ast})\big) ^{\frac{1}{4}} \varphi\big( (c_5^{\ast}c_5)^{2}(c_4 c_4^{\ast})\big) ^{\frac{1}{4}} \\
&\leq \varphi\big( (c_1^{\ast}c_1)^{2}\big)^{\frac{1}{4}}\varphi\big( (c_2^{\ast}c_2)^{2}\big)^{\frac{1}{4}}  \varphi\big( (c_3 c_3^{\ast})^4 \big)^{\frac{1}{8}}\varphi\big( (c_4 c_4^{\ast})^2 \big)^{\frac{1}{8}}\varphi\big( (c_5 c_5^{\ast})^4 \big)^{\frac{1}{8}}\varphi\big( (c_4 c_4^{\ast})^2 \big)^{\frac{1}{8}},
\end{align*}
so that the conclusion of Lemma \ref{algo} is achieved by setting $I_1(\bs{c})=I_{2}(\bs{c}) =\{1\}$, giving $2 = 2^{\frac{n-3}{2}}$ and $\frac{1}{4} = 2^{-\frac{n-1}{2}}$, and $I_{3}(\bs{c}) =I_{5}(\bs{c})=\{ 2\}$ so that $2^{2} = 2^{\frac{n-1}{2}}$, $I_4(\bs{c})=\{1,1\}$, so that $2+2=2^{\frac{n-1}{2}}$, and $\frac{1}{8} = 2^{-\frac{n+1}{2}}$.
\end{enumerate}

\end{exm}

\begin{proof}
Suppose first that $n$ is even, say $n=2k$: the proof will proceed by induction on $k$. If $n=2$, it is the standard Cauchy-Schwarz inequality. 

For $k > 1$, assume that the statement is true for $n=2h$, for all $h \leq k$. If $n=2(k+1)$, apply the Cauchy-Schwarz inequality and the trace property of $\varphi$ in the following way:
\begin{align*}
|\varphi\big((c_1 \cdots c_{k+1})& (c_{k+2}\cdots c_{n})\big)| \leq \varphi\big(c_1 c_2 \cdots c_{k+1}c_{k+1}^{\ast}\cdots c_{2}^{\ast}c_1^{\ast}\big)^{\frac{1}{2}} \varphi\big(c_{k+2}\cdots c_n c_n^{\ast}\cdots c_{k+2}^{\ast}\big)^{\frac{1}{2}}\\
&=\varphi\big((c_1^{\ast}c_1) c_2 \cdots c_k (c_{k+1}c_{k+1}^{\ast})\cdots c_3^{\ast} c_{2}^{\ast}\big)^{\frac{1}{2}} \varphi\big((c_{k+2}^{\ast}c_{k+2})c_{k+3}\cdots (c_n c_n^{\ast})\cdots c_{k+3}^{\ast}\big)^{\frac{1}{2}}.
\end{align*}

Set $A^{2} = \varphi\big((c_1^{\ast}c_1) c_2 \cdots c_k (c_{k+1}c_{k+1}^{\ast})\cdots c_3^{\ast} c_{2}^{\ast}\big)$  and $B^{2}=\varphi\big((c_{k+2}^{\ast}c_{k+2})c_{k+3}\cdots (c_n c_n^{\ast})\cdots c_{k+3}^{\ast}\big)$.

For $A^{2}$, set $\tilde{\bs{c}} = \tilde{c}_1\cdots \tilde{c}_{2k}$, with
\begin{itemize}
\item[-] $\tilde{c}_1 = c_1^{\ast}c_1$,
\item[-] for $j=2,\dots,k$, $\tilde{c}_j = c_j$, 
\item[-] $\tilde{c}_{k+1}= c_{k+1}c_{k+1}^{\ast}$,
\item[-] for $j=0,\dots,k-2$, $\tilde{c}_{k+2+j} = c_{k-j}^{\ast}$,
\end{itemize}
in such a way that $A^{2} = \varphi\big(\tilde{c}_1\tilde{c}_2\cdots \tilde{c}_{k}\cdots \tilde{c}_{2k}\big)$. Since $A^{2} = \varphi(a a^{\ast}) \geq 0$, with $a= c_1 \cdots c_{k+1}$, by the induction hypothesis it follows that:
\begin{align*}
\big(\varphi &\big(\tilde{c}_1\tilde{c}_2\cdots \tilde{c}_{k}\cdots \tilde{c}_{2k}\big)\big)^{\frac{1}{2}}  \leq \prod_{l=1}^{2k}\prod_{s_{j} \in I_l(\bs{\tilde{c}})} \varphi\big((\tilde{c}_l \tilde{c}_l^{\ast})^{2^{s_{j}}}\big)^{2^{-(k+1)}} \\
&= \prod_{s_{j}\in I_1(\bs{\tilde{c}})}\varphi\big((\tilde{c}_{1} \tilde{c}_1^{\ast})^{2^{s_{j}}}\big)^{2^{-(k+1)}}\;  \prod_{l=2}^{k}  \prod_{s_{j}\in I_l(\bs{\tilde{c}})} \varphi\big((\tilde{c}_l \tilde{c}_l^{\ast})^{2^{s_{j}}}\big)^{2^{-(k+1)}} \\
&\qquad \prod_{s_{j}\in I_{k+1}(\bs{\tilde{c}})}\varphi\big((\tilde{c}_{k+1} \tilde{c}_{k+1}^{\ast})^{2^{s_{j}}}\big)^{2^{-(k+1)}} \;\prod_{l=k+2}^{2k}\prod_{s_{j}\in I_l(\bs{\tilde{c}})} \varphi\big((\tilde{c}_l \tilde{c}_l^{\ast})^{2^{s_{j}}}\big)^{2^{-(k+1)}} \;.
\end{align*}
Keeping in mind the definition of the $\tilde{c}_l$'s, one has:
\begin{itemize}
\item[-] $\varphi\big((\tilde{c}_{1} \tilde{c}_1^{\ast})^{2^{s_{j}}}\big) = \varphi\big((c_{1} c_1^{\ast})^{2^{s_{j}+1}}\big)$ for every $s_j \in I_1(\bs{\tilde{c}})$,
\item[-] for $l=2,\dots,k$, $\varphi\big((\tilde{c}_{l} \tilde{c}_l^{\ast})^{2^{s_{j}}}\big) = \varphi\big((c_{l} c_l^{\ast})^{2^{s_{j}}}\big)$ for every $s_j \in I_{l}(\bs{\tilde{c}})$;
\item[-] $\varphi\big((\tilde{c}_{k+1} \tilde{c}_{k+1}^{\ast})^{2^{s_{j}}}\big) = \varphi\big((c_{k+1} c_{k+1}^{\ast})^{2^{s_{j}+1}}\big)$ for every $s_j \in I_{k+1}(\bs{\tilde{c}})$,
\item[-] for $l= k+2,\dots,2k$, $\varphi\big((\tilde{c}_{l} \tilde{c}_l^{\ast})^{2^{s_{j}}}\big) = \varphi\big((c_{2k-l+2} c_{2k-l+2}^{\ast})^{2^{s_{j}}}\big)$ for every $s_j \in I_l(\bs{\tilde{c}})$, so that:
$$ \prod_{l=k+2}^{2k}\prod_{s_{j}\in I_l(\bs{\tilde{c}})} \varphi\big((\tilde{c}_l \tilde{c}_l^{\ast})^{2^{s_{j}}}\big)^{2^{-(k+1)}} = \prod_{h=2}^{k}\prod_{s_{j}\in I_{2k-h+2}(\bs{\tilde{c}})} \varphi\big((c_h c_h^{\ast})^{2^{s_{j}}}\big)^{2^{-(k+1)}}.$$
\end{itemize}
Finally, writing $\bs{c} = c_1\cdots c_n$, and setting:
\begin{itemize}
\item[-] $I_1(\bs{c}) = I_1(\bs{\tilde{c}}) + 1 := \{ s_{j} +1: s_{j} \in I_1(\bs{\tilde{c}})\}$;
\item[-] $I_{k+1}(\bs{c}) = I_{k+1}(\bs{\tilde{c}}) + 1 :=\{ s_{j} +1: s_{j} \in I_{k+1}(\bs{\tilde{c}})\}$;
\item[-] for $l=2,\dots,k$, $I_{l}(\bs{c}) = I_{l}(\bs{\tilde{c}}) \cup I_{2k-l+2}(\bs{\tilde{c}})$,
\end{itemize}
in such a way that $\sum\limits_{s_{j} \in I_{l}(\bs{c})}2^{s_{j}} = 2^{k}$ for every $l=1,\dots,k+1$, it follows that:
$$ \big(\varphi\big(\tilde{c}_1\tilde{c}_2\cdots \tilde{c}_{k}\cdots \tilde{c}_{2k}\big)\big)^{\frac{1}{2}}
\leq \prod_{l=1}^{k+1}\prod_{s_{j} \in I_{l}(\bs{c})} \varphi\big( (c_l c_l^{\ast})^{2^{s_{j}}}\big)^{2^{-(k+1)}}.$$
 


In the same way, setting:
\begin{itemize}
\item[-] $d_1= c_{k+2}^{\ast}c_{k+2}$;
\item[-] $d_{j+1} = c_{k+j+2}$ for $j=1,\dots,k-1$;
\item[-] $d_{k+1} = c_n c_{n}^{\ast}$;
\item[-] $d_{k+j+1} = c_{n-j}^{\ast}$ for $j=1,\dots,k-1$.
\end{itemize}
a similar estimate for $B = \varphi\big(d_1 d_2\cdots d_{2k}\big)^{\frac{1}{2}}$ can be obtained. Indeed, if $\bs{d}= d_1\cdots d_{2k}$, from the induction hypothesis it follows that:
\begin{align*}
B &\leq \prod_{l=1}^{2k}\prod_{t_j \in I_l(\bs{d})} \varphi\big((d_l d_l^{\ast})2^{t_{j}}\big)^{2^{-(k+1)}} \\
&= \prod_{t_{j}\in I_1(\bs{d})}\varphi\big((d_1 d_1^{\ast})^{2^{t_{j}}}\big)^{2^{-(k+1)}}\;  \prod_{l=2}^{k}  \prod_{t_{j}\in I_l(\bs{d})} \varphi\big((d_l d_l^{\ast})^{2^{t_{j}}}\big)^{2^{-(k+1)}} \\
& \prod_{t_{j}\in I_{k+1}(\bs{d})}\varphi\big((d_{k+1} d_{k+1}^{\ast})^{2^{t_{j}}}\big)^{2^{-(k+1)}} \;\prod_{l=k+2}^{2k}\prod_{t_{j}\in I_l(\bs{d})} \varphi\big((d_l d_l^{\ast})^{2^{t_{j}}}\big)^{2^{-(k+1)}} 
\end{align*}

As for $A^2$,  by considering the definition of the $d_l$'s, one has that:
\begin{itemize}
\item[-] $\varphi\big((d_{1} d_1^{\ast})^{2^{t_{j}}}\big) = \varphi\big((c_{k+2} c_{k+2}^{\ast})^{2^{t_{j}+1}}\big)$ for every $t_j \in I_1(\bs{d})$,
\item[-] for $l=2,\dots,k$, $ \varphi\big((d_{l} d_l^{\ast})^{2^{t_{j}}}\big) = \varphi\big((c_{k+l+1} c_{k+l+1}^{\ast})^{2^{t_{j}}}\big)$ for every $t_j \in I_{l}(\bs{d})$, so that:
$$ \prod_{l=2}^{k}\prod_{t_{j}\in I_l(\bs{d})} \varphi\big((d_l d_l^{\ast})^{2^{t_{j}}}\big)^{2^{-(k+1)}} = \prod_{h=k+3}^{n-1}\prod_{t_{j}\in I_{h-k-1}} \varphi\big((c_h c_h^{\ast})^{2^{t_{j}}}\big)^{2^{-(k+1)}}.$$
\item[-] $\varphi\big((d_{k+1} d_{k+1}^{\ast})^{2^{t_{j}}}\big) = \varphi\big((c_{n} c_{n}^{\ast})^{2^{t_{j}+1}}\big)$ for every $t_j \in I_{k+1}(\bs{d})$,
\item[-] for $l= k+2,\dots,2k$, $\varphi\big((d_{l} d_l^{\ast})^{2^{t_{j}}}\big) = \varphi\big((c_{n-l+k+1} c_{n-l+k+1}^{\ast})^{2^{t_{j}+1}}\big)$ for every $t_j \in I_l(\bs{d})$, so that:
$$ \prod_{l=k+2}^{2k}\prod_{t_{j}\in I_l(\bs{d})} \varphi\big((d_l d_l^{\ast})^{2^{t_{j}}}\big)^{2^{-(k+1)}} = \prod_{h=k+3}^{n-1}\prod_{t_{j}\in I_{n-h+k+1}(\bs{d})} \varphi\big((c_h c_h^{\ast})^{2^{t_{j}}}\big)^{2^{-(k+1)}}.$$
\end{itemize}

Finally, writing $\bs{c} = c_1\cdots c_n$, and setting:
\begin{itemize}
\item[-] $I_{k+2}(\bs{c}) = I_1(\bs{d}) + 1 := \{ s_{j} +1: s_{j} \in I_1(\bs{\tilde{c}})\}$;
\item[-] $I_{n}(\bs{c}) = I_{k+1}(\bs{d}) + 1$;
\item[-] for $h=k+3,\dots,n-1$, $I_{h}(\bs{c}) = I_{h-k-1}(\bs{d}) \cup I_{n-h+k+1}(\bs{d})$,
\end{itemize}
in such a way that $\sum\limits_{t_{j} \in I_{l}(\bs{c})}2^{t_{j}} = 2^{k}$ for every $l=k+2,\dots,n$, it follows that:
$$ \big(\varphi\big(d_1 d_2\cdots d_{k}\cdots d_{2k}\big)\big)^{\frac{1}{2}}
\leq \prod_{l=k+2}^{n}\prod_{t_{j} \in I_{l}(\bs{c})} \varphi\big( (c_l c_l^{\ast})^{2^{t_{j}}}\big)^{2^{-(k+1)}}.$$

Hence, at the end:
\begin{align*}
|\varphi\big((c_1 \cdots c_{k+1}) (c_{k+2}\cdots c_{n})\big)| &\leq \prod_{l=1}^{n}\prod_{s_j \in I_l(\bs{c})} \varphi\big((c_l c_l^{\ast})^{2^{s_{j}}}\big)^{2^{-(k+1)}},
\end{align*}
with $\sum\limits_{s_j \in I_{l}(\bs{c})}2^{s_{j}} = 2^{k}$ for every $l=1,\dots,n$. Hence, the claim is true for all strings $c_1\cdots c_n$ of even length.\\

Assume now that $n$ is odd; the conclusion will follow again by induction. If $n=3$, apply the Cauchy-Schwarz inequality in the following way:
\begin{equation*} 
|\varphi\big(c_1 (c_2 c_3)\big)| \leq \varphi\big((c_1^{\ast}c_1)\big)^{\frac{1}{2}} \big(\varphi\big((c_2 c_2^{\ast})(c_3 c_3^{\ast})\big)\big)^{\frac{1}{2}} \leq \big(\varphi\big((c_1^{\ast}c_1)\big)\big)^{\frac{1}{2}} \big(\varphi\big((c_2^{\ast}c_2)^{2}\big)\big)^{\frac{1}{4}}\big(\varphi\big((c_3^{\ast}c_3)^{2}\big)\big)^{\frac{1}{4}}.
\end{equation*} 
For $k >1$, assume that the result holds true for every odd integer $n=2l+1$, with $l \leq k$. Let $n=2(k+1)+1 = 2k +3$ and apply the Cauchy-Schwarz inequality as follows:
\small{
\begin{align*}
|\varphi &\big((c_1 \cdots c_{\frac{n-1}{2}}) (c_{\frac{n+1}{2}}\cdots c_n)\big)| \leq \\
&\leq \bigg(\varphi\big((c_1^{\ast}c_1)c_2\cdots c_{\frac{n-3}{2}} (c_{\frac{n-1}{2}}c_{\frac{n-1}{2}}^{\ast})c_{\frac{n-3}{2}}^{\ast}\cdots c_{2}^{\ast}\big)\bigg)^{\frac{1}{2}}\; 
\bigg(\varphi\big( (c_{\frac{n+1}{2}}^{\ast}c_{\frac{n+1}{2}})c_{\frac{n+3}{2}}\cdots (c_n c_{n}^{\ast})c_{n-1}^{\ast}\cdots c_{\frac{n+3}{2}}^{\ast} \big)\bigg)^{\frac{1}{2}}.
\end{align*}
}\normalsize
If $A^{2}=\varphi\big((c_1^{\ast}c_1)c_2\cdots c_{\frac{n-3}{2}} (c_{\frac{n-1}{2}}c_{\frac{n-1}{2}}^{\ast})c_{\frac{n-3}{2}}^{\ast}\cdots c_{2}^{\ast}\big)$, set:
\begin{itemize}
\item[-] $\tilde{c}_1 := c_1^{\ast}c_1$,
\item[-] $\tilde{c}_j := c_j$, for $j=2,\dots,\frac{n-3}{2}$,
\item[-] $\tilde{c}_{\frac{n-1}{2}} := c_{\frac{n-1}{2}}c_{\frac{n-1}{2}}^{\ast}$, 
\item[-] $\tilde{c}_{\frac{n-1}{2}+j} := c_{\frac{n-1}{2}-j}^{\ast}$, for $j=1,\dots,\frac{n-5}{2}$,
\end{itemize}
in such a way that $ A^{2} = \varphi(\tilde{c}_1\cdots \tilde{c}_{2k}) = \varphi(a a^{\ast}) \geq 0$, and so, the statement for string of even length for $2k= n-3$ implies that:
\begin{align*}
\varphi (\tilde{c}_1\cdots \tilde{c}_{2k})^{\frac{1}{2}} &\leq \prod_{l=1}^{2k} \prod_{t_j \in I_l(\bs{\tilde{c}})} \varphi\big((\tilde{c}_l \tilde{c}_l^{\ast})^{2^{t_{j}}}\big)^{2^{-(k+1)}} \\
&= \prod_{t_j \in I_1(\bs{\tilde{c}})} \varphi\big((\tilde{c}_1 \tilde{c}_1^{\ast})2^{t_{j}}\big)^{2^{-(k+1)}}  \prod_{l=2}^{k} \prod_{t_j \in I_l(\bs{\tilde{c}})} \varphi\big((\tilde{c}_l \tilde{c}_l^{\ast})^{2^{t_{j}}}\big)^{2^{-(k+1)}} \\
& \prod_{t_j \in I_{k+1}(\bs{\tilde{c}})} \varphi\big((\tilde{c}_{k+1} \tilde{c}_{k+1}^{\ast})^{2^{t_{j}+1}}\big)^{2^{-(k+1)}} \cdot \prod_{l=k+2}^{2k}  \prod_{t_j \in I_l(\bs{\tilde{c}})} \varphi\big((\tilde{c}_l \tilde{c}_l^{\ast})^{2^{t_j}}\big)^{2^{-(k+1)}}
\end{align*}
where $\tilde{\bs{c}} = \tilde{c}_1\cdots \tilde{c}_{2k}$, with $\sum\limits_{t_j \in I_l(\bs{\tilde{c}})} 2^{t_{j}} = 2^{k-1}$ for every $l=1,\dots,2k$.
 
Again, by keeping in mind the definition of the $\tilde{c}_l$'s:
\begin{itemize}
\item[-] $\varphi\big((\tilde{c}_{1} \tilde{c}_1^{\ast})^{2^{t_{j}}}\big) = \varphi\big((c_{1} c_1^{\ast})^{2^{t_{j}+1}}\big)$ for every $t_j \in I_1(\bs{\tilde{c}})$,
\item[-] for $l=2,\dots, k=\frac{n-3}{2}$, $ \varphi\big((\tilde{c}_{l} \tilde{c}_l^{\ast})^{2^{t_{j}}}\big) = \varphi\big((c_{l} c_l^{\ast})^{2^{t_{j}}}\big)$ for every $t_j \in I_{l}(\bs{\tilde{c}})$;
\item[-] $\varphi\big((\tilde{c}_{k+1} \tilde{c}_{k+1}^{\ast})^{2^{t_{j}}}\big) = \varphi\big((c_{\frac{n-1}{2}} c_{\frac{n-1}{2}}^{\ast})^{2^{t_{j}+1}}\big)$ for every $t_j \in I_{k+1}(\bs{\tilde{c}})$ \;(note that $k+ 1 = \frac{n-1}{2}$),
\item[-] for $l= k+2,\dots,2k$, $\varphi\big((\tilde{c}_{l} \tilde{c}_l^{\ast})^{2^{s_{j}}}\big) = \varphi\big((c_{2k-l+2} c_{2k-l+2}^{\ast})^{2^{s_{j}+1}}\big)$ for every $s_j \in I_l(\bs{\tilde{c}})$, so that:
$$ \prod_{l=k+2}^{2k}\prod_{s_{j}\in I_l(\bs{\tilde{c}})} \varphi\big((\tilde{c}_l \tilde{c}_l^{\ast})^{2^{s_{j}}}\big)^{2^{-(k+1)}} = \prod_{h=2}^{\frac{n-3}{2}}\prod_{s_{j}\in I_{n-1-h}(\bs{\tilde{c}})} \varphi\big((c_h c_h^{\ast})^{2^{s_{j}}}\big)^{2^{-(k+1)}}.$$
\end{itemize}
Finally, writing $\bs{c} = c_1\cdots c_n$, and setting:
\begin{itemize}
\item[-] $I_1(\bs{c}) = I_1(\bs{\tilde{c}}) + 1 := \{ s_{j} +1: s_{j} \in I_1(\bs{\tilde{c}})\}$;
\item[-] $I_{k+1}(\bs{c}) = I_{k+1}(\bs{\tilde{c}}) + 1$;
\item[-] for $l=2,\dots,k = \frac{n-3}{2}$, $I_{l}(\bs{c}) = I_{l}(\bs{\tilde{c}}) \cup I_{n-1-l}(\bs{\tilde{c}})$,
\end{itemize}
so that $\sum\limits_{s_{j} \in I_{l}(\bs{c})}2^{s_{j}} = 2^{k}$ for every $l=1,\dots,k+1$, it follows that:
$$ \big(\varphi\big(\tilde{c}_1\tilde{c}_2\cdots \tilde{c}_{k}\cdots \tilde{c}_{2k}\big)\big)^{\frac{1}{2}}
\leq \prod_{l=1}^{\frac{n-1}{2}}\prod_{s_{j} \in I_{l}(\bs{c})} \varphi\big( (c_l c_l^{\ast})^{2^{s_{j}}}\big)^{2^{-(k+1)}},$$
with $\sum\limits_{t_j \in I_{l}(\bs{c})}2^{t_j} = 2^{k} = 2^{\frac{n-3}{2}}$ for every $l=1,\dots,\frac{n-1}{2}$.\\

Similarly, for  $B^{2}=\varphi\big( (c_{\frac{n+1}{2}}^{\ast}c_{\frac{n+1}{2}})c_{\frac{n+3}{2}}\cdots (c_n c_{n}^{\ast})c_{n-1}^{\ast}\cdots c_{\frac{n+3}{2}}^{\ast}\big)$, set:
\begin{itemize}
\item[-] $d_1:= c_{\frac{n+1}{2}}^{\ast}c_{\frac{n+1}{2}}$,
\item[-] for $j=2,\dots,\frac{n-1}{2}$, $\tilde{d}_{j} = c_{\frac{n+1}{2}+j-1}$ (so $\tilde{d}_{\frac{n-1}{2}} = c_{n-1}$),
\item[-] $d_{\frac{n+1}{2}} = c_n c_n^{\ast}$,
\item[-] for all $j=1,\dots,\frac{n-3}{2}, d_{\frac{n+1}{2} +j} = c_{n-j}^{\ast}$,
\end{itemize}
so that $B^{2} = \varphi\big(d_1\cdots d_{n-1}\big)$, and the claim for the string of even length $n-1= 2(k+1)$ applies to get:
\begin{align*}
&\varphi\big(d_1\cdots d_{n-1}\big)^{\frac{1}{2}} \leq \prod_{l=1}^{n-1}\prod_{t_j \in I_l(\bs{d})} \varphi\big((d_j d_j^{\ast})^{2^{t_{j}}}\big)^{2^{-(k+2)}} \\
&= \prod_{t_{j}\in I_1(\bs{d})}\varphi\big((d_1 d_1^{\ast})^{2^{t_{j}}}\big)^{2^{-(k+2)}}\;  \prod_{l=2}^{k+1}  \prod_{t_{j}\in I_l(\bs{d})} \varphi\big((d_l d_l^{\ast})^{2^{t_{j}}}\big)^{2^{-(k+2)}} \\
& \prod_{t_{j}\in I_{\frac{n+1}{2}}(\bs{d})}\varphi\big((d_{\frac{n+1}{2}} d_{\frac{n+1}{2}}^{\ast})^{2^{t_{j}}}\big)^{2^{-(k+2)}} \;\prod_{l=\frac{n+3}{2}}^{n-1}\prod_{t_{j}\in I_l(\bs{d})} \varphi\big((d_l d_l^{\ast})^{2^{t_{j}}}\big)^{2^{-(k+2)}} \;.
\end{align*}

As for $A^2$, by considering the definition of the $d_l$'s:
\begin{itemize}
\item[-] $\varphi\big((d_{1} d_1^{\ast})^{2^{t_{j}}}\big) = \varphi\big((c_{k+2} c_{k+2}^{\ast})^{2^{t_{j}+1}}\big)$ for every $t_j \in I_1(\bs{d})$, being $\dfrac{n+1}{2} = k+2$;
\item[-]for $l=2,\dots,k+1$, $\varphi\big((d_{l} d_l^{\ast})^{2^{t_{j}}}\big) = \varphi\big((c_{k+l+1} c_{k+l+1}^{\ast})^{2^{t_{j}}}\big)$ for every $t_j \in I_{l}(\bs{d})$, so that:
$$ \prod_{l=2}^{k+1}\prod_{t_{j}\in I_l(\bs{d})} \varphi\big((d_l d_l^{\ast})^{2^{t_{j}}}\big)^{2^{-(k+2)}} = \prod_{h=k+3}^{n-1}\prod_{t_{j}\in I_{h-k-1}(\bs{d})} \varphi\big((c_h c_h^{\ast})^{2^{t_{j}}}\big)^{2^{-(k+2)}}.$$
\item[-] $\varphi\big((d_{k+2} d_{k+2}^{\ast})^{2^{t_{j}}}\big) = \varphi\big((c_{n} c_{n}^{\ast})^{2^{t_{j}+1}}\big)$ for every $t_j \in I_{k+2}(\bs{d})$ (being $k+2= \frac{n+1}{2}$);
\item[-] for $l= k+3,\dots,n-1$, $\varphi\big((d_{l} d_l^{\ast})^{2^{t_{j}}}\big) = \varphi\big((c_{n-l+k+2} c_{n-l+k+2}^{\ast})^{2^{t_{j}+1}}\big)$ for every $t_j \in I_l(\bs{d})$, so that:
$$ \prod_{l=k+3}^{n-1}\prod_{t_{j}\in I_l(\bs{d})} \varphi\big((d_l d_l^{\ast})^{2^{t_{j}}}\big)^{2^{-(k+2)}} = \prod_{h=k+3}^{n-1}\prod_{t_{j}\in I_{n+k+2-h}(\bs{d})} \varphi\big((c_h c_h^{\ast})^{2^{t_{j}}}\big)^{2^{-(k+2)}}.$$
\end{itemize}

Finally, writing $\bs{c} = c_1\cdots c_n$, and setting:
\begin{itemize}
\item[-] $I_{k+2}(\bs{c}) = I_1(\bs{d}) + 1 := \{ s_{j} +1: s_{j} \in I_1(\bs{d})\}$;
\item[-] $I_{n}(\bs{c}) = I_{k+1}(\bs{d}) + 1$;
\item[-] for $h=k+3,\dots,n-1$, $I_{h}(\bs{c}) = I_{h-k-1}(\bs{d}) \cup I_{n+k+2-h}(\bs{d})$,
\end{itemize}
in such a way that $\sum\limits_{t_{j} \in I_{l}(\bs{c})}2^{t_{j}} = 2^{k+1} = 2^{\frac{n-1}{2}}$ for every $l=k+2,\dots,n$, it follows that:
$$ \big(\varphi\big(d_1 d_2\cdots d_{k}\cdots d_{2k}\big)\big)^{\frac{1}{2}}
\leq \prod_{l=k+2}^{n}\prod_{t_{j} \in I_{l}(\bs{c})} \varphi\big( (c_l c_l^{\ast})^{2^{t_{j}}}\big)^{2^{-(k+2)}},$$
yielding the desired conclusion.
\end{proof}

\subsection{The proof of Theorem \ref{Multiinvariance2}}

Let us start by assuming that both $\bs{X}$ and $\bs{Y}$ are composed of identically distributed random variables.\\

Consider the auxiliary ensembles $\bs{\mathcal{Z}}^{(i)} = (\bs{Y}_1,\dots,\bs{Y}_{i-1},\bs{\mathcal{X}}_i,\dots,\bs{\mathcal{X}}_n)$, with $\bs{Y}_i = \underbrace{(Y_i,\dots,Y_i)}_{d \text{ times }}$ and $\bs{\mathcal{X}}_i= (U_{h_1}(X_i),\dots,U_{h_d}(X_i))$, and the identities in \eqref{differenceBIS}.
By applying simultaneously the free binomial expansion (see Lemma \ref{freebin}) to each $W_j^{(i)} + V_j^{(i)}(\bs{\mathcal{X}}_i)$, for every $i=1,\dots,n$:
\begin{align*}
\varphi\bigg(\prod_{s=1}^{k}\big(W_1^{(i)}+ V_1^{(i)}&(\bs{\mathcal{X}}_i)\big)^{m_{1,s}}\cdots \big(W_p^{(i)}+ V_p^{(i)}(\bs{\mathcal{X}}_i)\big)^{m_{p,s}}\bigg) = \varphi\bigg(\prod_{s=1}^{k}(W_1^{(i)})^{m_{1,s}}\cdots (W_n^{(i)})^{m_{p,s}}\bigg)  \\
&+ \sum_{\bs{v} \in \mathcal{D}}\varphi\bigg(\prod_{s=1}^{k}\prod_{l=1}^{p}(W_{l}^{(i)})^{\alpha_{l,1}^{(s)}}V_l^{(i)}(\bs{\mathcal{X}}_i)^{\beta_{l,1}^{(s)}}\cdots (W_{l}^{(i)})^{\alpha_{l,r_l}^{(s)}}V_l^{(i)}(\bs{\mathcal{X}}_i)^{\beta_{l,r_l}^{(s)}}\bigg),
\end{align*}
where, in each summand, at least one $\beta_{l,j}^{(s)} \geq 1$ and with 
$$ \mathcal{D} = \{\bs{v} = (r_l^{(s)}, \bs{\alpha}_{l}^{(s)}, \bs{\beta}_{l}^{(s)}) \in \mathcal{D}_{n_{l,s},m_{l,s}}: s=1,\dots,k, \;l=1,\dots,p,\; n_{l,s}=1,\dots,m_{l,s} \},$$
\small
$$ \mathcal{D}_{n_{l,s},m_{l,s}} = \big\{(r_l^{(s)}, \bs{\alpha}_{l}^{(s)}, \bs{\beta}_{l}^{(s)}): r_{l}^{(s)} \in [n_{l,s}], \bs{\alpha}_{l}^{(s)},  \bs{\beta}_{l}^{(s)} \in \mathbb{N}^{r_{l}^{(s)}},\sum_{h=1}^{r_l^{(s)}}\alpha_{l,h}^{(s)} = m_{l,s}- n_{l,s}, \sum_{h=1}^{r_l^{(s)}}\beta_{l,h}^{(s)} = n_{l,s}\big\}\, .$$
\normalsize
Similarly,
\begin{align*}
\varphi\bigg(\prod_{s=1}^{k}\big(W_1^{(i)}+ V_1^{(i)}&(\bs{Y}_i)\big)^{m_{1,s}}\cdots \big(W_p^{(i)}+ V_p^{(i)}(\bs{Y}_i)\big)^{m_{p,s}}\bigg) = \varphi\bigg(\prod_{s=1}^{k}(W_1^{(i)})^{m_1,s}\cdots (W_p^{(i)})^{m_{p,s}}\bigg)  \\
&+ \sum_{\bs{v} \in \mathcal{D}}\varphi\bigg(\prod_{s=1}^{k}\prod_{l=1}^{p}(W_{l}^{(i)})^{\alpha_{l,1}^{(s)}}V_l^{(i)}(\bs{Y}_i)^{\beta_{l,1}^{(s)}}\cdots (W_{l}^{(i)})^{\alpha_{l,r_l}^{(s)}}V_l^{(i)}(\bs{Y}_i)^{\beta_{l,r_l}^{(s)}}\bigg),
\end{align*}
where at least on $\beta_{l,j}^{(s)} \geq 1$. Hence, the term $\varphi\Big(\prod\limits_{s=1}^{k}(W_1^{(i)})^{m_1,s}\cdots (W_p^{(i)})^{m_{p,s}}\Big)$ cancels out in the difference (\ref{differenceBIS}).
Set:
$$ a_{s,l}^{(i)} := (W_{l}^{(i)})^{\alpha_{l,1}^{(s)}}V_l^{(i)}(A)^{\beta_{l,1}^{(s)}}\cdots (W_{l}^{(i)})^{\alpha_{l,r_l}^{(s)}}V_l^{(i)}(A)^{\beta_{l,r_l}^{(s)}},$$
for $s=1,\dots,k, \text{and } l=1,\dots,p$ and $A = \mathcal{X}_{i,r}$ or $A = Y_i$.\\

If $\text{Alg}(R_1,\dots,R_k)$ denotes the algebra generated by the random variables $R_1,\dots,R_k$, for a fixed $i=1,\dots,n$, by virtue of Lemma \ref{3.1bis} applied with
\begin{enumerate}
\item  $\mathcal{A}_j = \text{Alg}(1,U_{h_1}(X_j),\dots,U_{h_d}(X_j))$ for every $j > i$;
\item $\mathcal{A}_j = \text{Alg}(1,Y_j)$ for every $j < i$;
\item  $\mathcal{B} = \text{Alg}(1,U_{h_1}(X_i),\dots,U_{h_d}(X_i))$;
\item $\mathcal{D} = \text{Alg}(1,Y_i)$,
\end{enumerate}
if $\gamma:=\sum\limits_{s=1}^{k}\sum\limits_{l=1}^{p}\sum\limits_{h=1}^{r_l}\beta_{l,h}^{(s)} \leq 2$, the terms $\varphi\Big(\prod\limits_{s=1}^{k}\prod\limits_{l=1}^{p}a_{s,l}^{(i)}\Big)$ relative to $A = \mathcal{X}_{i,r}$  either are zero or cancel with the corresponding ones associated with $A = Y_{i}$.

Indeed, if $\gamma = 1$, in the argument of $x:= \varphi\Big(\prod\limits_{s=1}^k \prod\limits_{l=1}^p a_{s,l}^{(i)}\Big)$, there will only be a factor of the type $U_{h_l}(X_i)$, so that $x=0$ by virtue of the first item in Lemma \ref{3.1bis}. If $\gamma = 2$, either there is only one exponent $\beta_{l,j}^{(s)} = 2$ or two different ones equal to $1$: in both cases,  either the second or the third item in Lemma \ref{3.1bis} applies, thanks to the hypothesis $h_i=h_{d-i+1}$ for $i=1,\dots, \lfloor \frac{d}{2}\rfloor$.\\

Therefore, the remaining terms to bound in \eqref{differenceBIS} are of the type
$$ \varphi\big((a_{1,1}^{(i)}\cdots a_{1,p}^{(i)}) \cdots (a_{k,1}^{(i)}\cdots a_{k,p}^{(i)})\big),$$
whose corresponding parameter $\gamma=\sum\limits_{s=1}^{k}\sum\limits_{l=1}^{p}\sum\limits_{h=1}^{r_l}\beta_{l,h}^{(s)}$ verifies $\gamma \geq 3;$ from here apply the triangle inequality for the absolute value.\\

The first step of the proof consists in applying the iterated Cauchy-Schwarz inequality, described in Lemma \ref{algo}, in the following way:
\begin{itemize}
\item[(i)] when $k$ is even, and therefore $kp$ is even, associate the first $\frac{k}{2}$ $p$-string $a_{j,1}^{(i)}\cdots a_{j,p}^{(i)}$ and the last ones, as follows:
$$ \big|\varphi\big[ \big( (a_{1,1}^{(i)}\cdots a_{1,p}^{(1)}) \cdots (a_{\frac{k}{2},1}^{(i)} \cdots a_{\frac{k}{2},p}^{(i)}) \big) \big( (a_{\frac{k}{2}+1,1}^{(i)}\cdots a_{\frac{k}{2}+1,p}^{(i)}) \cdots (a_{k,1}^{(i)} \cdots a_{k,p}^{(i)}) \big)  \big]\big|, $$
and apply the technique explained in the proof of Lemma \ref{algo}; 

\item[(ii)] if $p$ is even and $k$ is odd, first split the central $p$-string in the $\frac{k+1}{2}$-th position:
$$ (a_{\frac{k+1}{2},1}^{(i)}\cdots a_{\frac{k+1}{2}, \frac{p}{2}}^{(i)})  (a_{\frac{k+1}{2},\frac{p}{2}+1}^{(i)}\cdots a_{\frac{k+1}{2},p}^{(i)}) $$
so that the argument of $\varphi$ will be divided into two parts, each with $\frac{kp}{2}$ factors, and apply Lemma \ref{algo} to:
$$ \big|\varphi\big[ \big( (a_{1,1}^{(i)}\cdots a_{1,p}^{(i)}) \cdots  (a_{\frac{k+1}{2},1}^{(i)}\cdots a_{\frac{k+1}{2}, \frac{p}{2}}^{(i)}) \big) \big( (a_{\frac{k+1}{2},\frac{p}{2}+1}^{(i)}\cdots a_{\frac{k+1}{2},p}^{(i)}) \cdots (a_{k,1}^{(i)} \cdots a_{k,p}^{(i)}) \big)  \big]\big| \; ; $$

\item[(iii)] if both $k$ and $p$ are odd,  associate the argument of $\varphi$ by splitting between $a_{\frac{k+1}{2},\frac{p-1}{2}}^{(i)}$ and $a_{\frac{k+1}{2},\frac{p+1}{2}}^{(i)}$ (dividing the product into two parts, the first with $\frac{kp-1}{2}$ factors, the second with $\frac{kp+1}{2}$):
$$\big| \varphi\big[ \big( (a_{1,1}^{(i)}\cdots a_{1,p}^{(i)}) \cdots (a_{\frac{k+1}{2},1}^{(i)}\cdots a_{\frac{k+1}{2},\frac{p-1}{2}}^{(i)})\big)\big((a_{\frac{k+1}{2},\frac{p+1}{2}}^{(i)}\cdots a_{\frac{k+1}{2},p}^{(i)}) \cdots (a_{k,1}^{(i)} \cdots a_{k,p}^{(i)}) \big)  \big]\big|.$$
\end{itemize}
If $kp$ is even (both if $k$ is even or $k$ is odd), it follows straightforwardly from Lemma \ref{algo} that:
\begin{equation}
\label{pari}
\big|\varphi\big((a_{1,1}^{(i)}\cdots a_{1,p}^{(i)}) \cdots (a_{k,1}^{(i)}\cdots a_{k,p}^{(i)})\big)\big|  \leq \prod_{s=1}^{k}\prod_{l=1}^{p}\prod_{t_j \in I_{l,s}(\bs{a})} \Big( \varphi\big( (a_{s,l}^{(i)} (a_{s,l}^{(i)})^{\ast})^{2^{t_{j}}} \big)\Big)^{2^{-\frac{kp}{2}}},
\end{equation}
with $\sum\limits_{t_j \in I_{l,s}(\bs{a})} 2^{t_j} = 2^{\frac{kp}{2}-1}$ for every $l=1,\dots,p$, $s=1,\dots,k$.\\
If $p$ and $k$ are odd, then Lemma \ref{algo} gives:
\small{
\begin{align*}
&\big|\varphi\big((a_{1,1}^{(i)}a_{1,2}^{(i)} \cdots a_{1,p}^{(i)}) \cdots (a_{k,1}^{(i)}\cdots a_{k,p}^{(i)})\big)\big|  \\ 
&\qquad\leq \prod_{s=1}^{\frac{k-1}{2}} \prod_{l=1}^{p}\prod_{t_j \in I_{l,s}(\bs{a})} \Big(\varphi\big( (a_{s,l}^{(i)}  (a_{s,l}^{(i)})^{\ast})^{2^{t_{j}}}\big)\Big)^{2^{-\frac{kp-1}{2}}}  \prod_{b=1}^{\frac{p-1}{2}}  \prod_{t_j \in I_{b,\frac{k+1}{2}}(\bs{a})}  \Big(\varphi\big( (a_{\frac{k+1}{2},b}^{(i)} (a_{\frac{k+1}{2},b}^{(i)})^{\ast})^{2^{t_{j}}}\big)\Big)^{2^{-\frac{kp-1}{2}}}    \\ 
&\qquad\prod_{b= \frac{p+1}{2}}^{p}\prod_{t_j \in I_{b,\frac{k+1}{2}}(\bs{a})} \Big(\varphi\big( (a_{\frac{k+1}{2},b}^{(i)} (a_{\frac{k+1}{2},b}^{(i)})^{\ast})^{2^{t_{j}}}\big)\Big)^{2^{-\frac{kp+1}{2}}} \prod_{s=\frac{k+3}{2}}^{k}\prod_{l=1}^p   \prod_{t_j \in I_{l,s}(\bs{a})} \Big(\varphi\big( (a_{s,l}^{(i)}  (a_{s,l}^{(i)})^{\ast})^{2^{t_{j}}}\big)\Big)^{2^{-\frac{kp+1}{2}}}                 \numberthis \label{dispari}
\end{align*}
}
\normalsize
with
\begin{enumerate}
\item  $\sum\limits_{t_j \in I_{l,s}(\bs{a})}2^{t_{j}} = 2^{\frac{kp-3}{2}}$ for $s=1,\dots, \frac{k-1}{2}$ and $l=1,\dots,p$ and $s=\frac{k+1}{2}$, $l=1,\dots,\frac{p-1}{2}$;
\item $\sum\limits_{t_j \in I_{l,s}(\bs{a})}2^{t_{j}} = 2^{\frac{kp-1}{2}}$, for $l=1,\dots,p$ when $s=\frac{k+3}{2},\dots,k$, and when $s=\frac{k+1}{2}$, for $l=\frac{p+1}{2},\dots,p$.
\end{enumerate}

In every product of the type $a_{s,l}^{(i)}(a_{s,l}^{(i)})^{\ast}$, the factor $V_{l}^{(i)}(A)^{2\beta_{l,r_l}^{(s)}}$ appears exactly once, while for every $h=1,\dots,r_{l}-1$, $V_{l}^{(i)}(A)^{\beta_{l,h}^{(s)}}$ appears exactly twice. Therefore, for every fixed $s=1,\dots,k$, $l=1,\dots,p$ and $t_j \in I_{l,s}(\bs{a})$, in the argument of $\varphi\big( \big(a_{s,l}^{(i)}(a_{s,l}^{(i)})^{\ast}\big)^{2^{t_{j}}}\big)$, the property of trace of  $\varphi$ implies that there are exactly $2^{t_{j}}(2r_l - 1)$ paired products of the type $(W_l^{(i)})^{T_1}(V_l^{(i)}(A))^{T_2}$, for certain integers $T_1, T_2$.\\

Moreover, as follows by a direct application of Proposition \ref{prop02} and Lemma \ref{lemma01}  to the random variables $W_l^{(i)}$ and $V_l^{(i)}(A)$, with $A=Y_i \in \bs{Y}_i$  or $A= \mathcal{X}_{i,l} \in \bs{\mathcal{X}}_i$ and if $\hat{\bs{\mathcal{Z}}}^{(i)} = (\bs{Y}_1,\dots,\bs{Y}_{i-1},\bs{\mathcal{X}}_{i+1},\dots,\bs{\mathcal{X}}_n)$, for every $r \geq 1$ there exist constants $C_{r,d}$ and $D_{r,d}$ such that:
\begin{align*}
\varphi\big((W_j^{(i)})^{2r}\big) &\leq C_{r,d} \;\mu_{2^{rd-1}}^{\hat{\bs{\mathcal{Z}}}^{(i)}} \bigg( \sum_{j_1,\dots,j_d \in [n]\setminus \{i\}}f_n^{(j)}(j_1,\dots,j_d)^2 \bigg)^r \\
&\leq  C_{r,d} \;\mu_{2^{rd-1}}^{\hat{\bs{\mathcal{Z}}}^{(i)}}.
\end{align*} Similarly,
$$\varphi\big( V_j^{(i)}(A)^{2r}\big) \leq D_{r,d}\;\mu_{2^{rd-1}}^{\bs{\mathcal{Z}}^{(i)}} \big(\mathrm{Inf}_i(f_n^{(j)})\big)^{r}. $$
Indeed, by Proposition \ref{prop02},
$$ 
\varphi\big( V_j^{(i)}(A)^{2r}\big) \leq D_{r,d}\;\mu_{2^{rd-1}}^{\bs{\mathcal{Z}}^{(i)}} \varphi\big( V_j^{(i)}(A)^{2}\big),
$$
where
\begin{align*}
\varphi\big( V_j^{(i)}(A)^{2}\big) &= \sum_{l_1,l_2=1}^d \sum_{j_1,\dots,j_{d-1}\in [n]\setminus\{i\}}\sum_{s_1,\dots,s_{d-1}\in [n]\setminus\{i\}} \\
 & \qquad f_n^{(j)}(j_1,\dots,j_{l_1 -1},i,j_{l_1},\dots,j_{d-1}) f_n^{(j)}(s_1,\dots,s_{l_2 -1},i,s_{l_2},\dots,s_{d-1}) \\
 & \qquad \qquad \varphi\big((Z_{j_1}\cdots Z_{j_{l_1-1}}A Z_{j_{l_1}}\cdots Z_{j_{d-1}})(Z_{s_1}\cdots Z_{s_{l_2-1}}A Z_{s_{l_2}}\cdots Z_{s_{d-1}})\big).
\end{align*}
Every summand $\varphi\big((Z_{j_1}\cdots Z_{j_{l_1-1}}A Z_{j_{l_1}}\cdots Z_{j_{d-1}})(Z_{s_1}\cdots Z_{s_{l_2-1}}A Z_{s_{l_2}}\cdots Z_{s_{d-1}})\big) $ is non-zero if and only if $l_1= d-l_2 +1$ and $j_t = s_{d-t}$ for $t=1,\dots,d-1$ (see \cite[Lemma 5.8]{Speicher}), in which case equals $1$, giving:
$$ \varphi\big( V_j^{(i)}(A)^{2}\big)= \sum_{l=1}^d \sum_{j_1,\dots,j_{d-1}\in [n]\setminus \{i\}} f_n^{(j)}(j_1,\dots,j_{l-1},i,j_l,\dots,j_{d-1})f_n^{(j)}(j_{d-1},\dots,j_l,i,j_{l-1},\dots,j_1),$$
and the conclusion is achieved thanks to the mirror symmetry of $f_n^{(j)}$.  \\

The application of the generalized free H\"older inequality (Lemma \ref{lemma0}) yields:
\begin{align*}
\varphi\big(&(a_{s,l}^{(i)} (a_{s,l}^{(i)})^{\ast})^{2^{t_{j}}}\big) \leq\\
& \mathrm{C}\;\bigg\{ \bigg[ \varphi\bigg(V_l^{2\beta_{l,r_l}^{(s)}2^{2^{t_{j}}(2r_l-1)}}\bigg) \bigg]^{2^{-2^{t_{j}}(2r_{l}-1)}} \bigg\}^{2^{t_{j}}} \cdot \prod_{h=1}^{r_l - 1} \bigg\{\bigg[ \varphi\bigg( V_l^{\beta_{l,h}^{(s)}2^{2^{t_{j}}(2r_l-1)}}\bigg)\bigg]^{2^{-2^{t_{j}}(2r_l-1)}} \bigg\}^{2^{t_{j}+1}} \\
&\leq \mathrm{C}\;\bigg[\bigg( \big(\mathrm{Inf}_i(f_n^{(l)}\big)^{\beta_{l,r_l}^{(s)}2^{2^{t_{j}}(2r_l-1)}}\bigg)^{2^{-2^{t_{j}}(2r_l-1)}} \bigg]^{2^{t_{j}}} \cdot \prod_{h=1}^{r_l-1} \bigg[ \bigg(\big(\mathrm{Inf}_i(f_n^{(l)})\big)^{\beta_{l,h}^{(s)}2^{2^{t_{j}}(2r_l-1)}}\bigg)^{2^{-2^{t_{j}}(2r_l-1)}} \bigg]^{2^{t_{j}}}\\
& \leq \mathrm{C}\;\prod_{h=1}^{r_l}\big(\mathrm{Inf}_i(f_n^{(l)})\big)^{2^{t_{j}}\beta_{l,h}^{(s)}} \;= \;C \big(\mathrm{Inf}_i(f_n^{(l)})\big)^{2^{t_{j}}\sum_{h=1}^{r_l}\beta_{l,h}^{(s)}}  \numberthis \label{stima}
\end{align*}
(where the constant $C$ gathers all the estimates given by the application of Proposition \ref{prop02} to the $W_j^{\alpha}$'s, since they do not depend neither on the influence function, nor on $i$, due to the identically distributed assumption on the sequence $\bs{X}$ and on the sequence $\bs{Y}$).

Then, if $kp$ is even, back to \eqref{pari}, the product over all the integers $t_j$'s in $I_{l,s}(\bs{a})$, such that $\sum\limits_{j \in I_{l,s}(\bs{a})}2^{t_{j}} = 2^{\frac{kp}{2}-1}$, finally provides, up to a multiplicative coefficient:
\begin{align*}
\prod_{t_j \in I_{l,s}(\bs{a})}\Big(\varphi\big((a_{s,l}^{(i)} (a_{s,l}^{(i)})^{\ast})^{2^{t_{j}}}\big) \Big)^{2^{-\frac{kp}{2}}} \;&\leq \;  \big(\mathrm{Inf}_i(f_n^{(l)})\big)^{2^{-1}\sum_{h=1}^{r_l}\beta_{l,h}^{(s)}} \\
&\leq \; \big(\max_{t=1,\dots,p}\mathrm{Inf}_i(f_n^{(t)})\big)^{2^{-1}\sum_{h=1}^{r_l}\beta_{l,h}^{(s)}} 
\end{align*}
implying that:
\begin{align*}
\prod_{s=1}^{k}\prod_{l=1}^{p}\prod_{t_j \in I_{l,s}(\bs{a})}\Big(\varphi\big((a_{s,l}^{(i)} (a_{s,l}^{(i)})^{\ast})^{2^{t_{j}}}\big)\Big)^{2^{-\frac{kp}{2}}} &\leq \big(\max_{h=1,\dots,p}\mathrm{Inf}_i(f_n^{(h)})\big)^{2^{-1}\gamma} \\
&\leq \big(\max_{h=1,\dots,p}\mathrm{Inf}_i(f_n^{(h)})\big)^{\frac{3}{2}} \;=\; \max_{h=1,\dots,p}\big(\mathrm{Inf}_i(f_n^{(h)})\big)^{\frac{3}{2}},
\end{align*}
(recall that $\mathrm{Inf}_i (f_n^{(h)}) \leq 1 $ for all $h$ by hypothesis). 

Finally, up to a combinatorial coefficient one has:
\begin{align*}
\sum_{i=1}^{n} |\varphi\big((a_{1,1}^{(i)}&\cdots a_{1,p}^{(i)}) \cdots (a_{k,1}^{(i)}\cdots a_{k,p}^{(i)})\big)| \leq \sum_{i=1}^{n} \max_{h=1,\dots,p}\big(\mathrm{Inf}_i(f_n^{(h)})\big)^{\frac{3}{2}}\\
&\leq \sum_{i=1}^{n} \sum_{h=1}^{p}\big(\mathrm{Inf}_i(f_n^{(h)})\big)^{\frac{1}{2}}\big(\mathrm{Inf}_i(f_n^{(h)})\big) \, \leq \, \sum_{i=1}^{n}\sum_{h=1}^{p} \sqrt{\tau_n^{(h)}}\mathrm{Inf}_i(f_n^{(h)}) \\
&=\sum_{h=1}^p \sqrt{\tau_n^{(h)}} \sum_{i=1}^n \mathrm{Inf}_i(f_n^{(h)}) \, \leq \, p\, \max_{h=1,\dots,p}\sqrt{\tau_n^{(h)}} ,\numberthis \label{chain}
\end{align*}
due to $\sum\limits_{i=1}^{n}\mathrm{Inf}_i(f_n^{(h)}) = 1$, and the conclusion follows.\\

If $kp$ is odd, for every $s=1,\dots,\frac{k-1}{2}$ and every $l=1,\dots,p$,  and for $l=1,\dots,\frac{p-1}{2}$ when $s=\frac{k+1}{2}$, from \eqref{dispari} the estimate in (\ref{stima}) gives:
\begin{align*}
\prod_{t_{j} \in I_{l,s}(\bs{a})} \varphi\big( (a_{s,l}^{(i)} (a_{s,l}^{(i)})^{\ast})^{2^{t_{j}}} \big)^{2^{-\frac{kp-1}{2}}} &\leq  \bigg(\big(\mathrm{Inf}_i (f_n^{(l)})\big)^{2^{\frac{kp-3}{2}}\sum_{h=1}^{r_l}\beta_{l,h}^{(s)}}\bigg)^{2^{-\frac{kp-1}{2}}}\\
&= \big(\mathrm{Inf}_i (f_n^{(l)})\big)^{2^{-1}\sum_{h=1}^{r_l}\beta_{l,h}^{(s)}} \\
&\leq \Big( \max_{t=1,\dots,n} \mathrm{Inf}_i(f_n^{(t)}) \Big)^{2^{-1}\sum_{h=1}^{r_l}\beta_{l,h}^{(s)}},
\end{align*}
so that, for every $s=1,\dots,\frac{k-1}{2}$,
$$ \prod_{l=1}^{p}\prod_{t_{j} \in I_{l,s}(\bs{a})} \varphi\big( (a_{s,l}^{(i)} (a_{s,l}^{(i)})^{\ast})^{2^{t_{j}}} \big)^{2^{-\frac{kp-1}{2}}} \leq \Big( \max_{t=1,\dots,n} \mathrm{Inf}_i(f_n^{(t)}) \Big)^{2^{-1}\sum_{l=1}^{p}\sum_{h=1}^{r_l}\beta_{l,h}^{(s)}},$$
while for $s=\frac{k+1}{2}$, 
$$ \prod_{l=1}^{\frac{p-1}{2}}\prod_{t_{j} \in I_{l,s}(\bs{a})} \varphi\big( (a_{s,l}^{(i)} (a_{s,l}^{(i)})^{\ast})^{2^{t_{j}}} \big)^{2^{-\frac{kp-1}{2}}} \leq \Big( \max_{t=1,\dots,n} \mathrm{Inf}_i(f_n^{(t)}) \Big)^{2^{-1}\sum_{l=1}^{\frac{p-1}{2}}\sum_{h=1}^{r_l}\beta_{l,h}^{(s)}}.$$

Similarly, for every $s=\frac{k+3}{2},\dots,k$ and every $l=1,\dots,p$, and for $l=\frac{p+1}{2},\dots,p$ when $s=\frac{k+1}{2}$, the estimate in (\ref{stima}) gives:
\begin{align*}
\prod_{t_{j} \in I_{l,s}(\bs{a})} \varphi\big( (a_{s,l}^{(i)} (a_{s,l}^{(i)})^{\ast})^{2^{t_{j}}} \big)^{2^{-\frac{kp+1}{2}}} &\leq  \bigg(\big(\mathrm{Inf}_i (f_n^{(l)})\big)^{2^{\frac{kp-1}{2}}\sum_{h=1}^{r_l}\beta_{l,h}^{(s)}}\bigg)^{2^{-\frac{kp+1}{2}}}\\
&= \big(\mathrm{Inf}_i (f_n^{(l)})\big)^{2^{-1}\sum_{h=1}^{r_l}\beta_{l,h}^{(s)}} \\
&\leq \Big( \max_{t=1,\dots,n} \mathrm{Inf}_i(f_n^{(t)}) \Big)^{2^{-1}\sum_{h=1}^{r_l}\beta_{l,h}^{(s)}},
\end{align*}
so that for every $s=\frac{k+3}{2},\dots,k$
$$ \prod_{l=1}^{p}\prod_{t_{j} \in I_{l,s}(\bs{a})} \varphi\big( (a_{s,l}^{(i)} (a_{s,l}^{(i)})^{\ast})^{2^{t_{j}}} \big)^{2^{-\frac{kp-1}{2}}} \leq \Big( \max_{t=1,\dots,n} \mathrm{Inf}_i(f_n^{(t)}) \Big)^{2^{-1}\sum_{l=1}^{p}\sum_{h=1}^{r_l}\beta_{l,h}^{(s)}},$$
while for $s=\frac{k+1}{2}$, 
$$ \prod_{l=\frac{p+1}{2}}^{p}\prod_{t_{j} \in I_{l,s}(\bs{a})} \varphi\big( (a_{s,l}^{(i)} (a_{s,l}^{(i)})^{\ast})^{2^{t_{j}}} \big)^{2^{-\frac{kp-1}{2}}} \leq \Big( \max_{t=1,\dots,n} \mathrm{Inf}_i(f_n^{(t)}) \Big)^{2^{-1}\sum_{l=\frac{p+1}{2}}^{p}\sum_{h=1}^{r_l}\beta_{l,h}^{(s)}}.$$

In the end, the inequality in (\ref{dispari}) can be rewritten as $ |\varphi\big((a_{1,1}^{(i)}\cdots a_{1,p}^{(i)}) \cdots (a_{k,1}^{(i)}\cdots a_{k,p}^{(i)})\big)| \leq A_i\; B_i \;C_i\; D_i$, 
where
\begin{enumerate}
\item $ A_i = \prod\limits_{s=1}^{\frac{k-1}{2}}\prod\limits_{l=1}^p \prod\limits_{t_j \in I_{l,s}(\bs{a})} \varphi\big((a_{s,l}^{(i)} (a_{s,l}^{(i)})^{\ast})^{2^{t_{j}}}\big)^{2^{-\frac{kp-1}{2}}} ;$
\item $ B_i = \prod\limits_{l=1}^{\frac{p-1}{2}} \prod\limits_{t_j \in I_{l,\frac{k+1}{2}}(\bs{a})} \varphi\big((a_{s,\frac{k+1}{2}}^{(i)} (a_{s,\frac{k+1}{2}}^{(i)})^{\ast})^{2^{t_{j}}}\big)^{2^{-\frac{kp-1}{2}}} ;$
\item $ C_i = \prod\limits_{l=\frac{p+1}{2}}^{p} \prod\limits_{t_j \in I_{l,\frac{k+1}{2}}(\bs{a})} \varphi\big((a_{s,\frac{k+1}{2}}^{(i)} (a_{s,\frac{k+1}{2}}^{(i)})^{\ast})^{2^{t_{j}}}\big)^{2^{-\frac{kp+1}{2}}} ; $
\item $ D_i = \prod\limits_{s=\frac{k+3}{2}}^{k}\prod\limits_{l=1}^p \prod\limits_{t_j \in I_{l,s}(\bs{a})} \varphi\big((a_{s,l}^{(i)} (a_{s,l}^{(i)})^{\ast})^{2^{t_{j}}}\big)^{2^{-\frac{kp+1}{2}}} .$
\end{enumerate}
As for the estimates given in \eqref{stima},
\begin{enumerate}
\item $ A_i \leq \Big(\max\limits_{t=1,\dots,p} \mathrm{Inf}_i(f_n^{(t)})\Big)^{\gamma(A_i)},  \text{ with  }\; \gamma(A_i) = 2^{-1}\sum\limits_{s=1}^{\frac{k-1}{2}}\sum\limits_{l=1}^p \sum\limits_{h=1}^{r_l}\beta_{l,h}^{(s)}; $
\item $ B_i \leq \Big(\max\limits_{t=1,\dots,p} \mathrm{Inf}_i(f_n^{(t)})\Big)^{\gamma(B_i)} , \text{ with }\; \gamma(B_i) = 2^{-1}\sum\limits_{l=1}^{\frac{p-1}{2}} \sum\limits_{h=1}^{r_l}\beta_{l,h}^{(\frac{k+1}{2})} ;$
\item $ C_i \leq \Big(\max\limits_{t=1,\dots,p} \mathrm{Inf}_i(f_n^{(t)})\Big)^{\gamma(C_i)}, \text{ with } \;\gamma(C_i) = 2^{-1}\sum\limits_{l=\frac{p+1}{2}}^{p} \sum\limits_{h=1}^{r_l}\beta_{l,h}^{(\frac{k+1}{2})} ; $
\item $ D_i \leq \Big(\max\limits_{t=1,\dots,p} \mathrm{Inf}_i(f_n^{(t)})\Big)^{\gamma(D_i)}, \text{ with } \;\gamma(D_i) = 2^{-1}\sum\limits_{s=\frac{k+3}{2}}^{k}\sum\limits_{l=1}^p \sum\limits_{h=1}^{r_l}\beta_{l,h}^{(s)} $,
\end{enumerate}
yielding:
\begin{align*}
|\varphi\big((a_{1,1}^{(i)}\cdots a_{1,p}^{(i)}) \cdots (a_{k,1}^{(i)}\cdots a_{k,p}^{(i)})\big)|  &\leq \big(\max_{t=1,\dots,p} \mathrm{Inf}_i(f_n^{(t)})\big)^{2^{-1}\gamma} \\
& \leq\big(\max_{t=1,\dots,p} \mathrm{Inf}_i(f_n^{(t)})\big)^{\frac{3}{2}} 
\end{align*}
since $\gamma = \sum\limits_{s=1}^k \sum\limits_{l=1}^p \sum\limits_{h=1}^{r_l}\beta_{r,l}^{(s)} = \gamma(A_{i}) +\gamma(B_i) + \gamma(C_i) + \gamma(D_i)$.\\

To conclude in the case of sequences of identically distributed variables, it is sufficient to repeat the reasoning carried out in the chain of inequalities (\ref{chain}). \\

If the sequences $\bs{X}$ and $\bs{Y}$ were composed of independent random variables with uniformly bounded moments (not necessarily identically distributed), the proof would follow the same steps. The only modification to take into account would be relative to the hypercontractivity arguments in \eqref{stima}, and would require to replace $\mu_{2^{rd-1}}^{\hat{\bs{\mathcal{Z}}}^{(i)}}$ with $\mu_{2^{rd-1}}^{\bs{Y},\bs{\mathcal{X}}^{(n)}}$:
$$\mu_{2^{rd-1}}^{\hat{\bs{\mathcal{Z}}}^{(i)}} \; \leq \; \mu_{2^{rd-1}}^{\bs{Y},\bs{\mathcal{X}}^{(n)}} \, < \infty .$$
Indeed, since the moments of $\bs{X}$ and $\bs{Y}$ are uniformly bounded, so are the moments of the random variables composing the ensembles $\bs{\mathcal{X}}^{(n)}$.

\chapter{Universality of Chebyshev sums in every dimension}\label{Chapter_Universality}

As anticipated in the synopsis, the goal of the present chapter is to apply the  invariance principle stated via Theorem \ref{Multiinvariance2}, to derive other universal laws for semicircular and free Poisson approximations of (vectors) of homogeneous sums in freely independent random variables, in the sense of Definition \ref{UniversalLaw} below. So far, indeed, only the semicircle law is known to enjoy the feature under investigation: universality, and its interactions with the Fourth Moment Theorems, will allow us to establish that the semicircular asymptotic behaviour of any vector of Chebyshev sums, in semicircular entries, ensures that the same approximation holds for any vector of homogeneous sums. \\ 
Beyond the universality statements, it will be briefly outlined how the same technique leads to new universal laws for normal approximation of homogeneous sums in the classical probability setting.\\

The connection between Fourth Moment Theorems and universality statements will  be  further investigated in the subsequent Part \ref{FMT}. In this regard, note that the unidimensional results that can be drawn from the forthcoming discussion might be seen as a corollary of Theorem \ref{superTeo2}, even though they arise here from a different technique.  On the other hand,  the multidimensional case here presented will not be reached in Part \ref{FMT}.\\

If not otherwise specified, all random variables $Y$ in $(\mathcal{A},\varphi)$ are assumed to be centered and with unit variance (Assumption {\bf (1)} for short).

\begin{defn}\label{UniversalLaw}
Let $Y$ be a random variable  in $(\mathcal{A},\varphi)$, satisfying Assumption {\bf (1)}. $Y$ is said to be \textbf{ universal } (at the order $d\geq 2$) \index{Universal law, (free)} for semicircular approximations of homogeneous sums if, for any sequence $f_n:[n]^d\to \mathbb{R}$ of admissible kernels, the following conditions are equivalent as $n\to\infty$:
\begin{itemize}
\item[(i)] $Q_{\Y}(f_n) \xrightarrow{\text{\rm Law}}\mathcal{S}(0,1)$;
\item[(ii)] $Q_{\mathbf{W}}(f_n) \xrightarrow{\text{\rm Law}}\mathcal{S}(0,1)$ for any other sequence $\mathbf{W} = \{W_i\}_{i \geq 1}$ of identically distributed freely independent random variables, satisfying Assumption {\bf (1)}.
\end{itemize}
\end{defn}

\begin{rmk}
The choice $d\geq 2$ is motivated by the fact that it is well-known that there is no universality for linear polynomials. For instance, for any sequence of real numbers $f_n(i)$ such that $\sum\limits_{i=1}^n f_n(i)^2 = 1$, if $\bs{S}=\{S_i\}_{i\geq 1}$ is a sequence of freely independent random variables with the standard semicircle distribution, $Q_{\bs{S}}(f_n) = \sum\limits_{i=1}^n f_n(i)S_i \sim \mathcal{S}(0,1)$, while for general coefficients $f_n(i)$, and a sequence $\bs{X}$ of freely independent random variables, $Q_{\bs{X}}(f_n)$ does not converge in law to the semicircle law. \\
\end{rmk}

The \textit{universality} feature is not simple to detect and describe: for instance, it is well-known that the free Walsh chaos is not universal for $d=2$ (see \cite{NourdinDeya}). So far, the only example of universal law for 1-dimensional semicircular approximations of homogeneous sums (with symmetric coefficients) has been provided with Theorem \ref{invnoncom}, recalled in the previous Chapter, giving a partial  free counterpart to \cite[Theorem 1.2]{NourdinPeccatiReinert} (see Theorem \ref{inv} in Part \ref{FMT}).  Dually, other limit laws for which a universality phenomenon can be satisfied have not been explored, nor even in the simpler case of homogeneous sums in semicircular entries. \\

The goal pursued in this chapter is to provide further examples of laws that are universal for semicircular and free Poisson approximations  of homogeneous sums of degree $d\geq 2$, combining Theorem \ref{Multiinvariance2} and the Fourth Moment Theorems \ref{knps} and \ref{FPoissAppr}. \\


\section{Fourth Moment Theorem for Chebyshev sums}

The free counterpart to the \textit{Nualart-Peccati Criterion} (see Theorem \ref{FMTClassic}) has been established in full generality in \cite[Theorems 1.3, 1.6]{KempNourdinPeccatiSpeicher} for Wigner integrals of mirror symmetric functions: for the purposes of the present discussion, Theorem \ref{knps} recalls the free version of the \textit{Nualart-Peccati Criterion}, in a simplified version and only for homogeneous sums in a sequence  of freely independent standard semicircular random variables $\SSw = \{S_i\}_{i\geq 1}$.

\begin{thm}\label{knps}
For any $d\geq 2$,  and for every sequence of admissible kernels $f_n:[n]^d \rightarrow \mathbb{R}$, the following statements are equivalent as $n\rightarrow \infty$:
\begin{itemize}
\item[(i)] $Q_{\SSw}(f_n) \xrightarrow{\text{\rm Law}}\mathcal{S}(0,1)$;
\item[(ii)] $\varphi(Q_{\SSw}(f_n)^4) \longrightarrow \varphi(S^4)=2, S \sim \mathcal{S}(0,1)$;
\item[(iii)] if $k_n= \sum\limits_{i_1,\dots,i_d=1}^{n}f_n(i_1,\dots,i_d)e_{i_1}\otimes\cdots \otimes e_{i_d}$, with $\{e_j\}_{j\geq 1}$ an orthonormal basis of $\mathrm{L}^2(\mathbb{R}_+)$,  for every $r=1,\dots, d-1$,
$$\|k_n  \otimes_r k_n \|_{\mathrm{L}^{2}(\mathbb{R}_{+}^{2d-2r})} \longrightarrow 0 \,,$$
where $\otimes_r$ denotes the contraction  introduced in \eqref{GenContraction}.
\end{itemize}
\end{thm}

A similar simplified version of the method of moments and cumulants has been also provided for free Poisson approximations of Wigner integrals in \cite[Theorem 1.4, Lemma 5.1]{NourdinPeccati1}. As above, for the sake of simplicity, Theorem \ref{FPoissAppr} records this statement only for homogeneous sums in semicircular entries (Propositions \ref{contraction5} and \ref{contraction6} will be applied).

\begin{thm}\label{FPoissAppr}
For $\lambda >0$, let $Z(\lambda)$ denote a centred free Poisson random variable of parameter $\lambda$. If $d\geq 2 $ is even, consider a sequence of admissible kernels $f_n: [n]^d \rightarrow \mathbb{R}$ such that $ \varphi\big(Q_{\SSw}(f_n)^2\big) \rightarrow \lambda$ as $n\rightarrow \infty$. Then, in the limit, the following statements are equivalent:
\begin{itemize}
\item[(i)] $Q_{\SSw}(f_n) \stackrel{\text{Law}}{\longrightarrow} Z(\lambda)$;
\item[(ii)] $\varphi\big(Q_{\SSw}(f_n)^4\big)-2\varphi\big(Q_{\SSw}(f_n)^3\big) \longrightarrow \varphi\big(Z(\lambda)^4\big)-2\varphi\big(Z(\lambda)^3\big) = 2\lambda^2 - \lambda$;
\item[(iii)] $\|f_n \stackrel{\frac{d}{2}}{\smallfrown} f_n - f_n \| \rightarrow 0$, and for every $r=1,\dots,d-1, r \neq \frac{d}{2}, \| f_n \stackrel{r}{\smallfrown} f_n\| \rightarrow 0$,\\
where the contraction $\stackrel{r}{\smallfrown}$ has been introduced in \eqref{contraction}.
\end{itemize}
\end{thm}

\begin{rmk}
Free Poisson approximations can be established only in Wigner chaos of even order since, if $d$ is odd, $\varphi(Q_{\SSw}(f_n)^3) = 0$ while $\varphi(Z(\lambda)^3) = \lambda >0$.\\
\end{rmk}

In the following, the focus will be on Chebyshev sums based on freely independent semicircular random variables: the forthcoming Theorems \ref{ConvChebySemicircular} and \ref{ConvFreePoissonPARI} aim to state the Fourth Moment Theorem for such Chebyshev sums  in terms of the contraction operators, for semicircular and free Poisson limit respectively (Theorem \ref{knps}  and Theorem \ref{FPoissAppr}). Further, the following auxiliary lemma (whose proof requires only simple computations), is inspired by the proof of  \cite[Proposition 4.1]{PeccatiZheng1} and will be useful in the sequel.

\begin{lemma}
\label{stimacontr} 
Let $d\geq 2$ and $f:[n]^{d}\rightarrow \mathbb{R}$ be an admissible kernel as in Definition \ref{admissible_Univ}. Then, for every $q=1,\dots,d-1$, the norms of the contraction operators introduced in Definition \ref{starcontraction} satisfy the following inequalities:
\begin{eqnarray}
\|f \stackrel{q}{\smallfrown} f \| &\geq &  \|f \star_{q+1}^{q} f \| \; ,\nonumber \\
\|f \stackrel{d-1}{\smallfrown} f\| &\geq & \|f \star_{1}^{0}\;f\| \; .\nonumber
\end{eqnarray}
\end{lemma}



\begin{thm}
\label{ConvChebySemicircular}
Fix $d\geq 2$, as well as integers $h_1,\dots,h_d \geq 1$, with $h_i = h_{d-i+1}$ for all $i=1,\dots,\lfloor \frac{d}{2} \rfloor$, and let $Q_{\mathbf{S}}^{(\bs{h})}(f_n)$  denote a sequence of Chebyshev sums, as in (\ref{SumCheby2}), with $f_n$ admissible kernel as in Definition \ref{admissible_Univ}. Then, the following conditions are equivalent as $n\rightarrow \infty $:
\begin{itemize}
\item[(i)] $Q_{\mathbf{S}}^{(\bs{h})}(f_n)  \stackrel{\text{ Law }}{ \longrightarrow} \mathcal{S}(0,1)$;
\item[(ii)] for every $q=1,\dots, d-1$, $\|f_n \stackrel{q}{\smallfrown} f_n \| \longrightarrow 0$.
\end{itemize}
\end{thm}

\begin{proof}
Assume that $(ii)$ holds. Then, it is sufficient to remark that $Q_{\mathbf{S}}^{(\bs{h})}(f_n) = I_{m}^{S}(k_n)$, with $m=h_1 + \cdots + h_d$, and $k_n$ the kernel given by (\ref{kN1}) (see also \eqref{Integral}). Theorem \ref{knps} then entails the vanishing of all the non trivial contractions $\|k_n \otimes_r k_n \|$, for $r=1,\dots,m-1$, which, by virtue of Proposition \ref{contraction5}, in turn implies that the norm $\|f_n \stackrel{q}{\smallfrown} f_n\|$ vanishes in the limit as well for every $q=1,\dots,d-1$.
To show the converse, it is sufficient to repeat the same reasoning but keeping in mind also  Lemma \ref{stimacontr}.
\end{proof}

\begin{thm}
\label{ConvFreePoissonPARI}
Assume that $d\geq 2$ and $h_1+\cdots +h_d$  are even integers, and let $Z(\lambda)$ denote a (centered) free Poisson distributed random variable of parameter $\lambda > 0$. For a sequence of admissible kernels $f_n:[n]^d \rightarrow \mathbb{R}$ such that
\begin{equation}
\lim_{n \rightarrow \infty}\varphi\Big(\big(Q_{\mathbf{S}}^{(\bs{h})}(f_n)\big)^{2}\Big) = \lambda \, ,
\end{equation}
the following conditions are equivalent as $n \rightarrow \infty$:
\begin{itemize}
\item[(i)] $Q_{\mathbf{S}}^{(\bs{h})}(f_n) \stackrel{ \text{Law} }{ \longrightarrow} Z(\lambda)$;
\item[(ii)]
\begin{enumerate}
\item for every $q=1,\dots, d-1$, $q \neq \dfrac{d}{2}$, $\|f_n \stackrel{q}{\smallfrown} f_n \| \longrightarrow 0$;
\item $\|f_n \star_{\frac{d}{2}+1}^{\frac{d}{2}}f_n \|\longrightarrow 0 $, and $\|f_n \stackrel{\frac{d}{2}}{\smallfrown}f_n - f_n\|\longrightarrow 0$.\\
\end{enumerate}
\end{itemize}
\end{thm}

\begin{proof}
Again, $Q_{\mathbf{S}}^{(\bs{h})}(f_n) = I_m^{S}(k_n)$, with $m=h_1 + \cdots + h_d$, and $k_n$ as in (\ref{kN1}). The conclusion then  simply follows by applying Theorem \ref{FPoissAppr}, together with Proposition \ref{contraction5}, Proposition \ref{contraction6} and Lemma \ref{stimacontr} (in particular, the vanishing of the norm $\|f_n \star_{\frac{d}{2}+1}^{\frac{d}{2}}f_n \| \longrightarrow 0 $ entails the vanishing in the limit of all the norms $\|k_n \otimes_r k_n \|$, for $r= h_1+\cdots + h_{\frac{d}{2}-1}+1,\dots, h_1+\cdots + h_{\frac{d}{2}}-1$). 
\end{proof}

\begin{rmk}[\textbf{On the parity of} $d$]
Note that for the convergence of a Chebyshev sum towards the free Poisson law, it is not sufficient that only $m:=h_1 + \cdots+ h_d$ is even. Indeed, if $d$ is odd 
and $Q_{\mathbf{S}}^{(\bs{h})}(f_n) = I_m^{S}(k_n)$ converges to $Z(\lambda)$, then $\|k_n \stackrel{r}{\smallfrown} k_n \|$ would vanish in the limit for every $r=1,\dots,m-1$, $r\neq \frac{m}{2}$. In particular, if $r=h_1 + \cdots + h_{\frac{d-1}{2}}$, $\|k_n \otimes_r k_n\| = \| f_n \stackrel{\frac{d-1}{2}}{\smallfrown} f_n \|\rightarrow 0$. By virtue of Lemma \ref{stimacontr}, this would imply in turn $\|f_n \star_{\frac{d+1}{2}}^{\frac{d-1}{2}} f_n \| = \| k_n \otimes_{\frac{m}{2}} k_n\|\rightarrow 0 $, which contradicts the fact that $\| k_n \otimes_{\frac{m}{2}} k_n\|$ should not vanish in the limit. Hence, it is possible to establish conditions for the convergence of a Chebyshev sum towards the free Poisson law only if both $d$ and $h_1+\cdots +h_d$ are even integers.\\
\end{rmk}

\begin{rmk}
From Theorem \ref{ConvChebySemicircular} and Theorem \ref{ConvFreePoissonPARI} with $h_j = 2$ for every $j=1\dots,d$ and with $d$ even, since $U_2(S) \stackrel{\text{Law}}{=} Z(1)$, explicit conditions for the convergence of a homogeneous sum as in (\ref{QN}) in freely independent random variables with the centered free Poisson distribution of parameter $1$, towards the semicircular law, can be stated, generalizing to the free setting the findings of \cite{PeccatiZheng1}). Similarly, if $d$ is even and the target distribution  is the free Poisson law. 

Recently, a Fourth Moment Theorem has been proved for stochastic integrals with respect to a free Poisson measure in \cite{Solesne}.
\end{rmk}

\section{Main results}\label{SectionResults_Universality}

The combination of Propositions \ref{contraction5} and \ref{contraction6} with Theorems  \ref{ConvChebySemicircular} and \ref{ConvFreePoissonPARI}, allows one to exhibit  new universal limit laws for vectors of homogeneous sums as straightforward consequences of  Theorem \ref{Multiinvariance2} (see Theorems \ref{MultUniversality} and \ref{Poiss2} below). Nevertheless, even if the multidimensional invariance principle  holds for Chebyshev sums with mirror symmetric kernels, here it will be necessary to deal only with fully symmetric coefficients. Indeed, as shown in \cite{NourdinDeya} with a counterexample, the strategy of proof here proposed cannot be extended to deal with the weaker assumption of mirror symmetric coefficients, even though this would be the most natural framework. \\

Henceforth, the admissible kernels will be assumed to be symmetric functions: in this case, the following upper bounds for $\tau_n = \max\limits_{i=1,\dots,n}\mathrm{Inf}_i(f_n)$ holds. 

\begin{lemma}
\label{magg}
Let $d \geq 2$, and let $f_n:[n]^{d}\rightarrow \mathbb{R}$ be a symmetric kernel, vanishing on diagonals. Then, the following inequality holds:
\begin{equation}
\label{magg1}
\|f_n \stackrel{d-1}{\smallfrown}f_n\| \geq \dfrac{1}{d}\tau_n.
\end{equation}
Moreover, if $d=2$, then
\begin{equation}
\label{magg2}
\|f_n \stackrel{1}{\smallfrown}f_n - f_n \| \geq \dfrac{1}{2}\tau_n.
\end{equation}
\end{lemma}

\begin{proof}
If $d \geq 2$, by carrying out the same estimates as in the proof of in \cite[Theorem 1.4]{NourdinDeya}, the following lower bound for $\|f_n \stackrel{d-1}{\smallfrown}f_n\|^{2}$ holds:
\begin{align*}
\|f_n \stackrel{d-1}{\smallfrown}f_n\|^{2} &= \sum\limits_{i_1,i_2=1}^{n}\Big(f_n \stackrel{d-1}{\smallfrown} f_n (i_1,i_2)\Big)^{2} \;\geq \;\sum\limits_{i=1}^{n}\Big(f_n \stackrel{d-1}{\smallfrown} f_n(i,i)\Big)^{2} \\
&= \sum\limits_{i=1}^{n}\bigg(\sum_{j_2,\dots,j_d=1}^{n}f_n(i,j_2,\dots,j_d)^{2}\bigg)^{2} \; \geq\; \bigg(\sum_{j_2,\dots,j_d=1}^{n}f_n(i,j_2,\dots,j_d)^{2}\bigg)^{2} = \big(\dfrac{1}{d}\;\mathrm{Inf}_{i}(f_n)\big)^2
\end{align*}
for every $i=1,\dots, n$. By taking the square root on both sides, in particular one has: 
$$ \| f_n \stackrel{d-1}{\smallfrown}f_n\|\geq \dfrac{1}{d}\;\mathrm{Inf}_{i}(f_n) \geq \dfrac{1}{d}\max\limits_{i=1,\dots,n}\mathrm{Inf}_{i}(f_n) = \dfrac{1}{d}\tau_{n}.$$

When $d=2$, the desired upper bound for $\tau_n$ is obtained as follows:
\begin{align*}
\|f_n \stackrel{1}{\smallfrown}f_n - f_n \|^{2} &= \sum_{i,j=1}^{n}\Big(f_n \stackrel{1}{\smallfrown} f_n (i,j) - f_n(i,j)\Big)^{2} \\
&= \sum_{\substack{i,j=1\\ i \neq j}}^{n}\Big(f_n\stackrel{1}{\smallfrown} f_n (i,j) - f_n(i,j)\Big)^{2} + \sum_{i=1}^{n}\big(f_n \stackrel{1}{\smallfrown} f_n (i,i)\big)^{2}   \\
&\geq \sum_{i=1}^{n}\bigg(\sum_{k=1}^{n}f_n(i,k)^{2}\bigg)^{2} \;\geq \bigg(\sum_{k=1}^{n}f_n(i,k)^{2}\bigg)^{2}
\end{align*}
for ever $i=1,\dots, n$, from which $\|f_n \stackrel{1}{\smallfrown}f_n - f_n \| \geq  \dfrac{1}{2}\mathrm{Inf}_i(f_n)$, and finally
 $$\|f_n \stackrel{1}{\smallfrown}f_n - f_n \| \geq \dfrac{1}{2}\tau_n.\qedhere $$
\end{proof}

The proof of Theorem \ref{MultUniversality} will exploit the following statement, recalling that componentwise convergence of multiple Wigner integrals towards the semicircle law implies the joint convergence (see \cite[Theorem 1.3]{NouSpeiPec} for the original statement).

\begin{thm}\label{Equi_Free}
For $d \geq 2$ and $m\geq 1$, let $k_n^{(j)}$ be a mirror symmetric function in $\mathrm{L}^2(\mathbb{R}_+^{d})$ for every $j=1,\dots,m$. Let $C = (C_{i,j})_{i,j=1,\dots,m}$ be a real-valued, positive definite symmetric matrix, such that, for $i,j=1,\dots,m$,
$$ \lim_{n\rightarrow \infty} \varphi\big(I_{d}^{S}(k_n^{(i)}) I_d^{S}(k_n^{(j)}) \big) = C_{i,j}.$$
If $(s_1,\dots,s_m)$ denotes a semicircular system, with covariance determined by $C$, the following statements are equivalent as $n\rightarrow \infty$:
\begin{itemize}
\item[(i)] $I_d^{S}(k_n^{(j)}) \stackrel{ \text{ Law }}{\longrightarrow } s_{j}$;
\item[(ii)] $(I_d^{S}(k_n^{(1)}),\dots,I_d^{S}(k_n^{(m)})) \stackrel{ \text{ Law }}{\longrightarrow } (s_1,\dots,s_m).$\\
\end{itemize}
\end{thm}

\begin{thm}
\label{MultUniversality}
Let $d\geq 2$, and $m\geq 1$. Consider a semicircular system $(s_1,\dots, s_m)$ with covariance $\varphi(s_i s_j) = C_{i,j}$ for every $i,j=1,\dots,m$, and assume that the matrix $C = (C_{i,j})_{i,j=1,\dots,m}$ is real-valued, positive definite and symmetric. For fixed order $\bs{h}=(h_1,\dots,h_d)$ with $h_i = h_{d-i+1}$ for $i=1,\dots,\lfloor \frac{d}{2}\rfloor$, suppose that, for every $i,j=1,\dots,m$:
$$ \lim_{n \rightarrow \infty} \varphi\big(Q_{\mathbf{S}}^{(\bs{h})}(f_n^{(i)}) Q_{\mathbf{S}}^{(\bs{h})}(f_n^{(j)})\big) = C_{i,j}.$$ Then the following assertions are equivalent as $n\rightarrow \infty$:
\begin{itemize}
\item[(i)]  $Q_{\mathbf{S}}^{(\bs{h})}(f_n^{(j)}) \stackrel{\text{Law}}{\longrightarrow} s_j$, for every $j=1,\dots,m$;
\item[(ii)] $(Q_{\bs{X}}(f_n^{(1)}),\dots, Q_{\mathbf{X}}(f_n^{(m)})) \stackrel{\text{ Law }}{\longrightarrow }(s_1,\dots,s_m)$ for every sequence $\mathbf{X} = \{X_i\}_{i\geq 1}$ of freely independent and identically distributed random variables satisfying Assumption {\bf (1)}.
\end{itemize}

\end{thm}

\begin{proof}
$ $
\begin{itemize} 
\item[$(i) \Rightarrow (ii)$]
In view of Theorem \ref{Equi_Free}, and since $Q_{\mathbf{S}}^{(\bs{h})}(f_n^{(j)}) = I_{h_1+\cdots + h_d}^{S}(k_n^{(j)})$ (with $k_n^{(j)}$ as in \eqref{kN1}), the convergence $Q_{\mathbf{S}}^{(\bs{h})}(f_n^{(j)})\stackrel{ \text{Law}}{\longrightarrow} s_j$ for all $j=1,\dots,m$ is equivalent to the joint convergence:
$$(Q_{\mathbf{S}}^{(\bs{h})}(f_n^{(1)}),\dots, Q_{\mathbf{S}}^{(\bs{h})}(f_n^{(m)})) \stackrel{\text{ Law }}{\longrightarrow }(s_1,\dots,s_m) .$$ 
In particular, Lemma \ref{magg} and Theorem \ref{knps} together imply, for every $j$, the limit relation $\|f_n^{(j)} \stackrel{d -1}{\smallfrown} f_n^{(j)} \| \rightarrow 0$, yielding that $\tau_n^{(j)} \rightarrow 0$ for $j=1,\dots,m$, and, in turn, the vanishing in the limit of $\max\limits_{j=1,\dots,m}\tau_n^{(j)}$. The conclusion then follows by Theorem \ref{Multiinvariance2}.

\item[$(ii)\Rightarrow (i)$] In particular, $(Q_{\bs{S}}(f_n^{(1)}),\dots, Q_{\bs{S}}(f_n^{(m)})) \stackrel{\text{ Law }}{\longrightarrow }(s_1,\dots,s_m)$. Then, the Fourth Moment Theorem \ref{knps}, along with Lemma \ref{magg}, implies in particular that $\max\limits_{j=1,\dots,m}\tau_n^{(j)} \rightarrow 0$, yielding first $(Q_{\mathbf{S}}^{(\bs{h})}(f_n^{(1)}),\dots, Q_{\mathbf{S}}^{(\bs{h})}(f_n^{(m)})) \stackrel{\text{Law}}{\longrightarrow} (s_1,\dots,s_m)$ by virtue of Theorem \ref{Multiinvariance2}, and then the desired componentwise convergence.\qedhere
\end{itemize}
\end{proof}

By very similar arguments, and with a suitable parity assumption on the orders of the Chebyshev sums, relation (\ref{magg2}) provides immediate proof of Theorem \ref{Poiss2}, concerning free Poisson approximations of vectors of Chebyshev sums with symmetric coefficients. Note that 
 the stronger assumption on the joint convergence of the vector $(Q_{\mathbf{S}}^{(\bs{h})}(f_n^{(1)}),\dots, Q_{\mathbf{S}}^{(\bs{h})}(f_n^{(m)}))$ will be  necessary, since there is no counterpart to Theorem \ref{Equi_Free} for free Poisson approximations. \\

\begin{thm}
\label{Poiss2}
Let $d\geq 2$ be even and  $(z_1,\dots, z_m)$ be a system of random variables, with $z_j \stackrel{\text{Law}}{=} Z(1)$, and with covariance $\varphi(z_i z_j) = C_{i,j}$ for every $i,j=1,\dots,m$, such that $C = (C_{i,j})_{i,j=1,\dots,m}$ is a real-valued, positive definite, symmetric matrix. Assume further that $h_1+\cdots +h_d$ is even as well, and that
 $$ \lim_{n \rightarrow \infty} \varphi\big(Q_{\mathbf{S}}^{(\bs{h})}(f_n^{(i)}) Q_{\mathbf{S}}^{(\bs{h})}(f_n^{(j)})\big) = C_{i,j}, \quad \forall i,j=1,\dots,m.$$
Then, the following assertions are equivalent as $n \rightarrow \infty$:
\begin{itemize}
\item[(i)] $(Q_{\mathbf{S}}^{(\bs{h})}(f_n^{(1)}),\dots, Q_{\mathbf{S}}^{(\bs{h})}(f_n^{(m)})) \stackrel{\text{ Law }}{\longrightarrow }(z_1,\dots,z_m)$;
\item[(ii)] $(Q_{\bs{X}}(f_n^{(1)}),\dots, Q_{\bs{X}}(f_n^{(m)})) \stackrel{\text{ Law }}{\longrightarrow }(z_1,\dots,z_m)$ for every sequence $\mathbf{X} = \{X_i\}_{i\geq 1}$ of freely independent and identically distributed random variables, verifying Assumption {\bf (1)}.
\end{itemize}
\end{thm}

\begin{rmk}
If $h_j =h\geq 1$ for all $j=1,\dots,d$, the previous universality results state that sequences of the type $\{U_h(S_i)\}_{i\geq 1}$ (belonging to the $h$-th Wigner Chaos) behave universally (for vectors of homogeneous sums of degree $d\geq 2$) with respect to both semicircular and free Poisson approximations (if $d$ is even), in the sense that, for any $h\geq 1$, if $X=U_h(S)$, $S \sim \mathcal{S}(0,1)$, then $Q_{\bs{X}}(f_n) \stackrel{\text{Law}}{\rightarrow} \mathcal{S}(0,1)$ (or $Z(\lambda)$) implies $Q_{\bs{Y}}(f_n) \stackrel{\text{Law}}{\rightarrow} \mathcal{S}(0,1)$ (or $Z(\lambda)$) for every sequence $\bs{Y}$ of freely independent random variables verifying Assumption {\bf (1)}, generalizing the universality results established in \cite[Theorem 1.4]{NourdinDeya}, corresponding to the case $h=1$.
In particular, if $h =2$, the corresponding statements concern vectors of homogeneous sums in centered free Poisson random variables of parameter $1$, with respect to both semicircular and free Poisson approximations (when $d$ is an even integer).\\
\end{rmk}

For the subsequent remarks, it is convenient to explicitly reformulate the above universality phenomena in the case $m=1$.

\begin{cor}
\label{Univ1}
If $d\geq 2$, let $f_n:[n]^d \rightarrow \mathbb{R}$ be a sequence of symmetric admissible kernels. If $\bs{h}=(h_1,\dots,h_d)$ is a vector of positive integers, such that $h_i = h_{d-i+1}$ for $i=1,\dots,\lfloor \frac{d}{2}\rfloor$,  the following statements are equivalent as $n \rightarrow \infty$:
\begin{itemize}
\item[(i)] $Q_{\bs{S}}^{(\bs{h})}(f_n) \stackrel{\text{Law}}{\longrightarrow} \mathcal{S}(0,1)$;
\item[(ii)] for every sequence $\mathbf{Y}=\{Y_i\}_{i\geq 1}$ of freely independent and identically distributed random variables, verifying Assumption {\bf (1)},  $Q_{\bs{Y}}(f_n) \stackrel{\text{Law}}{\longrightarrow} \mathcal{S}(0,1)$.
\end{itemize}
\end{cor}


\begin{cor}
\label{Univ2}
Let the hypotheses of Corollary \ref{Univ1} prevail, and assume that both $d$ and $h_1+\cdots +h_d$  are even integers. For a sequence of symmetric admissible kernels $f_n:[n]^d \rightarrow \mathbb{R}$,  such that $ \lim\limits_{n \rightarrow \infty}\varphi\big((Q_{\bs{S}}^{(\bs{h})}(f_n))^{2}\big) = \lambda > 0,$ the following statements are equivalent as $n\rightarrow \infty$:
\begin{itemize}
\item[(i)] $Q_{\bs{S}}^{(\bs{h})}(f_n) \stackrel{\text{Law}}{\longrightarrow} Z(\lambda) $;
\item[(ii)] for every sequence $\bs{Y}=\{Y_i\}_{i\geq 1}$ of freely independent and identically distributed random variables, verifying Assumption {\bf (1)}, $Q_{\bs{Y}}(f_n) \stackrel{\text{Law}}{\longrightarrow} Z(\lambda)$.\\
\end{itemize}
\end{cor}


\begin{rmk}
By virtue of Theorems \ref{ConvChebySemicircular}, \ref{ConvFreePoissonPARI}, and Propositions \ref{contraction5}, \ref{contraction6}, the conditions required for the kernels $f_n$  for the convergence of $Q_{\bs{S}}^{(\bs{h})}(f_n)$ towards the semicircular and the free Poisson laws, do not depend on the choice of the orders $\bs{h}=(h_1,\dots,h_d)$. Therefore, the convergence of a vector of Chebyshev sums of given orders $\bs{h}=(h_1,\dots,h_d)$, based on a semicircular system, towards both the semicircular and the free Poisson law, is equivalent to the convergence towards that laws for any other vector of Chebyshev sums with the same kernels. In particular, this equivalence holds true for homogeneous sums based on the $h$-th Chebyshev polynomial, for different $h$'s. For notational convenience, these remarks are stated explicitly only in the one dimensional case. 
\end{rmk}

\begin{cor}
Let $d\geq 2$ and $f_n:[n]^d \rightarrow \mathbb{R}$ be a symmetric admissible kernel, for every $n\geq 1$. The following assertions are equivalent as $n \rightarrow \infty$:
\begin{itemize}
\item[(i)] there exist integers $\bs{h}=(h_1,\dots,h_d)$, with $h_i = h_{d-i+1}$ for $i= 1,\dots, \lfloor \frac{d}{2}\rfloor$, such that:
$$ Q_{\bs{S}}^{(\bs{h})}(f_n) \stackrel{\text{ Law }}{\longrightarrow} \mathcal{S}(0,1) \;;$$
\item[(ii)] for every $\bs{k}=(k_1,\dots,k_d)$ such that $k_i = k_{d-i+1}$ for $i= 1,\dots, \lfloor \frac{d}{2}\rfloor$, 
$$Q_{\bs{S}}^{(\bs{k})}(f_n) \stackrel{\text{ Law }}{\longrightarrow} \mathcal{S}(0,1).$$
\end{itemize}
\end{cor}

\begin{cor}
\label{coroll1}
Let $d\geq 2$ and $f_n:[n]^d \rightarrow \mathbb{R}$ be a symmetric kernel for every $n\geq 1$. The following assertions are equivalent as $n\rightarrow \infty$:
\begin{itemize}
\item[(i)] $Q_{\bs{S}}(f_n) \stackrel{\text{ Law }}{\longrightarrow} \mathcal{S}(0,1)\;;$
\item[(ii)] if $\bs{Z}=\{Z_i\}_{i \geq 1}$ is a sequence of freely independent, centered random variables with the free Poisson distribution of parameter $1$, then $$Q_{\bs{Z}}(f_n) \stackrel{\text{ Law }}{\longrightarrow} \mathcal{S}(0,1).$$
\end{itemize}
\end{cor}

Under a suitable parity assumption on the degree of the homogeneous sums, similar results can be formulated for free Poisson approximations. 

\begin{cor}
Let $d\geq 2$ be even, and consider a sequence $f_n:[n]^d \rightarrow \mathbb{R}$ of  symmetric kernels, vanishing on diagonal, such that $\vert\vert f_n\vert\vert^2 \rightarrow \lambda > 0$. Then, the following assertions are equivalent as $n \rightarrow \infty$:
\begin{itemize}
\item[(i)] there exist integers $\bs{h}=(h_1,\dots,h_d)$, with $h_i = h_{d-i+1}$ for $i= 1,\dots, \lfloor \frac{d}{2}\rfloor$, such that $$ Q_{\bs{S}}^{(\bs{h})}(f_n) \stackrel{\text{ Law }}{\longrightarrow} Z(\lambda) ;$$
\item[(ii)] for every $\bs{k}=(k_1,\dots,k_d)$ such that $k_i = k_{d-i+1}$ for $i= 1,\dots, \lfloor \frac{d}{2}\rfloor$, 
$$Q_{\bs{S}}^{(\bs{k})}(f_n) \stackrel{\text{ Law }}{\longrightarrow} Z(\lambda).$$
\end{itemize}
\end{cor}

\begin{cor}
\label{coroll2}
Let $d\geq 2$ be even, and consider a sequence $f_n:[n]^d \rightarrow \mathbb{R}$ of  symmetric kernels, vanishing on diagonal, such that $\vert\vert f_n\vert\vert^2 \rightarrow \lambda > 0$. The following assertions are equivalent as $n\rightarrow \infty$:
\begin{itemize}
\item[(i)] $Q_{\bs{S}}(f_n) \stackrel{\text{ Law }}{\longrightarrow} Z(\lambda);$
\item[(ii)] if $\bs{Z}=\{Z_i\}_{i \geq 1}$ is a sequence of freely independent centered random variables with the free Poisson distribution of parameter $1$, then $$Q_{\bs{Z}}(f_n) \stackrel{\text{ Law }}{\longrightarrow} Z(\lambda).$$
\end{itemize}
\end{cor}

\begin{rmk}
If $\bs{Z}=\{Z_i\}_{i \geq 1}$ is a sequence of freely independent centered random variables with the free Poisson distribution of parameter $1$, the random variable $Q_{\bs{Z}}(f_n)$ belongs to the so called \textit{Free Poisson algebra}. 
In regard to free Poisson approximations, the reference \cite[Theorem 1.5]{Solesne2} extends the  transfer principle between Free Poisson and Wigner homogeneous sums (of even degree $d$), stated with Corollary \ref{coroll1}, to integrals of tamed mirror symmetric functions in $\mathrm{L}^2(\mathbb{R}_+^{d})$. Moreover,  \cite[Theorem 1.5]{Solesne2} provides a general counterexample showing that the transfer principle fails for free Poisson approximations of integrals of odd order.\\
\end{rmk}

\begin{exm}
As an application of Corollary \ref{coroll1}, consider the homogeneous sum:
$$ Q_{\bs{x}}(f_n) = \dfrac{1}{\sqrt{2n-2}}\sum_{i=2}^{n}(x_1 x_i + x_i x_1).$$
As shown in \cite{NourdinDeya} in the first counterexample, if $\bs{S}=\{S_i\}_{i \geq 1}$ denotes a sequence of freely independent, standard semicircular random variables, 
$Q_{\bs{S}}(f_n)$ converges in law to $\frac{1}{\sqrt{2}}(s_1 s_2 + s_2 s_1)$, where $s_1,s_2$ are freely independent standard semicircular random variables, and therefore its limit is  Tetilla distributed\index{Tetilla law}. Corollary \ref{coroll1} gives the additional information that $Q_{\bs{Z}}(f_n)$, $Z \sim Z(1)$, cannot converge towards the semicircular law or the free Poisson law, as well as any other sequence $\{Q_{\bs{Y}}(f_n)\}_{n\geq 1}$, for $Y \stackrel{\text{Law}}{=} U_h(S)$, for any $h\geq 3$. 

Moreover, remark that with the same counterexample, the authors were meant to show that the free symmetric Rademacher law $\frac{1}{2}(\delta_1 + \delta_{-1})$ is not universal for semicircular approximations of homogeneous sums. Indeed, the authors proved that if $\bs{X}=\{X_i\}_{i\geq 1}$ is a sequence of freely independent Rademacher random variables, then $Q_{\bs{X}}(f_n)$ has asymptotically semicircular distribution. This is consistent with the fact that the free Rademacher law is not admissible for any chaotic random variable of the type $U_n(S)$, and it implies in turn that the Tetilla law cannot be a universal limit law for homogeneous sums in any sequence of freely independent random variables: some restrictions might be necessary (in this regard, for instance, Remark \ref{Extension_Tetilla}).\\
\end{exm}

\begin{rmk}
By considering the estimate (\ref{magg1}), it follows that if $d\geq 2$, and $f_n:[n]^d \rightarrow \mathbb{R}$ is a sequence of symmetric admissible kernels, satisfying $ \|f_n \stackrel{d-1}{\smallfrown} f_n \| \rightarrow 0$ as $n \rightarrow \infty$, then the asymptotic distribution of $Q_{\bs{X}}^{(\bs{h})}(f_n)$  for any  vector of orders $(\bs{h})$ (and, in particular, that of $Q_{\bs{X}}(f_n)$), never depends on the distribution of the sequence $\bs{X}=\{X_i\}_{i \geq 1}$. 
In order to provide an instance where the universality behaviour does not occur, consider the homogeneous sum of the previous counterexample:
$$ Q_{\bs{X}}(f_n) = \dfrac{1}{\sqrt{n-2}}\sum_{i=2}^{n}(X_1 X_i + X_i X_1) = \sum_{i,j=1}^{n}f_n(i,j)X_i X_j,$$
with
$$ f_n(i,j) = 
\begin{cases}
0 & \text{ if } i,j \neq 1, \\
0 & \text{ if } i = j,\\
\dfrac{1}{\sqrt{n-2}} & \text{ if } i \neq j, \text{and } 1 \in \{i,j\}.
\end{cases}
$$
Simple computations give:
\begin{align*}
\|f_n \stackrel{1}{\smallfrown} f_n\|^{2} &= \sum_{i,j=1}^{n}\big(f_n  \stackrel{1}{\smallfrown} f_n(i,j)\big)^{2} = \sum_{i,j=1}^n \bigg(  \sum_{k=1}^n f_n(i,k)f_n(k,j)\bigg)^2 \\
&= \sum_{i,j=2}^{n}\Big( f_n(i,1)f_n(1,j)\Big)^2 + \bigg(\sum_{k=2}^{n}f_n(1,k)f_n(k,1)\bigg)^2 \\
&= 2\dfrac{(n-1)^2}{(n-2)^2} \, ,
\end{align*}
so that $\lim\limits_{n\rightarrow \infty} \|f_n \stackrel{1}{\smallfrown} f_n\| = 2 \neq 0$, and in turn we can conclude that $Q_{\bs{S}}(f_n)$ does not have asymptotic semicircular law, if $\bs{S}$ denotes a sequence of freely independent standard semicircular random variables. As to influence functions,
\begin{itemize}
\item[-] $\mathrm{Inf}_1(f_n) = 2 \sum\limits_{j=2}^{n}f_n(1,j)^{2} = 2 \sum\limits_{j=2}^{n}\dfrac{1}{n-2} = 2\dfrac{n-1}{n-2}$;
\item[-] for every $i=2,\dots,n$, $\mathrm{Inf}_i(f_n) = 2\sum\limits_{j=1}^{n}f_n(i,j)^{2} = 2\;f_n(i,1)^{2} = \dfrac{2}{n-2},$
\end{itemize}
so that $\tau_n = \max\limits_{i=1,\dots,n}\mathrm{Inf}_i(f_n) = \mathrm{Inf}_1(f_n)$ does not vanish in the limit.\\
\end{rmk}

\subsection{The commutative counterpart: Hermite sums}

Consider the (monic) Hermite polynomials\index{Hermite polynomials}: 
$$H_0(x)=1, H_1(x)=x, \quad H_{n+1}(x) = xH_n(x) -nH_{n-1}(x) \quad \forall n\geq 1,$$
and recall that it forms the (unique) family of  polynomials that is orthogonal with respect to the Gaussian distribution. Let $\bs{X}= \{X_i\}_{i \geq 1}$ be a sequence of independent random variables on a fixed classical probability space $(\Omega, \mathcal{F}, \mathbf{P})$. As we shall see in Part \ref{FMT}, in the commutative setting, \textit{admissible kernels} for homogeneous sums are symmetric functions vanishing on diagonals (see Definition \ref{Admissible_Classic}). 

\begin{defn}
The \textit{Hermite sum} \index{Hermite polynomials!Hermite sum} of orders $\bs{m} = (m_1,\dots,m_d)$, and symmetric admissible coefficient $f:[n]^d \rightarrow \mathbb{R}$ is a random variable based on $\bs{X}$ of the type:
\begin{equation*}
Q_{\bs{X}}^{(\bs{m})}(f) =\sum_{i_1,\dots,i_d=1}^n f(i_1,\dots,i_d) H_{m_1}(X_{i_1})\cdots H_{m_d}(X_{i_d}). 
\end{equation*}
\end{defn}

Note that, with the language of ensembles, and adapting the notation introduced in \eqref{ensemble} and in \eqref{HomInEnsemble}, one can also write:
$$Q_{\bs{X}}^{(\bs{m})}(f)= Q_{\bs{\mathcal{X}}^{(n)}}(f),$$
where $\bs{\mathcal{X}}^{(n)} = (\bs{\mathcal{X}}_{1},\dots, \bs{\mathcal{X}}_{n}), \quad \bs{\mathcal{X}}_{i} = \{H_{m_1}(X_i),\dots, H_{m_d}(X_i)\}$.

For $d\geq 2$, assume that for the given integers $m_1,\dots,m_d$, the sequence $\bs{X}=\{X_i\}_{i\geq 1}$ of independent random variables is such that, for every $j$, $H_{m_j}(X_i)$ is centered, has unit variance and uniformly bounded third moments (namely, that there exists a constant $B>0$ such that $\sup\limits_{i\geq 1}\mathbb{E}[|H_{m_j}(X_i)|^{3}] < B$ for all $j=1,\dots, d$). Under these assumptions, the ensembles $\bs{\mathcal{X}}_{i} = \{H_{m_1}(X_i),\dots, H_{m_d}(X_i)\}$ are $(2,3,\eta)$-hypercontractive (in the sense of \cite{Mossel}), and therefore Theorem \ref{invMossel} can be exploited to produce the commutative counterpart to Theorem \ref{MultUniversality}, in the sense of the forthcoming Proposition \ref{HermiteSum}. \\

Let  $\mathcal{H}$ be a (separable) real Hilbert space, and let $\{e_j\}_{j\geq 1}$ be one orthonormal basis. Consider an isonormal Gaussian process  $G = \{G(h): h \in \mathcal{H}\}$ (note that, for a given covariance function, there exists an isonormal Gaussian process determined by the given covariance, see \cite{NourdinPeccatilibro}). Then, if $G_j = G(e_j)$, $G_j \sim \mathcal{N}(0,1)$. It is a standard result that $\frac{1}{n!}H_{n}(G_i) = I_{n}^{G}(e_i^{\otimes n})$ is centered, has unit variance, and satisfies a hypercontractivity property (see, for instance, \cite{NourdinPeccatilibro}).  Therefore, the sequence $\bs{N}=\{N_i\}_{i\geq 1}$ is a sequence of independent standard normal variables. Moreover, the Hermite sum  $Q_{\bs{\mathcal{N}}^{(n)}}(f)$ of orders $\bs{m}=(m_1,\dots,m_d)$, based on the sequence $\bs{N}$, satisfies  $Q_{\bs{\mathcal{N}}^{(n)}}(f) = I_M^{G}(k(f))$, with $k(f)$ as in (\ref{kN1}), $M=m_1 + \cdots + m_d$. Note that, in general, the symmetry of the kernel $f$ does not imply the symmetry of  $k:=k(f)$, but if $\widetilde{k}$ denotes its standard symmetrization, then $I_M^{G}(k) = I_M^{G}(\widetilde{k})$ (see \cite[Chapter 5.5]{PeccatiTaqqu}).\\

In conclusion, all the \textit{fourth moment}-type statements for Normal and Gamma  approximations of chaotic random variables  (\cite{NualartPeccati}, \cite[Theorem 1.2]{NourdinPeccati2}), imply the corresponding universality results for Hermite sums, in the following sense.

\begin{prop}\label{HermiteSum}
Let $f_n:[n]^d \rightarrow \mathbb{R}$ be a sequence of symmetric admissible kernels. If the above notation prevails, the following statements are equivalent as $n \rightarrow \infty$:
\begin{itemize}
\item[(i)] $Q_{\bs{N}}^{(\bs{m})}(f_n) \stackrel{\text{Law}}{\longrightarrow} \mathcal{N}(0,1) $;
\item[(ii)] for every sequence $\bs{Y}= \{Y_i\}_{i\geq 1}$ of i.i.d. centered random variables, with unit variance, $Q_{\bs{Y}}(f_n) \stackrel{\text{Law}}{\longrightarrow} \mathcal{N}(0,1)$.
\end{itemize}
\end{prop}

A similar statement holds whenever $d\geq 2$ is even,  and for the given sequence of symmetric admissible kernels, $\lim\limits_{n\rightarrow \infty}\E[Q_{\bs{X}}^{(\bs{m})}(f_n)^2] = 2\nu, \;\lim\limits_{n\rightarrow \infty}\E[Q_{\bs{X}}^{(\bs{m})}(f_n)^3] = 8\nu$, for a given  $\nu >0$, and whenever the target distribution is replaced by $F(\nu)=2\Gamma(\frac{\nu}{2}) - \nu$, with $\Gamma(\nu)$ denoting the Gamma distribution $\Gamma(\nu,1)$ (see, for comparison, Theorem \ref{GenGammaAppr}).

In particular, the choice $m_j= m \geq 1$ for all $j=1,\dots,d$, establishes  that homogeneous sums based on chaotic random variables of the form $H_m(N_i)$, $N_i\sim \mathcal{N}(0,1)$,  behave universally for both Gaussian and Gamma approximations of homogeneous sums, extending \cite[Theorem 1.10 and Theorem 1.12]{NourdinPeccatiReinert}, that correspond to $m=1$.

\begin{rmk}The same consequences will be partially recovered through a different approach in Part \ref{FMT}, Chapter \ref{Classic}, where a whole class of universal laws will be provided.
\end{rmk}

%

\thispagestyle{empty}
\part{A general \textit{Fourth Moment} criterion}\label{FMT}

\chapter*{Synopsis}
\addcontentsline{toc}{chapter}{Synopsis}

The topics covered in the present part are taken from \cite{NPPS}.\\

In the following, the focus will be on statistics having the form of homogeneous sums:
$$ Q_{\X}(f) = \sum_{i_1,\dots,i_d=1}^n f(i_1,\dots,i_d) X_{i_1}\cdots X_{i_d},$$
where $d\geq 2$, $f:[n]^d \rightarrow \mathbb{R}$ is a symmetric function, such that $f(i_1,\dots,i_d)=0$ if $i_j=i_k$ for $j \neq k$,  and $\X=\{X_i\}_{i\geq 1}$ denotes a sequence of independent copies of a random variable $X$, defined on  a classical probability space $(\Omega,\mathcal{F},\mathbb{P})$.\\

In \cite{deJong2}, it is shown that the Central Limit Theorem (CLT, for short) holds for a sequence  of the type $Q_{\X}(f_n)$, if $X$ has finite fourth moment and under the assumptions:
\begin{enumerate}
\item $\tau_n:=\max\limits_{i=1,\dots,n}\sum\limits_{i_2,\dots,i_d=1}^n f_n(i,i_2,\dots,i_d)^2 \rightarrow 0$ as $n\rightarrow \infty$;
\item $\E[Q_{\X}(f_n)^4]\rightarrow 3 = \E[N^4], N \sim \mathcal{N}(0,1)$,  as $n\rightarrow \infty$,
\end{enumerate}
generalizing the Lindberg condition for linear random sums (see, for instance, \cite[Theorem 11.1.1]{NourdinPeccatilibro}); this statement is customarily referred to as \textit{de Jong's Criterion} for central convergence. \\

Gaussian homogeneous sums, that is, random variables of the type $Q_{\mathbf{N}}(f)$, where $\mathbf{N}=\{N_i\}_{i \geq 1}$ is a sequence of independent Gaussian random variables, are, in this regard, special. Indeed, in 2005, Nualart and Peccati proved that, when dealing with the Gaussian Wiener Chaos, the vanishing condition on $\tau_n$ can be dropped, and hence the convergence of the fourth moments  is sufficient for the normal approximation of multiple Wiener integrals (see \cite[Theorem 1]{NualartPeccati} for the original statement). This result is usually referred to as the \textit{Fourth Moment Theorem}, or \textit{Nualart-Peccati Criterion}, and represents a useful simplification of the method of moments and cumulants, which is in general employed to prove convergence in law when the target law is determined by its moments. A similar result holds for the Poisson homogeneous Chaos, that is, for random variables of the type $Q_{\mathbf{P}}(f_n)$, with $\mathbf{P}$ being a sequence of independent random variables with the Poisson distribution (see \cite[Theorem 3.2]{PeccatiZheng1} for the original statement).\\
When the convergences of the second and of the fourth moments are sufficient for the CLT to hold for a sequence $\{Q_{\bs{X}}(f_n)\}_{n\geq 1}$, it will be customarily said that the \textit{fourth moment phenomenon} occurs for $X$ (see the forthcoming Definition \ref{defcom}). The question under consideration in the sequel is the following: are there other examples of \textit{fourth moment phenomenon}, other than the Gaussian and the Poisson Chaos? \\

The Gaussian and the Poisson Wiener Chaos share another peculiar feature: the \textit{universality}, in the sense that, if central convergence is established for a sequence of Gaussian (resp. Poisson) homogeneous sums, then one can obtain the same asymptotic behaviour by replacing the Gaussian (resp. Poisson) sequence with another sequence of independent and identically distributed variables (satisfying some minimal moment assumptions). The analysis of the universality properties within Gaussian Wiener Chaos is addressed in \cite{NourdinPeccatiReinert}, where the authors examine normal and $\chi^2$-approximations, both in the unidimensional and the multidimensional setting. The Poisson counterpart for normal approximations has been established  in \cite[Theorems 3.4, 3.8]{PeccatiZheng1}. As a consequence of Proposition \ref{HermiteSum} of Part \ref{Invariance}, analogous statements hold for homogeneous sums in independent copies of $H_m(N)$, for every $m\geq 1$, $N\sim \mathcal{N}(0,1)$, and $H_m(x)$ denoting the $m$-th Hermite polynomial. Beyond the analysis of the fourth moment phenomenon, in the following we will seek for
 the properties of the distribution of $X$ that entails the universality phenomenon for a sequence $Q_{\X}(f_n)$, trying to determine  if there is any dependence with the Fourth Moment Theorem.\\

The discussion will also cover the non-commutative setting. If $Y$ is a random variable in a fixed free probability space $(\mathcal{A},\varphi)$,  consider the random variable:
$$ Q_{\Y}(f) = \sum_{i_1,\dots,i_d=1}^n f(i_1,\dots,i_d) Y_{i_1}\cdots Y_{i_d},$$
where $\Y=\{Y_i\}_{i\geq 1}$ is a sequence of freely independent copies of $Y$, and $f:[n]^d \rightarrow \mathbb{R}$ is a suitable coefficient. By virtue of Theorems \ref{knps} and \ref{invnoncom} of Part \ref{Invariance}, it is known that when $Y$ has the standard semicircle law ($Y\sim \mathcal{S}(0,1)$ for short), both the fourth moment and the universality phenomena  occur for semicircular approximations of homogeneous sums. As a consequence of Corollaries \ref{Univ1}, \ref{Univ2} of Part \ref{Invariance}, the universality property holds, in general, for homogeneous sums in freely independent random variables distributed according to the law of $Y= U_k(S)$, for every $k\geq 1$, with $U_k(x)$ denoting the $k$-th Chebyshev polynomial (of the second kind), and $S \sim \mathcal{S}(0,1)$. \\
Similarly to the commutative setting, we will focus on the properties of the distribution of $Y$, that determine the Fourth Moment  and the universality phenomena  for $Q_{\Y}(f_n)$. \\

The main achievements of the present Part are stated in Theorems \ref{superTeo1} and \ref{superTeo2}, where a general \textit{fourth moment criterion} for homogeneous sums is provided and, in the meantime, new universal laws for central convergence are exhibited. Extensions homogeneous sums in independent, non necessarily identically distributed random variables, and to non-central convergence, are also provided. The whole discussion covers both the classical (Chapter \ref{Classic}) and the free setting (Chapter \ref{Free}). \\

Remark that, while in Part \ref{Invariance} the universality phenomenon was described starting from an invariance principle for non-commutative spaces, here the proofs rely on new combinatorial formulae for the kurtosis of $Q_{\X}(f)$ and $Q_{\Y}(f)$ (see Propositions \ref{formulaGauss} and \ref{formula} respectively, for the classical and the free setting), which are of independent interest. More precisely, the answers to our questions are found by showing that  a sufficient condition for both the fourth moment and the universality phenomena to occur is the non-negativity of the kurtosis of $X$ (resp. $Y$). Trivially this condition is fulfilled for the Gaussian and the semicircular law.\\ 

As a consequence, the multidimensional transfer principle between Wiener and Wigner Chaos, stated in \cite[Theorem 1.6]{NouSpeiPec}, can be extended to a transfer principle between homogeneous sums in independent random variables, having non-negative kurtosis in classical probability spaces, and homogeneous sums in freely independent random variables, with non-negative free kurtosis, in free probability spaces (see Theorem \ref{Transfer} in the sequel). The main step to be accomplished in this direction is showing that componentwise and joint central convergence are equivalent not only for Gaussian/semicircular homogeneous sums \cite{PeccatiTudor,NouSpeiPec}, but for all homogeneous sums in independent random variables with non-negative kurtosis, in both the commutative and non-commutative framework (see Theorems \ref{ComponentJoint} and \ref{ComponentJointFree} respectively). \\ 

Finally, in the last chapter, the optimality of the conditions provided with Theorems \ref{superTeo1} and \ref{superTeo2} is discussed, and the \textit{problem of thresholds} is introduced. More precisely,  Theorem \ref{ThresholdClassic} (respectively, Theorem \ref{ThresholdFree} in the free setting), proves the existence of a lower bound $r_d$ for the fourth moment of $X$ (resp. $Y$), such that the fourth moment of $X$ being greater that $r_d$ is also a necessary condition for the Fourth Moment Theorem to hold, for homogeneous sums of degree $d$, in independent copies of $X$ (resp. freely independent copies of $Y$).\\

\section*{Bibliographic comments}

Since the pioneering works of \cite{Hoeffding}, normal approximations of $U$-statistics is a crucial area of research. \\

In the classical probability setting, the first CLTs subjected to fourth moment conditions were provided in \cite{deJong1}, for quadratic sums, and in \cite{deJong2} for higher order  homogeneous sums $Q_{\bs{X}}(f_n)$: if the coefficients $f_n$  verify a Lindberg-type condition,  and if $X$ has finite fourth moment, the convergence of the fourth moments $\E[Q_{\bs{X}}(f_n)^4] \rightarrow 3$ is sufficient for the CLT to hold for $Q_{\bs{X}}(f_n)$. In view of Theorem \ref{superTeo1},  the Lindberg-type condition can be dropped whenever $X$ has a fourth moment that is equal or greater than 3, and zero third moment. Further recent developments around de Jong's theorems have appeared in \cite{the, peth}.\\

As already remarked in the Introduction, the first improvement of de Jong's Theorem concerns Gaussian random fields: in \cite{NualartPeccati}, the authors established the Fourth Moment Theorem for random variables living in the Wiener Chaos,  through the combination of the Malliavin Calculus and the Stein's method (see \cite{NP_ptrf}), a technique that has allowed ever since a better understanding of the fourth moment phenomenon, and of related topics: see, for instance, \cite{web} for an update collection of results on the subject.  The reader can consult \cite{Guillaume1} for an alternative simple proof of the results of \cite{NualartPeccati}, and also \cite{Nourdin} for another proof of the Nualart-Peccati Criterion based on multiplication formulae for multiple Wiener integrals. Extension of the Fourth Moment Theorem can be found in \cite{Guillaume2}, to higher moments, in \cite{NourdinPeccatiReveillac, PeccatiTudor} for the multidimensional case, and in \cite{NourdinPeccatiSwan} for an information-theoretical setting, where entropic bounds are provided for the multidimensional Fourth Moment Theorem on Wiener Chaos. A self-contained introduction to the subject is contained in the monograph \cite{NourdinPeccatilibro}. \\ 
Other than for Gaussian fields, the fourth moment phenomenon has been studied for the Poisson Wiener Chaos: see \cite{PeccatiZheng1, PeccatiZheng2} for normal approximations of Poisson homogeneous sums and \cite{the, Peccati_Lachieze} for an analogous analysis for more general functionals of Poisson measures having the form of finite sums of multiple integrals with constant-sign kernels, with applications to geometric random graphs. See also \cite{PeccatiSole} for bounds in CLTs for Poisson functionals, involving Stein's method and Malliavin Calculus. Further generalizations include the Fourth Moment Theorem for Markov diffusion generators \cite{Guillaume2}, and for infinitely divisible laws \cite{Arizmendi1}:  see, also, \cite{Arizmendi1bis} for quantitative estimates for the Kolmogorov distance between infinitely divisible laws and the normal (resp. semicircular) law, assessing the distance between the corresponding fourth moments.  \\
  
In the free probability setting, the analysis of the \textit{fourth moment phenomenon} for non-linear functionals of a free Brownian motion started in \cite{KempNourdinPeccatiSpeicher}, where the authors provided the non-commutative counterpart to the findings in \cite{NualartPeccati} and deal with stochastic analysis via the free version of the Malliavin Calculus, introduced in \cite{BianeSpeicher}.
Extensions of the Fourth Moment Theorem are provided in \cite[Theorem 1.7]{qbrownian} for multiple integrals with respect to a  $q$-Brownian motion, and in \cite[Theorem 4.1]{Solesne} for the free Poisson Chaos.  See, moreover, \cite[Theorem 1.3]{NouSpeiPec} for a multidimensional version of the Fourth Moment Theorem for semicircular approximations, as well as \cite{Arizmendi1} for the Fourth Moment Theorem for freely infinitely divisible laws. New universality results for homogeneous sums  have been discussed in Part \ref{Invariance}. Recently, Poisson limits on the free Poisson algebra, in terms of fourth moment conditions,  have been studied in \cite{Solesne2}. \\


\chapter{The classical probability setting: Fourth Moment Theorem and universality}\label{Classic}

Throughout the present chapter, $(\Omega, \mathcal{F},\mathbb{P})$ will denote a fixed probability space, and $\mathbb{E}$ the corresponding expectation. 
\section{Preliminaries}
For every $n\in \mathbb{N}$, set $[n] := \{1,\dots,n\}$.  Here, the definition of admissible kernels has to be slightly revisited to ensure that the corresponding homogeneous sums are suitably scaled. Moreover, a stronger symmetry assumption will be required.

\begin{defn}\label{Admissible_Classic}
Let $d\geq 2$. For $n \in \mathbb{N}$, an \textbf{admissible kernel} is a function $f:[n]^d \to \mathbb{R}$  satisfying the following properties:
\begin{enumerate}
\item[(i)] vanishing on diagonals: $f(i_1,\dots,i_d)=0$ whenever  $i_j=i_k$ for some $k\neq j$;
\item[(ii)] symmetry: $f(i_1,\dots,i_d)=f(i_{\sigma(1)},\dots,i_{\sigma(d)})$ for any permutation $\sigma \in \mathfrak{S}_d$, and any \linebreak$(i_1,\dots,i_d)\in [n]^d$;
\item[(iii)] $f$ has unit variance: $d!\sum\limits_{i_1,\dots,i_d =1}^n f(i_1,\dots,i_d)^2=1\, .$\\
\end{enumerate}
\end{defn}

For instance, for $d=2$, the kernel  $f(i,j) =  {\bf 1}_{\{i\neq j\}}/\sqrt{2n(n-1)}$ is admissible.

\begin{rmk} Note that the assumptions $(ii)$ and $(iii)$ are matter of convenience: indeed, given a function $f:[n]^d\to\mathbb{R}$ verifying $(i)$, it is always possible to generate an admissible kernel $\tilde{f}$ by first symmetrizing $f$, and then by properly renormalizing it. \\
\end{rmk}

Let $X$ be a random variable defined on  $(\Omega,\mathcal{F},\mathbb{P})$, such that: 
\begin{itemize}
\item[-] $X$ is centered and has unit variance;
\item[-] $\E[X^3]=0$;
\item[-] there exists $\epsilon > 0$ such that $\E[X^{4+\epsilon}] < \infty$.
\end{itemize}
When $X$ satisfies these conditions, it will be said, for short, that $X$ satisfies Assumption {\bf (2)} (unless other specified, it will always be assumed that $X$ satisfies Assumption {\bf (2)}).

Let $\X=\{X_i\}_{i\geq 1}$ be a sequence of independent copies of $X$ (i.i.d. for short)\footnote{As usual, it will be assumed that the $X_i$'s are defined over the same probability space.}. For any admissible kernel $f:[n]^d\to\mathbb{R}$, consider the statistics $Q_{\X}(f)$ defined by:
\begin{eqnarray}\label{e:alv}
Q_{\X}(f) &=&\sum_{i_1,\dots,i_d =1}^n f(i_1,\dots,i_d) X_{i_1}\cdots X_{i_d}.\label{F}
\end{eqnarray}
Note that, since $f$ is admissible and $X$ satisfies Assumption {\bf (2)},  then $\E[Q_{\X}(f) ]=0$ and $\E[Q_{\X}(f)^2]=1$.

\begin{defn}\label{defcom}
Let $X$ be a random variable verifying Assumption {\bf (2)}, and let $\X=\{X_i\}_{i\geq 1}$ be a sequence of independent copies of $X$. 
\begin{itemize}
\item[(a)] We say that $X$ \textbf{satisfies the  Fourth Moment Theorem at the order $d\geq 2$} if\index{Fourth Moment Theorem (classic)}, for every sequence $f_n:[n]^d\to\mathbb{R}$ of admissible kernels, the following statements are equivalent for $n\to\infty$:
\begin{enumerate}
\item[(i)] $Q_{\X}(f_n) \xrightarrow{\text{\rm Law}}\mathscr{N}(0,1)$;
\item[(ii)]  $\E[Q_{\X}(f_n)^4]\to \E[N^4]=3$, where $N \sim \mathcal{N}(0,1)$.
\end{enumerate}
\item[(b)] $X$ is said to be \textbf{universal at the order $d$} \index{Universal law, (classic)}(for normal approximations of homogeneous sums) if, for any sequence $f_n:[n]^d\to\mathbb{R}$ of admissible kernels, $Q_{\X}(f_n) \xrightarrow{\text{\rm Law}}\mathscr{N}(0,1)$ implies: 
$$ \tau_n(f_n) := \max_{ i=1,\dots, n} \mathrm{Inf}_i(f_n) \longrightarrow 0, \text{ as } n\to\infty,$$
where $\mathrm{Inf}_i(f_n):=\sum\limits_{i_2,\ldots,i_d=1}^n f_n(i ,i_2,\ldots,i_d)^2$ is the $i$-th \textbf{influence function} of $f_n$\index{Influence function}.\\
\end{itemize}
\end{defn}

Note that, if $X$ is universal at the order $d$, Theorem \ref{invMossel}, Part \ref{Invariance}, yields that the convergence $Q_{\X}(f_n) \xrightarrow{\text{\rm Law}}\mathscr{N}(0,1)$ implies $Q_{\mathbf{Z}}(f_n) \xrightarrow{\text{\rm Law}}\mathscr{N}(0,1)$  for every sequence $\mathbf{Z}$ of independent centered random variables with unit variance, and with uniformly bounded moments. \\

The goal pursued in the present chapter is inspired by the  groundbreaking works \linebreak\cite[Theorem 1]{NualartPeccati} and  \cite[Theorem 1.10]{NourdinPeccatiReinert}, where it is shown that the Gaussian distribution meets Definition \ref{defcom}, as summarized in the subsequent Theorems \ref{FMTClassic} and Theorem \ref{inv}, respectively (Lemma \ref{magg} has to be taken into account).

\begin{thm}\label{FMTClassic}
For a fixed $d \geq 2$, let $\{k_n\}_{n\geq 1}$ be a sequence of symmetric function in $\mathrm{L}^2({\mathbb{R}_+^{d}})$. As $n\rightarrow \infty$, if $\E[I_d^{W}(k_n)^2] \rightarrow 1$, the following conditions are equivalent:
\begin{itemize}
\item[(i)] $I_d^{W}(k_n) \stackrel{\text{Law}}{\longrightarrow} \mathcal{N}(0,1)$;
\item[(ii)] $\E[I_d^W(k_n)^4]\rightarrow \E[N^4] = 3, N \sim \mathcal{N}(0,1)$;
\item[(ii)] for every $r=1,\dots,d-1$, $\|k_n \otimes_r k_n \| \rightarrow 0$ (where the contraction $\otimes_r$ has been introduced in \eqref{GenContraction}).
\end{itemize}
\end{thm}

\begin{thm}\label{inv}
For a fixed $d\geq 2$, let $f_n:[n]^d\to\mathbb{R}$ be a sequence of admissible kernels as in Definition \ref{Admissible_Classic}. If $\bs{N} = \{N_i\}_{i \geq 1}$ is a sequence of independent standard Gaussian random variables, the following statements are equivalent as $n\to\infty$:
\begin{itemize}
\item[(i)]  $Q_{\bf N}(f_n) \xrightarrow{\text{\rm Law}}\mathscr{N}(0,1)$, 
\item[(ii)] $Q_{\X}(f_n) \xrightarrow{\text{\rm Law}}\mathscr{N}(0,1)$ for every sequence $\bs{X}$ of independent, centered random variables, having unit variance.\\
\end{itemize} 
\end{thm}

Fourth Moment Theorem and universality should be combined as follows: assume that one wishes to check if the central convergence holds for a sequence of homogeneous sums $Q_{\Y}(f_n)$, with $\bs{Y}$ being a sequence of i.i.d. random variables. By virtue of Definition \ref{defcom}, it is sufficient to check for the convergence of the fourth moments of $Q_{\X}(f_n)$ to $3$, where $\X$ is a sequence of independent copies of a random variable $X$, satisfying Assumption {\bf (2)},  the Fourth Moment Theorem, and being universal at the fixed order $d$. \\

Roughly speaking, the limit condition $\tau_n(f_n) \rightarrow 0$ implies a weak dependence structure between the arguments of $Q_{\X}(f_n)$ for large $n$: indeed, $\mathrm{Inf}_i(f_n)$ can be interpreted  as the measure of the influence that the variable $X_i$ has on the overall fluctuations of the statistic $Q_{\X}(f_n)$, as suggested by the formula:
\begin{equation*}
d \mathrm{Inf}_i(f_n) =  \dfrac{1}{d!}\E[\big(Q_{\X}(f_n)- \E[Q_{\X}(f_n)| X_k,k \neq i]\big)^2]\,\\
\end{equation*}
(here, $\E[W|Z]$ denotes the conditional expectation of $W$ with respect to $Z$).\\

\begin{rmk}[\textit{Hypercontractivity}]\label{hypercontractivity}
Let $\X=\{X_i\}_{i\geq 1}$ be a sequence of independent, centered random variables, with unit variance, \index{Hypercontractivity} and let $q >2$ be such that $\gamma:= \sup\limits_{i\geq 1}\mathbb{E}[|X_i|^q] < \infty$. 
Then, for every $d\geq 1$, and any admissible kernel $f:[n]^d \rightarrow \mathbb{R}$, the following inequality holds:
$$ \mathbb{E}[|Q_{\X}(f)|^q] \leq \gamma^d (2\sqrt{q-1})^{dq} \mathbb{E}[Q_{\X}(f)^2]^{\frac{q}{2}}$$
 (see \cite[Lemma 4.2]{NourdinPeccatiReinert} or \cite[Propositions 3.11, 3.12, 3.16]{Mossel}), ensuring the uniform integrability of the random variables $|Q_{\X}(f)|^r$, for every $r < q$. This, in turn, ensures that the convergence in law implies the convergence of the moments up to the order $q-1$. In particular, for a sequence $\bs{X}$ of random variables with uniformly bounded moments of every order, convergence in law implies the convergence of all the moments \cite[Chapter 6]{DasGupta}.\\\
\end{rmk}

The problem of finding other universal laws is addressed only for homogeneous sums of order $d\geq 2$, since it is well-known that there is no universality for linear polynomials. Indeed, for every integer $n$ and any collection of real numbers $f_n(i)$ such that $\sum\limits_{i=1}^n f_n(i)^2 =1$, the statistics $Q_{\NN}(f_n)=\sum\limits_{i=1}^n f_n(i) N_i $ is always normally distributed. On the other hand, and as already recalled in the previous chapters, for a CLT to hold for $Q_{\X}(f_n)$ it is necessary to require an additional \textit{Lindberg type condition}, such as, for instance, $\max\limits_{i=1,\dots,n}|f_n(i)| \to 0$ as $n\to \infty$ (see \cite[Theorem 11.1.1]{NourdinPeccatilibro}).

\section{Main results}

The principal result to be proven is stated in the next theorem.
\begin{thm}
\label{superTeo1}
Fix $d\geq 2$, and let $X$ be a random variable satisfying Assumption {\bf (2)}. If $\E[X^4]\geq 3$ (or, equivalently, $\chi_4(X)\geq 0$), then $X$ satisfies the Fourth Moment Theorem, and its law is universal at the order $d$ for normal approximations of homogeneous sums, in the sense of Definition \ref{defcom}.
\end{thm}

In order to prove Theorem \ref{superTeo1}, we need some preliminary considerations. 

If $f:[n]^d \rightarrow \mathbb{R}$ is an admissible kernel, for $m=1,\dots, d$ and every $j_1,\dots,j_m \in [n]$, consider the kernel 
$f(j_1,\dots,j_m,\cdot): [n]^{d-m} \rightarrow \mathbb{R}$ defined by 
$$(i_1,\dots,i_{d-m}) \mapsto f(j_1,\dots,j_m,i_1,\dots,i_{d-m}),$$ and the corresponding Gaussian homogeneous sum of order $d-m$:
$$Q_{\NN}(f(j_1,\dots,j_m,\cdot)) = \sum\limits_{i_1,\dots,i_{d-m}=1}^n f(j_1,\dots,j_m, i_1,\dots,i_{d-m})N_{i_{1}}\cdots N_{i_{d-m}} ,$$ 
where $\NN=\{N_i\}_{i\geq 1}$ denotes a sequence of i.i.d. standard Gaussian random variables on the fixed probability space. \\

In addition to the notation introduced in Section \ref{App1}, for $m = 0,\dots,d$,  denote by $\mathcal{P}_{\lambda_m}^{\star}(d^{\otimes 4})$ the set of the partitions in $\mathcal{P}([4d])$ that respect 
$$\pi^{\star}= d^{\otimes 4}= 1\cdots d | (d+1)\cdots 2d| (2d+1)\cdots 3d| ((3d+1)\cdots 4d ,$$ whose class is the partition $(2^{2(d-m)}, 4^{m})$ (namely, the respectful partitions composed by $2(d-m)$ blocks of cardinality $2$ and $m$ blocks of cardinality $4$). In particular, $\mathcal{P}_2^{\star}(d^{\otimes 4}) = \mathcal{P}_{\lambda_0}^{\star}(d^{\otimes 4})$ denotes the set of  pairing partitions respecting $\pi^{\star}$.\\

The strategy here proposed involves the following new combinatorial formula for the computation of the fourth moment of $Q_{\X}(f)$.

\begin{prop}
Let $X$ be a random variable satisfying Assumption {\bf (2)}. For any admissible kernel $f:[n]^d \rightarrow \mathbb{R}$, if $\NN$ denotes a sequence of independent standard Gaussian random variables, then:
\begin{equation}
\label{formulaGauss}
\mathbb{E}[Q_{\X}(f)^4]   = \mathbb{E}[Q_\NN(f)^4]  + \sum_{m=1}^{d} \binom{d}{m}^4 m!^{4} \chi_4(X)^m \sum_{j_1,\dots,j_m=1}^n\mathbb{E}[Q_\NN(f(j_1,\dots,j_m,\cdot))^4],
\end{equation}
where $\chi_4(X)$ denotes the fourth cumulant of $X$.
\end{prop}

\begin{proof}
By virtue of the vanishing on diagonals of the kernel $f$:
\begin{equation}
\mathbb{E}[Q_\X(f)^4] = \sum_{\substack{\mathbf{i} = (i_1,\dots,i_{4d})\in [n]^{4d}\\ \mathrm{Ker}(\mathbf{i}) \wedge \pi^{\star}= \hat{0}}}f^{\otimes 4}(\mathbf{i})\; \mathbb{E}[X_{i_1}\cdots X_{i_d}\cdots X_{i_4d}],
\end{equation}
where $f^{\otimes 4}(\mathbf{i})=\prod\limits_{l=1}^4 f(i_{(l-1)d+1},\dots,i_{ld})$. 
Then, from the moment-cumulant formula \eqref{MomCum}, for every $i_1,\dots,i_{4d} \in [n]$, since $\E[X] = \E[X^3] =0$, the only non-zero values $\chi_{\sigma}(X_{i_1},\dots,X_{i_{4d}})$ correspond either to pairings or to partitions in $\mathcal{P}_{\lambda_m}^{\star}(d^{\otimes 4})$, for $m=1,\dots,d$, yielding:
\begin{align*}
\mathbb{E}[X_{i_1}\cdots X_{i_d}\cdots X_{i_{4d}}]&= \sum_{\substack{\sigma \in \mathcal{P}([4d])\\ \sigma \wedge \pi^{\star}= \hat{0}}} \chi_{\sigma}(X_{i_1},\dots,X_{i_{4d}}) \\
&= \sum_{\sigma \in \mathcal{P}_2^{\star}(d^{\otimes 4})}\prod_{\{r,s\}\in \sigma}\chi_2(X_{i_r},X_{i_s}) + \sum_{m=1}^{d} \sum_{\sigma \in \mathcal{P}_{\lambda_m}^{\star}(d^{\otimes 4})}\chi_{\sigma}(X_{i_1},\dots,X_{i_{4d}}).
\end{align*}

In particular, for every partition $\sigma$ in $\mathcal{P}_{\lambda_m}^{\star}(d^{\otimes 4})$, the cumulant $\chi_{\sigma}(X_{i_1},\dots,X_{i_{4d}})$ will be the product of a fourth-order cumulant $\chi_4(X_{i_{r_1}},X_{i_{r_2}},X_{i_{r_3}},X_{i_{r_4}})$ for every 4-block $\{r_1,r_2,r_3,r_4\}$, and a second order cumulant $\E[X_{i_l}X_{i_p}]$ for every pairing $\{l,p\}$. Then, the characterization of independence in terms of cumulants implies that $\chi_4(X_{i_{r_1}},X_{i_{r_2}},X_{i_{r_3}},X_{i_{r_4}}) \neq 0$ if and only if $i_{r_l} = i_{r_p}$ for $l,p=1,\dots,4$, and similarly $\E[X_{i_r} X_{i_s}] = \delta_{r,s}$, namely $\chi_{\sigma}(X_{i_1},\dots,X_{i_{4d}}) \neq 0 $ if and only if there exists $m=0,\dots,d$ such that $\mathrm{Ker}(\mathbf{i}) \in \mathcal{P}_{\lambda_m}^{\star}(d^{\otimes 4})$, in which case:
$$\chi_{\sigma}(X_{i_1},\dots,X_{i_{4d}}) = \chi_4(X)^m, \quad  \text{for } m=1,\dots,d, \qquad \chi_{\sigma}(X_{i_1},\dots,X_{i_{4d}}) =1 \text{ if } m=0 .$$
Then, it follows that:
\begin{align*}
\E[Q_{\X}(f)^4] &= \sum_{\substack{\mathbf{i} = (i_1,\dots,i_{4d})\in [n]^{4d}\\ \mathrm{Ker}(\mathbf{i}) \wedge \pi^{\star}= \hat{0}}} f^{\otimes 4}(\mathbf{i})\; \E[X_{i_1}\cdots X_{i_d}\cdots X_{i_{4d}}]  \\
&=   \sum_{m=0}^d  \sum_{\sigma \in \mathcal{P}_{\lambda_m}^{\star}(d^{\otimes 4})} \sum_{\substack{\mathbf{i} \in [n]^{4d}\\ \mathrm{Ker}(\mathbf{i}) = \sigma} }  f^{\otimes 4}(\mathbf{i}) \;\chi_{\sigma}(X_{i_1},\dots,X_{i_{4d}})\\ 
  &= \sum_{\sigma \in \mathcal{P}_{2}^{\star}(d^{\otimes 4})} \sum_{\substack{\mathbf{i} \in [n]^{4d}\\ \mathrm{Ker}(\mathbf{i}) = \sigma}} f^{\otimes 4}(\mathbf{i}) \;+ \sum_{m=1}^d  \chi_4(X)^m\sum_{\sigma \in \mathcal{P}_{\lambda_m}^{\star}(d^{\otimes 4}) } \sum_{\substack{\mathbf{i} \in [n]^{4d} \\ \mathrm{Ker}(\mathbf{i}) = \sigma}}  f^{\otimes 4}(\mathbf{i}).\\
\end{align*}
 
Finally, observe that, for every $m=1,\dots,d$, every partition $\sigma \in \mathcal{P}_{\lambda_m}^{\star}(d^{\otimes 4})$ is uniquely determined by a choice of $m$ 4-blocks and a choice of a pairing of the remaining $4(d-m)$ elements. Simple combinatorial arguments yield that the $m$ 4-blocks can be formed in  $\binom{d}{m}^{4}m!^4$ ways: indeed, for $m=1,\dots,d$, if $B_l = \{(l-1)d+1,\dots,ld\}$ is the $l$-th block of $\pi^{\star}, l=1,\dots,4$, there are $\binom{d}{m}$ ways of choosing  $m$ elements $h_{l,1},\dots,h_{l,m}$ in $B_l$, each of which to put in a $4$-block.  Now, to form the first $4$-block, there are $m^4$ choices (choose one element out of the $m$ selected ones in each $B_l$). For the second $4$-block, the second element in each $B_l$ can be chosen in $m-1$ ways, giving a contribution of $(m-1)^4$, and so on, yielding $m!^4$. 
Moreover, for every $\sigma \in \mathcal{P}_{\lambda_m}^{\star}(d^{\otimes 4}),$  and every $\mathbf{i} \in [n]^{4d}$ such that $\mathrm{Ker}(\mathbf{i}) = \sigma$, we have determined a partition $\tau(\sigma) \in  \mathcal{P}_2^{\star}((d-m)^{\otimes 4})$ (namely, the restriction of $\sigma$ to its pairings); then, if $C_1,\dots,C_m \in \sigma$ are the blocks of cardinality $4$, $(j_1,\dots,j_m) \in [n]^m$ is uniquely determined by setting $j_l = i_{p}$ for every $p \in C_l$, so that, for the sub vector $\bs{h} \in [n]^{4(d-m)}$ of $\bs{i}$ indexed by the elements of the pairings of $\sigma$, one has  $\mathrm{Ker}(\bs{h}) = \tau(\sigma)$. Then, the conclusion follows from the Wick formula \eqref{Wick}:
\begin{align*}
\E[Q_{\X}(f)^4] &=  \sum_{\sigma \in \mathcal{P}_{2}^{\star}(d^{\otimes 4})} \sum_{\substack{\mathbf{i} \in [n]^{4d}\\ \mathrm{Ker}(\mathbf{i}) = \sigma}} f^{\otimes 4}(\mathbf{i}) \;+ \sum_{m=1}^d  \chi_4(X)^m\sum_{\sigma \in \mathcal{P}_{\lambda_m}^{\star}(d^{\otimes 4}) } \sum_{\substack{\mathbf{i} \in [n]^{4d} \\ \mathrm{Ker}(\mathbf{i}) = \sigma}}  f^{\otimes 4}(\mathbf{i}) \\
&=\E[Q_{\NN}(f)^4] \;+ \\
& \qquad \quad \sum_{m=1}^d  \binom{d}{m}^4 m!^4  \;\chi_4(X)^m \sum_{j_1,\dots,j_m=1}^n\sum_{\tau \in \mathcal{P}_{2}^{\star}((d-m)^{\otimes 4})} \sum_{\substack{\mathbf{i} \in [n]^{4(d-m)} \\ \mathrm{Ker}(\mathbf{i}) = \tau}}  f(j_1,\dots,j_m,\cdot)^{\otimes 4}(\mathbf{i}) \\
&= \E[Q_{\NN}(f)^4] \; + \\
&\qquad \quad \sum_{m=1}^d \binom{d}{m}^4 m!^4 \;\chi_4(X)^m \sum_{j_1,\dots,j_m=1}^n \E[ Q_{\NN}(f(j_1,\dots,j_m,\cdot))^4].
\end{align*}

%
\end{proof}

\subsection{The proof of Theorem \ref{superTeo1}}
Thanks to formula \eqref{formulaGauss}, it is possible to exhibit a whole class of random variables $X$ (that includes the normal and the Poisson distributions), for which the fourth moment phenomenon occurs, thus proving Theorem \ref{superTeo1}.

\begin{proof}
Given a sequence of admissible kernels $f_n :[n]^d \rightarrow \mathbb{R}$, write formula \eqref{formulaGauss} as:
\begin{equation*}
\mathbb{E}[Q_{\X}(f_n)^4] -3 = \mathbb{E}[Q_\NN(f_n)^4] -3 + \sum_{m=1}^{d} \binom{d}{m}^4 m!^{4} \chi_4(X)^m \sum_{j_1,\dots,j_m=1}^n\mathbb{E}[Q_\NN(f_n(j_1,\dots,j_m,\cdot))^4]\, ,
\end{equation*}
and assume that $\E[Q_{\X}(f_n)^4] \rightarrow 3$ as $n\rightarrow \infty$.  Recalling that  $\chi_4(Q_{\NN}(f_n))=\E[Q_{\NN}(f_n)^4] -3$ is strictly positive (see \eqref{Pos4cum}), the assumption $\chi_4(X) \geq 0$ entails, in turn,  that  $\E[Q_{\NN}(f_n)^4]  \rightarrow 3$. Then, by virtue of Theorem \ref{FMTClassic}, it follows hat $Q_{\NN}(f_n) \stackrel{\text{Law}}{\longrightarrow}\mathcal{N}(0,1)$, and the conclusion follows by Theorem \ref{inv}, Part \ref{Invariance}.

Conversely, if $Q_{\X}(f_n) \xrightarrow{\text{\rm Law}}\mathscr{N}(0,1)$, the convergence of the sequence of the fourth moments $\E[Q_{\X}(f_n)^4] \rightarrow 3$ follows  by a classical hypercontractivity argument (see Remark \ref{hypercontractivity}). 
In conclusion, the universality of $X$ follows by keeping in mind that $Q_{\NN}(f_n) \xrightarrow{\text{\rm Law}}\mathscr{N}(0,1)$ implies
 $$\tau_n = \max_{i=1,\dots,n}\mathrm{Inf}_i(f_n) \rightarrow 0 .$$ 
 Indeed, $Q_{\X}(f_n) \xrightarrow{\text{\rm Law}}\mathscr{N}(0,1)$ implies, via  $\E[Q_{\X}(f_n)^4] \rightarrow 3$, that $Q_{\NN}(f_n) \xrightarrow{\text{\rm Law}}\mathscr{N}(0,1)$. The claim  is achieved by virtue of Theorem \ref{invMossel} in Part \ref{Invariance}.
\end{proof}

\begin{exm} $ $
\begin{enumerate}
\item Let $X_1, X_2$ be independent random variables satisfying Assumption {\bf (2)}, and such that $\chi_4(X_1), \chi_4(X_2) \geq 0$. Then $Z = X_1 + X_2$ satisfies in turn Assumption {\bf (2)} and $\chi_4(Z) \geq 0$ (due to the additivity property of cumulants), and hence satisfies the Fourth Moment Theorem. As to multiplicative convolution, $W:= X_1 X_2$ satisfies Assumption {\bf (2)} as well. By virtue of the moment-cumulant formula \eqref{MomCum}, $\chi_4(W) =  \E[X_1^4]\E[X_2^4] - 3$, and hence, according to Theorem \ref{superTeo1}, for $W$ to satisfy the Fourth Moment Theorem, it is sufficient that at least one of the $X_i$'s satisfies $\chi_4(X_i)\geq 0$. This remark explains the reason why the technique of the mixtures discussed in Section \ref{proof_mixtures} gives an alternative proof to Theorem \ref{superTeo1}.\\
\item Every random variable $X$, centered and with unit variance, whose law is infinitely divisible with respect to additive convolution, satisfies $\chi_4(X) = \E[X^4] - 3\geq 0$. Indeed, by definition, for every integer $n\in \mathbb{N}$, there exist i.i.d. random variables $X_{1,n},\dots,X_{n,n}$ such that $X \stackrel{\text{Law}}{=} X_{1,n} + X_{2,n} + \cdots + X_{n,n}$, which yields $\chi_4(X) = n \chi_4(X_{1,n})$. Moreover, for any center random variable $Y$ with unit variance, $\chi_4(Y) \geq -2$. Then, if $\chi_4(X) < 0$, for $n$ large enough one would find $\chi_4(X) < -2$, which is impossible. Hence, for $X$ infinitely divisible, satisfying Assumption {\bf (2)}, the Fourth Moment Theorem for homogeneous sums $Q_{\X}(f_n)$ holds at any order $d\geq 2$: for instance, for $\lambda >0$, let $P(\lambda)$ denote a Poisson distributed random variable of parameter $\lambda$, and consider a Compound Poisson-distributed random variable $Y= X_1 + \cdots + X_{P(\lambda)}$, with $X_1, X_2,\dots,$ independent copies of a random variable $X$ satisfying Assumption {\bf (2)}. Then, $Y$ is infinitely divisible (indeed, $Y \stackrel{\text{Law}}{=} Y_1 + \cdots + Y_n$, with $Y_j = X_{j,1}+ \cdots +X_{j, P(\lambda/n)}$, and $X_{j,i}$ independent copy of $X$), and $\E[Y^3] = \lambda\E[X^3] = 0$. \\

\item For $k\geq 1$, let $H_k(x)$ denote the $k$-th Hermite polynomial and let $N \sim \mathcal{N}(0,1)$. Then,
$$ \E[H_k(N)^4] = |\mathcal{P}_2^{\star}([k^{\otimes 4}]) | \geq 3 \; $$
where $ \mathcal{P}_2^{\star}([4k])$ denotes the set of pairing partitions of $[4k]$ respecting 
$$\pi^{\star}= k^{\otimes 4} = \{\{1,\dots,k\}, \{k+1,\dots,2k\},\dots,\{2k+1,\dots,3k\},\{3k+1,\dots,4k\}\}.$$
Since $\E[H_k(N)^3]=0$ if $k$ is odd,  for $X=H_k(N)$, Theorem \ref{superTeo1} applies when $k$ is odd: to relax the assumption on the third moment, it is not possible to proceed from formula \eqref{formulaGauss}, but it would be necessary to adopt a different strategy. Remark that the universality of the law of $H_k(N)$ for normal approximations of homogeneous sums can alternatively be deduced  from Part \ref{Invariance}.\\
\end{enumerate}
\end{exm}

\begin{rmk}
For the formula \eqref{formulaGauss} itself, the assumption $\E[X^3] = 0$ is a matter of pure convenience, made to ease the computations. For instance, if $d=2$, relaxing this hypothesis would yield some extra summands on the right-hand side of formula \eqref{formulaGauss} depending on the third cumulant of $X$, arising from partitions with two blocks of cardinality $3$ and no singletons (note that no block of cardinality $5$ can be considered because the partition should respect $\pi^{\star}=2^{\otimes 4}$). More precisely, consider the set $\mathcal{P}_{(2,3^2)}^{\star}([8])$ of the partitions in $\mathcal{P}([8])$, respecting $\pi{\star}$, and whose class is the partition $(2,3^2)$. Then, if $\E[X^3] = \chi_3(X) \neq 0$, formula \eqref{formulaGauss} would have the extra summand
$$ \chi_3(X)^2\sum_{\sigma \in \mathcal{P}_{(2,3^2)}^{\star}([8])} \sum_{\substack{ \mathbf{i} \in [n]^{8}\\ \mathrm{ker}(\mathbf{i})=\sigma}} f^{\otimes 4}(\mathbf{i}).$$
However, since  for $\sigma \in \mathcal{P}_{(2,3^2)}^{\star}([8])$, $\sum_{\substack{ \mathbf{i} \in [n]^{8}\\ \mathrm{ker}(\mathbf{i})=\sigma}} f^{\otimes 4}(\mathbf{i})$ needs not to be positive, in general, no conclusion can be drawn to achieve the (quadratic) Fourth Moment Theorem for $X$ from this version of the formula for $\E[Q_{\bs{X}}(f)^4]$.
For $d >2$, there might be dependence on higher-order cumulants of $X$. \\

It is worth noticing that one could attempt at generalizing the technique here presented to compare higher moments of $Q_{\X}(f)$ with the corresponding moments of $Q_{\NN}(f)$ and then exploit the findings in \cite{Guillaume1}, where the authors showed that the  normal approximations of multiple Wiener integrals is equivalent to the convergence of pairs of even moments, that are not necessarily equal to the second and the fourth. However, the resulting formula would depend on higher-order cumulants of $X$, with coefficients that might be non-constant in sign, and not easily describable. Therefore, the fourth moment has to be considered as a very special case: as Theorem \ref{superTeo1} showed, the convenience of formula \eqref{formulaGauss} is the fact that its right-hand side is a polynomial in $\chi_4(X)$, of degree at most $d$, with non-negative coefficients.\\
\end{rmk}

\paragraph{Poisson homogeneous Chaos}

As mentioned in the Synopsis, random variables living the discrete Poisson homogeneous Chaos satisfy both the Fourth Moment Theorem and the universality phenomenon for central convergence \cite{PeccatiZheng2,PeccatiZheng1}, as summarized in the forthcoming statement.

\begin{thm}\label{FMTPoisson}
Let $d\geq 2$, and let $\mathbf{P}=\{P_i\}_{i\geq 1}$ be a sequence of i.i.d. centred Poisson distributed random variables, with parameter $1$, and let $f_n:[n]^d \rightarrow \mathbb{R}$ be symmetric and vanishing on diagonals for all $n \in \mathbb{N}$. If $\E[Q_{\mathbf{P}}(f_n)^2] \rightarrow 1$, as $n\rightarrow \infty$, the following conditions are equivalent:
\begin{itemize}
\item[(i)] $Q_{\mathbf{P}}(f_n) \stackrel{\text{Law}}{\longrightarrow} \mathcal{N}(0,1)$;
\item[(ii)] $\E[Q_{\mathbf{P}}(f_n)^4] \rightarrow \E[N^4] = 3, N \sim \mathcal{N}(0,1)$;
\item[(iii)] if $\{e_i\}_{i\geq 1}$ is an orthonormal basis in  $\mathrm{L}^2(\mathbb{R}_+^d)$,  then $\|g_n \otimes_r g_n \| \rightarrow 0$ for every $r=1,\dots,d-1$, with  $g_n = \sum\limits_{i_1,\dots,i_d=1}^n f_n(i_1,\dots,i_d) e_{i_1}\otimes \cdots \otimes e_{i_d}$, and where $\otimes_r$ is defined as in \eqref{GenContraction}.
\end{itemize}
Moreover, $Q_{\mathbf{P}}(f_n) \stackrel{\text{Law}}{\longrightarrow} \mathcal{N}(0,1)$ implies $Q_{\mathbf{X}}(f_n) \stackrel{\text{Law}}{\longrightarrow} \mathcal{N}(0,1)$ for every sequence $\mathbf{X}$ of i.i.d. centered random variables, with unit variance.
\end{thm}

Unfortunately, Theorem \ref{superTeo1}  does not cover Theorem \ref{FMTPoisson}: to achieve this goal, it would be necessary to drop the assumption $\E[X^3]=0$. Further remarks to this concern will be addressed in Section \ref{proof_mixtures}.\\

\subsection{The non identically distributed case}

Formula \eqref{formulaGauss}, and subsequently Theorem \ref{superTeo1}, deals with homogeneous sums in i.i.d. random variables, just to ease the notation and the discussion, but the same conclusion could be drawn for a sequence $\X=\{X_i\}_{i\geq 1}$ of independent centered random variables, satisfying Assumption {\bf (2)}, possibly non identically distributed, but starting from an inequality other than an equality. Indeed, for every $m=1,\dots,d$ and every $n$, set:
$$\gamma_{n}^{(m)} := \min\limits_{i_1,\dots,i_m \in [n]} \prod\limits_{l=1}^m \chi_4(X_{i_l}), \qquad A_n := \min\limits_{m=1,\dots,d}\gamma_n^{(m)}. $$
Assume further that there exists $A > 0$ such that $ \inf\limits_{n\geq 1}A_n > A$. Then, for every admissible kernel $f_n:[n]^d \rightarrow \mathbb{R}$, 
\begin{align*}
\mathbb{E}[Q_{\X}(f_n)^4] &- 3  \geq \mathbb{E}[Q_{\mathbf{N}}(f_n)^4] - 3  + \sum_{m=1}^d \gamma_n^{(m)} \sum_{\sigma \in \mathcal{P}_{\lambda_m}^{\star}(d^{\otimes 4})} \sum_{\substack{ \mathbf{i} \in [n]^{4d}  \\ \mathrm{Ker}(\mathbf{i})= \sigma }} f_n^{\otimes 4}(\mathbf{i}) \\
&\geq \mathbb{E}[Q_{\mathbf{N}}(f_n)^4] - 3 + A_n \sum_{m=1}^{d} \sum_{\sigma \in \mathcal{P}_{\lambda_m}^{\star}(d^{\otimes 4})} \sum_{\substack{ \mathbf{i} \in [n]^{4d}  \\ \mathrm{Ker}(\mathbf{i})= \sigma }} f_n^{\otimes 4}(\mathbf{i}) \\
&\geq \mathbb{E}[Q_{\mathbf{N}}(f_n)^4] - 3 +  \\
& \quad + A_n \sum_{m=1}^d \binom{d}{m}^4 (m!)^4 \sum_{j_1,\dots,j_m=1}^n \sum_{ \tau \in \mathcal{P}_2^{\star}((d-m)^{\otimes 4})} \sum_{\substack{ \mathbf{i} \in [n]^{4(d-m)} \\ \mathrm{Ker}(\mathbf{i}) = \tau}} f_n(j_1,\dots,j_m,\cdot)^{\otimes 4}(\mathbf{i}) \\
&\geq \mathbb{E}[Q_{\mathbf{N}}(f_n)^4] - 3 + A \sum_{m=1}^d \binom{d}{m}^4 (m!)^4 \sum_{j_1,\dots,j_m=1}^n \mathbb{E}[Q_{\mathbf{N}}(f_n(j_1,\dots,j_m,\cdot))^4] \numberthis \label{ineq_non_id}.
\end{align*}

Keeping in mind Remark \ref{hypercontractivity}, Theorem \ref{superTeo1} admits the following extension.
\begin{thm}\label{Niid_classic}
Let the above notation prevail. Then, for every sequence of admissible kernels $f_n:[n]^d \rightarrow \mathbb{R}$, the condition $\mathbb{E}[Q_{\X}(f_n)^4] - 3 \longrightarrow 0$ is  necessary and sufficient  for the convergence $$Q_{\X}(f_n)\stackrel{\text{Law}}{\longrightarrow} \mathcal{N}(0,1).$$
Moreover,  $Q_{\X}(f_n)\stackrel{\text{Law}}{\longrightarrow} \mathcal{N}(0,1)$ implies that $Q_{\Z}(f_n)\stackrel{\text{Law}}{\longrightarrow} \mathcal{N}(0,1)$ for every sequence $\Z=\{Z_i\}_{i\geq 1}$ of independent random variables, non necessarily identically distributed, satisfying Assumption {\bf (2)}.
\end{thm}

\begin{proof}
If $\mathbb{E}[Q_{\X}(f_n)^4] - 3 \longrightarrow 0$ as $n \rightarrow \infty$, and since $A >0$ by hypothesis, from inequality \eqref{ineq_non_id} it follows that $\mathbb{E}[Q_{\mathbf{N}}(f_n)^4] - 3 \longrightarrow 0$ in turn. The conclusion is then achieved referring  to Theorem \ref{FMTClassic} and Theorem \ref{inv}.
\end{proof}

In particular, Theorem \ref{Niid_classic} applies for Hermite sums $Q_{\mathbf{N}}^{(\bs{h})}(f_n)$, with $\bs{h}=(h_1,\dots,h_d)$, and $h_j$ odd integer for every $j=1,\dots,d$ (compare with  Proposition \ref{HermiteSum}, Part \ref{Invariance}).

\subsection{Gamma approximations of homogeneous sums}

If $\nu \in \mathbb{R}_+ \setminus\{0\}$, let $G(\frac{\nu}{2})$ denote a random variable with Gamma distribution $\Gamma(\frac{\nu}{2},1)$, and set $F(\nu) \stackrel{\text{Law}}{=} 2G(\frac{\nu}{2}) - \nu$, so that 
$$\E[F(\nu)]=0, \, \E[F(\nu)^2] = 2\nu , \, \E[F(\nu)^3] = 8 \nu , \,  \E[F(\nu)^4] = 12\nu^4 + 48 \nu.$$
Note that if $\nu \in \mathbb{N}$, then $F(\nu)$ has the centered $\chi^2$ distribution with $\nu$ degrees of freedom, namely, $ F(\nu) \stackrel{\text{ Law}}{=} \sum_{i=1}^{\nu}(N_i -1)$, where $N_1,\dots,N_{\nu}$ are i.i.d. random variables with the standard Gaussian distribution. In this case, Gamma approximations of Gaussian Wiener integrals corresponds to approximations in the second Wiener Chaos.\\

A Fourth Moment-type statement for Gamma approximations of sequences of multiple Wiener integrals has been provided in \cite[Theorem 1.2]{NourdinPeccati2}, and it is summarized in the next theorem for Gaussian homogeneous sums. Note that Gamma approximations can be established only in Wiener chaoses of even order since, if $d$ is odd, $\E[I_d^W(f)^3] = 0$, while $\E[F(\nu)^3] = 8\nu >0$.\\

\begin{thm}\label{GammaAppr}
Let $d\geq 2$ be an even integer and $f_n:[n]^d \rightarrow \mathbb{R}$ a sequence of admissible kernels such that $\lim\limits_{n \to \infty} \E[Q_{\mathbf{N}}(f_n)^2] = 2\nu$, $\lim\limits_{n\to \infty} \E[Q_{\mathbf{N}}(f_n)^3] = 8 \nu$ for a fixed $\nu \in \mathbb{R}_+ \setminus\{0\}$. The following statement are equivalent as $n\to \infty$:
\begin{itemize}
\item[(i)] $ \E[Q_{\mathbf{N}}(f_n)^4] - 12\E[Q_{\mathbf{N}}(f_n)^3] \longrightarrow \E[F(\nu)^4]-12\E[F(\nu)^3] = 12\nu^2 - 48\nu$;
\item[(ii)] $Q_{\mathbf{N}}(f_n) \stackrel{\text{Law}}{\longrightarrow} F(\nu)$;
\item[(iii)] $\|f_n \stackrel{r}{\smallfrown} f_n \| \rightarrow 0$ for every $r=1,\dots,d-1$, $r\neq \frac{d}{2}$, and $\| f_n \stackrel{d/2}{\frown}f_n - f_n \| \rightarrow 0$ (with the contraction $\frown$ defined as in \eqref{contraction}). \\
\end{itemize}
\end{thm}

\begin{rmk}
Theorem \ref{GammaAppr} follows from the original statement, by applying Propositions \ref{contraction5} and \ref{contraction6} to provide conditions in terms of the contractions of the coefficients of the involved homogeneous sums. Similarly, one has to take into account that, for every $r=1,\dots,d-1$, the symmetry of the admissible kernels $f_n$ yields:
$$ \| k(f_n) \widetilde{\otimes_r} k(f_n) \| \rightarrow 0  \, \Leftrightarrow \, \| k(f_n) \otimes_r k(f_n) \| \rightarrow 0,$$
with $k(f_n)$ defined as in \eqref{kN1} for $h_i=1$ for every $i$. In particular, this equivalence (see \cite{NualartPeccati}) gives that:
$$\| k(f_n) \widetilde{\otimes_{\frac{d}{2}}} k(f_n) - k(f_n)\| \rightarrow 0  \, \Leftrightarrow \, \| f_n \stackrel{\frac{d}{2}}{\frown} f_n -f_n \|  \rightarrow 0 .$$
\end{rmk}

Starting from formula \eqref{formulaGauss}, an extension of Theorem \ref{superTeo1} in the setting of Gamma approximations of homogeneous sums of even degree $d\geq 2$, can be achieved in the following way. Let $f_n:[n]^d \rightarrow \mathbb{R}$ be a sequence of admissible kernels satisfying the assumptions of Theorem \ref{GammaAppr}: then, from the proof of \cite[Theorem 1.2]{NourdinPeccati2}, it follows that:
$$ \E[Q_{\mathbf{N}}(f_n)^4] -12 \E[Q_{\mathbf{N}}(f_n)^3]  - (12 \nu^2 - 48\nu) > 0,$$
for sufficiently large $n$. Since, under Assumption {\bf (2)}, $\E[Q_{\X}(f_n)^3] = \E[Q_{\mathbf{N}}(f_n)^3]$, formula \eqref{formulaGauss} entails the identity:
\begin{align*}
\E[Q_{\mathbf{X}}(f_n)^4]  & -12\E[Q_{\X}(f_n)^3] - (12 \nu^2 - 48\nu) =  \E[Q_{\mathbf{N}}(f_n)^4]-12\E[Q_{\mathbf{N}}(f_n)^3] - (12 \nu^2 - 48\nu) \\
&\qquad \qquad+ \sum_{m=1}^d \binom{m}{d}^{4}m!^4 \chi_4(X)^m \sum_{j_1,\dots,j_m=1}^n \E[Q_{\mathbf{N}}(f_n(j_1,\dots,j_m,\cdot))^4].
\end{align*}

By exploiting Theorem \ref{GammaAppr}, and with the same strategy of the proof of Theorem \ref{superTeo1}, it is then possible to provide a Fourth Moment statement for $Q_{\mathbf{X}}(f_n)$, when the target law is the  Gamma distribution $F(\nu)$. 

\begin{thm}\label{GenGammaAppr}
Let $d\geq 2$ be an even integer. If $X$ satisfies Assumption {\bf (2)} and $\chi_4(X) \geq 0$, for every sequence of admissible kernels $f_n:[n]^d \rightarrow \mathbb{R}$ such that   $\E[Q_{\mathbf{X}}(f_n)^2] \rightarrow \nu >0$, the following statements are equivalent as $n \rightarrow \infty$:
\begin{itemize}
\item[(i)] $ \E[Q_{\X}(f_n)^4] - 12\E[Q_{\X}(f_n)^3] \longrightarrow  \E[F(\nu)^4] - 12\E[F(\nu)^3]  = 12\nu^2 - 48\nu$;
\item[(ii)] $Q_{\X}(f_n) \stackrel{\text{Law}}{\longrightarrow} F(\nu)$.
\end{itemize}
Besides, the law of $X$ is universal for $F(\nu)$-approximations of homogeneous sums at the fixed order $d$, that is, $Q_{\X}(f_n) \stackrel{\text{Law}}{\longrightarrow} F(\nu)$ implies $Q_{\mathbf{Z}}(f_n) \stackrel{\text{Law}}{\longrightarrow} F(\nu)$, for every sequence $\mathbf{Z} = \{Z_i\}_{i \geq 1}$ of independent copies of a centered random variable, having unit variance. \\
\end{thm}

\begin{rmk}
The universality of Gaussian homogeneous sums w.r.t. $\chi^2$ approximations has been established in \cite[Theorem 1.12]{NourdinPeccatiReinert}, both for homogeneous sums with i.i.d. entries and with only independent entries:  moreover, the vanishing of all the non-trivial contractions in Theorem \ref{GammaAppr}, along with Lemma \ref{magg}, yields that $Q_{\NN}(f_n) \stackrel{\text{Law}}{\longrightarrow} F(\nu)$ implies $\tau_n(f_n)\longrightarrow 0$.\\
\end{rmk}

\subsection{The quadratic case}\label{QuadraticCaseClassic}
The sufficient condition $\E[X^4]\geq 3$ for the Fourth Moment Theorem to hold, provided with Theorem \ref{superTeo1}, could be not optimal in every dimension $d\geq 2$. For instance, set $d=2$, and for a non-zero admissible kernel $f:[n]^2\rightarrow \mathbb{R}$, set $k(f):= \sum\limits_{i_1,i_2=1}^n f(i_1,i_2)e_{i_1}\otimes e_{i_2}$, with $\{e_i\}_{i \geq 1}$ orthonormal system in $\mathrm{L}^2(\mathbb{R}_+)$, 
so that $Q_{\bs{N}}(f) = I_{2}^{W}(k(f))$. From \cite[Lemma 5.2.4]{NourdinPeccatilibro}, it follows that:
$$\E[  Q_{\mathbf{N}}(f)^4 ]  = 3 + 16\big( \| k(f) \otimes_1 k(f) \|^2 + 2 \| k(f) \widetilde{\otimes}_1 k(f) \|^2\big)\, .$$
Due to the symmetry of $f$, one has that $k(f) \widetilde{\otimes}_1 k(f) = k(f) \otimes_1 k(f)$. Moreover,  by virtue of Proposition \ref{contraction5}, $\|k(f) \otimes_1 k(f) \| = \|f \stackrel{1}{\smallfrown} f\|$, and then:
\begin{align*}
\E[  Q_{\mathbf{N}}(f)^4 ]&=3 + 48 \| f \stackrel{1}{\smallfrown} f  \|^2  = 3 + 48 \big( \sum_{i,j=1}^n ( f \stackrel{1}{\smallfrown} f (i,j) )^2\big)\\
&\geq  3+ 48 \sum\limits_{i=1}^n \big(f \stackrel{1}{\smallfrown} f(i,i)\big)^2 
\end{align*}
(see also \cite[Proposition 2.7.13]{NourdinPeccatilibro}).  Then, from formula \eqref{formulaGauss}, and setting \linebreak$\alpha := \sum\limits_{i=1}^n \big(f \stackrel{1}{\smallfrown} f(i,i)\big)^2$, it follows that:
\begin{align*} 
\E[  Q_{\mathbf{X}}(f)^4] - 3 &=  48 \| f \stackrel{1}{\smallfrown} f  \|^2   +  16 \chi_4(X) \sum_{k=1}^n \E[  Q_{\mathbf{N}}(f(k,\cdot))^4 ]  + 16 \chi_4(X)^2 \sum_{k_1,k_2=1}^n f(k_1,k_2)^4 \\ 
&\geq 48 \| f \stackrel{1}{\smallfrown} f  \|^2   +  16 \chi_4(X) \sum_{k=1}^n \E[  Q_{\mathbf{N}}(f(k,\cdot))^4 ]   \\
&\geq 48 \alpha    +  16 \chi_4(X) \sum_{k=1}^n 3\E[  Q_{\mathbf{N}}(f(k,\cdot))^2 ]^2   \\
&= 48 \alpha\,(1 + \chi_4(X)) \numberthis \label{quadraticFMT},
\end{align*}
due to $\sum\limits_{k=1}^n \E[  Q_{\mathbf{N}}(f(k,\cdot))^2 ]^2 = \alpha$. Therefore, since $\alpha >0$ and $f \neq 0$, if $\chi_4(X) > -1$, \eqref{quadraticFMT} is strictly positive, yielding in turn $\E[  Q_{\mathbf{X}}(f)^4] - 3 > 0$.

\begin{prop}
Let $d=2$ and assume that $X$ satisfies Assumption {\bf (2)} and $\E[X^4] > 2$ (or, equivalently, $\chi_4(X) > -1$). Then, $X$ satisfies the quadratic Fourth Moment Theorem. Besides, the law of $X$ is universal for normal approximations of quadratic homogeneous sums.
\end{prop}

\begin{proof}
Formula  \eqref{formulaGauss} for  a sequence of admissible kernels $f_n:[n]^2 \rightarrow \mathbb{R}$ gives:
$$\E[  Q_{\mathbf{X}}(f)^4] - 3 = \E[  Q_{\mathbf{N}}(f)^4] - 3  +  48 \chi_4(X) \alpha_n  + 16 \chi_4(X)^2 \beta_n ,$$
where 
 $$\alpha_n =\sum\limits_{i=1}^n \big(f_n \stackrel{1}{\smallfrown} f_n (i,i)\big)^2 \,\geq\, \beta_n = \sum_{i,k=1}^n f_n(i,k)^4 =   \sum_{i,k=1}^n \E[Q_{\mathbf{N}}(f_n(i,k))^4].$$
If $\E[  Q_{\mathbf{X}}(f_n)^4] - 3 \longrightarrow 0$, as $n\rightarrow \infty$, then, from \ref{quadraticFMT}, it follows that:
 $$\alpha_n = \sum_{i=1}^n \big(f_n \stackrel{1}{\smallfrown} f_n(i,i)\big)^2   \rightarrow 0, $$ implying $\beta_n \rightarrow 0$, and, in turn, $\E[ Q_{\mathbf{N}}(f_n)^4] - 3 \longrightarrow 0$. The conclusion then follows as in the proof of Theorem \ref{superTeo1}.
\end{proof}

Despite the fact that for $d=2$ a stronger sufficient condition can be provided, no information about the optimality of such condition, neither of its necessity, can be easily achieved with the tools so far introduced. A more general discussion of this problem will be addressed in Chapter \ref{Threshold}.


\subsection{Multidimensional CLT}

The invariance principle for homogeneous sums in independent random variables, stated via Theorem \ref{invMossel} in Part \ref{Invariance}, has been extended to the multidimensional setting in \cite[Theorem 4.1]{Mossel2}, in the case one of the sequences is composed of discrete random variables, and  in \cite[Theorem 7.1]{NourdinPeccatiReinert}, where the authors provided an explicit bound for the distance in law between $(Q_{\bs{X}}(f_n^{(1)}),\dots,Q_{\bs{X}}(f_n^{(m)}))$ and its Wiener-Chaos counterpart $(Q_{\bs{N}}(f_n^{(1)}),\dots,Q_{\bs{N}}(f_n^{(m)}))$, as  summarized in the next statement in a simplified version, that is sufficient for the present framework.

\begin{thm}\label{MultiMossel}
Let $m\geq 1$ and $d\geq 1$. Let $\bs{X}=\{X_i\}_{i\geq 1}$ be a sequence of centered independent random variables, with unit variance, whose third moments are uniformly bounded (namely, such that there exists $\beta >0$ such that $\sup\limits_{i\geq 1}\E[|X_i|^{3}] < \beta$). For $j=1,\dots,m$, let $f_n^{(j)}:[n]^d \rightarrow \mathbb{R}$ be an admissible kernel according to Definition \ref{Admissible_Classic}. If $\bs{N}= \{N_i\}_{i\geq 1}$ denotes a sequence of i.i.d. standard Gaussian random variables,  for every thrice differentiable function $\psi:\mathbb{R}^m \rightarrow \mathbb{R}$, with $\|\psi^{\prime\prime \prime} \|_{\infty} < \infty$, there exists a constant $C=C(\beta, m, d, \psi)$ such that:
$$ \big| \E[\psi(Q_{\bs{X}}(f_n^{(1)}),\dots,Q_{\bs{X}}(f_n^{(m)}) )] -   \E[\psi(Q_{\bs{N}}(f_n^{(1)}),\dots,Q_{\bs{N}}(f_n^{(m)}) )]\big| \leq C \sqrt{\max_{j=1,\dots,m}\max_{i=1,\dots,n} \mathrm{Inf}_i(f_n^{(j)})}. \\$$ 
\end{thm}

The main result of the present subsection is a multidimensional version of Theorem \ref{superTeo1}, stated via Theorem \ref{ComponentJoint}: the proof we will provide 
exploits the findings of  \cite[Proposition 2]{PeccatiTudor}, where it is shown that, for vectors of the type $(Q_{\bs{N}}(f_n^{(1)}),\dots,Q_{\bs{N}}(f_n^{(m)}) )$, joint convergence towards the multidimensional normal distribution is equivalent to componentwise convergence, as summarized in the next statement (note that the original statement does not concern exclusively homogeneous sums, but deals, in full generality, with vectors of multiple Wiener integrals of symmetric functions). 

\begin{thm}\label{TeoPeccatiTudor}
For $d\geq 2$ and $m\geq 1$, assume that $C=(C_{i,j})_{i,j=1,\dots,m}$ is a real valued, positive definite, symmetric matrix. For every $j=1,\dots,m$, let $Q_{\bs{X}}(f_n^{(j)})$ be a sequence of homogeneous polynomials of degree $d$, with $f_n^{(j)}:[n]^d \rightarrow \mathbb{R}$ symmetric kernel, vanishing on diagonals, such that:
$$ \lim_{n \rightarrow \infty}\E[ Q_{\bs{X}}(f_n^{(j)}) Q_{\bs{X}}(f_n^{(i)})] = C_{i,j}\quad \forall \, i,j=1,\dots,m.$$
Then, the following statements are equivalent as $n\rightarrow \infty$:
\begin{itemize}
\item[(i)] $Q_{\bs{N}}(f_n^{(j)}) \stackrel{\text{ Law }}{ \longrightarrow} \mathcal{N}(0,C_{j,j})$ for every $j=1,\dots,m$;
\item[(ii)] $(Q_{\bs{N}}(f_n^{(1)}),\dots,Q_{\bs{N}}(f_n^{(m)}) ) \stackrel{\text{ Law }}{ \longrightarrow} \mathcal{N}(0, C)$, with $\mathcal{N}(0,C)$ denoting the $m$-dimensional Gaussian distribution with covariance matrix given by $C$.\\
\end{itemize}
\end{thm}

Combining Theorem \ref{TeoPeccatiTudor} and Theorem \ref{superTeo1}, it is possible to conclude that  the equivalence between joint and componentwise convergence for normal approximations of random vectors $(Q_{\bs{X}}(f_n^{(1)}),\dots,Q_{\bs{X}}(f_n^{(m)}) )$  always holds true under the assumption $\E[X^4] \geq 3$, as made precise in the following statement.

\begin{thm}\label{ComponentJoint}
Fix $m\geq 1$ and $d\geq 2$. Let $\bs{X} = \{X_{i}\}_{i\geq 1}$ be a sequence of independent copies of a random variable $X$ verifying Assumption {\bf (2)} and $\E[X^4]\geq 3$. For every $j=1,\dots,m$, let $Q_{\bs{X}}(f_n^{(j)})$ be a sequence of homogeneous polynomials of degree $d$, with $f_n^{(j)}:[n]^d \rightarrow \mathbb{R}$ symmetric admissible kernel, such that:
$$ \lim_{n \rightarrow \infty}\E[ Q_{\bs{X}}(f_n^{(j)}) Q_{\bs{X}}(f_n^{(i)})] = C_{i,j} \quad \forall\, i,j=1,\dots,m \,,$$
where $C= (C_{i,j})_{i,j=1,\dots,m}$ is a real valued, positive definite, symmetric matrix. Then, the following statements are equivalent as $n\rightarrow \infty$:
\begin{itemize}
\item[(i)] $Q_{\bs{X}}(f_n^{(j)}) \stackrel{\text{ Law }}{ \longrightarrow} \mathcal{N}(0,C_{j,j})$ for every $j=1,\dots,m$;
\item[(ii)] $(Q_{\bs{X}}(f_n^{(1)}),\dots,Q_{\bs{X}}(f_n^{(m)}) ) \stackrel{\text{ Law }}{ \longrightarrow} \mathcal{N}(0, C)$,
with $\mathcal{N}(0,C)$ denoting the $m$-dimensional Gaussian distribution with covariance matrix given by $C$.
\end{itemize}
\end{thm}

\begin{proof}
It is sufficient to prove  that $(i) \Rightarrow (ii)$, since the reverse implication always holds.\\
Assume that $(i)$ occurs. Under the assumption $E[X^4] \geq 3$, and by virtue of Theorem \ref{superTeo1}, $X$ satisfies the Fourth Moment Theorem, and its law is universal, at the order $d$, for normal approximations of homogeneous sums of degree $d$, implying, in particular, that:
$$Q_{\bs{N}}(f_n^{(j)}) \stackrel{\text{ Law }}{ \longrightarrow} \mathcal{N}(0,C_{j,j}) \quad \text{for every } j=1,\dots,m ,$$
for a sequence $\bs{N}$ of independent standard Gaussian random variables.  Besides, from Theorem \ref{inv}, $\tau_n^{(j)} = \max\limits_{i=1,\dots,n}\mathrm{Inf}_i(f_n^{(j)}) \longrightarrow 0$ as $n \rightarrow \infty$, for every $j=1,\dots,m$. Since 
$$\E[Q_{\bs{X}}(f_n^{(j)}) Q_{\bs{X}}(f_n^{(i)})] = \E[Q_{\bs{N}}(f_n^{(j)}) Q_{\bs{N}}(f_n^{(i)})] \quad \forall i,j=1,\dots,m,$$
by virtue of Theorem \ref{MultiMossel}, the random vectors $(Q_{\bs{N}}(f_n^{(1)}),\dots,Q_{\bs{N}}(f_n^{(m)}) )$ and \linebreak$(Q_{\bs{X}}(f_n^{(1)}),\dots,Q_{\bs{X}}(f_n^{(m)}) )$ are asymptotically close in distribution. The conclusion follows by Theorem \ref{TeoPeccatiTudor}.
\end{proof}

\begin{rmk}
By virtue of Theorem \ref{Niid_classic}, the identical distribution hypothesis on the sequence $\bs{X}$ can be dropped in the statement of Theorem \ref{ComponentJoint}.
\end{rmk}


\section{Alternative proofs}

The aim of this section is to provide two alternative proofs of Theorem \ref{superTeo1}: the first one uses mixtures of random variables, and it applies to every order $d\geq 2$, while the second appeals to the \textit{Stein's method of exchangeable pairs}, but only in the quadratic case $d=2$. Despite this deficiency, the Stein's method approach will allow us to derive a quantitative version of the quadratic Fourth Moment Theorem for homogeneous sums $Q_{\X}(f_n)$, with an explicit bound on the Wasserstein's distance between the law of $Q_{\X}(f_n)$ and the standard Gaussian law.\\

\subsection{The technique of mixtures}\label{proof_mixtures}

The aim of this section is to discuss a different approach to prove Theorem \ref{superTeo1}, involving mixtures of random variables. The strategy here adopted will turn out to be particularly helpful in  Chapter \ref{Threshold}, which is dedicated to the discussion of the optimality of the condition $\E[X^4]\geq 3$. 

\begin{lemma}\label{ExistenceT}
For every $\theta >0$, there exists a square-integrable random variable $T \in \mathrm{L}^{2}(\Omega)$\glossary{name={$\mathrm{L}^{2}(\Omega)$},description={Space of the square-integrable random variables defined on $(\Omega,\mathcal{F}, \mathbb{P})$}}, with values in $\mathbb{R}_+\setminus\{0\}$ and with compact support, such that $\mathbb{E}[T]= 1$ and $\mathbb{E}[T^2]= 1+\theta$.
\end{lemma}

\begin{proof}
For the fixed $\theta >0$,  set $\alpha_q:=\sqrt{(1+\theta)^{1/q}-1}$, for $q\in\mathbb{N}$. If $q \rightarrow \infty$, then
$(1+\theta)^{1/q}\to 1$ and hence, for $q$ large enough, $\alpha_q \in [0,1)$. If $V_1,\ldots,V_q$ are independent random variables with distribution $\dfrac{1}{2}(\delta_{1-\alpha_q} + \delta_{1+\alpha_q})$ (where $\delta_{y}$ denotes the Dirac's function in $y$), the random variable
$T:=V_1\cdots V_q$ takes its values in $[x,\infty[$ with $x=(1-\alpha_q)^{q}>0$, and satisfies $\E[T]=1$ and $\E[T^2]= (1+\alpha_q^2)^q= 1+\theta $.
%
\end{proof}

Observe that, for the chaotic random variable $Q_{\mathbf{N}}(f_n)$, the positiveness of the fourth cumulant (see \eqref{Pos4cum}): 
$$\E[Q_{\mathbf{N}}(f_n)^4]-3 \E[Q_{\mathbf{N}}(f_n)^2]^2 > 0$$ 
corresponds to the condition $ \E[Q_{\mathbf{N}}(f_n)^4] - 6 \E[Q_{\mathbf{N}}(f_n)^2] + 3 >0$. Indeed, one can write:
$$ \E[Q_{\mathbf{N}}(f_n)^4]-3 \E[Q_{\mathbf{N}}(f_n)^2]^2 + 3(\E[Q_{\mathbf{N}}(f_n)^2]-1)^2 = \E[Q_{\mathbf{N}}(f_n)^4] - 6 \E[Q_{\mathbf{N}}(f_n)^2] + 3 .$$

\begin{thm}
Let $X$ be a random variable satisfying Assumption {\bf (2)} and such that $\E[X^4] \geq 3$. Then, for every $d \geq 2$, $X$ satisfies the Fourth Moment Theorem, and is universal (at the order $d$) for normal approximations of homogeneous sums of degree $d$.
\end{thm}

\begin{proof}
For the sake of clarity, the proof is divided into two steps. Before starting, note that for sequences of homogeneous sums with kernels having non-constant normalizations,  Definition \ref{defcom} should be extended as follows: if $f_n:[n]^d \rightarrow \mathbb{R}$ is a sequence of symmetric and vanishing on diagonals kernels, such that $\E[Q_{\mathbf{X}}(f_n)^2] = \sigma_n^2 < \infty$ for all $n\geq 1$, and if $\lim\limits_{n \rightarrow \infty}\sigma_n^2 = \sigma^2  >0$, then we shall say that $X$ satisfies the $CFMT_d$ if the convergence
$$ \chi_4(Q_{\mathbf{X}}(f_n)) = \E[Q_{\mathbf{X}}(f_n)^4] - 3\E[Q_{\mathbf{X}}(f_n)^2]^2 \rightarrow 0 $$
implies, as $n \rightarrow \infty$,  $Q_{\mathbf{X}}(f_n) \stackrel{\text{Law}}{\rightarrow} \mathcal{N}(0,\sigma^2)$.

\begin{enumerate}
\item 
If $N \sim \mathcal{N}(0,1)$, then there exists a random variable $T $, with finite variance and independent of $N$, such that, for every $i=1,2,3,4$, $\E[X^i] = \E[(TN)^{i}] = \mathbb{E}[T^i]\E[N^{i}]$. Indeed, setting $\theta:=\E[X^4]/3 -1$,  by virtue of Lemma \ref{ExistenceT}, there exists a positive random variable $V$, with compact support non containing the zero, such that $\mathbb{E}[V^2] = 1+\theta$; then, it suffices to set $T := \sqrt{V}$. Therefore, $\mathbb{E}[T^4] = \E[X^4]/3 \geq 1$.  Consider, then, a sequence $\mathbf{T}=\{T_i\}_{i\geq 1}$ of independent copies of $T$, as well as a sequence $\mathbf{N} = \{N_i\}_{i\geq 1}$ of independent standard normally distributed random variables, such that $\mathbf{T}$ and $\mathbf{N}$ are independent. Then, for every sequence of admissible kernels $f_n:[n]^d \rightarrow \mathbb{R}$, the homogeneous sums
$$ Q_{\mathbf{TN}}(f_n) =\sum_{i_1,\dots,i_d=1}^{n}f_n(i_1,\dots,i_d) T_{i_1}N_{i_1}\cdots T_{i_d}N_{i_d},$$
satisfy $\mathbb{E}[Q_{\mathbf{TN}}(f_n)^i] = \E[Q_{\mathbf{X}}(f_n)^i]$, for $i=1,2,3,4$. 
\item Given a sequence of admissible kernels $f_n:[n]^d \rightarrow \mathbb{R}$, assume that $\E[Q_{\mathbf{X}}(f_n)^4] \rightarrow 3$ as $n\rightarrow \infty$, and  write:
\begin{align*}
\E[Q_{\mathbf{X}}(f_n)^4]- 3 &= \E[Q_{\mathbf{TN}}(f_n)^4]- 3  \\
&= \E\big[\E[Q_{\mathbf{TN}}(f_n)^4 | \mathbf{T}] -3\big]  \\
&=\E\big[ \E[Q_{\mathbf{TN}}(f_n)^4 |\mathbf{T}] - 6\E[Q_{\mathbf{TN}}(f_n)^2|\mathbf{T}] + 3 \big]\\ 
&= \E\big[ \E[Q_{\mathbf{TN}}(f_n)^4 |\mathbf{T}]-3\E[Q_{\mathbf{TN}}(f_n)^2 |\mathbf{T}]^2  +3 \left(\E[Q_{\mathbf{TN}}(f_n)^2|\mathbf{T}]-1\right)^2 \big].
\end{align*}

Since  $\mathbf{T}$ is independent of $\mathbf{N}$, $Q_{\mathbf{TN}}(f_n)$ is in the $d$-th Wiener chaos $\mathbf{T}$-a.s., and hence: 
$$\E[Q_{\mathbf{TN}}(f_n)^4 |\mathbf{T}]-3\E[Q_{\mathbf{TN}}(f_n)^2 |\mathbf{T}]^2 \,>\, 0 \quad \mathbf{T} \text{-}a.s. $$ 
Then, under the assumption $\E[Q_{\mathbf{X}}(f_n)^4] \rightarrow 3$, and up to extracting a subsequence, almost surely in $\{T_i\}_{i\ge1}$, and for $n\rightarrow \infty$, it follows that:
\begin{itemize}
\item[-]  $\mathbb{E}[Q_{\mathbf{TN}}(f_n)^4| \mathbf{T}] - 3\E[Q_{\mathbf{TN}}(f_n)^2 |\mathbf{T}]^2 \rightarrow 0 \; $;
\item[-] $\E[Q_{\mathbf{TN}}(f_n)^2|\mathbf{T}]-1 \, \rightarrow 0 \; ,$
\end{itemize}
yielding together that $Q_{\mathbf{TN}}(f_n) \stackrel{\text{ Law}}{\longrightarrow} \mathcal{N}(0,1)$ $\mathbf{T}$-a.s., 
which in turn implies 
$$   \max_{i_1 = 1,\dots,n} \sum_{i_2,\dots,i_d=1}^n f_n(i_1,\dots,i_d)^2 T_{i_1}^2 \cdots T_{i_d}^2  \longrightarrow 0   \; \mathbf{T} \text{- a.s.}$$
Since  $T \geq x>0$, this condition in turn implies that:
$$
\tau_n(f_n):= \max_{i_1=1,\dots,n}\sum_{i_2,\cdots,i_d=1}^n f_n(i_1,i_2,\cdots,i_d)^2  \xrightarrow[n\to\infty]{\text{}}0.
$$
The asymptotic normality of $Q_{\X}(f_n)$ then follows by de Jong's Criterion (see \cite[Theorem 1.9]{NourdinPeccatiReinert} for a modern proof). Moreover, the law of $X$ is universal: if $Q_{\X}(f_n) \stackrel{\text{Law}}{\longrightarrow} \mathcal{N}(0,1)$, then $\E[Q_{\X}(f_n)^4]-3 \rightarrow 0$ which, as just shown, implies $\tau_n(f_n) \rightarrow 0$.\qedhere
\end{enumerate}
\end{proof}

\begin{rmk}\label{relaxMom3}
In order to drop the assumption $\E[X^3]=0$, one should extend Lemma \ref{ExistenceT} to prove the existence of a positive random variable $V$ such that $\E[V]=1, \E[V^2] = 1+\theta$, and $\E[ \sqrt{V^3}] = \E[X^3]$, and consider mixtures between $T$ and a random variable $Z$ satisfying a Fourth Moment Theorem and all the necessary properties needed in the proof. For instance, one can choose $Z$ to be a centered Poisson random variable $P$ with parameter $1$. In this case, $\E[Q_{\X}(f_n)^i] =\E[Q_{\mathbf{TP}}(f_n)^i]$ for every $i=1,2,3,4$. Indeed, random variables in the Poisson Wiener Chaos satisfy both the Fourth moment Theorem and the universality phenomenon (see Theorem \ref{FMTPoisson} and \cite[Theorem 3.4]{PeccatiZheng1}), and enjoy as well the  feature of having strictly positive fourth cumulant: $\E[Q_{\mathbf{P}}(f_n)^4] -3 > 0$ (see identities (4.10) and (4.11) in \cite{PeccatiZheng1}), which has been  crucial property for the proof.
\end{rmk}

\subsection{$\lambda$-Stein pairs and normal approximation of quadratic homogeneous sums}

In this subsection, we will provide an alternative proof of the quadratic Fourth Moment Theorem for homogeneous sums in independent copies of a random variables $X$, satisfying Assumption {\bf (2)} and $\chi_4(X) \geq 0$, using the tools of the Stein's method of exchangeable pairs \cite{SteinBook, Chen_Goldstein_Shao, Ross}.\\

Let $X$ denote a random variable satisfying Assumption {\bf (2)}, such that $\chi_4(X) \geq 0$. As already seen in Subsection \ref{QuadraticCaseClassic}, for a quadratic homogeneous sum, with admissible coefficient \linebreak $f:[n]^2 \rightarrow \mathbb{R}$, formula \eqref{formulaGauss} can be explicitly written as:
\begin{equation}\label{QuadFormula}
\E[  Q_{\mathbf{X}}(f)^4] - 3 =  48 \| f \stackrel{1}{\smallfrown} f  \|^2   +  16 \chi_4(X) \sum_{k=1}^n \E[  Q_{\mathbf{N}}(f(k,\cdot))^4 ]  + 16 \chi_4(X)^2 \sum_{k_1,k_2=1}^n f(k_1,k_2)^4 \,.
\end{equation}
The aim of this subsection is to analyse normal approximations of quadratic homogeneous sums in independent copies of $X$, by combining the Lindberg method of influence functions with the Stein's method of exchangeable pairs \cite{Chen_Goldstein_Shao}, in the particular setting of $\lambda$-Stein pairs. Recall that an exchangeable pair $(T,T^{\prime})$ is called a $\lambda$-\textit{Stein pair}\index{$\lambda$-Stein Pair} if there exists $\lambda \in (0,1]$ such that, almost surely:
$$ \mathbb{E}[T^{\prime}|T] = (1-\lambda) T.$$
The analysis will be performed via the following bound for the Wasserstein distance $d_{\mathcal{W}}(\cdot, \cdot)$ between $Q_{\bs{X}}(f)$ and a random variable having the standard normal distribution (see, for instance, \cite[Theorem 3.7]{Ross}).
\begin{thm}
\label{WasserBound}
Let $T=T(X_1,\dots,X_n)$ be a symmetric statistics of the independent observations $X_1,\dots,X_n$, and let $T^{\prime}$ be a random variable such that $(T,T^{\prime})$ is a $\lambda$-Stein pair. Then, for $N\sim \mathcal{N}(0,1)$,
$$ d_\mathcal{W}(T, N) \leq \dfrac{\sqrt{\mathrm{Var}(  \mathbb{E}[  (T-T^{\prime})^2   | T  ]   )   }     }{ \sqrt{2\pi}\lambda} +         \dfrac{  \mathbb{E}| T - T^{\prime} |^3 }{3\lambda}  .$$
\end{thm}

More specifically, the next proposition provides a bound for the Wasserstein distance\linebreak $d_{\mathcal{W}}(Q_{\bs{X}}(f_n), \mathcal{N}(0,1))$, depending on the maximum of the influence functions $\tau_n$, and on the fourth moment of $Q_{\bs{X}}(f)$. As a consequence, and as an application of the formula \eqref{formulaGauss}, an alternative proof of the Fourth Moment Theorem for $X$ can be achieved, independently of the Fourth Moment Theorem and of the universality property of the Gaussian distribution.
\begin{prop}\label{SteinPairProof}
If $X$ satisfies Assumption {\bf (2)} and $\chi_4(X) \geq 0$, there exist constants $M_1, M_2$, only depending on $\E[X^4]$, and a constant $R_3$ (not depending on $X$) such that:
$$
d_{\mathcal{W}}(Q_{\bs{X}}(f), N)  \leq \dfrac{\sqrt{M_1 \big(\E[Q_{\bs{X}}(f)^4] - 3\big) + M_2 \tau(f) }}{2\sqrt{2\pi}} + \dfrac{4R_3 (\E[|X|^3])^2 \sqrt{\tau(f)}}{3},$$
where $\tau(f)= \max\limits_{i=1,\dots,n}\mathrm{Inf}_i(f)$.\\
\end{prop} 
Note that, since $T=T(X_1,\dots,X_n)$,  by virtue of the  inequality:
$$ \mathrm{Var}\big(\mathbb{E}[  (T-T^{\prime})^2   | T ] \big)  \leq     \mathrm{Var}\big(\mathbb{E}[  (T-T^{\prime})^2   | X_1,\dots, X_n  ]\big), $$
one can derive bounds directly for $\mathrm{Var}(\mathbb{E}[  (T-T^{\prime})^2   | X_1,\dots, X_n  ])$, instead that for \linebreak$\mathrm{Var}(   \mathbb{E}[  (T-T^{\prime})^2   | T  ])$.

The following result is known as \textit{Rosenthal inequality}\index{Rosenthal inequality}, and will be of use to derive the desired bounds (see \cite{Rosenthal} for generalizations of the Rosenthal's inequality for symmetric statistics of higher orders, in non identically distributed variables).

\begin{prop}
For every $t\geq 2$, let $X_1,\dots,X_n$ be centered independent random variables, such that  $\mathbb{E}[|X_i|^t] < \infty$ for every $i=1,\dots,n$. Then, there exists a positive constant $R_t$ such that:
$$ \mathbb{E}\bigg[\bigg| \sum_{i=1}^n X_i   \bigg|^t  \bigg] \leq R_t \max\bigg(  \sum_{i=1}^n \mathbb{E}[|X_i|^t], \bigg( \sum_{i=1}^n\mathbb{E}[X_i^2]\bigg)^{\frac{t}{2}}      
\bigg).$$
\end{prop}

The proof of Proposition \ref{SteinPairProof}, some preliminary arguments are needed. 
Assume  that $X$ satisfies Assumption {\bf (2)} and  $\chi_4(X) \geq 0$, and consider a homogeneous sum of degree $2$, based on independent copies of $X$, say
$$Q:=Q_{\bs{X}}(f)= \sum\limits_{1 \leq i < j \leq n}f(i,j)X_i X_j, $$
where $f:[n]^2 \rightarrow \mathbb{R}$ is an admissible kernel. Then, for every $k=1,\dots,n$:
$$ Q= \sum_{i < k}f(i,k)X_i X_k + \sum_{j > k}f(j,k)X_j X_k + \sum_{\substack{1 \leq l < r \leq n\\ l,r \neq k}}f(l,r)X_l X_r.$$
If $I$ is a random index, chosen uniformly from $\{1,\dots,n\}$, consider $Q^{\prime}$ obtained from $Q$ by replacing $X_{I}$ with an independent copy $X_{I}^{\prime}$. Then
$$ Q- Q^{\prime} = \sum_{i < I}f(i,I)X_i (X_I - X_{I}^{\prime}) + \sum_{j > I}f(j,I)X_j (X_I - X_{I}^{\prime}). $$

\begin{prop}\label{lambda}
$(Q, Q^{\prime})$ is a $\lambda$-Stein pair, with $\lambda = \dfrac{2}{n}$.
\end{prop}

\begin{proof} 
From:
\begin{equation*}
Q - Q^{\prime} = \sum_{i < I} f(i,I) X_i (X_I - X_I^{\prime}) + \sum_{j > I} f(j,I) X_j (X_I - X_I^{\prime}),
\end{equation*}
it follows that:
\begin{align*}
\mathbb{E}&[Q - Q^{\prime}|X_1,\dots, X_n] = \mathbb{E}[ \sum_{i < I} f(i,I) X_i (X_I - X_I^{\prime})+ \sum_{j > I} f(j,I) X_j (X_I - X_I^{\prime})|X_1,\dots, X_n] \\
&= \frac{1}{n}\sum_{k=1}^n \big(\sum_{i < k} f(i,k) X_i \mathbb{E}[X_k - X_k^{\prime}|X_1, \dots, X_n ]      +    \sum_{j >k} f(j,k) X_j \mathbb{E}[X_k- X_k^{\prime}|X_1, \dots, X_n ]          \big) \\
&= \frac{1}{n}\sum_{k=1}^n \big(\sum_{i < k} f(i,k) X_i X_k      +    \sum_{j >k} f(j,k) X_j X_k      \big) \\
&= \frac{2}{n} Q. \qedhere
\end{align*}
\end{proof}

\begin{rmk}
Proposition \ref{lambda} is a particular case of the following general picture. 
If $X_1,\dots,X_n$ are independent and identically distributed random variables, let $V=V(X_1,\dots,X_n)$ be a symmetric, degenerate $U$-statistics of order $d <n$, that is:
$$ V= \sum_{1 \leq i_1 < i_2 < \cdots < i_d \leq n} f(X_{i_1},\dots,X_{i_d}),$$
where the kernel $f$ satisfies $\E[f(X_1,\dots,X_d)| X_1,\dots,X_{d-1}] = 0$ a.s.. For a random index $I$, chosen uniformly from $\{1,\dots,n\}$, let $V^{\prime}$ be obtained from $V$ by replacing $X_I$ with an independent copy $X_I^{\prime}$. Then, thanks to independence,
$$ \E[V^{\prime}| X_1,\dots,X_n] = \dfrac{1}{n} \sum_{j=1}^n \bigg(\sum_{\substack{1 \leq i_1  < \cdots < i_{d} \leq n\\ i_l \neq j}} f(X_{i_1},\dots,X_{i_d}) \bigg).$$
Moreover, due to the exchangeability of $X_1,\dots,X_n$, $\E[f(X_{i_1},\dots,X_{i_d}) | V  ] = \dfrac{1}{\binom{n}{d}}V$, for every $i_1,\dots,i_d \in [n]$, and, in turn:
$$ \E[V^{\prime}| V] = \dfrac{ \binom{n-1}{d} }{\binom{n}{d}}V = \big(1 - \dfrac{d}{n}\big)V.$$
Hence, $(V, V^{\prime})$ is a $\lambda$-Stein pair for $\lambda = \dfrac{d}{n}$. 

In particular, if $f:[n]^d \rightarrow \mathbb{R}$ is an admissible kernel, homogeneous sums $Q_{\bs{X}}(f)$ of degree $d<n$, in i.i.d. random variables, are instances of symmetric, degenerate $U$-statistics, and hence we obtain a $\frac{d}{n}$-Stein pair by replacing $X_I$ with an independent copy $X_I^{\prime}$. 
Remark that, for homogeneous Rademacher sums, the above construction of an exchangeable pair has already been observed and exploited for the analysis on the Rademacher Chaos (see \cite[Section 3.3]{NourdinPeccatiReinert2}). \\
\end{rmk}

\begin{rmk}
Note that, throughout the previous chapters, the homogeneous sums have been defined as:
$$ Q_{\bs{X}}(f) = \sum_{i,j=1}^n f(i,j)X_i X_j = 2 \sum_{1\leq i < j \leq n}f(i,j)X_i X_j ,$$
where the last equality is due to the symmetry of $f$. Therefore, up to replace $\lambda$ with $\frac{4}{n}$, in the sequel the focus will be on
$$ Q_{\bs{X}}(f) = \sum_{1\leq i < j \leq n}f(i,j)X_i X_j .$$
\end{rmk}

Thanks to Proposition \ref{lambda}, it is possible to apply Theorem \ref{WasserBound} to prove Proposition \ref{SteinPairProof}. Remark that, even if Proposition \ref{lambda} can be generalised for homogeneous sums of every order $d \geq 2$, the forthcoming discussion only holds, with the tools available so far, in the quadratic case.

\begin{proof}
In order to bound the second summand in the right-hand side of the inequality stated with Theorem \ref{WasserBound}, write:
\begin{align*}
\mathbb{E}[|Q - Q^{\prime}|^3] &= \frac{1}{n}\sum_{k=1}^n \mathbb{E}\bigg[ \bigg|\sum_{i < k}f(i,k)X_i (X_k - X_{k}^{\prime}) + \sum_{j > k}f(j,k)X_j (X_k - X_{k}^{\prime})  \bigg|^3 \bigg] \\
&= \frac{1}{n}\sum_{k=1}^n \mathbb{E}\bigg[\bigg|(X_k - X_k^{\prime})\sum_{\substack{i=1\\i \neq k}}^n f(i,k)X_i     \bigg|^3\bigg]\\
&=\frac{1}{n}\sum_{k=1}^n \mathbb{E}\bigg[\big|X_k - X_k^{\prime}\big|^3 \bigg|\sum_{\substack{i=1\\i \neq k}}^nf(i,k)X_i \bigg|^3\bigg]\\
&=\frac{1}{n}\sum_{k=1}^n \mathbb{E}\big[|X_k - X_k^{\prime}|^3\big] \mathbb{E}\bigg[ \bigg|\sum_{\substack{i=1\\i \neq k}}^nf(i,k)X_i \bigg|^3\bigg],
\end{align*}
where the last equality follows by independence, and apply the Rosenthal inequality:
\begin{align*}
\mathbb{E}\bigg[ \bigg|\sum_{\substack{i=1,\dots,n\\i \neq k}}f(i,k)X_i \bigg|^3\bigg] &\leq R_3 \max\bigg( \mathbb{E}[|X|^3]\sum_{\substack{i=1,\dots,n\\i \neq k}}|f(i,k)|^3\; ,\; \bigg(\sum_{\substack{i=1,\dots,n\\ i \neq k}}f(k,i)^2 \mathbb{E}[X_i^2]\bigg)^{\frac{3}{2}} \bigg) \\
&= R_3 \max\bigg( \mathbb{E}[|X|^3]\sum_{\substack{i=1,\dots,n\\i \neq k}}|f(i,k)|^3\; ,\; \big(\mathrm{Inf}_k(f) \big)^{\frac{3}{2}} \bigg).
\end{align*}
By applying the H\"{o}lder inequality, one obtains the estimates:
\begin{align*}
\sum_{\substack{i=1,\dots,n\\i \neq k}}|f(i,k)|^3 &= \sum_{\substack{i=1,\dots,n\\i \neq k}}|f(i,k)| f(i,k)^2
 \\
 &\leq \bigg( \sum_{\substack{i=1,\dots,n\\i \neq k}} f(i,k)^4  \bigg)^{\frac{1}{2}} \bigg( \sum_{\substack{i=1,\dots,n\\i \neq k}} f(i,k)^2  \bigg)^{\frac{1}{2}} \\
 &= \bigg( \sum_{\substack{i=1,\dots,n\\i \neq k}} f(i,k)^4  \bigg)^{\frac{1}{2}} \big( \mathrm{Inf}_k(f)\big)^{\frac{1}{2}} \\
 &\leq \bigg( \big(\sum_{\substack{i=1,\dots,n\\i \neq k}} f(i,k)^2\big)^2 \bigg)^{\frac{1}{2}} \bigg( \mathrm{Inf}_k(f)\bigg)^{\frac{1}{2}} \\
&= \mathrm{Inf}_k(f)  \big( \mathrm{Inf}_k(f)\big)^{\frac{1}{2}} = \big(\mathrm{Inf}_k(f)\big)^{\frac{3}{2}} 
\end{align*}
yielding:
\begin{align*}
\mathbb{E}\bigg[ \bigg|\sum_{\substack{i=1,\dots,n\\i \neq k}}f(i,k)X_i \bigg|^3\bigg] &\leq R_3 \max\bigg( \mathbb{E}[|X|^3] \big(\mathrm{Inf}_k(f)\big)^{\frac{3}{2}}, \big(\mathrm{Inf}_k(f)\big)^{\frac{3}{2}} \bigg).
\end{align*}
Since $\max( \mathbb{E}[|X_1|^3],1) = \mathbb{E}[|X_1|^3]$ (indeed, by Jensen's inequality, $1 = \E[|X_1|^2 ]^{\frac{3}{2}} \leq \mathbb{E}[|X_1|^3]$), finally one has:
\begin{align*}
\mathbb{E}[|Q - Q^{\prime}|^3] &\leq \frac{1}{n}\sum_{k=1}^n \mathbb{E}\big[|X_k - X_k^{\prime}|^3\big]
R_3 \max\bigg( \mathbb{E}[|X_1|^3] \big(\mathrm{Inf}_k(f)\big)^{\frac{3}{2}}, \big(\mathrm{Inf}_k(f)\big)^{\frac{3}{2}} \bigg) \\
&\leq   \dfrac{R_3 \mathbb{E}[|X|^3]\,\mathbb{E}\big[|X_1 - X_1^{\prime}|^3\big]}{n} \sum_{k=1}^n  \big(\mathrm{Inf}_k(f)\big)^{\frac{3}{2}} \\
&\leq \dfrac{R_3 \mathbb{E}[|X|^3]\,\mathbb{E}\big[|X_1 - X_1^{\prime}|^3\big]}{n}   \bigg(\max_{k=1,\dots,n}\mathrm{Inf}_k(f)\bigg)^{\frac{1}{2}}  \sum_{k=1}^n  \mathrm{Inf}_k(f) \\
&=  \dfrac{R_3 \mathbb{E}[|X|^3] \,\mathbb{E}\big[|X_1 - X_1^{\prime}|^3\big]}{n}  \bigg(\max_{k=1,\dots,n}\mathrm{Inf}_k(f)\bigg)^{\frac{1}{2}} 
\end{align*}
(due to $\sum\limits_{k=1}^n \mathrm{Inf}_k(f) =1$). In conclusion, for $\lambda = \frac{2}{n}$,
$$\dfrac{ \mathbb{E}[|Q - Q^{\prime}|^3]}{3\lambda} \leq  \frac{R_3 \mathbb{E}[|X_1|^3]}{6}\mathbb{E}\big[|X_1- X_1^{\prime}|^3\big]   \big(\tau(f)\big)^{\frac{1}{2}},$$
with $\tau(f) = \max\limits_{k=1,\dots,n}\mathrm{Inf}_k(f)$. The conclusion then follows by virtue of the inequality $\mathbb{E}\big[|X_1- X_1^{\prime}|^3\big] \leq 8\mathbb{E}[|X_1|^3]$, that can be proved by first expanding the cube, and then by applying the H\"{o}lder inequality.\\

%
%
%

Finding a desirable bound for the first summand appearing in the bound stated via Theorem \ref{WasserBound} is a bit more demanding:
\begin{align*}
\mathbb{E}[  (Q-Q^{\prime})^2   | X_1,\dots, X_n  ] &= \dfrac{1}{n} \sum_{k=1}^n \mathbb{E}\bigg[  \bigg(  \sum_{ \substack{ j=1 \\ j \neq k}}^n  f(j,k)X_j (X_k - X_k^{\prime})  \bigg)^2                  | X_1,\dots, X_n \bigg]  \\
 &=  \dfrac{1}{n} \sum_{k=1}^n \bigg(   \sum_{ \substack{ j=1  \\ j \neq k}}^n f(j,k) X_j  \bigg)^2    \mathbb{E}[(X_k - X_k^{\prime})^2 | X_1,\dots, X_n] \\
 &=  \dfrac{1}{n} \sum_{ k=1}^n  (X_k^2 + 1)   \bigg(   \sum_{ \substack{ j=1\\ j \neq k}} ^nf(j,k) X_j  \bigg)^2    \,.
\end{align*}
Then,
$$\E\big[\mathbb{E}[  (Q-Q^{\prime})^2   | X_1,\dots, X_n  ] \big]  = \dfrac{2}{n}\sum_{k=1}^n \E\bigg[\bigg(\sum_{\substack{j=1\\j \neq k}}^n f(j,k)X_j \bigg)^2   \bigg]  = \dfrac{2}{n}\sum_{k=1}^n \mathrm{Inf}_k(f) =\dfrac{2}{n} $$
so that
$$ \big(\E\big[\mathbb{E}[  (Q-Q^{\prime})^2   | X_1,\dots, X_n  ] \big] \big)^2 = \dfrac{4}{n^2}$$
(recall that $ \sum\limits_{i=1}^n \mathrm{Inf}_i(f) =1$). 
On the other hand,
$$ \big( \mathbb{E}[  (Q-Q^{\prime})^2   | X_1,\dots, X_n  ] \big)^2  = A + B,     $$ 
where we have set:
\begin{equation}
\label{A}
A :=  \dfrac{1}{n^2}  \sum_{ k=1}^n  (X_k^2 + 1)^2   \bigg(   \sum_{ \substack{ j=1, \dots, n  \\ j \neq k}} f(j,k) X_j  \bigg)^4
\end{equation}
and 
\begin{equation}
\label{B}
B := \dfrac{1}{n^2} \sum_{ \substack{ k,l=1,\dots, n \\ k \neq l  }}  (X_k^2 + 1)(X_l^2 + 1)  \bigg(   \sum_{ \substack{ j=1, \dots, n  \\ j \neq k}} f(j,k) X_j  \bigg)^2 \bigg(   \sum_{ \substack{ i=1, \dots, n  \\ i\neq l}} f(i,l) X_i \bigg)^2.
\end{equation}

Since $X$ satisfies Assumption {\bf (2)}, and the $X_i$'s are independent, straightforward computations yield that:
\begin{equation}
\label{E(A)}
\E[A]= \dfrac{\E[(X^2+1)^2]}{n^2}\bigg(\E[X^4] \sum_{k=1}^n \sum_{\substack{j=1\\ j\neq k}}^n f(j,k)^4 \, + 3 \sum_{k=1}^n \sum_{\substack{j_1,j_2=1 \\j_1 \neq j_2}}^n f(j_1,k)^2\, f(j_2,k)^2 \bigg)\, ;
\end{equation}
\begin{align*}
\E[B]  &= \dfrac{4\E[X^4]}{n^2} \sum_{\substack{k,l=1\\k \neq l}}^n \sum_{\substack{j=1\\j\neq k,l}}^n f(j,l)^2 f(j,k)^2  + \dfrac{(\E[X^4]+1)^2}{n^2}  \sum_{\substack{k,l=1\\k \neq l}}^n f(k,l)^4  \\
& \quad + \dfrac{2(\E[X^4]+1)}{n^2}\sum_{\substack{k,l=1\\ k\neq l}}^n f(k,l)^2\,\sum_{\substack{j=1\\ j\neq k,l}}^n f(j,k)^2 + \dfrac{2(\E[X^4]+1)}{n^2}\sum_{\substack{k,l=1\\ k\neq l}}^n f(k,l)^2\,\sum_{\substack{i=1\\ i\neq k,l}}^n f(i,l)^2\\
& \qquad + \dfrac{4}{n^2} \sum_{\substack{k,l=1\\k \neq l}}^n \sum_{\substack{j,i=1\\ j\neq i\\ j,i \neq k,l}}^n f(i,l)^2 \,f(j,k)^2  + \dfrac{8}{n^2} \sum_{\substack{k,l=1\\k \neq l}}^n \sum_{\substack{j,i=1\\ j\neq i \\ j,i \neq k,l}}^n f(j,k)f(i,k)f(i,l)f(j,l) \, . \numberthis \label{E(B)}
\end{align*}
In the end, setting:
\begin{enumerate}
\item $P_1(\E[X^4]) = \E[X^4]\E[(X^2+1)^2] + (\E[X^4]+1)^2$,
\item $P_2(\E[X^4]) = 3 \E[(X^2+1)^2] + 4\E[X^4]$,
\end{enumerate}
one can write:
\begin{align*}
\E\big[\big( \mathbb{E}[ & (Q-Q^{\prime})^2  | X_1,\dots, X_n  ] \big)^2\big] = \dfrac{P_1(\E[X^4])}{n^2} \sum_{k=1}^n \sum_{\substack{j=1\\ j\neq k}}^n f(j,k)^4 \, \,\\
&\qquad +  \dfrac{P_2(\E[X^4])}{n^2}  \sum_{k=1}^n\sum_{\substack{j_1,j_2=1\\ j_1 \neq j_2}}^n f(j_1,k)^2 f(j_2,k)^2 \, +\, \dfrac{4}{n^2} \sum_{\substack{k,l=1 \\ k \neq l}}^n \sum_{\substack{i,j=1\\ i \neq j, i,j \neq k,l    }}^n f(i,l)^2 f(j,k)^2    \\
& \qquad +  \dfrac{8}{n^2}\sum_{\substack{k,l=1 \\ k \neq l}}^n \sum_{\substack{i,j=1\\ i \neq j, i,j \neq k,l    }}^n f(j,k)f(i,k)f(i,l)f(j,l)  \,+ \,\dfrac{2(\E[X^4]+1)}{n^2}\sum_{\substack{k,l=1\\ k\neq l}}^n f(k,l)^2\,\sum_{\substack{j=1\\ j\neq k,l}}^n f(j,k)^2\\
& \qquad \quad + \dfrac{2(\E[X^4]+1)}{n^2}\sum_{\substack{k,l=1\\ k\neq l}}^n f(k,l)^2\,\sum_{\substack{j=1\\ j\neq k,l}}^n f(j,k)^2 
\numberthis \label{Firstineq}.
\end{align*}
Note that $P_1(\E[X]^4), P_2(\E[X^4]) >0$, and  that, under the assumption $\chi_4(X) \geq 0$, \linebreak$\max\big(8,P_1(\E[X^4]),P_2(\E[X^4])) =P_1(\E[X^4])$ . Therefore, considering the inequalities:
\begin{align*}
\sum_{\substack{k,l=1\\ k\neq l}}^n f(k,l)^2\,\sum_{\substack{j=1\\ j\neq k,l}}^n f(j,k)^2  &\leq \sum_{k=1}^n \mathrm{Inf}_k(f)^2 \\
&\leq   \tau(f)\, ,
\end{align*}
and, similarly,
$$ \sum_{\substack{k,l=1\\ k\neq l}}^n f(k,l)^2\,\sum_{\substack{i=1\\ j\neq k,l}}^n f(i,l)^2 \leq  \tau(f),$$
as well as:
$$\sum_{\substack{k,l=1 \\ k \neq l}}^n \sum_{\substack{i,j=1\\ i \neq j, i,j \neq k,l    }}^n f(i,l)^2 f(j,k)^2 \,<\, \bigg( \sum_{k=1}^n \mathrm{Inf}_k(f)\bigg)^2 = 1, $$
it follows that:
\begin{align*}
\E\big[\big( \mathbb{E}[  (Q-Q^{\prime})^2  & | X_1,\dots, X_n  ] \big)^2\big] \leq  \dfrac{4}{n^2}   \,+\,  \dfrac{P_1(\E[X^4])}{n^2} \| f \stackrel{1}{\smallfrown} f \|^2  \\ 
&\quad +  \dfrac{4(\E[X^4]+1)}{n^2} \sum_{k=1}^n \mathrm{Inf}_k(f)^2 \\
&\leq \dfrac{4}{n^2}  + \dfrac{P_1(\E[X^4])}{n^2}  \bigg( 16\chi_4(X) ^2 \sum_{k=1}^n \sum_{\substack{j=1\\ j\neq k}}^n f(j,k)^4  \\
& \qquad \quad + 48 \chi_4(X) \sum_{k=1}^n\sum_{\substack{j_1,j_2=1\\ j_1 \neq j_2}}^n f(j_1,k)^2 f(j_2,k)^2 + 48 \|f \stackrel{1}{\smallfrown} f\|^2 \| \bigg) \\
&\qquad \qquad + \dfrac{4(\E[X^4]+1)}{n^2}\tau(f)
\end{align*}
and, in the end,
$$ \mathrm{Var}\big(\mathbb{E}[  (Q-Q^{\prime})^2   | X_1,\dots, X_n  ]\big) \leq \dfrac{P_1(\E[X^4])}{n^2} \big( \E[Q_{\bs{X}}(f)^4] - 3\big) + \dfrac{4(\E[X^4]+1)}{n^2} \tau(f).$$
Hence, for $\lambda = \dfrac{2}{n}$, and $N\sim \mathcal{N}(0,1)$,
\begin{align*}
d_{\mathcal{W}}(Q_{\bs{X}}(f), N) & \leq \dfrac{\sqrt{\mathrm{Var}\big(\mathbb{E}[  (Q-Q^{\prime})^2   | X_1,\dots, X_n  ]\big)}}{\sqrt{2\pi} \lambda} + \dfrac{\E[|Q-Q^{\prime}|^3]}{3\lambda} \\
& \leq \dfrac{\sqrt{P_1(\E[X^4]) \big(\E[Q_{\bs{X}}(f)^4] - 3\big) +  4(\E[X^4]+1) \tau(f) }}{2\sqrt{2\pi}} \\
& \qquad + \dfrac{4(\E[|X_1|^3])^2 R_3 \sqrt{\tau(f)}}{3}, \numberthis \label{IneqWass}
\end{align*}
and the claim follows.
\end{proof}
As a consequence, another proof of the quadratic Fourth Moment Theorem for $Q_{\bs{X}}(f_n)$ can be achieved.

\begin{cor}
Let $X$ be a random variable satisfying Assumption {\bf (2)} and $\chi_4(X) \geq 0$. Then, $X$ satisfies the Fourth Moment Theorem at the order $d=2$.
\end{cor}

\begin{proof}
Apply Proposition \ref{SteinPairProof} to a sequence $Q_{\bs{X}}(f_n)$,    with $f_n:[n]^2\rightarrow \mathbb{R}$ admissible kernel, and 
note that formula \eqref{QuadFormula}, together with the assumption $\chi_4(X) \geq 0$, implies $\E[Q_{\bs{X}}(f_n)^4 ] - 3 \geq 48 \| f_n\stackrel{1}{\smallfrown} f_n \|^2$. Then, if $\E[Q_{\bs{X}}(f_n)^4 ] - 3 \rightarrow 0$ as $n\rightarrow \infty$,   from  \eqref{magg1} it follows  that $\tau_n:=\tau(f_n) \rightarrow 0$ and,  in turn,
$$ d_{\mathcal{W}}(Q_{\bs{X}}(f_n), N) \rightarrow 0.$$
The conclusion follows considering that the topology induced by the Wasserstein distance is stronger than the topology of convergence in distribution.
\end{proof}


\chapter{The free probability setting: Fourth Moment Theorem and universality}\label{Free}

In this chapter, we will focus on free probability spaces, where the contents presented in the Chapter \ref{Classic} will be adapted to deal with homogeneous polynomials in freely independent random variables. The idea of the proofs developed in the following sections are similar to those exploited in the previous chapter: however, the peculiar structure of the lattice of non-crossing partitions will allow us to simplify some arguments (compare, for instance, the formulae \eqref{formula} and \eqref{formulaGauss}). In particular, one consequence of dealing with non-crossing partitions is that no assumption on the vanishing of the third moment will be required, so that Theorem \ref{superTeo2} covers a wider class of random variables than Theorem \ref{superTeo1} in the classical setting.

On the other hand, there will be no analogue of the alternative proof based on mixtures of random variables, since this technique trivializes when dealing with free independence. Finally, the equivalence between joint and componentwise convergence for the whole class of random variables $Y$ with $\varphi(Y^2) \geq 2$ will be established (see Theorem \ref{ComponentJointFree}), allowing one to have available a general multidimensional transfer principle for central convergence of symmetric homogeneous sums in independent copies of random variables having non-negative kurtosis (see Theorem \ref{Transfer}).

\section{Preliminaries}

For every $n\in \mathbb{N}$, set $[n] := \{1,\dots,n\}$. 

\begin{defn}\label{Admissible_Free}
Let $d\geq 2$. An  \textbf{admissible kernel} is a function $f:[n]^d \to \mathbb{R}$ satisfying the following properties:
\begin{enumerate}
\item[(i)] vanishing on diagonals: $f(i_1,\dots,i_d)=0$ whenever  $i_j=i_k$ for some $k\neq j$;
\item[(ii)] symmetry: $f(i_1,\dots,i_d)=f(i_{\sigma(1)},\dots,i_{\sigma(d)})$ for any permutation $\sigma\in \mathfrak{S}_d $ and any \linebreak$(i_1,\dots,i_d)\in [n]^d$;
\item[(iii)] $f$ has unit variance: $\sum\limits_{i_1,\dots,i_d =1}^n f(i_1,\dots,i_d)^2=1$.\\
\end{enumerate}
\end{defn}

Let $(\mathcal{A},\varphi)$ be a fixed $W^{\star}$-probability space. As in Part \ref{Invariance}, for a centered random variable $Y$ having unit variance,  namely $\varphi(Y)=0$ and $\varphi(Y^2)=1$, it will be said, for short, that $Y$ satisfies Assumption {\bf (1)}.\\ 

Let $\Y=\{Y_i\}_{i\geq 1}$ be a sequence of freely independent copies of $Y$, that are assumed to be defined on $(\mathcal{A},\varphi)$\footnote{Up to take the free product of the spaces $\mathcal{A}_i$, with $Y_i \in \mathcal{A}_i$.}. If $f:[n]^d\to \mathbb{R}$ is an admissible kernel, consider the homogeneous sum $Q_{\Y}(f)$ defined by:
\begin{eqnarray}
Q_{\Y}(f) &=&\sum_{i_1,\dots,i_d =1}^n f(i_1,\dots,i_d) Y_{i_1}\cdots Y_{i_d}.\label{Fnoncom}
\end{eqnarray}

Assumption {\bf (1)} and the properties of $f$ ensure that $\varphi(Q_{\Y}(f) )=0$ and $\varphi(Q_{\Y}(f)^2)=1$.\\

\begin{rmk}
A different normalization for the admissible kernels is chosen here to ensure that the homogeneous polynomial $Q_{\Y}(f)$ has unit variance.\\
\end{rmk}

In the free setting, the natural choice for the coefficients of a homogeneous sum would be a mirror symmetric function, namely a kernel $f:[n]^d \rightarrow \mathbb{C}$ such that $f(i_1,i_2, \dots, i_d) = \overline{f(i_d,\dots, i_2,i_1)}$ for every $i_1,\dots, i_d \in [n]$, with $\bar{z}$ denoting the complex conjugate of $z$. This assumption is the weakest possible to ensure that the element $Q_{\Y}(f)$ is self-adjoint. However, the forthcoming discussion will heavily rely on the universality property of the Wigner Semicircle law, that has been so far established only for homogeneous sums with symmetric real-valued coefficients: indeed, both in \cite{NourdinDeya} and in \cite{Solesne2}, counterexamples to the universality for mirror symmetric kernels have been provided.  On the other hand, the symmetry assumption on $f$ will allow us a better handling of $Q_{\Y}(f)$ for the computation of its fourth moment. \\

For several reasons, the semicircular distribution is considered  the non-commutative analogue of the Gaussian distribution: for instance, it is the limit law for the free version of the Central Limit Theorem, and joint moments of a semicircular system satisfy a Wick-type formula \cite{Speicher}. One further reason, most interesting for our purposes, is that the semicircular law satisfies both the Fourth Moment Theorem and the universality property (as recalled in Theorems \ref{knps} and \ref{invnoncom} of Part \ref{Invariance}), inspiring the following definition.

\begin{defn}\label{defnoncom}
Fix $d\geq 2$, let $Y$ satisfy Assumption {\bf (1)} and let $S\sim \mathcal{S}(0,1)$.  
\begin{itemize}
\item[(a)]  We shall say that $Y$ \textbf{satisfies the  Fourth Moment Theorem at the order $d$} \index{Fourth Moment Theorem (free)} if, for any sequence $f_n:[n]^d\to \mathbb{R}$ of admissible kernels, the following conditions are equivalent as $n\to\infty$:
\begin{itemize}
\item[(i)] $Q_{\Y}(f_n) \xrightarrow{\text{\rm Law}}\mathcal{S}(0,1)$;
\item[(ii)]  $\varphi(Q_{\Y}(f_n)^4)\to \varphi(S^4)= 2$.
\end{itemize}
\item[(b)] We shall say that $Y$ is   \textbf{universal at the order $d$} (for semicircular approximations of homogeneous sums) if, for any sequence $f_n:[n]^d\to\mathbb{R}$ of admissible kernels, $Q_{\Y}(f_n) \xrightarrow{\text{\rm Law}}\mathcal{S}(0,1)$ implies, as $n\to\infty$,
$$ \tau_n(f_n) := \max_{ i=1,\dots, n} \mathrm{Inf}_i(f_n) \longrightarrow 0, $$
where $\mathrm{Inf}_i(f_n):=\sum\limits_{i_2,\ldots,i_d=1}^n f_n(i ,i_2,\ldots,i_d)^2$ is the $i$-th \textbf{influence function} of $f_n$.\\
\end{itemize}
\end{defn}

\begin{rmk}
By virtue of Theorem \ref{teoNourdin}, $Y$ is {\it  universal} at the order $d$ if, equivalently, for any sequence $f_n:[n]^d\to \mathbb{R}$ of admissible kernels, the following conditions are equivalent as $n\to\infty$:
\begin{itemize}
\item[(i)] $Q_{\Y}(f_n) \xrightarrow{\text{\rm Law}}\mathcal{S}(0,1)$;
\item[(ii)] $Q_{\mathbf{W}}(f_n) \xrightarrow{\text{\rm Law}}\mathcal{S}(0,1)$ for any other sequence $\mathbf{W} = \{W_i\}_{i \geq 1}$ of freely independent random variables satisfying Assumption {\bf (1)}.\\
\end{itemize}
\end{rmk}

\begin{rmk}
Recall that, in the free probability setting, the convergence in law of a sequence of random variables is realized, by definition, with the convergence of the corresponding moments. This is why no hypercontractivity argument is required  to prove the convergence of the moments under the assumption of convergence in law.
\end{rmk}

\section{Main results}

The goal of this section is to  prove the free counterpart of Theorem \ref{superTeo1}, which is established with the next statement.

\begin{thm}
\label{superTeo2}
Fix $d\geq 2$ and consider a random variable $Y$ verifying Assumption {\bf (1)} and such that $\varphi(Y^4) \ge 2$. Then, $Y$ satisfies the Fourth Moment Theorem and it is universal at the order $d$ for semicircular approximations of homogeneous sums.
\end{thm}

As for the commutative case, some combinatorial arguments are needed for the proof. In the sequel, ${\bf S} = \{S_i\}_{i\geq 1}$  will denote a sequence of freely independent standard semicircular random variables. For every $k=1,\dots,n$ and a given admissible kernel $f:[n]^d \rightarrow \mathbb{R}$, consider the function $f(k,\cdot):[n]^{d-1} \rightarrow \mathbb{R}$, defined via:
$$ (i_1,\dots,i_{d-1}) \mapsto f(k,i_1,\dots,i_{d-1}),$$
and the corresponding semicircular homogeneous sums of order $d-1$:
$$Q_{\Y}(f(k,\cdot)) = \sum_{i_1,\dots,i_{d-1}=1}^n f(k,i_1,\dots,i_{d-1}) S_{i_1}\cdots S_{i_{d-1}}.$$

\begin{rmk}
In contrast to Assumption {\bf (2)} in the commutative case,  no extra assumption on the vanishing of the third moment $\varphi(Y^3)$ will be needed thanks to the simpler combinatorics of the non-crossing partitions that will emerge in the proofs.\\
\end{rmk}

The first step towards the free counterpart to Theorem \ref{superTeo1} is the following new formula for the fourth moment of $Q_{\Y}(f)$.

\begin{prop}
Let the above notation prevail. If $Y$ verifies Assumption {\bf (1)}, then, for every admissible kernel $f:[n]^d \rightarrow \mathbb{R}$:
\begin{equation}
\label{formula}
\varphi(Q_\Y(f)^4) = \varphi(Q_\SSw(f)^4) + \kappa_4(Y)\sum_{k=1}^{n}\varphi(Q_\SSw(f(k,\cdot))^4),
\end{equation}
where $\kappa_4(Y)$ denotes the free fourth cumulant of $Y$.
\end{prop} 

\begin{proof}
Write:
\begin{equation*}
\varphi \big(Q_\Y(f)^4\big)= \sum_{\mathbf{i}=(i_1,\dots,i_{4d}) \in [n]^{4d}}f^{\otimes 4}(\mathbf{i})\varphi(Y_{i_1} \cdots Y_{i_d} \cdots Y_{i_{4d}}),
\end{equation*}
where $f^{\otimes 4}(\mathbf{i}) = \prod\limits_{l=1}^4 f(i_{(l-1)d+1},\dots,i_{ld})$. Since $f$ vanishes on diagonals, the moment-cumulant  formula \eqref{MomCumFree} reduces to:
\begin{equation*}
\varphi\big(Y_{i_1}\cdots Y_{i_d}\cdots Y_{i_{2d}}\cdots Y_{i_{3d}}\cdots Y_{i_{4d}}\big) = \sum_{\substack{\sigma \in \mathcal{NC}([4d])\\ \sigma \wedge \pi^{\star}= \hat{0}}} \prod_{b \in \sigma} \kappa_{|b|}(Y_{i_j}:j \in b),
\end{equation*}
where $\pi^{\star}= d^{\otimes 4}$ denotes the interval partition with 4 consecutive blocks of cardinality $d$. In the right-hand side of the above equation, the only partitions that give a non-zero contribution are those whose blocks have at most cardinality $4$, since they have to intersect each block of $\pi^{\star}$  at most at one element. Therefore,  only cumulants up to the order $4$ will be involved. 
Recalling that $\mathcal{NC}^{\star}(d^{\otimes 4})$ denotes the set of the partitions $\sigma \in \mathcal{NC}([4d])$, such that $\sigma \wedge \pi^{\star}= \hat{0}$, every $\sigma \in\mathcal{NC}^{\star}(d^{\otimes 4})$   will give a non-zero contribution only if its blocks have cardinality $2$ or $4$. Indeed, since $Y$ is centered, whenever $B \in \sigma$ is a singleton, say $B=\{j\}$, then $\kappa_{1}(Y_j)=0$. Similarly, $\sigma$ cannot have any block of cardinality $3$, otherwise there would be at least one singleton, and the corresponding cumulant would vanish.

Therefore, the only non-vanishing terms are those relative either to full pairings that respect $\pi^{\star}$, or to partitions that respect $\pi^{\star}$ whose blocks have cardinality $2$ or $4$: denote this set by $\mathcal{NC}_{2,4}^{\star}(d^{\otimes 4})$. The crucial point in the following discussion is that such a partition can only have exactly one $4$-block and $2(d-1)$ pairings \footnote{This difference with the formula in the classical setting is due to the fact that for every choice of the 4-block, the remaining elements can be paired in exactly one non-crossing way.}.

To count the partitions in $\mathcal{NC}_{2,4}^{\star}(d^{\otimes 4})$, start by forming the 4-block. Choose $j_1 \in \{1,\dots,d\}$. Then, if $j_2 \sim j_1$ is selected in $\{d+1,\dots,2d\}$, then necessarily, to avoid crossings, $j_2 = 2d-j_1$ and every $l$, for $l= j_1+1,\dots,d$, has to be matched with $d+l$, $l=1,\dots, d-j-1$. Continuing in this way, the block of cardinality 4 has to be determined by $j_1+1 \sim 2d-j_1 \sim 2d + j_1+ 1 \sim 4d-j_1$. The same reasoning allows us to show that there cannot exist another block of cardinality 4. Indeed, assume that there exist two blocks of size 4, say $h+1 \sim 2d-h \sim 2d + h+ 1 \sim 4d-h$ and $j+1 \sim 2d-j \sim 2d + j+ 1 \sim 4d-j$. Without loss of generality, say $h< j$, but then  $2d+h +1 < 2d + j+1$ and there would be a crossing $h < j < 2d +h +1 < 2d + j + 1$ (if $j < h$, then $2d-j < 2d-h$ and there would be the crossing $j <h < 2d-j < 2d-h$).  
After having formed the 4-block (say, $h+1 \sim 2d-h \sim 2d + h+ 1 \sim 4d-h$, for a given $h=0,\dots,d-1$), the remaining $4(d-1)$ elements have to be paired in such a way that there are no pairings within a block of $\pi^{\star}$. The only possibility is then determined by the conditions: 
\begin{enumerate}
\item $j \sim 2d-j+1$, for $j=h+2,\dots,d$;
\item $j \sim 4d-j+1$, for $j=1,\dots,h$;
\item $2d + j \sim 2d-j+1$, for $j=1,\dots,h$;
\item $ 2d+j \sim 4d-j+1$, for $j=h+2,\dots,d$,
\end{enumerate}
or, equivalently, for $j=1,\dots,d$, the element $\rho_j \in \mathcal{NC}_{2,4}^{\star}(d^{\otimes 4})$ is determined by the following conditions:
\begin{itemize}
\item[-] the $4$-block is determined  by $j \sim 2d-j+1 \sim 2d+j \sim 4d-j+1$;
\item[-] the pairings are determined by:
\begin{enumerate}
\item[-] $h \sim 4d-h+1$ and $2d+h \sim 2d-h+1$, for $h=1,\dots,j-1$,
\item[-] $h \sim 2d-h+1$, for $h= j+1,\dots,d$;
\item[-] $3d+h \sim 3d-h+1$, for $h= 1,\dots,d-j$, 
\end{enumerate}
\end{itemize}
 yielding $|\mathcal{NC}_{2,4}^{\star}(d^{\otimes 4})| = d$. 
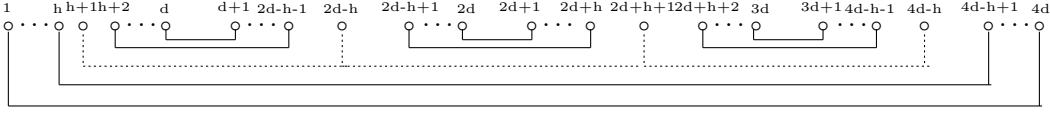
\begin{figure}[h]
\begin{picture}(600, 35)
\put(5,7){\makebox(-1,0){\tiny{$1$}}}
\put(5,0){\circle{3}}
\put(5,-30){\line(0,0){28}}
\put(5,-30){\line(1,0){387}}
\put(9,0){\dots}
\put(24,7){\makebox(-1,0){\tiny{h}}}
\put(24,0){\circle{3}}
\put(24,-22){\line(0,0){20}}
\put(24,-22){\line(1,0){349}}
\put(33,7){\makebox(-1,0){\tiny{h+1}}}
\put(33,0){\circle{3}}
\multiput(33,-3)(0,-2){6}{\line(0,-1){0.8}}
\multiput(33,-15)(2,0){50}{\line(1,0){0.8}}

\put(45,7){\makebox(-1,0){\tiny{h+2}}}
\put(45,0){\circle{3}}
\put(45,-8){\line(0,0){7}}
\put(45,-8){\line(1,0){65}}
\put(49,0){\dots}
\put(64,7){\makebox(-1,0){\tiny{d}}}
\put(64,0){\circle{3}}
\put(64,-5){\line(0,0){4}}
\put(64,-5){\line(1,0){26}}

\put(90,7){\makebox(-1,0){\tiny{d+1}}}
\put(90,0){\circle{3}}
\put(90,-5){\line(0,0){4}}
\put(94,0){\dots}
\put(108,7){\makebox(-1,0){\tiny{2d-h-1}}}
\put(110,-8){\line(0,0){7}}
\put(110,0){\circle{3}}
\put(130,7){\makebox(-1,0){\tiny{2d-h}}}
\put(130,0){\circle{3}}
\multiput(130,-3)(0,-2){6}{\line(0,-1){0.8}}
\multiput(130,-15)(2,0){56}{\line(1,0){0.8}}

\put(155,7){\makebox(1,0){\tiny{2d-h+1}}}
\put(155,0){\circle{3}}
\put(155,-8){\line(0,0){7}}
\put(155,-8){\line(1,0){68}}
\put(160,0){\dots}
\put(176,7){\makebox(1,0){\tiny{2d}}}
\put(175,0){\circle{3}}
\put(175,-5){\line(0,0){4}}
\put(175,-5){\line(1,0){26}}

\put(197,7){\makebox(-1,0){\tiny{2d+1}}}
\put(201,0){\circle{3}}
\put(201,-5){\line(0,0){4}}
\put(205,0){\dots}
\put(220,7){\makebox(-1,0){\tiny{2d+h}}}
\put(223,-8){\line(0,0){7}}
\put(223,0){\circle{3}}
\put(243,7){\makebox(-1,0){\tiny{2d+h+1}}}
\put(243,0){\circle{3}}
\multiput(243,-3)(0,-2){6}{\line(0,-1){0.8}}
\multiput(243,-15)(2,0){54}{\line(1,0){0.8}}
\put(266,7){\makebox(1,0){\tiny{2d+h+2}}}
\put(265,0){\circle{3}}
\put(265,-8){\line(0,0){7}}
\put(265,-8){\line(1,0){65}}
\put(269,0){\dots}
\put(286,7){\makebox(1,0){\tiny{3d}}}
\put(285,0){\circle{3}}
\put(284,-5){\line(0,0){4}}
\put(284,-5){\line(1,0){26}}

\put(310,7){\makebox(-1,0){\tiny{3d+1}}}
\put(310,0){\circle{3}}
\put(310,-5){\line(0,0){4}}
\put(315,0){\dots}
\put(328,7){\makebox(-1,0){\tiny{4d-h-1}}}
\put(330,-8){\line(0,0){7}}
\put(330,0){\circle{3}}
\put(348,7){\makebox(0,0){\tiny{4d-h}}}
\put(348,0){\circle{3}}
\multiput(348,-3)(0,-2){6}{\line(0,-1){0.8}}

\put(372,7){\makebox(1,0){\tiny{4d-h+1}}}
\put(372,0){\circle{3}}
\put(372,-22){\line(0,0){20}}
\put(376,0){\dots}
\put(391,7){\makebox(1,0){\tiny{4d}}}
\put(391,0){\circle{3}}
\put(391,-30){\line(0,0){28}}
\end{picture}
\vspace{0.5cm}
\caption{Diagram of $\rho_{h+1}$, for $h=0,\dots,d-1$}
\end{figure}

\vspace{1cm}

Therefore, the moment-cumulant formula applied  to $\varphi\big(Y_{i_1}\cdots Y_{i_d}\cdots Y_{i_{2d}}\cdots Y_{i_{3d}}\cdots Y_{i_{4d}}\big)$ can be rewritten as:
\begin{equation*}
\varphi\big(Y_{i_1}\cdots Y_{i_d}\cdots Y_{i_{2d}}\cdots Y_{i_{3d}}\cdots Y_{i_{4d}}\big) = \sum_{\substack{\sigma \in \mathcal{NC}_2([4d])\\ \sigma \wedge \pi^{\star}= \hat{0}}} \prod_{\{r,s\} \in \sigma}\varphi(Y_{i_r}Y_{i_s}) + 
\sum_{j=1}^d \kappa_{\rho_j},
\end{equation*}
where: 
\begin{align*}
\small
\kappa_{\rho_j}= \kappa_4&(Y_{i_j},Y_{i_{2d-j+1}},Y_{i_{2d+j}}, Y_{i_{4d-j+1}}) \\& \prod_{h=1}^{d-j} \varphi(Y_{i_{3d+h}} Y_{i_{3d-h+1}})  \prod_{h=j+1}^{d} \varphi(Y_{i_h}Y_{i_{2d-h+1}}) \prod_{h=1}^{j-1} \varphi(Y_{i_h}Y_{i_{4d-h+1}})\varphi(Y_{i_{2d+h}}Y_{i_{2d-h+1}}).
\end{align*}
The vanishing property of cumulants of freely independent random variables then entails that $\varphi\big(Y_{i_1}\cdots Y_{i_d}\cdots Y_{i_{2d}}\cdots Y_{i_{3d}}\cdots Y_{i_{4d}}\big)\neq 0 $ if and only if $\mathrm{Ker}(\mathbf{i}) \in \mathcal{NC}_2^{\star}(d^{\otimes 4})$, in which case equals $1$, or $\mathrm{Ker}(\mathbf{i})\in \mathcal{NC}_{2,4}^{\star}(d^{\otimes 4})$, in which case equals $\kappa_4(Y)$, where $\mathbf{i} = (i_1,\dots,i_{4d})$.  Moreover,  the Wick formula for semicircular elements establishes that:
$$ \varphi(S_{i_1}\cdots S_{i_d}\cdots S_{i_{4d}}) = \sum_{\sigma \in \mathcal{NC}_2^{\star}(d^{\otimes 4})} \prod_{\{r,t\} \in \sigma}\varphi(S_{i_r}S_{i_t})=\sum_{\sigma \in \mathcal{NC}_2^{\star}(d^{\otimes 4})} \prod_{\{r,t\} \in \sigma}\varphi(Y_{i_r}Y_{i_t}),$$
yielding
\begin{equation}
\label{relation}
\varphi\big(Q_\Y(f)^4\big) = \varphi\big(Q_\SSw(f)^4\big)  +  \kappa_4(Y)\sum_{j=1}^d \sum_{\substack{ \mathbf{i}  \in [n]^{4d}  \\ \mathrm{Ker}(\mathbf{i}) = \rho_j}} f^{\otimes 4}( \mathbf{i}),\numberthis 
\end{equation}
where, if $\mathbf{i}  \in [n]^{4d}$, with $ \mathrm{Ker}(\mathbf{i}) = \rho_j$,  
\begin{align*}
 f^{\otimes 4}(\mathbf{i}) &= f(i_1,\dots,i_{j-1},i_j,i_{j+1},\dots,i_d)f(i_d,\dots,i_{j+1},i_j,i_{2d-j+2},\dots,i_{2d})
 \\& f(i_{2d},\dots,i_{2d-j+2},i_j,i_{2d+j+1},\dots,i_{3d})f(i_{3d},\dots,i_{2d + j +1},i_j,i_{j-1},\dots,i_1). 
\end{align*}

Note that, for $h=0,\dots,d-1$, the restriction of $\rho_{h+1}$ to its pairings corresponds naturally to the partition $\sigma_h \in \mathcal{NC}_2^{\star}((d-1)^{\otimes 4})$, determined by the matchings:
\begin{enumerate}
\item $ j \sim 2(d-1) - j +1$, for $j=h+1,\dots,d-1$;
\item $j \sim 4(d-1) - j +1$, for $j=1,\dots,h$;
\item $2(d-1)+j \sim 2(d-1)-j+1$, for $j=1,\dots,h$;
\item $2(d-1) + j \sim 4(d-1) - j +1$, for $j=h+1,\dots,d-1$.
\end{enumerate}

\vspace{0.8cm}
\begin{figure}[h!]
\begin{picture}(600, 30)
\put(5,7){\makebox(-1,0){\tiny{1}}}
\put(5,0){\circle{3}}
\put(5,-30){\line(0,0){28}}
\put(5,-30){\line(1,0){376}}
\put(9,0){\dots}
\put(24,7){\makebox(-1,0){\tiny{h}}}
\put(24,0){\circle{3}}
\put(24,-22){\line(0,0){20}}
\put(24,-22){\line(1,0){339}}

\put(41,7){\makebox(-1,0){\tiny{h+1}}}
\put(41,0){\circle{3}}
\put(41,-10){\line(0,0){8}}
\put(41,-10){\line(1,0){71}}
\put(45,0){\dots}
\put(61,7){\makebox(-1,0){\tiny{d-1}}}
\put(61,0){\circle{3}}
\put(61,-5){\line(0,0){4}}
\put(61,-5){\line(1,0){30}}

\put(91,7){\makebox(-1,0){\tiny{d}}}
\put(91,0){\circle{3}}
\put(91,-5){\line(0,0){4}}
\put(95,0){\dots}
\put(113,7){\makebox(-1,0){\tiny{2d-h-2}}}
\put(112,-10){\line(0,0){8}}
\put(112,0){\circle{3}}

\put(140,7){\makebox(1,0){\tiny{2d-h-1}}}
\put(140,0){\circle{3}}
\put(140,-10){\line(0,0){8}}
\put(140,-10){\line(1,0){85}}
\put(143,0){\dots}
\put(164,7){\makebox(1,0){\tiny{2d-2}}}
\put(164,0){\circle{3}}
\put(164,-5){\line(0,0){4}}
\put(164,-5){\line(1,0){37}}

\put(201,7){\makebox(-1,0){\tiny{2d-1}}}
\put(201,0){\circle{3}}
\put(201,-5){\line(0,0){4}}
\put(205,0){\dots}
\put(225,7){\makebox(-1,0){\tiny{2d-2+h}}}
\put(225,-10){\line(0,0){8}}
\put(225,0){\circle{3}}
\put(255,7){\makebox(1,0){\tiny{2d-1+h}}}
\put(257,0){\circle{3}}
\put(257,-10){\line(0,0){8}}
\put(257,-10){\line(1,0){75}}
\put(264,0){\dots}
\put(283,7){\makebox(1,0){\tiny{3d-3}}}
\put(283,0){\circle{3}}
\put(283,-5){\line(0,0){4}}
\put(283,-5){\line(1,0){30}}

\put(313,7){\makebox(-1,0){\tiny{3d-2}}}
\put(313,0){\circle{3}}
\put(313,-5){\line(0,0){4}}
\put(318,0){\dots}
\put(333,7){\makebox(-1,0){\tiny{4d-4-h}}}
\put(333,-10){\line(0,0){8}}
\put(333,0){\circle{3}}
\put(360,7){\makebox(1,0){\tiny{4d-3-h}}}
\put(362,0){\circle{3}}
\put(362,-22){\line(0,0){20}}
\put(362,0){\dots}
\put(381,7){\makebox(1,0){\tiny{4d-4}}}
\put(381,0){\circle{3}}
\put(381,-30){\line(0,0){28}}
\end{picture}
\vspace*{0.4cm}
\caption{Diagram of $\sigma_h$, $h=0,\dots,d-1$}
\end{figure}
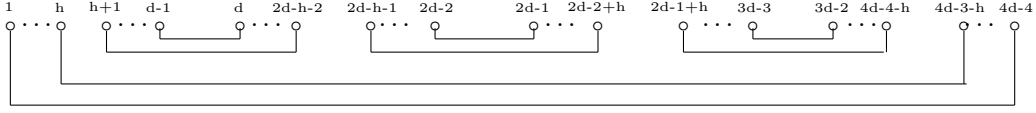

Therefore, a bijection $\sigma_h \longmapsto \rho_{h+1}$ is  determined by inserting in the diagram  of $\sigma_h$,  a block of cardinality $4$, with elements between $h$ and $h+1$, $2d-h-2$ and $2d-h-1$, $2d-2+h$ and $2d-1+h$, $4d-4-h$ and $4d-3+h$, in such a way that the diagram of $\rho_{h+1}$ is recovered (in particular, then, $ |\mathcal{NC}_2^{\star}((d-1)^{\otimes 4})| = |\mathcal{NC}_{2,4}^{\star}(d^{\otimes 4})| = d$). Therefore,
\begin{align*}
\varphi\big(Q_{\Y}(f_n)^4 \big) &= \varphi \big(Q_{\SSw}(f_n)^4 \big) + \kappa_4(Y) \sum_{j=1}^d \sum_{ \substack{\mathbf{i} \in [n]^{4d}\\ \mathrm{Ker}(\mathbf{i}) = \rho_j }} f_n^{\otimes 4}(\mathbf{i}) \\
&= \varphi \big(Q_{\SSw}(f_n)^4 \big) + \kappa_4(Y) \sum_{h=0}^{d-1} \sum_{ \substack{\mathbf{i} \in [n]^{4d}\\ \mathrm{Ker}(\mathbf{i}) = \rho_{h+1} }} f_n^{\otimes 4}(\mathbf{i}).
\end{align*}
In conclusion, observe that every $\mathbf{i} \in [n]^{4d}$, with $\mathrm{Ker}(\mathbf{i}) = \rho_{h+1}$, is uniquely determined by the value $k:=i_{h+1} = i_{2d+h+1} = i_{2d - h} = i_{4d-h}$, corresponding to the 4-block, and by the sub-vector $\mathbf{j} \in [n]^{4(d-1)}$ with $\mathrm{Ker}(\mathbf{j}) = \sigma_h$, from which it follows that:
\begin{align*}
\varphi\big(Q_{\Y}(f)^4 \big) &= 
\varphi \big(Q_{\SSw}(f)^4 \big) + \kappa_4(Y) \sum_{h=0}^{d-1} \sum_{ \substack{\mathbf{i} \in [n]^{4d}\\ \mathrm{Ker}(\mathbf{i}) = \rho_{h+1} }} f^{\otimes 4}(\mathbf{i}) \\
&= \varphi \big(Q_{\SSw}(f)^4 \big) + \kappa_4(Y) \sum_{k=1}^n \sum_{h=0}^{d-1} \sum_{\substack{\mathbf{j} \in [n]^{4(d-1)}\\ \mathrm{Ker}(\mathbf{j}) = \sigma_h} }f(k,\cdot)^{\otimes 4}(\mathbf{j})\\
&= \varphi \big(Q_{\SSw}(f)^4 \big) + \kappa_4(Y) \sum_{k=1}^n \varphi\big(Q_{\SSw}(f(k,\cdot))^4\big).\qedhere
\end{align*}
\end{proof}

Thanks to formula \eqref{formula}, it is now possible to prove Theorem \ref{superTeo2}.

\begin{proof}
Let ${\bf S}=\{S_i\}_{i\geq 1}$  be a sequence of freely independent standard semicircular random variables, and assume that $\kappa_4( Q_{\Y}(f_n))=\varphi(Q_{\Y}(f_n)^4) -2 \rightarrow 0$ as $n\to \infty$. Keeping in mind that $\kappa_4(Q_{\bf S}(f_n))=\varphi(Q_{\bf S}(f_n)^4) -2$ is positive (see \eqref{Pos4cumFree}), the assumption $\varphi(Y^4)\geq 2$ entails that:
$$ 
\kappa_4(Q_{\Y}(f_n))=  \varphi(Q_{\Y}(f_n)^4) - 2  \geq  \varphi(Q_{\bf S}(f_n)^4) - 2,$$
and, in turn,  $\varphi(Q_{\bf S}(f_n)^4)  \rightarrow 2$. Here,  Theorem \ref{knps} applies implying that $Q_{\bf S}(f_n) \stackrel{\text{Law}}{\longrightarrow} \mathcal{S}(0,1)$, and finally Theorem \ref{invnoncom} yields the desired conclusion $Q_{\bf Y}(f_n) \stackrel{\text{Law}}{\longrightarrow} \mathcal{S}(0,1)$. 
The reverse implication, that is, $Q_{\bf Y}(f_n) \stackrel{\text{Law}}{\longrightarrow} \mathcal{S}(0,1) \Rightarrow \kappa_4(Q_{\Y}(f_n)) \rightarrow 0$, holds trivially because the convergence in law, in the free case, is exactly the convergence of all the moments (equivalently, all the cumulants).
\end{proof}

\begin{rmk}\label{mixture doesn't work}
In order to generalize,  in the free probability setting, the technique of the mixtures used in Subsection \ref{proof_mixtures}, one should consider a sequence $\{Z_i\}_{i\geq 1}$ of freely independent random variables, freely independent of $\{S_i\}_{i\geq 1 }$ in such a way that the $Z_i$'s and the $S_j$'s commute (to suitably handle the conditional expectation). But this is possible only if $Z_i$ has vanishing variance (see \cite[Lecture 5]{Speicher}). 
\end{rmk}

\subsection{The non identically distributed case}

Even in the present non-commutative framework, the choice of dealing with homogeneous sums in identically distributed entries is made just to ease the notation: indeed, the findings proved with Theorem \ref{superTeo2} admit a generalization in the case the sequence $\Y=\{Y_i\}_{i\geq 1}$ is composed of freely independent centered random variables, with unit variance, possibly non identically distributed, but the starting point would be an inequality rather than an equality.

For every $n\geq 1$, set $\beta_n = \min\limits_{i=1,\dots,n}\kappa_4(Y_i)$, and assume that there exists $\beta >0$ such that $\inf\limits_{n \geq 1}\beta_n > \beta$. 
Repeating the reasoning that led to the proof of formula \eqref{formula}, and recalling that for $j=1,\dots,d$, $\rho_j$ denotes the partition whose only 4-block is determined by the condition $j \sim 2d-j+1 \sim 2d+j \sim 4d-j+1$, it follows that:
\begin{align*}
\varphi(Q_{\Y}(f_n)^4) -2 &= \varphi(Q_{\SSw}(f_n)^4) -2 + \sum_{j=1}^d\sum_{\substack{\mathbf{i} \in [n]^{4d} \\ \mathrm{Ker}(\mathbf{i}) = \rho_j}} \kappa_4(Y_{i_j}) f_n^{\otimes 4}(\mathbf{i})  \\
 &>  \varphi(Q_{\SSw}(f_n)^4) - 2 + \beta_n \sum_{k=1}^n \sum_{h=0}^{d-1} \sum_{\substack{ \mathbf{i} \in [n]^{4(d-1)} \\ \mathrm{Ker}(\mathbf{i}) = \sigma_h}} f_n(k,\cdot)^{\otimes 4}(\mathbf{i})  \\
 &>  \varphi(Q_{\SSw}(f_n)^4) - 2 + \beta \sum_{k=1}^n  \varphi(Q_{\SSw}(f_n(k,\cdot))^4), \numberthis \label{ineq}
\end{align*}
where $\sigma_h$ is the restriction of $\rho_{h+1}$ to its pairings. In particular, the estimate in \eqref{ineq} applies whenever $\kappa_4(Y_i) >0$ for all $i\geq 1$. 

\begin{thm}\label{FreeNiid}
Let the above notations and assumptions prevail. Then, $\varphi\big(Q_{\Y}(f_n)^4\big) \longrightarrow 2$ is a necessary and sufficient condition for the convergence $Q_{\Y}(f_n) \stackrel{ \text{ Law} }{\longrightarrow } \mathcal{S}(0,1)$. Moreover, $Q_{\Y}(f_n)\stackrel{ \text{ Law} }{\longrightarrow } \mathcal{S}(0,1)$ implies $Q_{\mathbf{Z}}(f_n)\stackrel{ \text{ Law} }{\longrightarrow } \mathcal{S}(0,1)$ for every sequence $\mathbf{Z}=\{Z_i\}_{i\geq 1}$ of freely independent random variables, non necessarily i.i.d., satisfying Assumption {\bf (1)}.
\end{thm} 

\begin{proof}
Assume that $\varphi(Q_{\Y}(f_n)^4) -2   \rightarrow 0$ as $n\rightarrow \infty$.  Since $\beta >0$ and 
$\varphi(Q_{\SSw}(f_n)^4) -2 > 0$, the inequality \eqref{ineq} implies $\varphi(Q_{\SSw}(f_n)^4) -2 \longrightarrow 0$.  The conclusion then follows by applying Theorems \ref{invnoncom} and \ref{knps} of Part \ref{Invariance}.
\end{proof}

In particular, since $\varphi(U_n(S)^4) \geq 2$ for every $n\geq 1$, Theorem \ref{FreeNiid} allows us to recover Corollaries \ref{Univ1}, \ref{Univ2} for Chebyshev sums in semicircular entries.\\

\subsection{Free Poisson approximations of homogeneous sums}

Assume that $d \geq  2$ is even. By virtue of Theorem \ref{FPoissAppr} and Lemma \ref{magg}, $Q_{\SSw}(f_n) \stackrel{\text{Law}}{\longrightarrow} Z(\lambda)$ implies the asymptotic vanishing of the influence functions $\tau_n(f_n)\rightarrow 0$. Moreover, homogeneous sums $Q_{\SSw}(f_n)$ are universal at the order $d$  also as to $Z(\lambda)$-approximations: see Corollary \ref{Univ2} in Part \ref{Invariance}. \\

If $f$ is an admissible kernel as in Definition \ref{Admissible_Free},  then $\varphi\big(Q_\Y(f)^3\big) = \varphi\big(Q_\SSw(f)^3\big)$: indeed, if $\mathcal{NC}_{>0}([3d])$ denotes the set of non-crossing partitions of $[3d]$ with no singleton, then 
$$|\mathcal{NC}_{>0}^{\star}(d^{\otimes 3})| = |\mathcal{NC}_2^{\star}(d^{\otimes 3})|. $$ Therefore, assuming that  $\varphi\big(Q_\Y(f)^2\big) = \varphi\big(Q_\SSw(f)^2\big) = \lambda> 0$, from (\ref{formula}) it follows that:
\begin{align*}
\varphi(Q_\Y(f)^4)-2\varphi(Q_\Y(f)^3) - (2\lambda^2- \lambda) &= \varphi(Q_\SSw(f)^4)-2\varphi(Q_\SSw(f)^3) -(2\lambda^2- \lambda) \\
&+ \kappa_4(Y)\sum_{k=1}^{n}\varphi(Q_\SSw(k,\cdot)^4) \numberthis \label{freePoissGen}
\end{align*}

From formula \eqref{freePoissGen}, Theorem \ref{FPoissAppr} can be generalized to a Fourth Moment Theorem for homogeneous sums in freely independent copies of any centered random variable $Y$, with unit variance, and such that $\kappa_4(Y) \geq 0$, providing the analogous of Theorem \ref{superTeo2} with respect to the Free Poisson limit.

\begin{thm}
Let $d\geq 2$ be even. If $Y$ satisfies Assumption {\bf (1)} and  $\kappa_4(Y) \geq 0$, for every sequence of admissible kernels $f_n:[n]^d \rightarrow \mathbb{R}$, with $\varphi\big(Q_{\Y}(f_n)^2\big) \rightarrow \lambda$, the following statements are equivalent in the limit as $n\rightarrow \infty$:
\begin{itemize}
\item[(i)] $Q_{\Y}(f_n) \stackrel{\text{Law}}{\longrightarrow} Z(\lambda)$;
\item[(ii)] $\varphi\big(Q_{\Y}(f_n)^4\big)-2\varphi\big(Q_{\Y}(f_n)^3\big) \longrightarrow \varphi\big(Z(\lambda)^4\big)-2\varphi\big(Z(\lambda)^3\big) = 2\lambda^2 - \lambda$.
\end{itemize}
Besides, the law of $Y$ is universal for free Poisson approximations of homogeneous sums at the order $d$, that is, $Q_{\Y}(f_n) \stackrel{\text{Law}}{\longrightarrow} Z(\lambda)$ implies $Q_{\mathbf{Z}}(f_n) \stackrel{\text{Law}}{\longrightarrow} Z(\lambda)$ for every other sequence $\mathbf{Z} = \{Z_i\}_{i \geq 1}$ of freely independent random variables, satisfying Assumption {\bf (1)}.
\end{thm}

\begin{rmk}\label{Extension_Tetilla}
Let $\mathcal{T}$ denote a Tetilla distributed free random variable on a fixed $W^{\star}$-probability space.  Once a combinatorial formula for the sixth moment of a homogeneous sum $Q_{\Y}(f)$ is provided, in the spirit of formula \eqref{formula}, a similar approach could lead to a Fourth Moment type statement for the Tetilla \index{Tetilla law}approximation of homogeneous sums, extending the results in \cite[Theorem 1.1]{NourdinDeya2}, where the authors proved that, for a sequence of (mirror) symmetric kernels $f_n:[n]^d\rightarrow \mathbb{R}$, the conditions $\varphi(Q_{\bs{S}}(f_n)^6) \rightarrow \varphi(\mathcal{T}^6)$ and $\varphi(Q_{\bs{S}}(f_n)^4) \rightarrow \varphi(\mathcal{T}^4)$, are sufficient for the Tetilla approximation of the sequence $Q_{\bs{S}}(f_n)$.\\
\end{rmk}

\begin{exm}$ $
\begin{enumerate}
\item
Every random variable $Y$, satisfying Assumption {\bf (1)}, and whose law is infinitely divisible with respect to the additive free convolution, satisfies $\kappa_4(Y) = \varphi(Y^4) - 2\geq 0$. Indeed, by definition, for every integer $n\in \mathbb{N}$, there exist freely independent and identically distributed random variables $Y_{1,n},\dots,Y_{n,n}$, such that $Y \stackrel{\text{Law}}{=} Y_{1,n} + Y_{2,n} + \cdots + Y_{n,n}$, which yields $\kappa_4(Y) = n\, \kappa_4(Y_{1,n})$ due to the additivity of cumulants. Moreover, for every random variable $Z$,  satisfying Assumption {\bf (1)}, $\kappa_4(Z) = \varphi(Z^4)-2 \geq -1$ (since $\varphi(Z^4) \geq 1$). Therefore, if $\kappa_4(Y) < 0$, for $n$ large enough one would find $\kappa_4(Y) < -1$, which is impossible. Hence, every  freely infinitely divisible law satisfies the Fourth Moment Theorem (and the universality) as to semicircular and free Poisson approximations, at any order $d\geq 2$. \\

\item For $k\geq 1$, if $U_k(x)$ denotes the $k$-th Chebyshev polynomial (of the second kind) and $S \sim \mathcal{S}(0,1)$, then:
$$ \varphi[U_k(S)^4] = |\mathcal{NC}_2^{\star}(k^{\otimes 4}) | \geq 2 \; .$$
Therefore, $U_k(S)$ satisfies the Fourth Moment Theorem and is universal at any order $d\geq 2$. Note that the universality of the law of $U_k(S)$ for semicircular (and free Poisson) approximations of homogeneous sums has been also obtained in Part \ref{Invariance}.\\

\item Let $\mathcal{T}$ be a Tetilla distributed random variable\index{Tetilla law}, namely $\mathcal{T} \stackrel{ \text { Law} }{=} \frac{1}{\sqrt{2}}(S_1 S_2 + S_2 S_1)$, where $S_1, S_2$ are freely independent standard semicircular random variables. Since $\kappa_4(\mathcal{T}) = \frac{1}{2}$,  $\mathcal{T}$ satisfies both the Fourth Moment Theorem and the universality property for semicircular approximations of homogeneous sums, at any order $d\geq 2$, and for free Poisson approximations when $d$ is even (see \cite{NourdinDeya2}).\\

\item Let $X \sim \mathcal{G}_q(0,1)$, with $\mathcal{G}_q(0,1)$ denoting the $q$-Gaussian distribution \cite{BozejkoSpeicher, qbrownian}. Then, $\kappa_4(X)  = \varphi_q(X^4) - 2 = q$, and hence, if $q \in [0,1]$, $X$ satisfies the Fourth Moment Theorem and the law $\mathcal{G}_q(0,1)$ is universal (at any order $d \geq 2$) (see \cite[Theorem 3.1 and Proposition 3.2]{qbrownian} for the general Fourth Moment Theorem for integrals with respect to a $q$-Brownian motion of symmetric kernels, for non-negative values of $q$). Equivalently, the fourth moment and the universality phenomena for $X$ can be alternatively deduced from the fact that, for positive values of $q$, the $q$-Gaussian distribution is also freely infinitely divisible \cite{Lehner}.\\
\end{enumerate}
\end{exm}

\subsection{The quadratic case}

Similarly to the classical setting, the condition $\varphi(Y^4) \geq 2$ might not be the best in every dimension $d$. This is the case, for instance, when $d=2$. Indeed, the multiplication formula for Wigner stochastic integrals \eqref{MultFormulaFree} entails that formula (\ref{formula}) can be rewritten as:
$$ \varphi\big(Q_\Y(f_n)^4\big) = 2 +  \|f_n \stackrel{1}{\smallfrown} f_n \|^2 + \kappa_4(Y)\sum_{k=1}^{n} \varphi\big(Q_\SSw(f_n(k,\cdot))^4\big),$$
where:
\begin{itemize}
\item[-] $\|f_n \stackrel{1}{\smallfrown} f_n \|^2 = \sum\limits_{i,j=1}^n \bigg( \sum\limits_{k=1}^n f_n(i,k)f_n(k,j)\bigg)^2 \; ;$
\item[-] $ \varphi\big(Q_\SSw(f_n(k,\cdot))^4\big) = 2\sum\limits_{i,j=1}^n f_n(i,k)^2 f_n(k,j)^2.$
\end{itemize}
Then, the chain of inequalities:
\small
\begin{align*}
\varphi\big(Q_\Y(f_n)^4\big) &= 2 + \sum_{i,j=1}^n \bigg( \sum_{k=1}^n f_n(i,k)f_n(k,j)\bigg)^2 + 2\kappa_4(Y)\sum_{k=1}^n\sum\limits_{i,j=1}^n f_n(i,k)^2 f_n(k,j)^2 \\
&\geq 2 + \sum_{i=1}^n \bigg(\sum_{k=1}^n f_n(i,k)^2\bigg)^2 + 2 \kappa_4(Y)\sum_{k=1}^n\sum_{i,j=1}^n f_n(i,k)^2 f_n(k,j)^2 \\
&= 2 + \sum_{i=1}^n \bigg(\sum_{k=1}^n f_n(i,k)^2\bigg)^2  \big(1 + 2\kappa_4(Y)\big) \numberthis \label{ineq2},
\end{align*}
\normalsize
provides that, if $\kappa_4(Y) > -\frac{1}{2}$ (or equivalently $\varphi(Y^4) > \frac{3}{2}$), then $\varphi\big(Q_\Y(f_n)^4\big) - 2 >0$.

\begin{prop}
Let $Y$ be a random variable verifying Assumption {\bf (1)}. Then, if $\varphi(Y^4) > \frac{3}{2}$ (or, equivalently, $\kappa_4(Y) >-\frac{1}{2}$),  $Y$ satisfies the quadratic Fourth Moment Theorem. Besides, the law of $X$ is universal at the order $d=2$, for semicircular and free Poisson approximation of quadratic homogeneous sums.
\end{prop}

\begin{proof}
Given a sequence of admissible kernels $f_n$, assume that $ \varphi\big(Q_\Y(f_n)^4\big) \to 2$ as $n\rightarrow \infty$.  Then, from \eqref{ineq2}, it follows that: 
 $$\alpha_n = \sum_{i=1}^n \bigg(\sum\limits_{k=1}^n f_n(i,k)^2\bigg)^2 = \sum\limits_{k=1}^n\varphi\big(Q_\SSw(f_n(k,\cdot))^4\big)\to 0 .$$
Finally, considering the limit in equation (\ref{formula}),  it follows that $\varphi(Q_\SSw(f_n)^4)\to 2$, and then Theorems \ref{knps} and \ref{invnoncom} provide together that  $Q_\Y(f_n) \stackrel{\text{Law}}{\longrightarrow} \mathcal{S}(0,1)$.
\end{proof}

Despite the stronger sufficient condition $\varphi(Y^4) > \frac{3}{2}$ for the validity of the quadratic Fourth Moment Theorem, no inference can be fruitfully done to claim its optimality nor even its being necessary. This problem will absorb the bulk of the next chapter.

\section{Multidimensional CLT in the free setting}

In \cite[Theorem 1.3]{NouSpeiPec}, the free counterpart to the findings in \cite[Proposition 2]{PeccatiTudor} was achieved, showing that, for semicircular approximations on the Wigner Chaos, joint convergence is equivalent to componentwise convergence (see Theorem \ref{Equi_Free} in Part \ref{Invariance}).\\

Combining Theorem \ref{Multiinvariance2} (see Part \ref{Invariance}), applied for $h_i = 1$ for every $i=1,\dots,d$ and $d \geq 2$, and  Theorem \ref{ComponentJointFree}, it is possible to extend Theorem \ref{Equi_Free}  to all random variables with non-negative free kurtosis, providing the free counterpart to Theorem \ref{ComponentJoint}. 

\begin{thm}\label{ComponentJointFree}
Fix $m\geq 1$ and $d\geq 2$. Let $\bs{Y} = \{Y_{i}\}_{i\geq 1}$ be a sequence of freely independent copies of a random variable $Y$, verifying Assumption {\bf (1)} and $\varphi(Y^4)\geq 2$. For every $j=1,\dots,m$, let $Q_{\bs{Y}}(f_n^{(j)})$ be a sequence of homogeneous polynomials of degree $d$, with \linebreak$f_n^{(j)}:[n]^d \rightarrow \mathbb{R}$ symmetric admissible kernels, such that: 
$$ \lim_{n \rightarrow \infty}\varphi\big( Q_{\bs{Y}}(f_n^{(j)}) Q_{\bs{Y}}(f_n^{(i)})\big) = C_{i,j} \quad \forall i, j=1,\dots,m. $$
If $C=(C_{i,j})_{i,j=1,\dots,m}$ is a real-valued, positive definite, symmetric matrix, and $(s_1,\dots,s_m)$ denotes a semicircular system with covariance determined by $C$, the following statements are equivalent as $n\rightarrow \infty$:
\begin{itemize}
\item[(i)] $Q_{\bs{Y}}(f_n^{(j)}) \stackrel{\text{ Law }}{ \longrightarrow} s_j$ for every $j=1,\dots,m$;
\item[(ii)] $(Q_{\bs{Y}}(f_n^{(1)}),\dots,Q_{\bs{Y}}(f_n^{(m)}) ) \stackrel{\text{ Law }}{ \longrightarrow} (s_1,\dots,s_m)$.
\end{itemize}
\end{thm}

\begin{proof}
It is sufficient to prove  that $(i)$ $\Rightarrow$ $(ii)$, since the reverse implication always holds.\\
Assume that $(i)$ occurs. Under the assumption $\varphi(Y^4) \geq 2$, and by virtue of Theorem \ref{superTeo2}, $Y$ satisfies the Fourth Moment Theorem and its law is universal for semicircular approximations of homogeneous sums, at the given order $d$. In particular, $Q_{\SSw}(f_n^{(j)}) \stackrel{\text{ Law }}{ \longrightarrow} s_j$ for every $j=1,\dots,m$; besides, from Theorem \ref{invnoncom} in Part \ref{Invariance}, $\tau_n^{(j)} = \max\limits_{i=1,\dots,n}\mathrm{Inf}_i(f_n^{(j)}) \longrightarrow 0$ for every $j=1,\dots,m$. Since:
$$\varphi\big(Q_{\bs{Y}}(f_n^{(j)}) Q_{\bs{Y}}(f_n^{(i)})\big) = \varphi\big(Q_{\bs{S}}(f_n^{(j)}) Q_{\bs{S}}(f_n^{(i)})\big) \; \forall i, j =1,\dots,m$$
by virtue of Theorem \ref{Multiinvariance2} it follows that $(Q_{\bs{S}}(f_n^{(1)}),\dots,Q_{\bs{S}}(f_n^{(m)}) )$ and $(Q_{\bs{Y}}(f_n^{(1)}),\dots,Q_{\bs{Y}}(f_n^{(m)}) )$ are asymptotically close in distribution. Hence, the conclusion follows  by Theorem \ref{Equi_Free}.
\end{proof}

In \cite[Theorem 1.6]{NouSpeiPec}, the authors established the following transfer principle for the multidimensional CLT between Wiener and Wigner Chaos, here recalled only for homogeneous sums.

\begin{thm}\label{PrimoTransfer}
Let $d \geq 1$ and $m\geq 1$ be fixed integers, and let $C = (C_{i,j})_{i,j=1,\dots,m}$ be a real-valued, positive definite, symmetric matrix. For every $j=1,\dots, m$, let $f_n^{(j)}:[n]^d \rightarrow \mathbb{R}$ be an admissible kernel, and assume that, for every $i,j=1,\dots,m$:
$$ \varphi(  Q_{\bs{S}}(f_n^{(i)})  Q_{\bs{S}}(f_n^{(j)})  ) \rightarrow C_{i,j},$$
$$ \E[  Q_{\bs{N}}(f_n^{(i)})  Q_{\bs{N}}(f_n^{(j)})  ]\rightarrow d!C_{i,j},$$
where $\bs{S}$ denotes a sequence of freely independent, standard semicircular random variables, and $\bs{N}$ denotes a sequence of independent, standard Gaussian random variables. Then, if $(s_1,\dots,s_m)$ denotes a semicircular system, with covariance given by $C$, and $\mathcal{N}(0,C)$ denotes the multivariate normal distribution of covariance $C$, the following statements are equivalent as $n\rightarrow \infty$:
\begin{itemize}
\item[(i)] $(Q_{\bs{S}}(f_n^{(1)}),\dots,  Q_{\bs{S}}(f_n^{(m)}) ) \stackrel{\text{Law}}{\rightarrow} (s_1,\dots,s_m)$ 
\item[(ii)] $(Q_{\bs{N}}(f_n^{(1)}), \dots,  Q_{\bs{N}}(f_n^{(m)}) )  \stackrel{\text{Law}}{\rightarrow}  d! \mathcal{N}(0,C)$.
\end{itemize}
\end{thm}

Thanks to Theorems \ref{ComponentJoint} and \ref{ComponentJointFree}, Theorem \ref{PrimoTransfer} can be completely generalized to a transfer principle, for central convergence, between homogeneous sums $\frac{1}{\sqrt{d!}}Q_{\X}(f_n)$, with $X$ satisfying Assumption {\bf (2)} and with non-negative kurtosis, over a classical probability space, and free homogeneous sums $Q_{\Y}(f_n)$, with $Y$ satisfying Assumption {\bf (1)} and with non-negative free kurtosis, over a free probability space $(\mathcal{A},\varphi)$.

\begin{thm}\label{Transfer}
Let $X$ be a random variable (in the classical sense), satisfying Assumption {\bf (2)} and such that $\E[X^4]\geq 3$, and $Y$ be a free random variable  satisfying Assumption {\bf (1)} and $\varphi(Y^4)\geq 2$. Let $m\geq 1$, and $f_n^{(j)}:[n]^d \rightarrow \mathbb{R}$, with $d\geq 2$, be a symmetric admissible kernel as in Definition \ref{Admissible_Free} for every $j=1,\dots,m$, such that
$$ \lim_{n\rightarrow\infty}\varphi\big(Q_{\Y}(f_n^{(i)})Q_{\Y}(f_n^{(j)})\big) = \dfrac{1}{d!} \lim_{n\rightarrow\infty}\E[Q_{\X}(f_n^{(i)})Q_{\X}(f_n^{(j)})] = C_{i,j}, \; \forall i,j=1,\dots,m,$$
with $C=(C_{i,j})_{i,j=1,\dots,m}$  real-valued, positive definite, symmetric matrix. Then the following conditions are equivalent as $n \rightarrow \infty$:
\begin{itemize}
\item[(i)] $\big( \dfrac{1}{\sqrt{d!}}Q_{\X}(f_n^{(1)}),\dots,  \dfrac{1}{\sqrt{d!}}Q_{\X}(f_n^{(m)}) \big)  \stackrel{\text{Law}}{\longrightarrow} \mathcal{N}(0,C)$;
\item[(ii)] $\big( Q_{\Y}(f_n^{(1)}),\dots,  Q_{\Y}(f_n^{(m)})\big) \stackrel{\text{Law}}{\longrightarrow} (s_1,\dots,s_m)$,
\end{itemize}
with $(s_1,\dots,s_m)$ denoting a semicircular system with covariance determined by $C$.
\end{thm}

\begin{proof}
Assume first that $(i)$ holds: then, for every $j=1,\dots, m$, $\dfrac{1}{\sqrt{d!}}Q_{\X}(f_n^{(j)}) \stackrel{\text{Law}}{\longrightarrow} \mathcal{N}(0,C_{j,j})$, implying,  by virtue of Theorem \ref{superTeo1}, that $\frac{1}{\sqrt{d!}}Q_{\NN}(f_n^{(j)}) \stackrel{\text{Law}}{\longrightarrow} \mathcal{N}(0,C_{j,j})$. By virtue of Theorem \ref{TeoPeccatiTudor}, then, we have the joint convergence 
$(\frac{1}{\sqrt{d!}}Q_{\NN}(f_n^{(1)}),\dots,\frac{1}{\sqrt{d!}}Q_{\NN}(f_n^{(m)})) \stackrel{\text{Law}}{\longrightarrow} \mathcal{N}(0,C)$, which is, in turn, equivalent to the joint convergence $\big( Q_{\SSw}(f_n^{(1)}),\dots,  Q_{\SSw}(f_n^{(m)})\big) \stackrel{\text{Law}}{\longrightarrow} (s_1,\dots,s_m)$, By virtue of \cite[Theorem 1.6]{NouSpeiPec}. Finally, Theorem \ref{Equi_Free}  implies that $ Q_{\Y}(f_n^{(j)}) \stackrel{\text{Law}}{\longrightarrow} s_j$ and the conclusion follows by Theorem \ref{ComponentJointFree}.

To prove the reverse implication, start with Theorem \ref{superTeo2} and consider Theorem \ref{ComponentJoint} instead of Theorems \ref{superTeo1} and \ref{ComponentJointFree}, respectively.
\end{proof}

\begin{rmk}
As remarked in Chapter \ref{Classic}, Theorem \ref{superTeo1} does not fit the Poisson homogeneous Chaos, due to the necessity of working under the assumption $\E[X^3]=0$. In view of the Transfer principle provided with Theorem \ref{Transfer}, this failure accounts for the lack of a Transfer principle,  for central convergence, between classical and free Poisson Chaos, as highlighted with a counterexample in \cite{Solesne}.\\
\end{rmk}

\chapter{The threshold problem}\label{Threshold}

In view of Theorem \ref{superTeo1} (respectively, Theorem \ref{superTeo2} in non-commutative probability spaces), a random variable having non-negative kurtosis (customarily called \textit{leptokurtic}) satisfies the Fourth Moment Theorem and the universality principle for normal (resp. semicircular) approximations of homogeneous sums at every order $d\geq 2$. 

On the other hand, no further information about the optimality of such conditions for a fixed $d$ can be inferred from the tools so far developed, and no conclusion can be drawn about the condition being also necessary. 
More precisely, when speaking of an \textit{optimal threshold} at the order $d$, it is meant the smallest real number $r_d$ such that $\E[X^4] \geq r_d$ (resp. $\varphi(Y^4) \geq r_d$) is a necessary and sufficient condition for $Q_{\X}(f_n)$ (resp. $Q_{\Y}(f_n)$) to satisfy a CLT under the only condition that $\E[Q_{\X}(f_n)^4] \rightarrow 3$ (resp. $\varphi(Q_{\Y}(f_n)^4) \rightarrow 2$). 

The first logical step to accomplish in order to determine the threshold in every dimension $d \geq 2$ is the prove of its existence. Once this goal is achieved, several questions arise: for instance, are the thresholds increasing (namely, $r_d < r_{d+1}$ for every $d$)? If this is the case, which is their supremum? Might it be $3$? Unfortunately, so far it has been possible to establish only the existence of the optimal threshold $r_d$ in every dimension $d\geq 2$.\\

Despite the main results proved in Chapter \ref{Classic} and \ref{Free} have been reached following similar approaches, to discuss the threshold problem it will be necessary to adopt different strategies within the two settings. In particular, the result achieved in the classical probability setting is weaker than the one in the non commutative framework, in the sense that the existence of the threshold for the Fourth Moment Theorem is determined under the hypothesis of universality. On the other hand, in the classical setting it is possible to provide a \textit{dimension-free} lower bound for the thresholds $r_d$'s.

\section{Threshold in the classical setting}\label{ClassicThreshold}

Few auxiliary statements are needed for the proof of the main theorem of the section: Theorem \ref{ThresholdClassic}.

\begin{prop}\label{interpolation}
Assume that $X$  is universal and satisfies the Fourth Moment Theorem at a fixed order $d\geq 2$. Then, either $\chi_4(Q_\X(f)) < 0$ for every admissible kernel $f$, or $\chi_4(Q_\X(f)) > 0$ for every admissible kernel $f$.
\end{prop}

\begin{proof} 
The proof is divided into two steps.
\begin{itemize}
\item[Step 1:]
First, note that if $X$ satisfies both the Fourth Moment Theorem and the universality property, then $\chi_4(Q_\X(f)) \neq 0$ for every admissible kernel $f$. Indeed, if there exists $f$ such that $\chi_4(Q_{\X}(f)) = 0,$ then the constant sequence $Q_{\X}(f)$ will be normally distributed, and then, the universality of $X$  would yield $Q_{\mathbf{N}}(f) \stackrel{\text{Law}}{=} \mathcal{N}(0,1)$, which is absurd, because random variables living in Wiener Chaoses of order $d\geq 2$ cannot be normally distributed (see, \eqref{Pos4cum}, or \cite[Corollary 5.2.11]{NourdinPeccatilibro}).

\item[Step 2:]
Assume that there exist two admissible kernels $f_{0}$ and $f_{1}$ such that $\E[Q_{\mathbf{X}}(f_{0})^4]>3$ and $\E[Q_{\mathbf{X}}(f_{1})^4]<3$, and consider, for every $t \in [0,1]$, the admissible kernel
$$ 
f_t = \dfrac{ t f_1+ (1-t)f_0}{\sqrt{\E[(t Q_{\mathbf{X}}(f_{1}) + (1-t)Q_{\mathbf{X}}(f_{0}))^2] }}.
$$
Since $\chi_4\big(Q_{\mathbf{X}}(f_{1})\big) < 0$ and $\chi_4\big(Q_{\mathbf{X}}(f_{0})\big) > 0$, there exists $t^{\star} \in (0,1)$ such that \linebreak$\chi_4\big(Q_{\mathbf{X}}(f_{t^\star})\big) = 0$, which contradicts the conclusion of the first step.
To establish which case applies, it is enough to check for $Q_{\X}(f)= X_1 \cdots X_d$. \qedhere
\end{itemize}
\end{proof}

\begin{rmk}[The Rademacher Chaos]\label{Rademacher} 
In \cite[Proposition 4.6]{NourdinPeccatiReinert}\index{Rademacher law}, the authors provided the quadratic Fourth Moment Theorem when $\E[X^4]=1$, that is, for elements in the Rademacher chaos of order $2$ (the case $d\geq 3$ is still open). Nevertheless, the reader should keep in mind that Rademacher chaos is not universal (see, for instance, \cite{NourdinPeccatiReinert}), and hence such result is not in contrast with the forthcoming Theorem \ref{ThresholdClassic}. Therefore, for the present discussion, it is legitimate to exclude the case $\E[X^4]=1$, corresponding to Rademacher random variables. \\
\end{rmk}

\begin{rmk}
If $\mathbb{E}[X^4] > 1$, it is always possible to consider a homogeneous sum with positive fourth cumulant. Indeed, for $n \in \mathbb{N}$ large enough, and the fixed $d\geq 2$, set $N:=1 +(n-1)(d-1)$, and consider the homogeneous polynomial:
$$
Q_{\mathbf{X}}(g_N) = \dfrac{X_1}{\sqrt{n-1}}\big( \sum_{j=1}^{n-1}\prod_{l=2}^d X_{(j-1)(d-1)+l} \big)  = \sum_{i_1,\dots,i_d=1}^N g_N(i_1,\dots,i_d)X_{i_1}\cdots X_{i_d},$$
with 
$$ g_N(i_1,\dots,i_d) =
\dfrac{1}{d!\sqrt{n-1}} $$
 if $\{i_1,\dots,i_d\} = \{1, (j-1)(d-1)+2,\dots, (j-1)(d-1)+d\}$ for a certain $j=1,\dots,n-1,$ and $g_N(i_1,\dots,i_d)= 0$  otherwise. 
Note that $g_N$ is an admissible kernel, since  $g_N(i_1,\dots,i_d) = g_N(i_{\sigma(1)},\dots,i_{\sigma(d)})$, for every $\sigma \in \mathfrak{S}_d$, and every $\{i_1,\dots,i_d\} \subset [N]$, it is suitably normalized, and vanishes on diagonals, by definition. A direct computation, then, provides:
$$ 
\E[Q_{\mathbf{X}}(g_N)^4]=\E[X^4]\left(3+\dfrac{\E[X^4]^{d-1}-3}{n-1}\right) \underset{n \rightarrow \infty}{\longrightarrow} 3\E[X^4] > 3.
$$
\end{rmk}

\begin{prop}
Let $X$ satisfy the Fourth Moment Theorem and the universality property at the order $d\geq 2$. Then, necessarily, $\E[X^4] > \sqrt[d]{3}$.
\end{prop}

\begin{proof}
As a consequence of Proposition \ref{interpolation}, $\E[X^4] \neq \sqrt[d]{3}$ (otherwise, for $Q_\X(f) = X_1 \cdots X_d$, one would have $\chi_4(Q_{\X}(f))=0$). 
By contradiction, assume that $\E[X^4] \in (1,\sqrt[d]{3})$. Then, for $Q_\X(f) = X_1 \cdots X_d$, $\E[Q_{\X}(f)^4] < 3$, and hence  $\E[Q_{\X}(g)^4] < 3$ for any other admissible kernel $g$, which contradicts the previous remark.  In conclusion, if $\E[X^4] < \sqrt[d]{3}$, $X$ cannot satisfy the Fourth Moment Theorem.
\end{proof}

\begin{thm}
\label{ThresholdClassic}
For every $d\geq 2$,  there exists a real number $r_d\in(\sqrt[d]{3},3]$ such that, for any centered random variable $X$ satisfying Assumption {\bf (2)}, that is universal (for normal approximations of homogeneous sums, at the order $d$), the following are equivalent:
\begin{enumerate}
\item $X$ satisfies the Fourth Moment Theorem at the order $d$;
\item $\E[X^4] \geq r_d$.
\end{enumerate}
\end{thm}

\begin{proof}
Assume that $X$ satisfies the Fourth Moment Theorem. Then, as a consequence of the above discussion, $\E[Q_{\X}(f)^4] > 3$ for every admissible kernel $f$. Let $Z$ be a random variable satisfying Assumption {\bf (2)} as well as $\E[Z^4] \geq \E[X^4]$: to obtain the existence of the desired threshold $r_d$, it is enough to show that $Z$ satisfies the Fourth Moment Theorem as well. The proof involves several steps, considering mixtures between $X$ and a suitable random variable $T$.

Before starting, it is convenient to adapt Definition \ref{defcom} to sequences of homogeneous sums with kernels having non-constant normalizations. Let $f_n:[n]^d \rightarrow \mathbb{R}$ be a sequence of symmetric and vanishing on diagonals kernels, such that $\E[Q_{\mathbf{X}}(f_n)^2] = \sigma_n^2 < \infty$ for all $n\geq 1$. If $\lim\limits_{n \rightarrow \infty}\sigma_n^2 = \sigma^2  >0$, then we shall say that $X$ satisfies the $CFMT_d$ if the convergence
$$ \chi_4(Q_{\mathbf{X}}(f_n)) = \E[Q_{\mathbf{X}}(f_n)^4] - 3\E[Q_{\mathbf{X}}(f_n)^2]^2 \rightarrow 0 $$
implies, as $n \rightarrow \infty$,  $Q_{\mathbf{X}}(f_n) \stackrel{\text{Law}}{\rightarrow} \mathcal{N}(0,\sigma^2)$.
In particular, Proposition \ref{interpolation} still holds when dropping the unit normalization for admissible kernels. 

\begin{itemize}
\item[Step 1:] Consider a random variable $T$, independent of $X$, with values in $[x,\infty[$ for some $x>0$ and with $\E[T^2]=1$ and $\E[T^4]=\E[Z^4]/\E[X^4]$ (the existence of $T$ is ensured by Lemma \ref{ExistenceT}). Let $\bs{X}=\{X_i\}_{i\geq 1}$, $\bs{T}=\{T_i\}_{i\geq 1}$, $\bs{Z}=\{Z_i\}_{i\geq 1}$ be  sequences of independent copies of $X$, $T$ and $Z$, respectively, such that $\bs{X}, \bs{T}$ and $\bs{Z}$ are independent between each others.\\

\item[Step 2:] Set $Q_{\TX}(f_n)=\sum\limits_{i_1,\cdots,i_d=1}^nf_n(i_1,\cdots,i_d)(T_{i_1}X_{i_1})\cdots(T_{i_d}X_{i_d})$. From Step 1 it follows that $\E[Q_{\TX}(f_n)^2]=1$ and
$\E[Q_{\TX}(f_n)^4]=\E[Q_{\mathbf{Z}}(f_n)^4]$. Moreover, one can write:
\begin{align*}
\E[Q_{\mathbf{Z}}(f_n)^4]- 3 &= \E\big[\E[Q_{\TX}(f_n)^4| \bs{T}]-3\big ]  \\ 
&=\E\big[\E[Q_{\TX}(f_n)^4 | \bs{T}] - 6\E[Q_{\TX}(f_n)^2|\bs{T}] +3 \big] \\ 
&= \E\big[\big(\E[Q_{\TX}(f_n)^4|\bs{T}]-3\E[Q_{\TX}(f_n)^2|\bs{T}]^2\big )+3\big(\E[Q_{\TX}(f_n)^2| \bs{T}]-1 \big)^2\big] \numberthis \label{eq-T-N-2}
\end{align*}
As already underlined, since $X$ satisfies the Fourth Moment Theorem, $\chi_4(Q_{\TX}(g_n)) > 0$ for every sequence of admissible kernels: in particular,   almost surely in $\{T_i\}_{i\ge 1}$ and due to the independence between $\bs{T}$ and $\bs{X}$, it holds true that: 
$$\chi_4(\E[Q_{\TX}(f_n)|\bs{T}]) = \E[Q_{\TX}(f_n)^4|\bs{T}]-3\E[Q_{\TX}(f_n)^2|\bs{T}]^2 \, >0.$$  
Therefore, from \ref{eq-T-N-2}, if $\E[Q_{\mathbf{Z}}(f_n)^4]\to 3$ as $n \rightarrow \infty$ and up to extracting a subsequence,
almost surely in $\{T_i\}_{i\ge1}$ it holds true that:
\begin{eqnarray}
\E[Q_{\TX}(f_n)^2|\bs{T}]& \underset{n \rightarrow \infty}{\longrightarrow}   & 1,\notag\\
\E[Q_{\TX}(f_n)^4|\bs{T}]&\underset{n \rightarrow \infty}{\longrightarrow}& 3.\label{ymca}
\end{eqnarray}
\item[Step 3:] Since $X$  satisfies the Fourth Moment Theorem, from (\ref{ymca}) it follows that, $\bs{T}$-a.s. 
\begin{eqnarray*}
Q_{\TX}(f_n)&\xrightarrow[n\to\infty]{\text{Law}}&\mathcal{N}(0,1).
\end{eqnarray*}
Since $X$ is assumed to be universal, by Theorem \ref{inv}, it follows that, almost surely in $\{T_i\}_{i\ge1}$,
\begin{eqnarray*}
\max_{i_1=1,\dots,n}\sum_{i_2,\cdots,i_d=1}^n f_n(i_1,i_2,\cdots,i_d)^2 (T_{i_1}\cdots T_{i_d})^2 &\xrightarrow[n\to\infty]{\text{}}&0.
\end{eqnarray*}
Finally, being $T_i\geq x>0$ for all $i$, one has:
$$
\tau_n:= \max_{i_1=1,\dots,n}\sum_{i_2,\dots,i_d=1}^n f_n(i_1,\dots,i_d)^2  \xrightarrow[n\to\infty]{\text{}}0.
$$
Then, as in the proof given in Subsection \ref{proof_mixtures},  de Jong's Criterion, ensures that $Z$ satisfies the Fourth Moment Theorem as well. Besides, the law of $Z$ is universal.\\

\item[Step 4:] In conclusion, the desired threshold $r_d$ is given as 
the smallest real number $t\in (\sqrt[d]{3},3]$ such that there exists $X$ satisfying Assumption {\bf (2)}, the Fourth Moment Theorem and the universality property at the order $d$, and such that $\E[X^4]=t$.\qedhere
\end{itemize}
\end{proof}

\begin{rmk}
To relax Assumption {\bf (2)}, it would be necessary to prove the existence of a random variable $T$ such $\E[Z^{i}] = \E[T^i]\E[X^{i}]$ for $i=1,2,3,4$, whatever are the third moments of $Z$ and of $X$ (in this regard, see Remark \ref{relaxMom3}).
\end{rmk}

\section{Threshold in the free setting}\label{FreeThreshold}

The problem under consideration in the present section is the free counterpart to the questions analysed in Section \ref{ClassicThreshold}.

Compared to the commutative setting, here the linearity in the fourth cumulant of formula \eqref{formula} allows a simpler argument to derive the existence of the threshold for the fourth moment, with the advantage that there will be no need in putting the extra assumption on the universality of $Y$. On the other hand, since the technique of the mixtures would be trivial in this setting (see Remark \ref{mixture doesn't work}), a different approach would be nevertheless required to reach the free counterpart of Theorem \ref{ThresholdClassic}. 
Conversely, since formula \eqref{formulaGauss} is not linear in $\chi_4(X)$, the following strategy cannot be adapted for the determination of the threshold in the classical case.

\begin{rmk}
For every sequence $\Y = \{Y_i\}_{i\geq 1}$ of identically distributed, freely independent random variables satisfying Assumption {\bf (1)}, if $d\geq 2$, it is always possible to exhibit a sequence $Q_{\Y}(g_n)$, with $g_n$ admissible kernel according to Definition \ref{Admissible_Free}, and with strictly positive fourth cumulant. Indeed, for every $n \in \mathbb{N}$ and every $i=1,\dots,d$, set:
$$ Z_{n}^{(i)} = \dfrac{1}{\sqrt{n}}\sum_{j=1}^n Y_{(i-1)d + j}, $$
and consider the homogeneous sum (with admissible kernel $g_n$, say):
$$ Q_{\Y}(g_n)= \dfrac{1}{d!} \sum_{ \sigma \in \Sigma_d } Z_n^{(\sigma(1))}\cdots Z_n^{(\sigma(d))} .$$
Then, according to the free CLT (see, for instance, \cite[Theorem 8.10]{Speicher})  $Z_n^{(i)} \underset{n \rightarrow \infty}{\stackrel{\text{Law}}{\longrightarrow}} S^{(i)} \sim \mathcal{S}(0,1)$ for every $i=1,\dots,d$, with the $S^{(i)}$'s freely independent: finally, the multidimensional CLT (see \cite[Theorem 8.17]{Speicher}) assures that $ Q_{\Y}(g_n)$ converges in law to an element in the $d$-th Wigner Chaos.
\end{rmk}

\begin{thm}
\label{ThresholdFree}
For any $d\ge2$, there exists $s_d\in ]1,2]$ such that, for every  random variable $Y \in \mathcal{A}$, with $\varphi(Y)=0, \varphi(Y^2)=1,$ the following statements are equivalent:
\begin{itemize}
\item[(i)] $Y$ satisfies the Fourth Moment Theorem at the order $d$;
\item[(ii)]$\varphi(Y^4) \geq s_d$.
\end{itemize}
\end{thm}

\begin{proof}$ $
The proof is divided into 3 steps.
\begin{itemize}
\item[Step 1:] Let $t\in ]1,2]$ be such that for all $Y\in\mathcal{A}$ satisfying Assumption {\bf (1)} the following implication holds
$$\Big(\varphi(Y^4)=t\Big) ~~~\Rightarrow~~~~ \text{$Y$ satisfies the Fourth Moment Theorem}$$
(Theorem \ref{superTeo2} hints the existence of such $t$: indeed, at the worst, $t=2$, but, for instance, when $d=2$, $t$ can be chosen to be any real number in $(\frac{3}{2},2]$).

Then, $\kappa_4(Q_{\Y}(f))$ is constant in sign, no matter the admissible kernel $f$: indeed, assume that there exists an admissible kernel $f$ such that  $\kappa_4(Q_{\Y}(f))<0$, and consider an admissible kernel $g$ such that $\kappa_4(Q_{\Y}(g))>0$. For every $r \in [0,1]$, consider the homogeneous sum:
$$ Q_{\Y}(f_r) =   \dfrac{r Q_{\Y}(f) + (1-r)Q_{\Y}(g)}{\sqrt{\varphi\big( (\; r Q_{\Y}(f) + (1-r)Q_{\Y}(g) \;)^2\big)}  },  $$
such that $f_0= g$ and $f_1 =f$. Then, as for the proof of Proposition \ref{interpolation}, there exists $r^{\star} \in (0,1)$ such that the admissible kernel $h:=f_{r^{\star}}$ verifies $\kappa_4(Q_{\Y}(h))=0$. It follows that, for every free random variable $Z$, centered and with unit variance, such that $\ph(Z^4)=\ph(Y^4)=t$,  $\kappa_4(Q_{\Z}(h))=0$ as well. Since $Z$ satisfies the Fourth Moment Theorem, $\kappa_4(Q_{\Z}(h))=0$ entails that $\kappa_{p}(Q_{\Z}(h))=0$ for any $p\ge 3$.

Consider, then,  the set $\mathcal{L}_t$ of random variables $Z \in \mathcal{A}$, with $\varphi(Z)=0$, $\varphi(Z^2)=1$ and $\varphi(Z^4)=t$, and set:
\begin{eqnarray*}
E_t&=&\Big\{(a,b,c)\in \mathbb{R}^3\,\Big|\, \exists Z\in\mathcal{L}_t\,:\,\ph(Z^3)=a,\ph(Z^5)=b,\ph(Z^6)=c\Big\}\\
&=&\Big\{(\ph(Z^3),\ph(Z^5),\ph(Z^6))\,\Big|\,Z\in \mathcal{L}_t\}.
\end{eqnarray*}

Since $\kappa_6(Q_{\Z}(h))=0$ for every $Z \in \mathcal{L}_t$, $E_t$ has zero Lebesgue measure: indeed, by expanding $\kappa_6(Q_{\Z}(h))=0$ as a multivariate polynomial $P_t$ in $\ph(Z^3),\ph(Z^5),\ph(Z^6)$, it turns out that $E_t\subset \{(a,b,c)|P_t(a,b,c)=0\}$. 
On the other hand, since the criterion of solvability of the Hamburger's moment problem is a necessary condition for the solvability of the Hausdorff's moment problem\index{Moment problem}, if  $(a,b,c)\in E_t$, then in particular the Hankel matrix $M_t=(\ph(Z^{i+j}))_{0\le i,j \le 3}$ is positive definite (see \cite[Theorem 6.1]{Chihara}). However, the set of triplets $(\ph(Z^3),\ph(Z^5),\ph(Z^6))$ such that $M_t$ is positive definite is a non-empty open subset of $\mathbb{R}^3$ and has then positive Lebesgue measure (indeed, since $M_t$ is positive definite, all the upper-left minors are strictly positive\footnote{This characterization is sometimes referred to as Sylvester's criterion for the positive definiteness of matrices.}, that is:
$$
det(M_t) =
\begin{array}{|cccc|}
1 & 0 & 1 &a \\
0 & 1 & a & t\\
1 & a & t & b \\
a & t & b & c
\end{array}
>0
$$
Trivially, $F_t:=det(M_t)$ is a continuous function of $a, b, c$, and hence $F_t^{-1}((0,\infty))$ is an open set of matrices with respect to the metric induced by a given matrix norm (or precisely, the topology of the balls of such metrics). For instance, for the norm $\|\cdot \|_1$ for $N$-dimensional matrices, one has
$$ d(A,B) = \|A-B \|_1 := \max_{j=1,\dots,N}\sum_{i=1}^N |a_{i,j} - b_{i,j} |,$$
and hence if $b_{i,j}$ is in a neighbourhood of $a_{i,j}$ for all $i,j=1,\dots,N$, then $B$ is in a neighbourhood of $A$. Therefore, the open subset $F_t^{-1}((0,\infty))$ of the positive definite matrices can be obtained by considering matrices whose entries are in open neighbourhoods of the entries of $M_t$. Since such set of entries is  an open subset of $\mathbb{R}^3$, it should  have positive Lebesgue measure. By contradiction, it follows that the existence of the admissible kernel $f$ is impossible. So, for any admissible kernel $g$, $\kappa_4(Q_{\Y}(g)) > 0$.

\item[Step 2:] From the first step of the proof, $\kappa_4(Q_{\Y}(f_n)) >0$ for every sequence of admissible kernels $f_n$. Then, every free random variable $Z$ satisfying Assumption {\bf (1)} and $\ph(Z^4)>t$, satisfies the Fourth Moment Theorem as well: indeed, by applying the formula (\ref{formula}) to $Y$ and $Z$ and by taking the difference, one can write:
\begin{equation}
\label{equazione}
\kappa_4(Q_{\Z}(f_n))=\kappa_4(Q_{\Y}(f_n))+(\ph(Z^4)-t)\sum_{k=1}^{n} \varphi\big(Q_\SSw(f_n(k,\cdot))^4\big).
\end{equation}
Then, if $\kappa_4(Q_{\Z}(f_n)) \longrightarrow 0$ as $n\rightarrow \infty$, from equation \eqref{equazione} it follows that  $$\sum_{k=1}^n \varphi\big(Q_{\SSw}(f_n(k,\cdot))^{4}\big) \rightarrow 0 $$
and, in turn, from the formula (\ref{formula}) written for $Z$,  $Q_{\SSw}(f_n) \stackrel{\text{Law}}{\longrightarrow} \mathcal{S}(0,1)$ as $n\rightarrow \infty$; finally, Theorem \ref{invnoncom} completes the proof. Besides, the law of $Z$ is universal at the order $d$.

\item[Step 3:] In conclusion, the desired threshold $s_d$ is given as the smallest real number $t\in ]1,2]$ such that there exists $Y$ satisfying Assumption {\bf (1)}, $\varphi(Y^4)=t$, and the Fourth Moment Theorem at the order $d$.\qedhere
\end{itemize}
\end{proof}

\begin{rmk}
In principle, the first steps of the proof of Theorem \ref{ThresholdFree} could be adapted in the classical setting: the problem would arise in drawing the final conclusion, since the non-linearity in the cumulant of formula \eqref{formulaGauss} could not afford to proceed as from equation \eqref{equazione}.
\end{rmk}

%
%

\thispagestyle{empty}
\part{Invariants and semi-invariants: from orthogonal polynomials to cumulants}\label{Orth}

\chapter*{Synopsis}
\addcontentsline{toc}{chapter}{Synopsis}

In the modern probability scenario concerning stochastic integration, a prominent role is played by the so called \textit{multiplication formulae} for the products of multiple integrals. As a consequence of the orthogonality (isometry) property enjoyed by these random objects (see \cite[Proposition 5.5.3]{PeccatiTaqqu}, or \cite[Proposition 2.7.5]{NourdinPeccatilibro} for the Gaussian setting), diagram formulae for the moments of multiple integrals can be derived \cite[Theorem 7.1.3]{PeccatiTaqqu}.\\ 

Orthogonal polynomials are the gist of several other pages of stochastic analysis \cite{Schoutens} and random matrix theory \cite{EdeRao,Kuijlaars,Mehta}, as well as of the combinatorial theory of symmetric functions \cite{Macdonald2}. In other words, orthogonal polynomials are far from being an outdated mathematical subject. \\

The work that lie behind the contents here presented aimed at recasting the theory of orthogonal polynomials in a unified algebraic framework: so far, the most suitable to accomplish this goal appears to be the \textit{symbolic method of invariant theory} (for binary forms), through apolarity, as developed in \cite{KungRota}. Even if it is doubtless  not surprising that orthogonality can be settled in terms of apolarity, here the details of such intuition are defined, providing, among other results, explicit formulae for \textit{generalized orthogonal polynomials}  (equivalently, the apolar form of a given binary form) and for the moments of the so called \textit{random discriminants} \cite{Lu}. \\

The main contributions can be summarized as follows:
\begin{enumerate}
\item generalized orthogonal polynomials are triangular arrays of polynomials satisfying partial orthogonality properties: indeed, they have already been considered in the literature under the name of \textit{partial orthogonal polynomials} \cite{PartOrth}.  In Chapter \ref{Chapter_OPS}, an  algebraic representation and a determinantal formula for generalized orthogonal polynomials associated with a probability distribution are given, both in a univariate (Theorem \ref{GOPS1}) and in a multivariate setting (Theorems \ref{detformulti} and \ref{morthpol}). These formulae are consistent with the corresponding representations for orthogonal and biorthogonal polynomials \cite{IserNor}, both of which are encoded in a sequence of generalized orthogonal polynomials.

\item The choice of dealing with orthogonality in a separate chapter is made to introduce in a perhaps more reader-friendly way the topic under consideration: apolarity, which is at the core of Chapter \ref{Chapter_WhatIsOrtho}. The starting point is the definition of a family of covariants (see identity \eqref{eq:main1}) that allows to show that generalized orthogonal systems can be naturally embedded in the invariant theory of binary forms via apolarity (in the sense specified in Theorem \ref{Th:main1}). As a matter of fact, explicit determinantal formulae for these covariants are provided (Theorem \ref{Th:det}), corresponding to the representations for generalized orthogonal polynomials given in Chapter \ref{Chapter_OPS}: the transition between covariants  and orthogonal polynomials is then explicitly described in Theorem \ref{ApolarityOrth}. The framework so set allows to derive an immediate application to probability theory, since two explicit formulae for the moments of the statistics usually called \textit{random discriminants} are given, for simple random samples drawn from any distribution (see \cite{Lu}). These formulae are achieved via a suitable multivariate extension of Sylvester's Theorem and involve the so called Christoffel's numbers (see Corollary \ref{MomDiscriminant2} and Theorem \ref{MomDiscriminant}).  
Finally, in the last section, apolarity and invariant theory are discussed in a general multivariable setting. Even if the proofs and the whole presentation will proceed analogously to the first sections, the choice of dealing separately with the two settings is mainly due to the necessity of highlighting some important differences, arising from the fact that there is no standard way of defining orthogonality nor apolarity in several variables. However, with Theorem \ref{MultiOrth}, what appears to be the most suitable and natural definition is discussed. 


\item Last, Chapter \ref{Chapter_Diagonal} deals with the most prominent example of semi-invariants in probability and statistics: the cumulants, analysed with the tools deriving from the combinatorial approach to stochastic integration initiated in \cite{Rota1}. The starting point is the representation of cumulants as the expectation of the so-called diagonal measures. This approach turns out to be particularly suitable to manage cumulants of L\'{e}vy processes and of their process of variations. Moreover, this setting for diagonal measures allows to provide a measure-theoretical description of $\kappa$-statistics and polykays for positive random measures, in both the classical and the free setting.\\
\end{enumerate}

The findings exposed in the present part are essentially based from the references \cite{PasqualeSenatoSimone} and \cite{SimonePreprint}.

\section*{Bibliographic comments}

For a survey on the classical invariant theory, as set up by to Gordan, Clebsch, Capelli, Hodge, Igusa, and many others, see \cite{KraftProcesi} or \cite{Janson_Invariant}.\\

The symbolic method of invariant theory was actually born with the pioneering work of  Grace and Young, back to the early 900's \cite{GraceYoung}, but it was resettled and organized in the eighties, in the main reference \cite{KungRota}, through the language and the techniques of the \textit{umbral calculus} as developed in \cite{RotTay}. More recent outcomes of the umbral methods can be found in \cite{DiNSen}. For a survey on the use of umbral calculus in invariant theory, see \cite{Brini1}; a general outline of the contributions of Gian-Carlo Rota in invariant theory can be found in \cite{Grosshans} or \cite{Rota_TurningPoints}.\\

Other methods other than the symbolic one have been exploited to study and develop the classical invariant theory: for instance the combinatorics of Young tableaux \cite{DRS} and superalgebras \cite{Brini2}.  Back to the seventies, a characteristic free approach to the invariants of classical groups has been provided \cite{DeCProc}. For a focus on apolarity and its applications, see \cite{EreRota,SB}.\\

As to classical orthogonal polynomials in one variable, standard references include \cite{Chihara, Szego, Ism, Totik}. Orthogonal polynomials (OPs, for short) have a long history: nevertheless they are still one of the mainstream subjects in modern research areas, as they contribute to several applications and different topics, including  moments problems \cite{Chihara2}, random matrix theory \cite{Konig}, and stochastic integration: for instance, see \cite{Schoutens} for a comprehensive introduction to the subject, as well as \cite{Sole}, and \cite{Ans3} for a combinatorial interpretation of the \textit{linearization coefficients} of some classical OPs via stochastic processes. However, when speaking about orthogonality, it is important to specify if one refers to orthogonal polynomials in the classical sense, as several generalizations have been brought to life to answer to specific needs: see, for instance, \cite{IserNor, Chihara3, Kuijlaars}. See moreover \cite{AnshelevichNonCom} for a non-commutative counterpart, as well as \cite{EdeRao} and \cite{Kuijlaars} for connections between random matrix theory and the theory of multiple orthogonal polynomials.\\

As underlined throughout the whole part, there is no standard agree in the definition of multivariate orthogonal polynomials: some references on the topic  are \cite{Edelman,Xu}, and \cite{Withers} for Hermite polynomials. \\

The classical problem of decomposing a binary form of degree $n$ into a sum of $n$-th powers of linear forms usually goes under the name of \textit{Waring's problem}, originating in number theory. See \cite{WaringsPro1, WaringsPro2} for different techniques of solving the Waring's problem. In the present essay, for binary forms of odd degree, the solution is achieved via Sylvester's Theorem and to a suitable multivariable version (see Theorem \ref{GenQuadFormula}). See also \cite{Lyubich} for a  $n$-reducibility criterion for complex-valued Borel measures, related to the solvability criterion for the complex moment problem.\\
 
As to random discriminants, the reference \cite{Lu} provides also a short overview of the different ways of determining its distribution, other than surveying its applications: among others, the squared Vandermonde often occurs together with Jack symmetric polynomials within  random matrix theory \cite{EdeRao,Mehta}.\\

Since \cite{Rota1}, the combinatorial theory of stochastic integration has been deeply investigated and developed: a very comprehensive survey on the subject is the book by G. Peccati and M.S. Taqqu \cite{PeccatiTaqqu}, while some interesting works on related topics are those by J.L. Solé et al. (see \cite{Sole2} for instance). See \cite{Ans1,Ans2,Ans3} for the non-commutative probability setting.\\

Standard references about $\kappa$-statistics and polykays include \cite{Speed2,Speed3,Tukey}; recently, a very fast algorithm for the computation of such estimators and their generalizations has been given in \cite{Senato1} using umbral methods.


\chapter{Algebraic representation for orthogonal polynomials}\label{Chapter_OPS}

Throughout the present chapter, $(\Omega, \mathcal{F},\mathbb{P})$ will denote a fixed probability space (in the classical sense), and $\E$ the associated expectation. As usual, $\C$, $\N$ and $\mathbb{N}_0$ \glossary{name={$\mathbb{N}_0$}, description={Set of non-negative integers}} will denote the field of complex numbers, the set of  positive integers, and the set of non-negative integers, respectively.\\

\section{Generalized OPs}
In the sequel, let $X_0$ be a real random variable, with finite moments of every order, whose law is determined by its moments. Consider the linear functional $\EE:\C[x_0]\rightarrow \C$ such that $\EE[x_0^k] = \E[X_0^k]$ for every $k\in \mathbb{N}_0$, and a \textit{triangular array} of polynomials $\{p_{nm}(x_0)\}_{n,m\geq 1}=\{p_{nm}(x_0)\,|\,m,n\in \mathbb{N},m=1,\dots,n\}$,  satisfying $\deg\,p_{nm}(x_0)=n$ for every $m=1,\dots, n$. 


\begin{defn}
The triangular array $\{p_{nm}(x_0)\}_{n,m\geq 1}$ is called a \textbf{generalized orthogonal polynomial system} \index{GOPs}(GOPs, for short) for $X_0$ (or equivalently, for $\EE$) if and only if, for every $n \in \mathbb{N}$, and every $ m\leq n$,
\begin{equation}\label{genOPS}
\E[X_0^k \,p_{nm}(X_0)]=0 \qquad \forall \, k= 0,\dots,n-m,
\end{equation}
and $\E[X_0^{n-m+1} \,p_{nm}(X_0)] \neq 0$.
\end{defn}

\begin{rmk}
Note that the assumption that the law of $X$ is determined by its moments ensures that, if a GOPs $\{p_{nm}(x_0)\}_{n,m\geq 1}$ exists fo $X$, one could refer to $\{p_{nm}(x_0)\}_{n,m\geq 1}$ as a GOPs for the law of $X$, so that \eqref{genOPS} holds true whenever $X_0$ is replaced by any other random variable with the same moments as $X$.\\
\end{rmk}

If $\C[x_0]_{\leq d}$ \glossary{name={$\C[x_0]_{\leq d}$},description={Polynomials in $\C[x_o]$ of degree at most $d$}}denotes the subspace of $\C[x_0]$ consisting of the polynomials of degree at most $d$, then \eqref{genOPS} is equivalent to:
\begin{equation}
\E[q(X_0)p_{nm}(X_0)]=0 \quad \forall \, q(x_0)\in\C[x_0]_{\leq n-m}. 
\end{equation}
Observe that, for a fixed $p_0(x_0) \in \C$, if a generalized orthogonal polynomial system exists for $X_0$, the sequence $\{p_n(x_0)\}_{n \geq 0}$, with $p_{n}(x_0):=p_{n1}(x_0)$ for every $n \geq 1$, is an orthogonal polynomial system for $\EE$ in the classical sense \cite{Chihara} (OPs, for short)\index{GOPs!Orthogonal polynomials}, that is:
$$ \EE[p_k(x_0)p_n(x_0)] = 0  \quad \forall\, k\neq n, \, \text{and } \EE[p_n(x_0)^2]\neq 0.$$
Similarly, the sequence $q_n(x_0) := p_{n2}(x_0)$ is a \textit{quasi-orthogonal} polynomial sequence in the sense of \cite{Chihara3}, while $r_{n}(x_0):=p_{nn}(x_0)$ reduces to the biorthogonal polynomials \index{GOPs!Biorthogonal polynomials} introduced in \cite{IserNor} (in the sequel, BOPs for short). Generalized orthogonal polynomials are the topic of investigation in \cite{PartOrth}, where, for a fixed $m\geq 1$, the polynomials $p_{nm}(x_0)$, for $n\geq 1$, are called \textit{partially orthogonal polynomials of deficiency} $m$: in particular, the focus is on recursion relations and on examples of GOPs.\\

Assume that a sequence $\{X_i\}_{i \geq 0}$  of independent (not necessarily identically distributed) random variables, on the fixed probability space is given, whose elements have finite moments of every order: then, in particular, for every $k \in \mathbb{N}$, for every integers $i_1,\dots,i_k$ with $i_s \neq i_j$, $j\neq s$, and non-negative integers $l_1,\dots,l_k$, $\E[X_{i_1}^{l_1}\cdots X_{i_k}^{l_k}] = \E[X_i^{l_1}]\cdots \E[X_{i_k}^{l_k}]$.   If $\E_{0}[\cdot] = \E[\cdot | X_0]$ denotes the conditional expectation with respect to $X:=X_0$, the independence assumption yields  that $\E_{0}[X_0^{l_0} X_1^{l_1}\cdots X_n^{l_n}] = X_0^{l_0} \E[X_1^{l_1}\cdots X_n^{l_n}]$. 
Next statement aims to provide a determinantal formula for a GOPs for $\EE$ (equivalently, for $X_0$), that corresponds to the well-known determinantal expression for the OPs associated with $X_0$ when $m=1$ (see \cite[Exercise 3.1]{Chihara}), and to the one for BOPs when $m=n$ (see \cite{IserNor}).

\begin{thm}
\label{GOPS1}$ $
\begin{itemize}
\item[(i)] Let $X_0$ be a centered random variable, with finite moments of every order, say $a_j = \E[X_0^j]$ for $j\geq 1$, and set $a_0=1$. For every $n\geq 1$, assume that $X_1,\dots,X_{n}$ are independent random variables, not identically distributed, that are, in turn, independent of $X_0$. If $a_{jk} = \E[X_j^k]$ for $j= 1,\dots,n$,  for every $n\geq 1$ and every $m=1,\dots, n$, the polynomial sequence defined via
\begin{equation}
\label{detp}
p_{nm}(x_0)=\begin{vmatrix}
1&x_0&x_0^2&\ldots&x_0^n\\
a_{0}&a_{1}&a_{2}&\ldots&a_{n}\\
a_{1}&a_{2}&a_{3}&\ldots&a_{n+1}\\
\vdots&\vdots&\vdots&&\vdots\\
a_{n-m}&a_{n-m+1}&a_{n-m+2}&\ldots&a_{2n-m}\\
a_{2\,0}&a_{2\,1}&a_{2\,2}&\ldots&a_{2\,n}\\
\vdots&\vdots&\vdots&&\vdots\\
a_{m\,0}&a_{m\,1}&a_{m\,2}&\ldots&a_{m\,n}\\
\end{vmatrix},
\end{equation}
is a generalized orthogonal polynomial system for $X_0$, provided that $\deg\,p_{nm}=n$ for every $m=1,\dots, n$.
\item[(ii)] Let $\Delta(x_0,x_1,\dots,x_m)=\prod\limits_{0\leq i < j \leq m}(x_j-x_i)$ denote the Vandermonde polynomial, and assume that $X_1,\dots,X_n$ are independent random variables, independent of $X_0$, such that at least $X_1,\dots,X_{n-m+1}$ are identically distributed with $X_0$. Then, the random variable defined via:
\begin{equation}
\label{symbp}
p_{nm}(X_0)=\E_{0}[\Delta(X_{1},X_{2},\ldots,X_{n-m+1})\Delta(X_0,X_{1},\ldots,X_{n})],
\end{equation}
satisfies \eqref{genOPS}, provided that $\deg\,p_{nm}=n$ for all $m=1,\dots, n$.
\end{itemize}
\end{thm}

\begin{proof}
First, assume that the moments $a_{ij}$'s are such that $\deg\,p_{nm}=n$, namely, that for every $n$ and every $m\leq n$,
$$
\begin{array}{|c c c c c|}
a_0 & a_1 & \dots & \dots & a_{n-1} \\
a_1 & a_2 & \dots & \dots & a_{n} \\
\vdots &  & & & \vdots \\
a_{n-m} & a_{n-m+1} & \dots & \dots &a_{2n-m-1} \\
a_{2\,0} & a_{2\,1} & \dots & \dots &  a_{2\,n-1} \\
\vdots &  & & & \vdots \\
a_{m\,0} & a_{2n-m\,1} & \dots & \dots &  a_{m\,n-1} \\
\end{array}
 \neq 0\; .
 $$
In this case, for every $k=0,\dots,n-m$,
\small{
\begin{align*}
&\E[X_0^k p_{nm}(X_0)] = \\
&\,= \E X_0^k \begin{vmatrix}
1& X_0 & X_0^2 & \ldots & X_0^n\\
a_{0}&a_{1}&a_{2}&\ldots&a_{n}\\
a_{1}&a_{2}&a_{3}&\ldots&a_{n+1}\\
\vdots&\vdots&\vdots&&\vdots\\
a_{n-m}&a_{n-m+1}&a_{n-m+2}&\ldots&a_{2n-m}\\
a_{2\,0}&a_{2\,1}&a_{2\,2}&\ldots&a_{2\,n}\\
\vdots&\vdots&\vdots&&\vdots\\
a_{m\,0}&a_{2n-m\,1}&a_{2n-m\,2}&\ldots&a_{m\,n}\\
\end{vmatrix} \quad\,= \E  \begin{vmatrix}
X_0^k & X_0^{k+1 }& X_0^{k+2}&\ldots& X_0^{n+k}\\
a_{0}&a_{1}&a_{2}&\ldots&a_{n}\\
a_{1}&a_{2}&a_{3}&\ldots&a_{n+1}\\
\vdots&\vdots&\vdots&&\vdots\\
a_{n-m}&a_{n-m+1}&a_{n-m+2}&\ldots&a_{2n-m}\\
a_{2\,0}&a_{2\,1}&a_{2\,2}&\ldots&a_{2\,n}\\
\vdots&\vdots&\vdots&&\vdots\\
a_{m\,0}&a_{2n-m\,1}&a_{2n-m\,2}&\ldots&a_{m\,n}\\
\end{vmatrix}
 \\
&\quad \qquad= \begin{vmatrix}
a_k & a_{k+1}& a_{k+2}&\ldots& a_{n+k}\\
a_{0}&a_{1}&a_{2}&\ldots&a_{n}\\
\vdots&\vdots&\vdots&&\vdots\\
a_{k}&a_{k+1}&a_{k+2}&\ldots&a_{n+k}\\
\vdots&\vdots&\vdots&&\vdots\\
a_{n-m}&a_{n-m+1}&a_{n-m+2}&\ldots&a_{2n-m}\\
a_{2\,0}&a_{2\,1}&a_{2\,2}&\ldots&a_{2\,n}\\
\vdots&\vdots&\vdots&&\vdots\\
a_{m\,0}&a_{2n-m\,1}&a_{2n-m\,2}&\ldots&a_{m\,n}\\
\end{vmatrix} \, = 0\,
\end{align*}
}
\normalsize
(since two rows are equal in the determinant). Moreover, $\E[X_0^{n-m+1}p_{n,m}(X^0)] \neq 0$. Indeed: 
\small{
\begin{align*}
&\E[X_0^{n-m+1} p_{nm}(X_0)] = \\
&\,= \E  \begin{vmatrix}
X_0^{n-m+1}& X_0^{n-m+2} & \ldots & \ldots & X_0^{2n-m+1}\\
a_{0}&a_{1}&a_{2}&\ldots&a_{n}\\
a_{1}&a_{2}&a_{3}&\ldots&a_{n+1}\\
\vdots&\vdots&\vdots&&\vdots\\
a_{n-m}&a_{n-m+1}&a_{n-m+2}&\ldots&a_{2n-m}\\
a_{2\,0}&a_{2\,1}&a_{2\,2}&\ldots&a_{2\,n}\\
\vdots&\vdots&\vdots&&\vdots\\
a_{m\,0}&a_{2n-m\,1}&a_{2n-m\,2}&\ldots&a_{m\,n}\\
\end{vmatrix} \quad\,=  \begin{vmatrix}
a_{n-m+1} & a_{n-m+2}& \ldots&\ldots& a_{2n-m+1}\\
a_{0}&a_{1}&a_{2}&\ldots&a_{n}\\
a_{1}&a_{2}&a_{3}&\ldots&a_{n+1}\\
\vdots&\vdots&\vdots&&\vdots\\
a_{n-m}&a_{n-m+1}&a_{n-m+2}&\ldots&a_{2n-m}\\
a_{2\,0}&a_{2\,1}&a_{2\,2}&\ldots&a_{2\,n}\\
\vdots&\vdots&\vdots&&\vdots\\
a_{m\,0}&a_{2n-m\,1}&a_{2n-m\,2}&\ldots&a_{m\,n}\\
\end{vmatrix}
 \\
&\quad \qquad= (-1)^{n-m+1} \begin{vmatrix}
a_{0}&a_{1}&a_{2}&\ldots&a_{n}\\
\vdots&\vdots&\vdots&&\vdots\\
a_{k}&a_{k+1}&a_{k+2}&\ldots&a_{n+k}\\
\vdots&\vdots&\vdots&&\vdots\\
a_{n-m}&a_{n-m+1}&a_{n-m+2}&\ldots&a_{2n-m}\\
a_{n-m+1} & a_{n-m+2}& \ldots&\ldots& a_{2n-m+1}\\
a_{2\,0}&a_{2\,1}&a_{2\,2}&\ldots&a_{2\,n}\\
\vdots&\vdots&\vdots&&\vdots\\
a_{m\,0}&a_{2n-m\,1}&a_{2n-m\,2}&\ldots&a_{m\,n}\\
\end{vmatrix} \,,
\end{align*}
}\normalsize
that equals (at most up to a sign) the leading coefficient of $p_{n+1,m}(x_0)$.

To prove \eqref{symbp}, set:
$$ q_{n,m}(X_0,X_1,\dots,X_n) = X_2 X_3^2 \cdots X_l^{l-1}\cdots X_{n-m+1}^{n-m} \Delta(X_0, X_1,\dots,X_n),$$
and, for $i \notin \{1,\dots, n\}$, assume that $X_i$ is an independent copy of $X_1$. For a fixed $k=0,\dots,n-m$, if $\tau=(i, k+1)$ denotes the transposition exchanging $i$ and $k+1$, 
\begin{align*}
\E[\big(X_i^k q_{n,m}(X_i,X_1,X_2,\dots,X_n)\big)^{\tau} ] &= \E[\big(X_i^k X_2 \cdots  X_{n-m+1}^{n-m} \Delta(X_i,X_1,\dots,X_{k+1}, \dots,X_n)\big)^{\tau}] \\
&= - \E[X_i^k q_{n,m}(X_i,X_1,X_2,\dots,X_n) ],
\end{align*}
entailing $\E[X_0^k q_{n,m}(X_0,X_1,X_2,\dots,X_n)] = 0$ (here, $(a(x_1,\dots,x_{m}))^{\tau}$ denotes the polynomial obtained from $a(x_1,\dots,x_m)$ by letting $\tau$ acting on the variables $x_1,\dots,x_{m}$). The canonical symmetrization of $q_{n,m}(X_0,X_1,\dots,X_n)$ over $\mathfrak{S}_{n-m+1}$ yields:
\begin{align*}
\frac{1}{(n-m+1)!}&\sum_{\sigma \in \mathfrak{S}_{n-m+1}} \big( q_{n,m}(X_0,X_1,\dots,X_n)\big)^{\sigma} = \\
&= \frac{1}{(n-m+1)!}\sum_{\sigma \in \mathfrak{S}_{n-m+1}} \big(\Delta(X_0,X_1,\dots,X_n)X_2 X_3^2 \cdots X_{n-m+1}^{n-m}\big)^{\sigma} \\
&= \frac{1}{(n-m+1)!}\Delta(X_0,X_1,\dots,X_n)\sum_{\sigma \in \mathfrak{S}_{n-m+1}} (-1)^{\sigma} X_{\sigma(1)}^0 X_{\sigma(2)} X_{\sigma(3)}^2 \cdots X_{\sigma(n-m+1)}^{n-m} \\
&= \frac{1}{(n-m+1)!}\Delta(X_0,X_1,\dots,X_n)\Delta(X_1,\dots,X_{n-m+1}).
\end{align*}
Remarking  that  every $\sigma  \in \mathfrak{S}_{n-m+1}$ acts on $X_0$ as the identity, and thanks to the identical distribution assumption on $X_1,\dots,X_{n-m+1}$, the random variables
$$ p_{nm}(X_0) = \mathbb{E}_0[\Delta(X_0,X_1,\dots,X_n)\Delta(X_1,\dots,X_{n-m+1})]$$
satisfy:
\begin{align*}
\E[X_0^k p_{nm}(X_0)] &= \E[X_0^k \Delta(X_0,X_1,\dots,X_n)\Delta(X_1,\dots,X_{n-m+1})] \\
&= \sum_{\sigma \in \mathfrak{S}_{n-m+1}}\E[ X_0^k \big( q_{n,m}(X_0,X_1,\dots,X_n)\big)^{\sigma} ] \\
&= \sum_{\sigma \in \mathfrak{S}_{n-m+1}}\E[ \big(X_0^k  q_{n,m}(X_0,X_1,\dots,X_n)\big)^{\sigma} ] \\
&= (n-m+1)! \E[X_0^k q_{n,m}(X_0,X_1,\dots,X_n)] = 0,
\end{align*}
for every $k=0,\dots,n-m$. The conclusion, then, follows, by remarking that $\E[X_0^{n-m+1} p_{nm}(X_0)]$ equals the coefficient of $X_0^{n+1}$ in $p_{n+1,m}(X_0)$. Indeed, since $X_{n-m+2}$ is an independent copy of $X_0$, we can write:
\begin{align*}
\E[\Delta&(X_1,\dots,X_{n-m+2})\Delta(X_1,\dots,X_{n-m+2},\dots, X_{n+1})] = \\
&=\E[\Delta(X_1,\dots,X_{n-m+1},X_0)\Delta(X_1,\dots,X_{n-m+1},X_0,X_{n-m+3},\dots, X_{n+1})] \\
&= (n-m+2)! \E[X_2 X_3^2 \cdots X_{n-m+1}^{n-m} X_0^{n-m+1} \Delta(X_1,\dots,X_{n-m+1},X_0,X_{n-m+3},\dots, X_{n+1})]\\
&= (-1)^{\sigma} (n-m+2)! \E[X_2 X_3^2 \cdots X_{n-m+1}^{n-m} X_0^{n-m+1} \Delta(X_0,X_1,\dots,X_{n-m+1},X_{n-m+3},\dots, X_{n+1})]\\
&= (-1)^{\sigma} (n-m+2)! \E[X_0^{n-m+1} Y_2 Y_3^2 \cdots Y_{n-m+1}^{n-m} \Delta(X_0,Y_1,\dots,Y_{n-m+1},Y_{n-m+2},\dots, Y_{n})]\\
&= (-1)^{\sigma} (n-m+2)! \E[X_0^{n-m+1} p_{n,m}(X_0)],
\end{align*}
where $\sigma$ denotes the permutation of $\{0,\dots,n-m+1\}$ shifting $(X_1,\dots,X_{n-m+1},X_0)$ to $(X_0, X_1,\dots,X_{n-m+1})$, and where we have set $Y_j:=X_j$ for $j=1,\dots,n-m+1$, and $Y_j:=X_{j+1}$ for $j=n-m+2,\dots,n$.
\end{proof}

If the distribution of $X_0$ is absolutely continuous with respect to the Lebesgue measure, say with density (or weight) $\omega(t)$, then its moments admit the integral representation:
$$
a_{k}=\E[X_0^k]\,=\, \int_{I}\,t^k\,\omega(t)\,dt, \quad \forall k\in\N .$$
Similarly, assume that for every $j$, the random variable $X_j$ in \eqref{symbp}  has a density $\omega_j$ over $I_j$, so that:
\begin{equation}
\label{intL}a_{jk}=\E[X_{j}^k]\,=\,\int_{I_j}t^k\,\omega_j(t)\,dt \quad \forall \, k \in \N.
\end{equation}

In this case, the representation provided with \eqref{symbp} can be seen as a generalized \textit{Heine integral formula} (see, for instance \cite{Konig,Ism,Szego}).

\begin{thm}[Heine integral formula]\label{intr}
For every $n \geq 1$ and every $m=1,\dots,n$, the triangular array of polynomials $\{p_{nm}(x_0)\}_{n,m\geq 1}$ defined by
\small{
$$ p_{nm}(x_0):=\int\limits_{I_1\times I_2\times \cdots\times I_n}\,\Delta(x_1,x_2,\ldots,x_{n-m+1})\Delta(x_0,x_1,\ldots,x_n)\prod_{j=1}^{n-m+1}\omega_0(x_j)dx_j\prod_{i=n-m+2}^{n}\omega_i(x_i)dx_{i} $$}
\normalsize 
is a generalized orthogonal polynomial system for $X_0$, provided that $\deg\,p_{nm}=n$ for every $m=1,\dots,n$.
\end{thm}

\begin{exm} $ $
\begin{enumerate}
\item Assume that  $X_0, X_1,\dots,X_{n-m+1}$ are i.i.d.. Then, the leading coefficient of $p_{nm}(x_0)$ is given by:
$$\E[\Delta(X_1,X_2,\ldots,X_{n-m+1})\Delta(X_1,X_2,\ldots,X_{n})],$$
implying that $\deg\,p_{nm}=n$ if and only if 
$$\E[\Delta(X_1,X_2,\ldots,X_{n-m+1})\Delta(X_1,X_2,\ldots,X_{n})]\neq 0 .$$ Therefore, when $m=1$, Theorem \ref{intr} reduces  to the Heine integral formula for orthogonal polynomial sequences: in this case,  $\deg\,p_n=n$ is satisfied if and only if 
\begin{equation}
\label{HankelVand}
 \E[\Delta(X_1,X_2,\ldots,X_{n})^2] = n!\, det(a_{i+j})_{i,j=0,\dots,n-1} \neq 0, 
 \end{equation}
 while by setting $n=m$ and $\Delta(x_1)=1$,  Theorem \ref{intr} reduces to the integral formula for biorthogonal polynomials proved in \cite{IserNor}. In particular, $\deg\,p_{nn}=n$ if and only if 
$$ \E[\Delta(X_1,X_2,\ldots,X_{n})]\neq 0 .$$
\item
The orthogonal polynomials $\{p_{n1}(x_0)\}_{n\geq 1}$ are instances of \textit{multiple orthogonal polynomials} of the second kind (for short, II OPs) \cite{Kuijlaars}: given $p\in \mathbb{N}$, and real weight functions $\omega_1,\dots,\omega_p$, consider $\bs{n}=(n_1,\dots,n_p) \in \mathbb{N}_0^p$. A (monic) polynomial $P_{\bs{n}}(x_0)$ of degree $|\bs{n}| =n_1 + \cdots + n_p$ is of the type II OPS if
$$\int_{\mathbb{R}} P_{\bs{n}}(x_0) x_0^k \omega_j(x_0)dx_0 \, = 0 \quad \forall k=0,\dots,n_j-1, \forall j=1,\dots,p \,. $$
The determinantal and the integral formulae provided with Theorem \ref{GOPS1}, for $m=1$, should be then compared with the corresponding ones for multiple orthogonal polynomials,  provided in \cite{Kuijlaars}.\\
\end{enumerate}
\end{exm}


\begin{rmk}
Assume that, for a fixed $X:=X_0$, a GOPs $\{p_{nm}(x_0)\}_{n,m\geq 1}$ exists, and consider the associated OPs $p_n(x_0):=p_{n1}(x_0)$: in this case, \eqref{symbp} can be rewritten as:
\begin{equation}
\label{intorth}
p_n(X) = \E_0[(X_1-X)\cdots (X_n-X) \Delta(X_1,\dots,X_n)^2]. 
\end{equation}
Apart from occurring in the representation for the OPs, the statistics \linebreak$(X_1-X)\cdots(X_n-X)$ and
$\Delta(X_1,\ldots,X_n)^2$ play their own significant role
in probability: if $a=\E[X]$, then $\E\,[(X_1-X)\cdots(X_n-X)]=\E\,[(a-X)^n]$, the $n$-th central
moment of $X$ (up to a sign), while $\Delta(X_1,\ldots,X_n)^2$ is called
\textit{random discriminant}, and arises in spectral theory of random matrices (see \cite{Lu}).
\end{rmk}

\begin{exm}
Let $H_n(x)$ denote the $n$-th (monic) Hermite polynomial. If $N_1,\dots,N_k$ are independent $\mathcal{N}(0,1)$ distributed random variables, and $n_1+\cdots +n_k=n$, random variables of the form $\prod\limits_{j=1}^k \; H_{n_j}(N_j)$ are the generators of the so called $n$-th \textit{Gaussian Wiener homogeneous Chaos} (see \cite[Theorem 3.2.1]{Janson_GaussianSpaces},\cite[Theorem 1.1]{Engel}, \cite[Corollary 2.3]{Major_Integrals}, \cite{NourdinPeccatilibro}).  More generally, consider the polynomials 
$$p_{\bs{n}}(\bs{x})=p_{n_0}^{(0)}(x_{0})p_{n_1}^{(1)}(x_{1})\cdots p_{n_d}^{(d)}(x_{d}),$$ with $\bs{n}=(n_0,n_1,\dots,n_d) \in \mathbb{N}_0^{d+1}$,  $\bs{x}=(x_0,x_1,\dots,x_d)$ and $\{p_{n}^{(j)}(x_j)\}_{n\geq 0}$ is an OPs for a random variable $X_j$, admitting finite moments up to every order. Then, via \eqref{symbp} for $m=1$, $p_{\bs{n}}(\bs{x})$ admits the \textit{symbolic representation}:
$$ p_{\bs{n}}(X_0,X_1,\dots,X_d) = \prod_{j=0}^d \E_{0_j}[\Delta(X_{1,j},\dots,X_{n_j,j})\Delta(X_j,X_{1,j},\dots,X_{n_j,j} )],$$
with $\E_{0_j}$ denoting the conditional expectation with respect to $X_j$, and $X_{1,j},\dots,X_{n_j,j}$ independent copies of $X_j$. \\
\end{exm}

\section{Generalized OPs in several variables}

In the literature, there are several possible ways of defining orthogonality for polynomials in several indeterminates (see, for instance \cite{Xu}, \cite[Chapter 2]{Macdonald2}). In order to  enhance some differences with the univariate setting, we have chosen to discuss generalized orthogonal polynomials in $d+1$ indeterminates, for $d\geq 0$, in a separate section. \\

Consider a random vector $\bs{X}=(X_0,X_1,\dots,X_d)$. For the joint moments of $\bs{X}$ (all of which are assumed to exist finite), the following multi-index notation will be of use: if $\bs{k}=(k_0,k_1,\dots,k_d) \in \mathbb{N}_0^{d+1}$, then $\bs{X}^{\bs{k}} = \prod\limits_{j=0}^d X_j^{k_j}$.

Whenever a sequence $\{\bs{X}_i\}_{i\geq 0}$  of independent (non necessarily identically distributed) random vectors (defined on the same probability space) is given, say $\bs{X}_i=(X_{i0},X_{i1},\dots,X_{id})$, the independence assumption implies that, for every $\bs{k}_1,\dots, \bs{k}_l \in\N_0^{d+1}$, and every $i_1,\dots,i_l\geq 0$, $i_j \neq i_p$ for every $j \neq p$,
\begin{equation}\label{mE2}
\E[ \bs{X}_{i_1}^{\bs{k}_1} \bs{X}_{i_2}^{\bs{k}_2}\cdots \bs{X}_{i_l}^{\bs{k}_l}]= \E[\bs{X}_{i_1}^{\bs{k}_1}]\E[\bs{X}_{i_2}^{\bs{k}_2}]\cdots \E[\bs{X}_{i_l}^{\bs{k}_l}].
\end{equation}
In analogy with the first section, $\E_{0}$ will denote the conditional expectation with respect to $\bs{X}_0$, namely:
$$ \E_{0}[\bs{X}_0^{\bs{n}_0}\bs{X}_1^{\bs{n}_1} \cdots \bs{X}_r^{\bs{n}_r}] = \bs{X}_0^{\bs{n}_0}\E_0[\bs{X}_1^{\bs{n}_1} \cdots \bs{X}_r^{\bs{n}_r}],$$
and $\E_0[\bs{X}_1^{\bs{n}_1} \cdots \bs{X}_r^{\bs{n}_r}]= \E[\bs{X}_1^{\bs{n}_1} \cdots \bs{X}_r^{\bs{n}_r}]$.

Let $\leq$ denote the \textit{componentwise order} on $\mathbb{N}_0^{d+1}$, defined via: $(a_0,\dots,a_d) \leq (b_0,\dots,b_d)$ if and only if $a_i \leq b_i$ for every $i=0,\dots,d$. 

\begin{rmk}
The choice of the componentwise order is made to ensure that $(\mathbb{N}_0^{d+1},\leq)$ is a graded poset (see \cite{Stanley}), namely a partially ordered set with a rank function $\rho: \mathbb{N}_0^{d+1}\rightarrow \mathbb{N}_0$ such that:
\begin{itemize}
\item[-] if $\bs{n} \leq \bs{m}$, then $\rho(\bs{n}) \leq \rho(\bs{m})$ (where, with abuse of notation, $\leq$ denotes both the componentwise order on $\mathbb{N}_0^{d+1}$ and the usual order on $\mathbb{N}_0$);
\item[-] if $\bs{m}=(m_0,\dots,m_d)$ covers $\bs{n}=(n_0,\dots,n_d)$, namely if $m_j= n_j+1$ for the only $j$ such that $m_j > n_j$, then $\rho(\bs{m}) = \rho(\bs{n})+1$.
\end{itemize}
Moreover, the choice of such ordering implies that, if $\bs{m}\leq \bs{n}$, the element $\bs{n}-\bs{m}$ is always well defined (unlike in the case $\leq$ is, for instance, the \textit{lexicographical order}: $\bs{m}\leq \bs{n}$ if $m_i < n_i$ for $i=\min\{j: m_j \neq n_j\}$). In other words, the choice of the componentwise order leads to the most appropriate extension of $(\mathbb{N}_0,\leq)$. This setting will be crucial to encode orthogonality in a general unified framework for apolarity in higher dimensions, even if, in principle, $\leq$ might be replaced by any order $\preceq$ such that $(\mathbb{N}_0^{d+1},\preceq)$ is a graded poset verifying, at any stage, the required properties.\\
\end{rmk}

\begin{exm}
Consider  the set of positive integers $\mathbb{N}$ equipped with the divisibility relation: $a \leq b$ if and only if $a|b$, with minimal element $1$. $(\mathbb{N},|)$ is a graded poset, with rank function $\rho(a)$ equal to the number of prime factors of $a$, counted with multiplicity. Remark that $b$ covers $a$ if and only if $b/a$ is prime. Consider, then, the $k$-fold direct product of $(\mathbb{N},|)$, with partial order defined by $\bs{a}=(a_1,\dots,a_k) |\, \bs{b}=(b_1,\dots,b_k)$ if and only if $a_i | b_i$ for all $i=1,\dots,k$. Then, $(\mathbb{N}^k, |)$ is a graded poset with rank function $\rho_k(\bs{a}) = \sum_{i=1}^k \rho(a_i)$. In particular, if $\bs{a}| \bs{b}$, then $\bs{b} - \bs{a} \in \mathbb{N}^k$: indeed, $\bs{a} | \bs{b}$ implies $\bs{a} \leq \bs{b}$, if $\leq$ denotes the componentwise order. 
\end{exm}

\begin{defn}
For the fixed $\leq $, consider a triangular array of polynomials in $\C[\bs{x}_0]$ (where $\bs{x}_0=(x_0,x_1,\dots,x_d)$), say
\begin{equation}\label{polm}
p_{\bs{n}\bs{m}}(\bs{x}_0)=\sum_{\bs{0}\leq\bs{k}\leq\bs{n}}\binom{\bs{n}}{\bs{k}}\,p_{\bs{n}\bs{m}}^{(\bs{k})}\,\bs{x}_0^{\bs{k}}\, \,\text{ with }p_{\bs{n}\bs{m}}^{(\bs{n})}\neq 0,
\end{equation}
for every $\bs{n},\bs{m} \in \mathbb{N}_0^{d+1}$, with $\bs{m} \leq \bs{n}$, $\bs{m}\neq \bs{0}$. Then $\{p_{\bs{n}\bs{m}}(\bs{x}_0):\bs{m}\leq \bs{n}\}_{\bs{n} \in \mathbb{N}_0^{d+1}}$ is a \textbf{generalized orthogonal polynomial system} for $\bs{X}_0$ if it satisfies
\begin{equation}
\label{GMOPS}\E[\bs{X}_0^{\bs{k}}\,p_{\bs{n}\bs{m}}(\bs{X}_0)]=0 \quad \forall \, \bs{k}: \bs{0}\leq\bs{k}\leq\bs{n}-\bs{m},
\end{equation}
and $\E[\bs{X}_0^{\bs{a}}\,p_{\bs{n}\bs{m}}(\bs{X}_0)] \neq 0$ for every multi-index $\bs{a} \leq \bs{n}$, covering $\bs{n}-\bs{m}$.\\
\end{defn}

A polynomial of the type \eqref{polm} will be said of degree $\bs{n}$. If $\C[\bs{x}_0]_{\leq \bs{d}}$ denote the space of all polynomials having degree at most $\bs{d}$, then \eqref{GMOPS} means that $p_{\bs{n}\bs{m}}(\bs{x}_0)$ is orthogonal to all elements in $\C[\bs{x}_0]_{\leq \bs{n}-\bs{m}}$.


The next theorems are meant to generalize the formulae \eqref{detp} and \eqref{symbp} in the present multivariate setting: to this aim, for every $\bs{k} \in \mathbb{N}_0^{d+1}$, set $s(\bs{k}) = |\{\bs{h} \in \mathbb{N}_0^{d+1}: \bs{h} \leq \bs{k}\}|$.

\begin{thm}\label{detformulti}
For every $\bs{n},\bs{m} \in \mathbb{N}_0^{d+1}$, with $\bs{m} \leq \bs{n}$, and with the above notation, set:
\begin{enumerate}
\item $s:=s(\bs{n})-1$ and $\{\bs{k}\,|\,\bs{0}\leq\bs{k}\leq\bs{n}\}=\{\bs{0}=\bs{k}_0,\bs{k}_1,\ldots,\bs{k}_s:=\bs{n}\}$;
\item $r:=s(\bs{n}-\bs{m})-1$ and $\{\bs{h}\,|\,\bs{0}\leq\bs{h}\leq\bs{n}-\bs{m}\}=\{\bs{0}=\bs{h}_0,\bs{h}_1,\ldots,\bs{h}_r:=\bs{n}-\bs{m}\}$.
\end{enumerate}
Set $\bs{X}:=\bs{X}_0$, and, for every $\bs{n} \in \mathbb{N}_0$, let $\bs{X}_1,\dots, \bs{X}_{s}$ be independent, but non identically distributed random vectors, independent of $\bs{X}$, and set  $a_{\bs{k}} = \E[\bs{X}_0^{\bs{k}}]$, $a_{j {\bs{k}} } = \E[\bs{X}_j^{\bs{k}}]$ for every $j\geq 1$. Then, the polynomials $p_{\bs{n}\bs{m}}(\bs{x}_0)$  defined by:
\[p_{\bs{n}\bs{m}}(\bs{x}_0)=
\begin{vmatrix}
1&\bs{x}_0^{\bs{k}_1}&\bs{x}_0^{\bs{k}_2}&\ldots&\bs{x}_0^{\bs{n}}\\
a_{\bs{0}}&a_{\bs{k}_1}&a_{\bs{k}_2}&\ldots&a_{\bs{k}_s}\\
a_{\bs{h}_1}&a_{\bs{k}_1+\bs{h}_1}&a_{\bs{k}_2+\bs{h}_1}&\ldots&a_{\bs{n}+\bs{h}_1}\\
a_{\bs{h}_2}&a_{\bs{k}_1+\bs{h}_2}&a_{\bs{k}_2+\bs{h}_2}&\ldots&a_{\bs{n}+\bs{h}_2}\\
\vdots&\vdots&\vdots&&\vdots\\
a_{\bs{h}_r}&a_{\bs{k}_1+\bs{h}_r}&a_{\bs{k}_2+\bs{h}_r}&\ldots&a_{2\bs{n}-\bs{m}}\\
a_{2 \bs{k}_0}& a_{2 \bs{k}_1}& a_{2\bs{k}_2}&\ldots& a_{2 \bs{n}}\\
\vdots&\vdots&\vdots&&\vdots\\
a_{s-r\,\bs{k}_0}& a_{s-r\,\bs{k}_1}& a_{s-r\,\bs{k}_2}&\ldots& a_{s-r\,\bs{n}},
\end{vmatrix}\]
form a generalized orthogonal polynomial system for $\bs{X}_0$, provided that $p_{\bs{n}\bs{m}}(\bs{x}_0)$ is of degree $\bs{n}$ for every $\bs{m} \leq \bs{n}$, $\bs{m} \neq \bs{0}$.
\end{thm}

\begin{proof}
Assume that $p_{\bs{n}\bs{m}}(\bs{x}_0)$ is of degree $\bs{n}$ for every $\bs{m} \leq \bs{n}$ and consider $\bs{k}: \bs{0}\leq \bs{k}\leq\bs{n}-\bs{m}$, so that $\bs{k}=\bs{h}_i$ for some $i=0,\dots, r$. Then 
\[\bs{X}_0^{\bs{k}}\,p_{\bs{n}\bs{m}}(\bs{X}_0)=
\begin{vmatrix}
\bs{X}_0^{\bs{h}_i}&\bs{X}_0^{\bs{k}_1+\bs{h}_i}&\bs{X}_0^{\bs{k}_2+\bs{h}_i}&\ldots&\bs{X}_0^{\bs{k}_s+\bs{h}_i}\\
a_{\bs{0}}&a_{\bs{k}_1}&a_{\bs{k}_2}&\ldots&a_{\bs{k}_s}\\
a_{\bs{h}_1}&a_{\bs{k}_1+\bs{h}_1}&a_{\bs{k}_2+\bs{h}_1}&\ldots&a_{\bs{k}_s+\bs{h}_1}\\
a_{\bs{h}_2}&a_{\bs{k}_1+\bs{h}_2}&a_{\bs{k}_2+\bs{h}_2}&\ldots&a_{\bs{k}_s+\bs{h}_2}\\
\vdots&\vdots&\vdots&&\vdots\\
a_{\bs{h}_r}&a_{\bs{k}_1+\bs{h}_r}&a_{\bs{k}_2+\bs{h}_r}&\ldots&a_{2\bs{n}-\bs{m}}\\
a_{2 \bs{k}_0}& a_{2 \bs{k}_1}& a_{2\bs{k}_2}&\ldots& a_{2 \bs{n}}\\
\vdots&\vdots&\vdots&&\vdots\\
a_{s-r\,\bs{k}_0}&a_{s-r\,\bs{k}_1}&a_{s-r\,\bs{k}_2}&\ldots& a_{s-r\,\bs{n}}
\end{vmatrix}\]
so that two rows are equal in the determinant, and hence $\E[\bs{X}_0^{\bs{k}}\,p_{\bs{n}\bs{m}}(\bs{X}_0)]=0$. Similarly to the univariate setting, finally one has that $\E[\bs{X}_0^{\bs{a}}\,p_{\bs{n}\bs{m}}(\bs{X}_0)]\neq 0$ for every multi-index $\bs{a} \leq \bs{n}$, covering $\bs{n}-\bs{m}$. Indeed, set $\bs{a}:= \bs{n}-\bs{m} + \bs{\delta}_i$, with $\bs{\delta}_i$ denoting the multi-index with $1$ in the $i$-th position, and $0$ elsewhere, and assume that $\bs{a} \leq \bs{n}$: then, $\E[\bs{X}_0^{\bs{a}}\,p_{\bs{n}\bs{m}}(\bs{X}_0)]$ equals the leading coefficient of $p_{\bs{n}+\bs{\delta}_i\, \bs{m}}(\bs{x}_0)$.
\end{proof}

\begin{rmk}
Note that the choice of the labelling for the multi-indices $\bs{k}: \bs{k} \leq \bs{n}$ does not affect the result and the proof, since the value of the determinant would simply change in sign whenever some columns are switched.\\
\end{rmk}

In order to state the multivariable counterpart to the  symbolic representation \eqref{symbp}, consider the
following determinants: for all $\bs{n}\in\N_0^{d+1}$, if
$s=s(\bs{n})-1$ and $\{\bs{k}\,|\,\bs{0}\leq\bs{k}\leq\bs{n}\}=\{\bs{n}_0:=\bs{0},\bs{n}_1,\ldots,\bs{n}_s=\bs{n}\}$, set
\[\bs{\Delta}_{\bs{n}}(\bs{X}_0,\bs{X}_1,\ldots,\bs{X}_{s})=\begin{vmatrix}
1&\bs{X}_0^{\bs{n}_1} & \bs{X}_0^{\bs{n}_2}&\cdots & \bs{X}_0^{\bs{n}_s}\\
1&\bs{X}_1^{\bs{n}_1} & \bs{X}_1^{\bs{n}_2}&\cdots & \bs{X}_1^{\bs{n}_s}\\
1&\bs{X}_2^{\bs{n}_1} & \bs{X}_2^{\bs{n}_2}&\cdots & \bs{X}_2^{\bs{n}_s}\\
\vdots&\vdots&\vdots&&\vdots\\
1&\bs{X}_s^{\bs{n}_1} & \bs{X}_s^{\bs{n}_2}&\cdots & \bs{X}_s^{\bs{n}_s}
\end{vmatrix}  = \sum_{\sigma \in \mathfrak{S}_{\{0,\dots,s\}}}(-1)^{\sigma}\bs{X}_{0}^{\bs{n}_{\sigma(0)}} \bs{X}_{1}^{\bs{n}_{\sigma(1)}} \cdots \bs{X}_{s}^{\bs{n}_{\sigma(s)}}  ,\]
and
\[\bs{\Delta}^*_{\bs{n}}(\bs{X}_1,\bs{X}_2,\ldots,\bs{X}_{s})=\begin{vmatrix}
\bs{X}_1^{\bs{n}_1} & \bs{X}_1^{\bs{n}_2}&\cdots & \bs{X}_1^{\bs{n}_s}\\
\bs{X}_2^{\bs{n}_1} & \bs{X}_2^{\bs{n}_2}&\cdots & \bs{X}_2^{\bs{n}_s}\\
\vdots&\vdots&&\vdots\\
\bs{X}_s^{\bs{n}_1} & \bs{X}_s^{\bs{n}_2}&\cdots & \bs{X}_s^{\bs{n}_s}
\end{vmatrix} = \sum_{\sigma \in \mathfrak{S}_{\{1,\dots,s\}}}(-1)^{\sigma} \bs{X}_{1}^{\bs{n}_{\sigma(1)}} \cdots \bs{X}_{s}^{\bs{n}_{\sigma(s)}}.\]
\begin{thm}\label{morthpol}
Given $\bs{n}, \bs{m} \in \mathbb{N}_0^{d+1}$, with $\bs{m} \leq \bs{m}$, set $r:=s(\bs{n}-\bs{m})-1$ and $s:=s(\bs{n})-1$. If $\bs{X}_0,\bs{X}_1,\dots,\bs{X}_n$ are independent, and at least $\bs{X}_0,\dots,\bs{X}_r$ are identically distributed, then the random variable $p_{\bs{n}\bs{m}}(\bs{X}_0)$ defined by
\begin{equation*}
p_{\bs{n}\bs{m}}(\bs{X}_0)=\E_0[\bs{\Delta}^*_{\bs{n}-\bs{m}}(\bs{X}_1,\bs{X}_2,\ldots,\bs{X}_{r})\bs{\Delta}_{\bs{n}}(\bs{X}_0,\bs{X}_1,\ldots,\bs{X}_{s})]\, \text{ for every }\bs{m}: \bs{0}<\bs{m}\leq\bs{n},
\end{equation*}
with $p_{\bs{n}\bs{m}}(\bs{x}_0)$ of degree $\bs{n}$ for every $\bs{m}$: $\bs{0}<\bs{m}\leq\bs{n}$, satisfies $\E[\bs{X}_0^{\bs{k}}p_{\bs{n}\bs{m}}(\bs{X}_0)] = 0$ for every $\bs{k}\leq \bs{n}-\bs{m}$, and $\E[\bs{X}_0^{\bs{a}}p_{\bs{n}\bs{m}}(\bs{X}_0)]\neq 0$ for every multi-index $\bs{a} \leq \bs{n}$, covering $\bs{n}-\bs{m}$. Hence, $p_{\bs{n}\bs{m}}(\bs{x}_0)$ is a GOPs for $\bs{X}_0$.
\end{thm}

\begin{proof}
The proof is analogous to the proof of \eqref{symbp}, but one has to start from 
$$ q_{\bs{n},\bs{m}}(\bs{X}_0,\bs{X}_1,\ldots,\bs{X}_s):=\bs{X}_1^{\bs{h}_1}\bs{X}_2^{\bs{h}_2}\cdots\bs{X}_r^{\bs{h}_r}\bs{\Delta}_{\bs{n}}(\bs{X}_0,\bs{X}_1,\ldots,\bs{X}_{s}).$$
In particular, remark that the choice of the labelling for the elements in $\{\bs{k}: \bs{0}\leq \bs{k} < \bs{n}\}$ does not affect the result, thanks to the assumption of identical distribution on $\bs{X}_1,\dots,\bs{X}_r$.
\end{proof}


If the random vectors are real-valued, and for every $j$ there exists a density  $\omega_j(\bs{t})=\omega_j(t_0,t_1,\ldots,t_d)$ such that \begin{equation}\label{Eintmulti}
\E[\bs{X}_j^{\bs{n}}]=\int_{\R^{d+1}}\bs{x}^{\bs{n}}\omega_j(\bs{x})d\bs{x} \,,
\end{equation}
with $d\bs{x}=dx_0 dx_1\cdots dx_d$, then the following Heine integral formula can be deduced as a direct consequence of Theorem \ref{morthpol}.
\begin{cor}[Heine integral formula]
Under the hypotheses and notations of Theorem \ref{morthpol}, the polynomials
\begin{equation} \label{lastint}
p_{\bs{n}\bs{m}}(\bs{x}_0)=\int\limits_{\mathbb{R}^{s(d+1)}}\bs{\Delta}^*_{\bs{n}-\bs{m}}(\bs{x}_1,\bs{x}_2,\ldots,\bs{x}_{r})\bs{\Delta}_{\bs{n}}(\bs{x}_0,\bs{x}_1,\ldots,\bs{x}_{s})\prod_{i=1}^{s}\omega_i(\bs{x}_i)d\bs{x}_i
\end{equation}
form a GOPs for $\bs{X}_0$.
\end{cor}

\chapter{What are orthogonal polynomials, really?}\label{Chapter_WhatIsOrtho}

In his Fubini lecture ``\textit{What is invariant theory, really?}'' \cite{Rota98}, Gian-Carlo Rota disclosed new motivations to investigate the classical invariant theory (of binary forms) via the symbolic method introduced in \cite{KungRota}. As the title suggests, the aim of his lecture was to explore what really should be meant by \textit{invariant theory} in a simple and effective way, exploiting the features of the so-called \textit{umbral calculus} (see \cite{DiNSen,RotTay}). Indeed, as Rota himself said (see \cite{Rota_TurningPoints}),
\begin{center}
``\textit{The purpose
of invariant theory, from Boole to our day, is precisely the translation of geometric facts into invariant algebraic equations expressed in terms of tensors. This program of translation of geometry into algebra was to be carried out in two steps. The first step consisted in decomposing the tensor algebra into irreducible components under changes of coordinates. The second step consisted in devising an efficient notation for the invariants for each irreducible component. The first step was successfully carried out in this century; the second
was abandoned sometime in the twenties, and only recently  it has resurfaced}''. 
\end{center}

In this chapter, the symbolic method of invariant theory, as set in \cite{KungRota}, is slightly revisited in order to fit the theory of orthogonal polynomials (see Section \ref{Section_Symb}). In particular, the algebraic representation for generalized orthogonal polynomials  that has been achieved in Chapter \ref{Chapter_OPS}  will be encoded in a symbolic expression for a family of joint-covariants of binary forms: the \textit{apolar covariant}. In this direction, it will be shown that the classical orthogonality of a sequence of polynomials with respect to a linear functional, as in \cite{Chihara}, is nothing but apolarity of binary forms in disguise (see Section \ref{Section_Apol}). As a consequence, in Section \ref{Section_Quadrature}, new formulae for the computation of the moments of an important statistic in modern probability theory, known as the \textit{random discriminant}, are discussed: the most recent outcomes for the computation of the distribution of random discriminants can be found in \cite{Lu}.

\section{The symbolic method of invariant theory}\label{Section_Symb}

The symbolic method of invariant theory, in its modern setting, has been developed in the landmark paper by Kung and Rota~\cite{KungRota}, that is the starting point of the present discussion. The notation here is slightly revisited to better perform the algebraic approach to the theory of orthogonal polynomials. In particular, the roman variables $(u_1,u_2)$ and the Greek letters $(\alpha_1,\alpha_2), (\beta_1,\beta_2), \ldots$, denoting the \textit{umbrae}, used in \cite{KungRota}, are respectively replaced by pairs of indeterminates $(x_0,y_0)$, and $(x_1,y_1)$, $(x_2,y_2) \ldots$; the action of the \textit{umbral operator} $\U$ on a polynomial $p$ will be denoted by $\U\,p$, while the coefficients of a generic form of degree $n$, originally written $A_0, A_1,\ldots,A_n$, will occur
here as $(-1)^{n}a_0, (-1)^{n-1}a_1,\ldots,a_n$. Therefore,
covariants of binary forms of degree $n$ will be expressed as polynomials in
$a_0, a_1,\ldots, a_n, x_0, y_0$ instead of $A_0, A_1, \ldots, A_n, X, Y$.\\

Given two infinite sets of indeterminates $\bs{x}=\{x_i\,|\,i\in\N_0\}$ and $\bs{y}=\{y_i\,|\,i\in\N_0\}$, let $\C[\bs{x},\bs{y}]$  denote the ring of polynomials with coefficients in $\C$ \footnote{More generally, $\C$ could be replaced by any field of characteristic zero.}, and variables in $\bs{x}\cup\bs{y}$. The general linear group $GL_2(\C)$ \glossary{name={$GL_n(\C)$},description={General linear group of degree $n$}} acts on $\C[\bs{x},\bs{y}]$ via the standard matrix multiplication:
\begin{equation}\label{Glact}\left(\begin{matrix}g_{11}&g_{12}\\g_{21}&g_{22}\end{matrix}\right)\left(\begin{matrix}
x_i\\y_i\end{matrix}\right)=\left(\begin{matrix}
g_{11}x_i+g_{12}y_i\\g_{21}x_i+g_{22}y_i\end{matrix}\right) \;, \; \forall\, i\in\N_0.\end{equation}
Hence, if $p\in\C[\bs{x},\bs{y}]$ and $g\in GL_2(\C)$, $g\cdot
p$ denotes the polynomial obtained from $p$ by replacing all its indeterminates according to \eqref{Glact}. 

\begin{defn}
If $m \in \N_0$, a polynomial $p \in \C[\bs{x},\bs{y}]$  is said to be an \textbf{invariant of index} \index{Invariant} $m$ \index{Invariant!Index} if and only
if it satisfies:
\[g\cdot p=(\det\,g)^m\,p \;\; \quad \forall \,g \in GL_2(\C).\]
\end{defn}
An important class of invariants of the general linear group is the set of the \textit{brackets}\index{Brackets polynomials}, defined for $i\neq j$ as the polynomial:
$$[i\;j]=x_i\;y_j-x_j\;y_i \;.$$
Brackets polynomials are the generators of the subring $\C[\bs{x},\bs{y}]^{GL_2(\C)}$ of $\C[\bs{x},\bs{y}]$ consisting of all the  invariant polynomials in $\C[\bs{x},\bs{y}]$ (see \cite{Grosshans,KraftProcesi}).
\begin{defn}
For $n\in \N$, let $a_0,a_1,\ldots,a_n$ be independent indeterminates over $\C$. A \textbf{generic binary form of degree $n$} \index{Binary form!Generic} is a polynomial $f(a_0, a_1,\ldots,a_n ;x_0, y_0) \in \C[a_0, a_1,\ldots,a_n;x_0, y_0]$ of the type
\begin{equation}
\label{binf}f(a_0, a_1,\ldots, a_n; x_0, y_0)=\sum_{k=0}^n\binom{n}{k}(-1)^{n-k}a_k x_0^{n-k}y_0^{k}.
\end{equation}
\end{defn}

When $a_0,a_1,\ldots,a_n$ are replaced by elements in $\C$, with at least one $a_0 \neq 0$, then \eqref{binf} is a homogeneous polynomial $f(x_0,y_0) \in \C[x_0,y_0]$ of degree $n$, and it is usually referred to as \textit{binary form of degree $n$}\index{Binary form}. \\

For every $c_{11},c_{12},c_{21},c_{22} \in \C$ such that $c_{11}c_{22}-c_{12}c_{21} \neq 0$, a \textit{linear change of variables} is a mapping $\phi=\phi(c_{11},c_{12},c_{21},c_{22})$ defined on the pair of indeterminates $(x_0,y_0)$ via matrix multiplication:
\begin{equation}
\label{linearchange}\phi\,\bigg(\begin{array}{c}x_0\\y_0\end{array}\bigg)=\bigg(\begin{array}{cc}c_{11}&c_{12}\\c_{21}&c_{22}\end{array}\bigg)\bigg(\begin{array}{c}x_0\\y_0\end{array}\bigg).
\end{equation}

Under the action of a linear change of variables, a generic binary form of degree $n$, \linebreak$f(a_0, a_1,\ldots,a_n ;x_0, y_0) \in \C[a_0, a_1,\ldots,a_n;x_0, y_0]$ is mapped to the generic binary form of degree $n$ defined via:
\begin{equation}\label{Transform}
f(\bar{a}_0,\bar{a}_1,\ldots,\bar{a}_n;x_0,y_0):=f(a_0,a_1,\ldots,a_n;c_{11}x_0+c_{12}y_0,c_{21}x_0+c_{22}y_0).
\end{equation}

\begin{defn}\label{JointCovariant}
For $l\geq 1$, let $\left(f_i(a_{i0},a_{i1},\ldots,a_{in_i};x_0,y_0)\right)_{1\leq i\leq l}$ be an ordered sequence of generic binary forms, with $f_i$ of degree $n_i \in \N$, and set $\bs{n}=(n_1,n_2,\ldots,n_l)$. A \textbf{joint-covariant} \index{Joint covariant}of index $m \in \mathbb{N}_0$ \index{Joint covariant!index} of binary forms of degree $\bs{n}$ is a  non-constant polynomial \linebreak$\mathcal{I}(\ldots;a_{i0},a_{i1},\ldots,a_{in_i};\ldots;x_0,y_0) \in \C[\dots,a_{i0},a_{i1},\ldots,a_{in_i},\dots, x_0,y_0]$, homogeneous of \textrm{degree}\index{Joint covariant!degree} $\nu_i$ in $a_{i0},a_{i1},\ldots,a_{in_i}$, and of degree $\mu$ in $x_0,y_0$, called the \textrm{order}\index{Joint covariant!order}, satisfying
$$\mathcal{I}(\ldots;\bar{a}_{i0},\bar{a}_{i1},\ldots,\bar{a}_{in_i};\ldots;x_0,y_0)= (\det\,\phi)^m\,\mathcal{I}(\ldots;a_{i0},a_{i1},\ldots,a_{in_i};\ldots;x_0,y_0)$$
for every linear change of variables $\phi = \phi(c_{11},c_{12},c_{21},c_{22})$,
where 
$$\mathcal{I}(\ldots;\bar{a}_{i0},\bar{a}_{i1},\ldots,\bar{a}_{in_i};\ldots;x_0,y_0) := \mathcal{I}(\ldots;a_{i0},a_{i1},\ldots,a_{in_i};\ldots;c_{11}x_0 +c_{12}y_0,c_{21}x_0+c_{22}y_0).$$
In particular, $n_1 \nu_1 + \cdots + n_l \nu_l = 2m + \mu$.
\end{defn}

Covariants that do not depend on $x_0$ nor on $y_0$ are called \textit{invariants}\index{Invariant}. When $l=1$, joint-covariants are simply called covariants.

\begin{defn}
A \textbf{covariant} of index $m$ of generic binary forms of degree $n$ \index{Covariant}is a homogeneous non-constant polynomial $ \mathcal{I}(a_0,a_1,\ldots,a_n;x_0,y_0) \in \C[a_0, a_1,\ldots,a_n; x_0, y_0]$ satisfying
$$\mathcal{I}(\bar{a}_0,\bar{a}_1,\ldots,\bar{a}_n;x_0,y_0)=(\det\,\phi)^m \mathcal{I}(a_0,a_1,\ldots,a_n;x_0,y_0)$$
for every linear change of variables $\phi=\phi(c_{11},c_{12},c_{21},c_{22})$, where
$$\mathcal{I}(\bar{a}_0,\bar{a}_1,\ldots,\bar{a}_n;x_0,y_0):= \mathcal{I}(a_0,a_1,\ldots,a_n;c_{11}x_0+c_{12}y_0,c_{21}x_0+c_{22}y_0).$$
In particular, $n\,\nu = 2m+ \mu$, where  $\nu$ and $\mu$ are respectively the degrees of $\mathcal{I}(a_0,a_1,\ldots,a_n;x_0,y_0)$ as a polynomial in $a_0,\dots,a_n$ and in $x_0,y_0$.
\end{defn}

\begin{exm}
Some of the most important examples of joint covariants are listed below (see \cite{Janson_Invariant,KungRota}):
\begin{itemize}
\item[-] If $n=2q$, for the binary form $f=\sum\limits_{k=0}^n  A_k x^{n-k}y^k$, the Hankel determinant 
$$ \mathrm{Hank}(f) = \det(A_{i+j})_{i,j=0,\dots,q}$$
is the invariant of degree $\nu=q+1$ and index $m=q(q+1)$, usually called \textit{Catalecticant}\index{Covariant!Catalecticant}.
\item[-] If $f$ is a binary form of degree $n$, the \textit{Hessian} determinant \index{Covariant!Hessian}
$$ H(f,x_0,y_0) =  
\begin{array}{|c c |}
\dfrac{\partial^2 f}{\partial^2 x_0 } & \dfrac{\partial^2 f}{\partial x_0 \partial y_0 } \\
\dfrac{\partial^2 f}{\partial y_0 \partial x_0 } & \dfrac{\partial^2 f}{\partial^2 y_0 }
\end{array}
$$
is a covariant of degree $\nu=2$, order $\mu=2(n-2)$ and index $m=2$.
\item[-] If $f_i$ is a binary form of degree $n_i$ for $i=1,2$, the \textit{Jacobian} determinant \index{Joint covariant!Jacobian}
$$ J(f_1,f_2,x_0,y_0)=
\begin{array}{|c c|}
\dfrac{\partial f_1}{\partial x_0} & \dfrac{\partial f_1}{\partial y_0} \\
\dfrac{\partial f_2}{\partial x_0} & \dfrac{\partial f_2}{\partial y_0}
\end{array}
$$
is a joint covariant of degrees $\nu_1=\nu_2=1$, order $\mu=n_1+n_2 -2$, and index $m=1$.
\end{itemize}
\end{exm}

\begin{rmk}
Since every polynomial can be uniquely decomposed as a sum of homogeneous polynomials, there is no loss of generality in assuming that the joint covariants are homogeneous: indeed, the homogeneous components of an invariant polynomial are themselves invariant.\\
\end{rmk}

If the degree $n$ of a generic binary form is fixed, consider the linear operator
\[\U:=\U^{(n)}\colon\C[\bs{x},\bs{y}]\to\C[a_0,a_1,\ldots,a_n;x_0,y_0]\]
defined by the following conditions:
\begin{align}
\label{E(n)}\U\,x_i^{k_1}y_i^{k_2}&=\begin{cases}x_0^{k_1}y_0^{k_2}&\text{ if }i=0,\\a_{k_1}&\text{ if }k_1+k_2=n \text{ and }i\in \mathbb{N},\\0& \text{ otherwise},
\end{cases}\\
\label{E(n)bis}\U\,x_0^{k_1}y_0^{k_2}x_1^{l_1}y_1^{l_2}x_2^{m_1}y_2^{m_2}\cdots &=\U\,x_0^{k_1}y_0^{k_2}\,\U\,x_1^{l_1}y_1^{l_2}\,\U\,x_2^{m_1}y_2^{m_2}\cdots,
\end{align}
for every non-negative integers $k_1,k_2,l_1,l_2,m_1,m_2,\ldots$.

\begin{rmk}
As pointed out in the first place by Grace and Young \cite{GraceYoung}, in order to have a proper generic binary form, it is fundamental to have an infinite number of \textit{umbrae} representing the same coefficient (here, an infinite number of indeterminates $x_i$), over which the umbral functional $\U$ factorises.\\
\end{rmk}

The \textit{umbral operator} \index{Umbral operator} $\U$ allows one to associate a generic binary form with a bracket polynomial, in the sense that for every $i\in \mathbb{N}$:
\begin{equation}\label{fbracket}
f(a_0,a_1,\ldots,a_n;x_0,y_0) = \U\,[i\;0]^n=\U\,(x_i y_0-x_0 y_i)^n\,.
\end{equation}
As a consequence, the action of a linear change of variables $\phi = \phi(c_{11},c_{12},c_{21},c_{22})$ on $f$ can be expressed as an action of  $g^\phi\in GL_2(\C)$ on the pairs of indeterminates $(x_i,y_i)$'s with $i\in\mathbb{N}$, where:
\begin{equation}\label{map}
 g^\phi=\left(\begin{matrix}c_{22}&-c_{12}\\-c_{21}&c_{11}\end{matrix}\right).
\end{equation}
Then,
\begin{equation}\label{tranlsate1}
g^\phi\cdot[i\;0]^n=\big(x_i(c_{21}x_0+c_{22}y_0)-y_i(c_{11}x_0+c_{12}y_0)\big)^n,
\end{equation}
so that, by comparing \eqref{Transform}, \eqref{fbracket} and \eqref{tranlsate1}, it follows that:
\begin{equation}\label{tranlsate2}
\U\,g^\phi\cdot[i\;0]^n=f(\bar{a}_0,\bar{a}_1,\ldots,\bar{a}_n;x_0,y_0),
\end{equation}
which is, in turn, equivalent to:
\begin{equation}\label{tranlsate3}
\U\,(c_{22}x_i-c_{12}y_i)^{k_1}(-c_{21}x_i+c_{11}y_i)^{k_2}=\bar{a}_{k_1},
\end{equation}
for all $k_1,k_2\in\N_0$ such that $k_1+k_2=n$.\\

One of the major results in the invariant theory of binary forms is the so-called \textit{First Fundamental Theorem} (see \cite[Theorem 3.1]{KungRota}, as well as the reference therein), \index{First Fundamental Theorem} stating that   every covariant $\mathcal{I}(a_0,a_1,\ldots,a_n;x_0,y_0)$  of binary forms of degree $n$, of index $m$, is obtained as:
\[\mathcal{I}(a_0,a_1,\ldots,a_n;x_0,y_0)=\U\,p,\]
where $p\in\C[\bs{x},\bs{y}]$ is a product of a finite number of brackets involving exactly $m$ brackets of the type $[j\;i]$ (with $i,j\neq 0$), and $n$ being such that $n\nu=2m+\mu$, if $\nu, \mu$ are respectively the degree and the order of $\mathcal{I}(a_0,a_1,\ldots,a_n;x_0,y_0)$.


\begin{exm}
For $m\geq 1$, the polynomial in $\C[a_0,\dots,a_m,x_0,y_0]$ defined by:
$$ p_m(a_0, \dots, a_m,x_0,y_0) = \U\bigg( \prod_{j=1}^m [0 \; j] \; \prod_{1 \leq i < j \leq m} [i \; j]^2\bigg) $$ 
is a covariant of index $m(m-1)$ of binary forms of degree $n=2m-1$: indeed, both its order and its degree (as in Definition \ref{JointCovariant}) are equal to $m$. We shall see that the $p_m$'s correspond naturally to classical orthogonal polynomials.\\
\end{exm}

To represent joint covariants, the umbral operator defined in \eqref{E(n)} and
\eqref{E(n)bis} has to be generalized as follows: let $\mathbb{N}_1,\mathbb{N}_2,\ldots,\mathbb{N}_l$ be pairwise disjoint infinite sets, satisfying $\mathbb{N}_1\cup \mathbb{N}_2\cup\cdots
\cup \mathbb{N}_l=\mathbb{N}$, and consider the linear operator $\U:= \U(\mathbb{N}_1,\dots,\mathbb{N}_l)$
\[\U:= \U(\mathbb{N}_1,\dots,\mathbb{N}_l)\colon\C[\bs{x},\bs{y}]\to\C[\ldots,a_{i0},a_{i1},\ldots,a_{in_i},\ldots;x_0,y_0],\]
defined by the conditions:
\begin{align}
\label{E(nn)}
\U\,x_i^{k_1}y_i^{k_2}&=\begin{cases}x_0^{k_1}y_0^{k_2}&\text{ if }i=0,\\a_{jk_1}&\text{ if }k_1+k_2=n_j \text{ and }i\in \mathbb{N}_j,\\0& \text{ otherwise},\end{cases}\\
\label{E(nn)bis}\U\,x_0^{k_1}y_0^{k_2}x_1^{l_1}y_1^{l_2}x_2^{m_1}y_2^{m_2}\cdots
&=\U\,x_0^{k_1}y_0^{k_2}\,\U\,x_1^{l_1}y_1^{l_2}\,\U\,x_2^{m_1}y_2^{m_2}\cdots,
\end{align}
for all non-negative integers $k_1,k_2,l_1,l_2,m_1,m_2,\ldots$.

Then, a joint-covariant of index $m$ of binary forms of degree $\bs{n}=(n_1,\dots,n_l)$ can be represented as
\[\mathcal{I}(\ldots;a_{i0},a_{i1},\ldots,a_{in_i};\ldots;x_0,y_0)=\U\,p,\]
where $p\in\C[\bs{x},\bs{y}]$ is a product of a finite number of brackets, involving exactly $m$ brackets of the type $[j\;i]$, for $i,j\neq 0$.\\

For a fixed $l\geq 1$, and binary forms $f_i(x_0,y_0) \in \C[x_0,y_0]$ of degree $n_i \in \mathbb{N}$, for $i=1,\dots,l$, consider the linear operator
\[\U(f_1,f_2,\ldots,f_l)\colon\C[\bs{x},\bs{y}]\to\C[x_0,y_0], \]
defined as follows: if $p \in \C[\bs{x},\bs{y}]$, then $\U(f_1,f_2,\ldots,f_l)\,p$ is obtained from $\U \, p$  by evaluating the variable $a_{ij}$ with the corresponding coefficient of $f_i(x_0,y_0)$. Particularly relevant for the subsequent discussion is the polynomial 
$$\mathcal{I}(f_1,f_2,\ldots,f_l)(x_0,y_0): =\U(f_1,f_2,\ldots,f_l)\,p \, \in \C[x_0,y_0] $$ 
obtained by evaluating a joint-covariant $\mathcal{I}(\ldots;a_{i0},a_{i1},\ldots,a_{in_i};\ldots;x_0,y_0) = \U p$ at the coefficients of $f_1(x_0,y_0)$, $f_2(x_0,y_0)$, \ldots, $f_l(x_0,y_0)$.

\section{The apolar covariant and Sylvester's Theorem}\label{Section_Apol}

Binary forms that are the $n$-th power of a linear factor, say $f(x_0,y_0)= (r\,x_0 - s\,y_0)^n$, or sums of a finite number of such polynomials, are the simplest examples possible. Since the expression of a binary form $f(x_0,y_0)$ may be rather complicated, it is preferable to have available reducibility criteria for $f(x_0,y_0)$: the simplest form to which a binary form can be reduced is usually called \textit{canonical form}\index{Binary form!Canonical Form}. \\

With the notation introduced in the previous section, let $l=2$ and consider a partition of $\mathbb{N}$ into two infinite sets $\mathbb{N}_1 \cup \mathbb{N}_2=\mathbb{N}$.  Without loss of generality, assume that $1\in \mathbb{N}_1$ and $2\in \mathbb{N}_2$ (more generally, $1$ can be replaced by any $i \in \mathbb{N}_1$ and $2$ by any $j \in \mathbb{N}_2$). Finally, set $\bs{n}=(n,m)$ with $n\geq m$. 

\begin{defn} The polynomial in $\C[a_{10},a_{11},\ldots,a_{1n};a_{20},a_{21},\ldots,a_{2m};x_0,y_0]$ defined by:
\begin{equation}\label{apol}
\mathcal{A}(a_{10},a_{11},\ldots,a_{2n};a_{20},a_{21},\ldots,a_{2m};x_0,y_0)=\U\,[1\;0]^{n-m}[2\;1]^{m} \, ,
\end{equation}
is called the \textbf{apolar covariant} \index{Joint covariant!Apolar covariant} (here, uniqueness is meant up to a multiplicative constants). If binary forms $f_1(x_0,y_0)$ of degree $n$ and $f_2(x_0,y_0)$ of degree $m$ are given, $f_1$ and $f_2$ are said to be \textbf{apolar} if and only if $\mathcal{A}(f_1,f_2)(x_0,y_0)=0$ identically. The bilinear form $\{\cdot,\cdot\}$ induced by the apolar covariant via $\{f_1,f_2\}=\mathcal{A}(f_1,f_2)(x_0,y_0)$ is called the \textbf{apolar form}.
\end{defn}

For $n,m\in\mathbb{N}$ with $m\leq n$, assume that $l=2m-n\geq 1$ and let $f(x_0,y_0)$ and $g(x_0,y_0)$ be of degree $n$ and $m$,
respectively. From \eqref{apol}, it follows that
\begin{equation}\label{apolsum}\mathcal{A}(f,g)(x_0,y_0)\\=\U(f)\,\sum_{k=0}^{n-m}\binom{n-m}{k}(-1)^{n-m-k}x_{1}^ky_{1}^{n-m-k}\,g(x_{1},y_{1})\,x_0^{n-m-k}y_0^{k},
\end{equation}
where $\U$ is the umbral operator in \eqref{E(n)} and \eqref{E(n)bis}, so that $\U(f)\,p$ equals $\U\,p$ evaluated at the coefficients of $f(x_0,y_0)$. This means that $\{f,g\}=0$ if and only if
\begin{equation}\label{apol1}\U(f)\,x_{1}^ky_1^{n-m-k}\,g(x_{1},y_1)=0 \quad \forall\, k=0,\dots, n-m.
\end{equation}

The set of all the binary forms of degree $m$, which are apolar to a given form of degree $n$, is a $\C$-vector space, whose properties are summarized in the next statement.

\begin{prop}
For all $n\in \mathbb{N}$, let $V_n$ denote the $\C$-vector space of the binary forms of degree $n$. 
\begin{enumerate}
\item[(i)]  If $m\leq n$, every joint covariant map from $V_n \times V_m$ to $V_{n-m}$ is a constant multiple of the apolar form $\{\cdot, \cdot\}: V_n \times V_m \rightarrow V_{n-m}$ (see \cite[Lemma 5.1]{KungRota}).
\item[(ii)] If $f \in V_n, g \in V_m, h \in V_r$, with $m+r \leq n$, then $\{f, g\,h\} = \{\{f,g\},h\}$. In particular, $\{f,g\}=0$ implies $\{f,g\,h\}=0$ for every $h \in V_r$ (see \cite[Lemma 5.2, Corollary 5.1]{KungRota}).

\item[(iii)] Let $g \in V_m$ be a non-zero form. For every $n\geq m$, the dimension of the space $V_{m,n}(g)$ of the binary forms of degree $n$, apolar to $g$, equals $m$ (see \cite[Proposition 5.1]{KungRota}).
\item[(iv)] Let $n,m\in\mathbb{N}$, with $m\leq n$, and take $f \in V_n$, as in \eqref{binf}. The dimension of the $\C$-vector space $V_{n,m}(f)$ of the forms $g(x_0,y_0)= \sum_{h=0}^m \binom{m}{h}(-1)^h b_h x_0^h y_0^{m-k}$ of degree $m$ which are apolar to $f$ is at least $2m-n$. More precisely, it has dimension $m-r+1$, where $r$ is the rank of the system of linear equations given by the condition $\{f, g\}=0$, namely
$$ \sum_{h=0}^m \binom{m}{h}(-1)^h b_h a_{k+h} = 0 \quad \text{ for } k=0,\dots, n-m $$
(see \cite[Proposition 5.2 and Corollary 5.2]{KungRota}).\\
\end{enumerate}
\end{prop}

\begin{rmk}
The matrix $M_{n,m}$ of the linear system  arising from $\{f,g\}=0$  is the Hankel matrix $(a_{i+j})_{\substack{i=0,\dots,n-m \\ j=0,\dots,m}}$ in the coefficients $a_0,a_1,\dots,a_n$ of the binary form $f$. In this setting, these coefficients arise as the moments of  the linear functional $\U(f)$. For the purposes of the present discussion, it will be always assumed that all the Hankel determinants $\det(a_{i+j})_{i,j=0,\dots,k}$ are non-zero, for every $k\geq 1$. Under these assumptions, $M_{n,m}$ has maximum rank $r=n-m+1$, implying that the dimension of $V_{n,m}(f)$ equals $2m-n$. For instance, this is the case when the coefficients $a_k$'s are the moments of a probability measure admitting an OPs (see \cite[Theorem 3.1]{Chihara}.\\
\end{rmk}

The most celebrated theorem about apolarity is  Sylvester's Theorem \cite[Theorem 5.1]{KungRota}\index{Binary form!Sylvester's Theorem}, dealing with the case $2m-n=1$.

\begin{thm}[Sylvester's Theorem]
Let $f(x_0,y_0)$ be a binary form of odd degree $n=2m-1$. Then, there exists a unique non-zero form $g(x_0,y_0)$ of degree $m$, uniquely determined up a to multiplicative factor, such that $\mathcal{A}(f,g)=0$. Moreover, if $g(x_0,y_0)$ can be written as the product of $m$ distinct linear factors $r_i x_0 - s_i y_0$, for $i=1,\dots,m$, then there exist unique  $c_1,\dots,c_m \in \mathbb{C}$ such that
$$ f(x_0, y_0) = \sum_{i=1}^m c_i (r_i x_0 - s_i y_0)^n. \\$$
\end{thm}

Within the language of invariant theory and the notation introduced so far, Sylvester's Theorem says that there exists a covariant $\mathcal{J}(a_0,a_1,\ldots,a_{2m-1};x_0,y_0)$ of binary forms of degree \linebreak$n=2m-1$, of order $m$, such that if $f(x_0,y_0)$ is of degree $2m-1$ and $g(x_0,y_0)=\mathcal{J}(f)(x_0,y_0)$, then  $g(x,y)$ is a form of degree $m$, satisfying $\{f,g\}=0$. The covariant $\mathcal{J}(a_0,a_1,\ldots,a_{2m-1};x_0,y_0)$ so introduced is customarily referred to as the \textit{covariant} $J$ \cite{Rota98}\index{Binary form!Covariant $J$}.

In \cite[Lemma 5.3]{KungRota}, other than a symbolic expression, the authors provided an explicit
determinantal formula for the covariant $J$, which will be here generalized to a wider family of joint-covariants.\\

For $l\geq 2$, fix a partition $\mathbb{N}_1\cup\mathbb{N}_2\cup\cdots\cup\mathbb{N}_{l}=\mathbb{N}$, where $\mathbb{N}_1,\dots,\mathbb{N}_l$ are pairwise disjoint infinite sets. If $\bs{n}=(n,m,\ldots,m)\in\N^{l}$, consider the joint-covariant defined by:
\begin{equation}\label{eq:main1}
\mathcal{J}_{n,m}(\ldots;a_{i0},a_{i1},\ldots,a_{in_i};\ldots;x_0,y_0):=\U\,\prod_{1\leq i<j\leq n-m+1}[j\,i]\prod_{0\leq i<j\leq m}[j\,i]
\end{equation}
where, without loss of generality, it is assumed that $1,2,\ldots,n-m+1\in \mathbb{N}_1$, $n-m+2\in \mathbb{N}_2$, $\dots, n-m+s \in \mathbb{N}_s$, \dots,
$\ldots$, $m\in \mathbb{N}_{l}$, and $\U$ is the umbral operator defined in \eqref{E(nn)}, \eqref{E(nn)bis}. If a form $f(x_0,y_0)$ of degree $n$ is given, then a form $g(x_0,y_0)$ of degree $m$, and apolar to $f(x_0,y_0)$, can be obtained by suitably replacing each
$a_{ij}$ in $\mathcal{J}_{n,m}(\ldots;a_{i0},a_{i1},\ldots,a_{in_i};\ldots;x_0,y_0)$ with an element in $\C$, as made precise in Theorem \ref{Th:main1}, with the help of the following \textit{vanishing criterion}. 

\begin{lemma}[Vanishing criterion] Let $\U$ denote the operator defined by \eqref{E(nn)} and $\eqref{E(nn)bis}$, and assume that $p\in\C[\bs{x},\bs{y}]$ changes in sign by permuting two pairs of its indeterminates, say $(x_{i_1},y_{i_1})$ and $(x_{i_2},y_{i_2})$, with $i_1,i_2\in\mathbb{N}_i$, for some $i$. Then, $\U\,p=0$.
\end{lemma}
\begin{proof}
According to \eqref{E(nn)} and \eqref{E(nn)bis}, since $i_1,i_2\in\mathbb{N}_i$ for the same $i$, $\U\,p$ does not change if $(x_{i_1},y_{i_1})$ and $(x_{i_2},y_{i_2})$ are exchanged. On the other hand, under this swapping, $p$ changes in sign, implying $\U\,p=-\U\,p$, and therefore $\U\,p=0$.
\end{proof}
\begin{thm}\label{Th:main1}
Let $f(x_0,y_0)$ be a binary form of degree $n$, and let $m\leq n$ be such that $l=2m-n\geq 1$. For a sequence $(f_i(x_0,y_0))_{1\leq i\leq l}$ of forms of degrees $\bs{n}=(n,m,\ldots,m) \in \mathbb{N}^l$ such that $f_1(x_0,y_0)=f(x_0,y_0)$, set
\[g(x_0,y_0)=\mathcal{J}_{n,m}(f_1,f_2,\ldots,f_{l})(x_0,y_0).\]
Then, $g(x_0,y_0)=0$ or, if $f_1,\dots,f_l$ are linearly independent,  $g(x_0,y_0)$ is a form of degree $m$ such that $\{f,g\}=0$.
\end{thm}
\begin{proof}
Fix a partition $\mathbb{N}_1\cup \mathbb{N}_2\cup\cdots\cup \mathbb{N}_{l}=\mathbb{N}$ of $\mathbb{N}$ into infinite subsets, and assume that $1,2,\ldots,n-m+1\in \mathbb{N}_1$, $n-m+2\in \mathbb{N}_2$, $\ldots$, $m\in \mathbb{N}_{l}$. Consider the polynomial:
\begin{equation}
\label{q}
q(x_0,y_0,x_1,y_1,\ldots,x_m,y_m)\\
=y_1^{n-m}\,x_{2}y_2^{n-m-1}\,\cdots x_{k+1}^{k}y_{k+1}^{n-m-k}\,\cdots\,x_{n-m+1}^{n-m}\,\prod_{0\leq h<j\leq m}[j\;h],
\end{equation}
and choose $i\in\mathbb{N}_1\setminus\{1,2,\ldots,n-m+1\}$. Then, for every $ k=0,\dots, n-m$, since $n-m+1 \leq m$,
$$\left(x_{i}^ky_i^{n-m-k}\,q(x_i,y_i,x_1,y_1,\ldots,x_m,y_m)\right)^\tau=-x_{i}^ky_i^{n-k}\,q(x_i,y_i,x_1,y_1,\ldots,x_m,y_m),$$
where $\tau$ is the transposition of $\{1,2,\ldots,n-m+1,i\}$ such that $\tau(i)=k+1$. Since $k+1,i\in\mathbb{N}_1$,  the vanishing criterion applies, yielding:
\begin{equation}\label{eq1}
\U\,x_i^{k}y_i^{n-m-k}\,q(x_i,y_i,x_1,y_1,\ldots,x_m,y_m)=0 \quad \forall \, k=0,\dots, n-m.
\end{equation}
Moreover, if $h(x_0,y_0):=\U(f_1,$ $f_2,\dots,f_{l})\,q(x_0,y_0, $ $x_1,y_1, \dots, x_m,y_m)$, then $h(x_0,y_0)=0$ or $h(x_0,y_0)$ is a form of degree $m$ satisfying
\begin{equation}\label{eq2}
\U(f)\,x_i^{k}y_i^{n-m-k}\,h(x_i,y_i)=0 \quad \forall k=0,\dots, n-m.
\end{equation}
By virtue of \eqref{E(nn)} and \eqref{E(nn)bis}, the pair $(x_i,y_i)$ may be replaced in \eqref{eq2} by any pair
$(x_j,y_j)$, such that $j\in\mathbb{N}_1$. For $(x_j,y_j)=(x_1,y_1)$, it follows that $\{f,h\}=0$. Then, symmetrizing
$q(x_0,y_0,x_1,y_1,\ldots,x_m,y_m)$ with respect to $(x_1,y_1)$, $(x_2,y_2)$, $\ldots$, $(x_{n-m+1},y_{n-m+1})$, it follows that:
\[\sum_{\sigma\in\mathfrak{S}_{n-m+1}}\left(q(x_0,y_0,x_1,y_1,\ldots,x_m,y_m)\right)^\sigma=\prod_{1\leq i<j\leq n-m+1}[j\;i]\prod_{0\leq i<j\leq m}[j\;i] \, .\]
Indeed, consider the polynomial $Q(x_0,x_1,\dots,x_m)$ obtained from $q(x_0,y_0,\dots,x_m,y_m)$ by setting $y_i=1$ for all $i=0,\dots,m$, namely:
$$ Q(x_0,x_1,\dots,x_m) := x_2 x_3^2 \cdots x_{n-m+1}^{n-m} \Delta(x_0,\dots,x_m). $$ Since $\Delta(x_0,\dots,x_m)^{\tau} = (-1)^{\tau}\Delta(x_0,\dots,x_{m})$  for every $\tau \in \mathfrak{S}_{n-m+1}$,  the symmetrization of $Q(x_0,\dots,x_m)$ over $\mathfrak{S}_{n-m+1}$ can be written as:
\begin{align*}
\frac{1}{(n-m+1)!}&\sum\limits_{\tau \in \mathfrak{S}_{n-m+1}}\big(Q(x_0,x_1,\dots,x_m)\big)^{\tau} = \\
&= \frac{1}{(n-m+1)!}\Delta(x_0,\dots,x_m) \sum_{\tau \in \mathfrak{S}_{n-m+1}} (-1)^{\tau}x_{\tau(1)}^{0}x_{\tau(2)} x_{\tau(3)}^{2} \dots x_{\tau(n-m+1)}^{n-m} \\
&= \frac{1}{(n-m+1)!}\Delta(x_1,\dots,x_{n-m+1})\Delta(x_0,x_1,\dots,x_m).
\end{align*}
Then, the conclusion for $q(x_0,x_1,\dots,x_m)$ follows considering that:
$$ q(x_0,y_0,\dots,x_m,y_m) = y_0^{m} \prod_{j=1}^{n-m+1}y_j^{n} \prod_{r=n-m+2}^{m}y_r^{m} Q\bigg(\frac{x_0}{y_0},\dots,\frac{x_m}{y_m}\bigg). $$
Indeed,
\begin{align*}
\sum_{\sigma\in\mathfrak{S}_{n-m+1}}\big(q(x_0,y_0,&x_1,y_1,\ldots,x_m,y_m)\big)^{\sigma} = \\
&=\sum_{\sigma\in\mathfrak{S}_{n-m+1}}\bigg(y_0^{m} \prod_{j=1}^{n-m+1}y_j^{n} \prod_{r=n-m+2}^{m}y_r^{m} Q\bigg(\frac{x_0}{y_0},\dots,\frac{x_m}{y_m}\bigg)\bigg)^{\sigma} \\
&= y_0^{m} \prod_{j=1}^{n-m+1}y_j^{n} \prod_{r=n-m+2}^{m}y_r^{m} \sum_{\sigma\in\mathfrak{S}_{n-m+1}} \bigg( Q\bigg(\frac{x_0}{y_0},\dots,\frac{x_m}{y_m}\bigg)\bigg)^{\sigma}
\end{align*}
and
\begin{align*}
\sum_{\sigma\in\mathfrak{S}_{n-m+1}} & \bigg(Q\bigg(\frac{x_0}{y_0},\dots,\frac{x_m}{y_m}\bigg)\bigg)^{\sigma} = \Delta\bigg(\frac{x_0}{y_0},\dots,\frac{x_m}{y_m}\bigg)\Delta\bigg(\frac{x_1}{y_1},\dots,\frac{x_{n-m+1}}{y_{n-m+1}}\bigg)\\
&= \dfrac{1}{y_0^{m}}\prod_{j=1}^{n-m+1}\dfrac{1}{y_j^{n}}\prod_{r=n-m+2}^m \dfrac{1}{y_r^{m}} \prod_{1\leq i<j\leq n-m+1}[j\;i]\prod_{0\leq i<j\leq m}[j\;i].
\end{align*}

On the other hand, since $1,2,\ldots,n-m+1\in \mathbb{N}_1$, then 
$$\U\,\left(q(x_0,y_0,x_1,y_1,\ldots,x_m,y_m)\right)^\sigma=\U\,q(x_0,y_0,x_1,y_1,\ldots,x_m,y_m)$$
for all $\sigma\in\mathfrak{S}_{n-m+1}$, implying that:
\[g(x_0,y_0)=\mathcal{J}_{n,m}(f_1,f_2,\ldots,f_{2m-n})(x_0,y_0)=(n-m+1)!\,h(x_0,y_0),\]
and, finally, that $\{f,g\}=(n-m+1)!\,\{f,h\}=0$.
\end{proof}

Since the space of all the forms of degree $m$, which are apolar to a given form of degree $n$, has dimension $2m-n$,  $2m-n=1$ (i.e. $n=2m-1$) and $2m-n=m$ (i.e. $n=m$) are the minimum and the maximum values respectively for which a form $g(x_0,y_0)$ of degree $m$, and apolar to $f(x_0,y_0)$ of degree $n$ ,exists. So, when $2m-n=1$, $f(x_0,y_0)$ is of degree $2m-1$, $g(x_0,y_0)=\mathcal{J}_{n,m}(f)(x_0,y_0)\neq 0$ is of degree $m$, and $\{f,g\}=0$. In particular, since $n-m+1=m$, the covariant \eqref{eq:main1} reduces, up to a sign, to the covariant $J$ of Kung and Rota \cite{KungRota}:
\begin{equation*}
\mathcal{J}_{2m-1,m}(\ldots;a_{i0},a_{i1},\ldots,a_{in_i};\ldots;x_0,y_0)=\U\,\prod_{1\leq i<j\leq m}[j\;i]\prod_{0\leq i<j\leq m}[j\;i].
\end{equation*}

The determinantal formula for $\mathcal{J}_{n,m}(\ldots;a_{i0},a_{i1},\ldots,a_{in_i};\ldots;x_0,y_0)$ is provided by the next theorem.
\begin{thm}\label{Th:det}
Let $n,m\in\mathbb{N}$ with $l=2m-n\geq 1$ and let $\mathcal{J}_{n,m}(\ldots;a_{i0},a_{i1},\ldots,a_{in_i};\ldots;x_0,y_0)$ be the joint-covariant defined in \eqref{eq:main1}. Then:
\begin{multline}\label{det}
\mathcal{J}_{n,m}(\ldots;a_{i0},a_{i1},\ldots,a_{in_i};\ldots;x_0,y_0)\\=\frac{1}{(n-m+1)!}\,\begin{vmatrix}
y_0^m&x_0y_0^{m-1}&x_0^2y_0^{m-2}&\ldots&x_0^m\\
a_{1\,0}&a_{1\,1}&a_{1\,2}&\ldots&a_{1\,m}\\
a_{1\,1}&a_{1\,2}&a_{1\,3}&\ldots&a_{1\,m+1}\\
\vdots&\vdots&\vdots&&\vdots\\
a_{1\,n-m}&a_{1\,n-m+1}&a_{1\,n-m+2}&\ldots&a_{1\,n}\\
a_{2\,0}&a_{2\,1}&a_{2\,2}&\ldots&a_{2\,m}\\
\vdots&\vdots&\vdots&&\vdots\\
a_{l\,0}&a_{l\,1}&a_{l\,2}&\ldots&a_{l\,m}\\
\end{vmatrix}.
\end{multline}
\end{thm}
\begin{proof}
Fix a partition $\mathbb{N}_1, \mathbb{N}_2,\dots, \mathbb{N}_{l}$ of $\mathbb{N}$ into infinite subsets, and assume that $1,2,\ldots,n-m+1\in \mathbb{N}_1$, $n-m+2\in \mathbb{N}_2$, $\ldots$, $m\in \mathbb{N}_{l}$. By virtue of \eqref{E(nn)} and \eqref{E(nn)bis}, it follows that:
\small{
\begin{align*}
&\U\,\begin{vmatrix}
y_0^m&x_0y_0^{m-1}&x_0^2y_0^{m-2}&\ldots&x_0^m\\
y_1^n&x_1y_1^{n-1}&x_1^2y_1^{n-2}&\ldots&x_1^my_1^{n-m}\\
x_2y_2^{n-1}&x_2^2y_2^{n-2}&x_2^3y_2^{n-3}&\ldots&x_2^{m+1}y_2^{n-m-1}\\
\vdots&\vdots&&&\vdots\\
x_{n-m+1}^{n-m}y_{n-m+1}^{m}&x_{n-m+1}^{n-m+1}y_{n-m+1}^{m-1}&x_{n-m+1}^{n-m+2}y_{n-m+1}^{m-2}&\ldots&x_{n-m+1}^{n}\\
y_{n-m+2}^m&x_{n-m+2}y_{n-m+2}^{m-1}&x_{n-m+2}^2y_{n-m+2}^{m-2}&\ldots&x_{n-m+2}^m\\
\vdots&\vdots&&&\vdots\\
y_{m}^m&x_{m}y_{m}^{m-1}&x_{m}^2y_{m}^{m-2}&\ldots&x_{m}^m\\
\end{vmatrix}\\
&=\begin{vmatrix}
\U\,y_0^m&\U\,x_0y_0^{m-1}&\ldots&\U\,x_0^m\\
\U\,y_1^n&\U\,x_1y_1^{n-1}&\ldots&\U\,x_1^my_1^{n-m}\\
\U\,x_2y_2^{n-1}&\U\,x_2^2y_2^{n-2}&\ldots&\U\,x_2^{m+1}y_2^{n-m-1}\\
\vdots&\vdots&&\vdots\\
\U\,x_{n-m+1}^{n-m}y_{n-m+1}^{m}&\U\,x_{n-m+1}^{n-m+1}y_{n-m+1}^{m-1}&\ldots&\U\,x_{n-m+1}^{n}\\
\U\,y_{n-m+2}^m&\U\,x_{n-m+2}y_{n-m+2}^{m-1}&\ldots&\U\,x_{n-m+2}^m\\
\vdots&\vdots&&\vdots\\
\U\,y_{m}^m&\U\,x_{m}y_{m}^{m-1}&\ldots&\U\,x_{m}^m\\
\end{vmatrix} \\
&=\begin{vmatrix}
y_0^m&x_0y_0^{m-1}&x_0^2y_0^{m-2}&\ldots&x_0^m\\
a_{1\,0}&a_{1\,1}&a_{1\,2}&\ldots&a_{1\,m}\\
a_{1\,1}&a_{1\,2}&a_{1\,3}&\ldots&a_{1\,m+1}\\
\vdots&\vdots&\vdots&&\vdots\\
a_{1\,n-m}&a_{1\,n-m+1}&a_{1\,n-m+2}&\ldots&a_{1\,n}\\
a_{2\,0}&a_{2\,1}&a_{2\,2}&\ldots&a_{2\,m}\\
\vdots&\vdots&\vdots&&\vdots\\
a_{l\,0}&a_{l\,1}&a_{l\,2}&\ldots&a_{l\,m}\\
\end{vmatrix} \, \\
&= \U q(x_0,y_0,x_1,y_1,\ldots,x_m,y_m),
\end{align*}
}
\normalsize
where $q(x_0,y_0,x_1,y_1,\ldots,x_m,y_m)$ is the polynomial defined in \eqref{q}. Indeed,  
$$q(x_0,y_0,x_1,y_1,\ldots,x_m,y_m) = y_0^m y_1^n y_2^n\cdots y_{n-m+1}^n  y_{n-m+2}^m\cdots y_m^{m}Q\bigg(\frac{x_0}{y_0},\frac{x_1}{y_1},\dots,\frac{x_m}{y_m}\bigg) \, ,$$
with $Q(x_0,x_1,\dots,x_m) = x_2 x_3^2\cdots x_{n-m+1}^{n-m}\Delta(x_0,x_1,\dots,x_m)$, namely
\begin{align*}
q(x_0,y_0,x_1,y_1,&\ldots,x_m,y_m) = \\
&=\begin{vmatrix}
y_0^m&x_0y_0^{m-1}&x_0^2y_0^{m-2}&\ldots&x_0^m\\
y_1^n&x_1y_1^{n-1}&x_1^2y_1^{n-2}&\ldots&x_1^my_1^{n-m}\\
x_2y_2^{n-1}&x_2^2y_1^{n-2}&x_2^3y_2^{n-3}&\ldots&x_2^{m+1}y_2^{n-m-1}\\
\vdots&\vdots&\vdots&&\vdots\\
x_{n-m+1}^{n-m}y_{n-m+1}^{m}&x_{n-m+1}^{n-m+1}y_{n-m+1}^{m-1}&x_{n-m+1}^{n-m+2}y_{n-m+1}^{m-2}&\ldots&x_{n-m+1}^{n}\\
y_{n-m+2}^m&x_{n-m+2}y_{n-m+2}^{m-1}&x_{n-m+2}^2y_{n-m+2}^{m-2}&\ldots&x_{n-m+2}^m\\
\vdots&\vdots&\vdots&&\vdots\\
y_{m}^m&x_{m}y_{m}^{m-1}&x_{m}^2y_{m}^{m-2}&\ldots&x_{m}^m\\
\end{vmatrix} \, .
\end{align*}
Finally, \eqref{det} follows by symmetrizing $q(x_0,y_0,x_1,y_1,\ldots,x_m,y_m)$ with respect to $(x_1,y_1)$, $(x_2,y_2)$, \ldots, $(x_{n-m+1},y_{n-m+1})$ (since $1,\dots,n-m+1 \in \mathbb{N}_1$, as in the proof of Theorem \ref{Th:main1}).
\end{proof}

\subsection{Apolarity and orthogonality}

Formula \eqref{det} corresponds to \eqref{detp} whenever $y_i=1$ for all $i=0,\dots,m$. This is a consequence of the fact that orthogonality (in the general sense of \cite{Chihara}) is as a realization of apolarity, and vice versa, in the sense of  Theorem \ref{ApolarityOrth} below.\\

Assume that infinitely many, pairwise disjoint infinite sets $\mathbb{N}_0,\mathbb{N}_1,\mathbb{N}_2,\ldots$ are given, such that  $\mathbb{N}_0\cup\mathbb{N}_1\cup\mathbb{N}_2\cup\cdots=\mathbb{N}_0$, and that a subset $\{a_{ij}\,|\,i,j\in\N\}$ of $\C$ is given, with $a_k=a_{0k}=a_{1k}$ for all $k\in\N$. Consider the linear functional $\EE:\C[\bs{x}]\rightarrow \C$ defined by
\begin{align}
\label{E}\EE\,x_i^{k}&=a_{jk} \, \text{ if and only if }\, i\in\mathbb{N}_j,\\
\label{Ebis}\EE\,x_0^{k_0}x_1^{k_1}x_2^{k_2}\cdots
&=\EE\,x_0^{k_0}\,\EE\,x_1^{k_1}\,\EE\,x_2^{k_2}\cdots,
\end{align}
for every choice of non-negative integers $i,j,k,k_0,k_1,k_2,\ldots\in\N_0$ (in the sequel, $a_{jk}$ will be said the $k$-th moment of $\EE$ on $\mathbb{N}_j$). Any linear functional $\EE\colon\C[\bs{x}]\to\C$ of the type \eqref{E} and \eqref{Ebis} can be equivalently determined by the set $\{f_{jn}(x_0,y_0)\,|\,j,n\in\N\}$ of binary forms defined by:
\begin{equation}
\label{Lforms}f_{jn}(x_0,y_0)=\sum_{k=0}^n\binom{n}{k}(-1)^{n-k}a_{jk}x_0^{n-k}y_0^k.
\end{equation}
Therefore, for all $j,n\in\N$,
\begin{equation}\label{UvsE}
\U(f_{jn})\,x_i^{k}y_i^{n-k}=\EE\,x_i^k \quad \forall k=0,\dots,n \,,\, i\in\mathbb{N}_j.
\end{equation}
On the other hand, let $\EE_0\colon\C[\bs{x}]\to\C[x_0]$ denote the
linear operator defined by:
\begin{equation}
\label{E0}\EE_0\,x_0^{k_0} x_{1}^{k_1} x_{2}^{k_2}\cdots= x_0^{k_0}\EE\,x_{1}^{k_1} x_{2}^{k_2}\cdots
\end{equation}
for every non-negative integers $k_0,k_1,k_2,\ldots\in\N_0$, and then extended by linearity. Besides, let $\bs{e}_1\colon\C[\bs{x},\bs{y}]\to\C[x_0]$ denote the map evaluating each $p\in\C[\bs{x},\bs{y}]$ at $y_i=1$, for all $i$. Hence, by comparing
\eqref{E(nn)} and \eqref{E(nn)bis} with \eqref{E} and
\eqref{Ebis}, it follows that
\begin{equation}\label{UvsE0}\U(f_{j_1n_1},f_{j_2n_2},\ldots,f_{j_ln_l})\,p=\EE_0\,\bs{e}_1(p),
\end{equation}
for all $p\in\C[\bs{x},\bs{y}]$ whose indeterminates $(x_i,y_i)$'s satisfy $i\in\mathbb{N}_{j_1}\cup\mathbb{N}_{j_2}\cup\cdots\cup\mathbb{N}_{j_l}$. In particular, for every $N,M$ with $M\leq N$ and $l=2M-N \geq 1$, via \eqref{UvsE0}, the covariant defined in \eqref{eq:main1} can be written as:
\begin{equation}
\label{ApolOrth}
\mathcal{J}_{N,M}(f_{1\,N},f_{2\,M},\ldots,f_{l\,M})(x_0,1)\\=\EE_0\,\Delta(x_1,x_2,\ldots,x_{N-M+1})\Delta(x_0,x_1,\ldots,x_{M}),\\
\end{equation}
with $\EE_0$ as in \eqref{E0}. In particular, assume that $a_k = a_{jk}$ for all $j\geq 1$, namely assume that the restriction of $\EE$ to  $\C[x_0]$  is uniquely determined by the sequence $\{f_n(x_0,y_0)\}_{n\geq 1}$  of binary forms
\begin{equation}
\label{Lform}f_n(x_0,y_0)=\sum_{k=0}^n\binom{n}{k}(-1)^{n-k}a_k\,x_0^{n-k}y_0^k\, \quad a_n \neq 0 \, \forall n \in \mathbb{N}.
\end{equation}

Let $\{p_{nm}(x_0)\}_{n,m\geq 1}:=\{p_{nm}(x_0)\,|n \in \mathbb{N},m=1,\dots, n\}$ be a triangular array of polynomials in $\C[x_0]$, satisfying $\deg\,p_{nm}(x_0)=n$ for every $m=1,\dots, n$. 

\begin{defn}
The triangular array $\{p_{nm}(x_0)\}_{n,m\geq 1}$ is called a \textbf{generalized orthogonal polynomial system} (GOPs, for short) for $\EE: \C[x_0]\rightarrow \C$ if, and only if, for every $n \in \mathbb{N}$ and every $ m\leq n$,
\begin{equation}\label{genOPS2}
\EE[x_0^k \,p_{nm}(x_0)]=0 \quad \forall \, k= 0,\dots,n-m,
\end{equation}
and $\EE[x_0^{n-m+1} \,p_{nm}(x_0)] \neq 0$.\\
\end{defn}

Then, if $\{p_{nm}(x_0)\}_{n,m\geq 1}$ is a GOPs for $\EE$, for every $i\in\mathbb{N}$, by comparing \eqref{E(n)bis} and \eqref{Lform}, it follows that:
\[\U(f_n)\,x_i^{k}y_i^{n-k}=\EE[x_0^k] \quad \forall k=0,\dots,n \, ,\]
implying that the orthogonality condition \eqref{genOPS2} can be restated in the equivalent form:
$$\U(f_{2n-m})\,x_1^{k}y_1^{n-m-k}\,g_{nm}(x_1,y_1)=0 \qquad \forall k=0,\dots,n-m,$$
where $g_{nm}(x_0,y_0):=y_0^n\,p_{nm}\bigg(\dfrac{x_0}{y_0}\bigg)$, or also as:
\begin{equation}\label{apol2}
\{f_{2n-m},g_{nm}\}=0 \qquad \forall m=1,\dots,n.
\end{equation}

All the previous considerations are gathered in the next statement.
\begin{thm}
\label{ApolarityOrth}
Let $\EE$ be a linear functional satisfying \eqref{E} and \eqref{Ebis}. For every $n\in \mathbb{N}$ and every $m=1,\dots,n$, let $f_{2n-m}(x_0,y_0)$ be the binary form associated with $\EE$, as in  \eqref{Lform},  and assume that $g_2(x_0,y_0),\dots,g_m(x_0,y_0)$  are (linearly independent) binary forms of degree $n$. Then, every generalized orthogonal polynomial system
$\{p_{nm}(x_0)\}_{n,m\geq 1}$ for $\EE$ corresponds to a set 
$\{g_{nm}(x_0,y_0)\}_{n,m\geq 1}$ of binary forms such that $g_{nm}(x_0,y_0)$ is of
degree $n$, and is apolar to $f_{2n-m}(x_0,y_0)$, via $g_{nm}(x_0,y_0):=y_0^n\,p_{nm}\bigg(\dfrac{x_0}{y_0}\bigg)$. Dually, 
$$ p_{nm}(x_0) = \mathcal{J}_{2n-m,n}(f_{2n-m},g_2,\dots,g_m)(x_0,1),$$
with the covariant $\mathcal{J}$ as defined in \eqref{eq:main1}.
\end{thm}

\begin{proof}
Apply relation \eqref{ApolOrth} with $N=2n-m$ and $M=n$, and follow the same steps required in the proof of Theorem \ref{GOPS1}, replacing the random variable $X_j$ with the indeterminate $x_j$.
\end{proof}

\begin{rmk}
If $N=2M-1$ in \eqref{ApolOrth}, 
$$\mathcal{J}_{2M-1,M}(f_{1\,2M-1})(x_0,1) = \EE_0 \big[ \Delta(x_1,\dots,x_M)^2 \prod_{j=0}^M (x_j-x_0)\big]$$
corresponds to the  representation of the orthogonal polynomial $p_{M,1}(x_0)$ extracted from a GOPs for $\EE$. 
\end{rmk}

\begin{exm}
If $2m-n=m$, and $g(x_0,y_0)=\mathcal{J}_{n,m}(f,g_2,\ldots,g_m)(x_0,y_0)\neq 0$, both $g(x_0,y_0)$ and $f(x_0,y_0)$ are of degree $m$. In this case, since $n-m+1=1$, the apolarity condition $\{f,g\}=0$ can be equivalently referred to by saying that the covariants
\[\mathcal{J}_{m,m}(\ldots;a_{i0},a_{i1},\ldots,a_{im};\ldots;x_0,y_0)=\U\,\prod_{0\leq i<j\leq m}[j\;i]\]
correspond to the biorthogonal polynomials studied by Iserles and Norsett \cite{IserNor}: for every $m \geq 1$, assume that a set of distributions $\{F(x,\mu_l):l=1,\dots,m\}$ is given, where $\{\mu_1,\dots,\mu_n,\dots\}$ is a set of real parameters. A monic polynomial $p_m(x;\mu_1,\dots,\mu_m)$ of degree $m$ is said to satisfy the \textit{biorthogonality condition} if
$$ \int_{\mathbb{R}}p_m(x;\mu_1,\dots,\mu_m) dF(x,\mu_l) \,= 0    \quad \forall \, l=1,\dots, m. $$
For every $j=1,\dots,m$, set $h_j(x_0,y_0) = \sum\limits_{k=0}^m \binom{m}{k}(-1)^k b_{j,k}x_0^{m-k}y_0^k$, where
$b_{j,k} = \int_{\mathbb{R}} x^k d F(x,\mu_j)$. Then, by virtue of Theorem \ref{ApolarityOrth}, the polynomial 
$$p_m(x_0; h_1,\dots,h_m) := \mathcal{J}_{m,m}(h_1,h_2,\dots,h_m)(x_0,1) $$  satisfies the biorthogonality condition. In particular, the determinantal representation provided in \cite[Theorem 1]{IserNor} is recovered via \eqref{det}; likewise, for the integral representation \cite[Theorem 2]{IserNor}, which is a particular case of \eqref{ApolOrth}.
\end{exm}

\section{Moments of random discriminants}\label{Section_Quadrature}

In statistics, for a simple random sample $X_1,\dots,X_n$, for $n\in \mathbb{N}$, the random discriminant \index{Random discriminant}is defined as the square of the Vandermonde polynomial in the independent and identically distributed variables $X_1,\dots,X_n$, namely $\Delta(X_1,\dots,X_n)^2$. The interest in studying this statistics is motivated by the wide range of applications it is concerned: spectral theory of random matrices and hypothesis testing, just to cite a few. In \cite{Lu}, a direct investigation of the topic is provided: more precisely, the author applies Selberg's integral formula to obtain stochastic representations for  random discriminants, when the sample is drawn from a Normal, Gamma or Beta population. Due to the identity \eqref{HankelVand}, one could in principle apply the techniques available for moment matrices to study its distribution. However, as underlined in the same reference \cite{Lu}, the aforementioned stochastic representation hints that the distribution of $\Delta(X_1,X_2,\dots,X_n)^2$ might be rather complicated, reason why stochastic bounds, or other results that can supply some information in this regard, are found to be of major interest. \\

The main achievements of this section are  explicit formulae for the moments of the random discriminant. It is worth to note that no strict assumption on the distribution of the underlying sample is required: indeed, the algebraic technique here adopted applies for the computation of $\E[\Delta(X_1,\dots,X_n)^{2k}]$, $k\geq 1$, regardless of the distribution of the population the sample is drawn from.\\

The starting point is, once more, apolarity and, in particular, Sylvester's Theorem, but in terms of linear functionals $\EE\colon\C[\bs{x}]\to\C$ and $\EE_0\colon\C[\bs{x}]\to\C[x_0]$,  rather than the umbral functional $\U\colon\C[\bs{x},\bs{y}]\to\C[x_0,y_0]$. To ease the notation and highlight the connection with the theory of orthogonal polynomials, the binary form $f_n(x,y)$ and the corresponding polynomial $P_n(x):=f_n(x,1)$ will be used interchangeably, and apolarity will be referred to polynomials instead of binary forms. The bottom line is the generalization of the symbolic expression for the covariant $J$  in terms of its \textit{homogenized roots} \cite[Algorithm 4.1]{KungRota}, to a wider family of apolar covariants.\\

In general, for a binary form $f(x,y)$ of degree $2n-1$, its covariant $J$  is the unique apolar covariant of order $n$ (see \cite[Lemma 5.3]{KungRota}). Sylvester's Theorem, at this point, does not provide any extra information about when the covariant $J$ factorizes into $n$ distinct linear factors, giving rise to a decomposition of $f$ as sum of $(2n$-$1)$-th powers of $n$ linear factors. As a consequence of Theorem \ref{ApolarityOrth}, such a representation always occurs for binary forms of the type $f_n(x,y) = \mathbb{E}[(yX -x)^{2n-1}]$, where $X$ is a random variable whose OPs $p_{n}(x)$ provides its covariant $J$ via  $J(x,y) = y^n p_n\big(\frac{x}{y}\big)$ (indeed, real orthogonal polynomials admit real and simple roots \cite[Theorem 5.2]{Chihara}). \\

Before detailing these conclusions via Theorems \ref{AppelGen} and \ref{Appel}, it might be convenient to start with a brief example.   In the following, for a given random variable $X$, consider the sequence of polynomials $A_n(x) = \mathbb{E}[(X-x)^n]$, for all $n\geq 1$. If $N \sim \mathcal{N}(0,1)$ denotes a random variable distributed according to the standard Gaussian law, and  $H_n(x)$ denotes the sequence of the monic Hermite polynomials, then
\begin{equation}
A_3(x) = -(x^3+3x) = \dfrac{1}{2}(1-x)^3 - \dfrac{1}{2}(x+1)^3 \;,
\end{equation}
with $H_2(x)= x^2-1 = (x+1)(x-1)$. Similarly,
\begin{equation}
A_5(x) = -(x^5 +10x^3 + 15 x)  = -\dfrac{2}{3}x^5	+ \dfrac{1}{6}(\sqrt{3}-x)^5 - \dfrac{1}{6}(x+ \sqrt{3})^5 \, ,
\end{equation}
with $H_3(x) = x(x- \sqrt{3})(x+ \sqrt{3})$.\\

This phenomenon is the result of the following general picture. In particular, the exactness of the classical Gauss quadrature formula (see, for instance, \cite[Theorem 6.1]{Chihara}) can be embedded in the following formulation of Sylvester's Theorem.

\begin{thm}\label{AppelGen}
Let $f(x_0,y_0)$ be a binary form of degree $2n-1$. For any linear functional $\EE:\C[\bs{x}]\rightarrow \C$ satisfying \eqref{E} and \eqref{Ebis} with $a_k=a_{jk}$ for all
$j,k\in\N$,  assume that
$g(x_0,y_0)=\mathcal{J}_{2n-1,n}(f)(x_0,y_0)$ factorizes as:
\[g(x_0,y_0)=(x_0r_1-y_0s_1)(x_0r_2-y_0s_2)\cdots(x_0r_n-y_0s_n),\]
where the linear factors are pairwise distinct, and $s_i\neq 0$ for all $i=1,\dots, n$ \footnote{The coefficients $r_i, s_i$, for $i=1,\dots,n$, are usually called \textit{homogenized roots} of the covariant $\mathcal{J}$}. Then, for every $n \in \mathbb{N}$, there exist unique complex numbers $c_{1,n},\dots,c_{n,n}$, called \textit{Christoffel numbers}, such that:
\begin{equation}\label{QuadAppel}
\EE_0\,[(x_1-x_0)^{2n-1}]=\sum_{i=1}^n c_{i,n}(\zeta_i-x_0)^{2n-1},
\end{equation}
where $\zeta_i:=r_i/s_i$ for all $i=1,\dots,n$.
\end{thm}

By expanding both sides of equation \eqref{QuadAppel} as polynomials in $x_0$, and comparing the corresponding coefficients, one has:
\begin{equation}\label{qf0}a_k=\EE[x_0^{k}]=\sum_{i=1}^n c_{i,n}\,\zeta_i^{k}\quad \forall \,k=0,\dots, 2n-1,
\end{equation}
and, more generally, for every $p(x_0)\in\C[x_0]_{2n-1}$, since $\EE[ p(x_1)]= \EE[p(x_0)]$, it follows that:
\begin{equation}\label{qf}\EE[p(x_0)]=\sum_{i=1}^n c_{i,n}\,p(\zeta_i)\, .
\end{equation}

In this framework, this classical result arises as a corollary of Sylvester's Theorem. Indeed, Sylvester's Theorem guarantees that there exists a unique solution to the system
$$ a_k = \sum_{j=1}^n c_{j,n} \zeta_j^k \qquad k=0,\dots, 2n-1 \,,$$
obtained by extracting the equations corresponding to $k=0,\dots, n-1$. Then, Cramer's rule entails that, for $ i=1,\dots, n$,
\begin{align*}
c_{i,n} &=\dfrac{\EE[\Delta(\zeta_1,\ldots,\zeta_{i-1},x_0,\zeta_{i+1},\ldots,\zeta_n)]}{\Delta(\zeta_1,\zeta_{2},\ldots,\zeta_n)} \numberthis \label{ci} \\
&= \dfrac{\EE\big[\prod\limits_{ j\neq i} (x_0 - \zeta_j)\big]}{\prod\limits_{\substack{j=1\\ j \neq i}}^n (\zeta_j-\zeta_i)}
\end{align*}

\begin{rmk}
Equivalently, the Christoffel's numbers can be computed as:
$$ c_{i,n} = \dfrac{b_n}{b_{n-1}} \dfrac{1}{p_{n-1}(\zeta_i) \prod_{\substack{ j=1 \\ j \neq i }}^n (\zeta_i - \zeta_j)}, $$
where $b_n$ denotes the leading coefficient of $p_n(x)$ (see \cite{Szego}).\\
\end{rmk}

As a matter of fact, \eqref{ci} allows to detect a nice property enjoyed by the weights $c_{i,n}$'s: the invariance under translation  (see Proposition \ref{InvTrans} below). In order to enhance this property, it might be convenient to restate \eqref{QuadAppel} for $\EE = \E$ on a fixed probability space, yielding as a consequence an explicit expression for the translated moments $\E[(X-x)^n]$ of a (real) random variable $X$.

\begin{thm}\label{Appel}
Let $X$ be a random variable, admitting moments up to all orders, and let $\{p_n(x)\}_{n\geq 1}$ denote the associated OPs. If $A_n(x) = \mathbb{E}[(X-x)^{n}]$, $n \in \N$,  then
$$ \{A_{2n-1}(x), p_n(x)\} = 0.$$
Moreover, if $r_{1,n}(X), r_{2,n}(X), \dots, r_{n,n}(X)$ denote the $n$ (real) roots of $p_n(x)$, then there exist unique complex numbers $c_{1,n}(X),\dots,c_{n,n}(X)$ such that:
$$ A_{2n-1}(x) = \sum_{j=1}^{n} c_{j,n}(X)(r_{j,n}(X)-x)^{2n-1}.$$
\end{thm}

As in the general setting, Sylvester's Theorem guarantees that there exists a unique solution to the system
$$ a_k = \sum_{j=1}^n c_{j,n}(X) r_{j,n}(X)^k \qquad k=0,\dots, 2n-1,$$
that will be given by:
\begin{equation}
\label{ciBIS}
c_{k,n}(X) =\dfrac{\E[\Delta(r_{1,n}(X),\dots,r_{k-1,n}(X),X,r_{k+1,n}(X),\dots,r_{n,n}(X))]}{\Delta(r_{1,1}(X),\dots,r_{n,n}(X))},
\end{equation}
where $a_{j}(X)= \E[X^j]$.

\begin{prop}\label{InvTrans}
If the above notation prevails, then, for every  $t \in \mathbb{R}$ and for every $k=1,\dots,n$, $c_{k,n}(X+t) = c_{k,n}(X)$ .
\end{prop}

\begin{proof}
First of all, remark that the existence of an OPs for $X+t$, for all $t \in \mathbb{R}$, follows from the invariance under translation of the statistics $\Delta(X_1,\dots,X_n)^2$. For every $t\in \mathbb{R}$, let $r_{j,n}(X+t)$ denote the $j$-th root of the orthogonal polynomial $p_n(x,X+t)=p_n(x-t,X)$ for $X+t$. Then, $r_{j,n}(X+t)= r_{j,n}(X)+t$, and hence:
\begin{align*}
c_{k,n}(X+ t) &=  \dfrac{\E\big[ \Delta(r_1(X)+t,\dots,r_{k-1}(X)+t,X+t,r_{k+1}(X)+t,\dots,r_{n}(X)+t)\big]}{\Delta(r_{1}(X)+t,\dots,r_{n-1}(X)+t)}\\
&= \dfrac{\E\big[ \Delta(r_1(X),\dots,r_{k-1}(X),X,r_{k+1}(X),\dots,r_{n}(X))\big]}{\Delta(r_{1}(X),\dots,r_{n}(X))}\\
&= c_{k,n}(X).\qedhere
\end{align*}
\end{proof}

\begin{rmk}
Cumulants are an important class of invariants under translation associated with the law of a random variable $X$ (classically, they are called \textit{semi-invariants}): it might be interesting, then, to determine if there is an explicit relation between the weights $c_{k,n}(X)$ and the cumulants $\chi_n(X)$ of the random variable $X$. Note that a representation of cumulants in terms of (generalized) Vandermonde polynomials has been provided in \cite[Theorem 4.1]{Rota3},  using umbral methods.\\
\end{rmk}

\begin{rmk}
It is worth to remark that \eqref{qf0} says that the decomposition \eqref{qf} applies, in particular, if the moment problem for the $a_k$'s as in \eqref{qf0} has a solution. Indeed, in this case, all the Hankel determinants $\det(a_{i+j})_{i,j=0,\dots,n-1}$ are positive, and the corresponding OPs for $\EE: \EE[x^k]=a_k$ exists, with $p_n(x)$ having $n$ simple roots $\zeta_1,\dots,\zeta_n$ for every $n$. \\
\end{rmk}

Via Theorem \ref{AppelGen},  \eqref{qf} can be easily generalized in the following way.
\begin{thm}\label{GenQuadFormula}
Let $\EE\colon\C[\bs{x}]\to\C$ be a linear functional satisfying \eqref{E} and \eqref{Ebis}, with $a_k=a_{jk}$ for all
$j,k\in\N$. Assume that $\{p_n(x_0)\}_{n\geq 1}$ is an orthogonal polynomial system associated with $\EE\colon\C[x_0]\to\C$ \footnote{With abuse of notation, $\EE$ is used to denote both $\EE:\C[x_0]\rightarrow \C$ and its linear extension to $\EE:\C[\bs{x}]\rightarrow\C$.}, with $p_n(x_0)$ having pairwise distinct roots $\zeta_1,\zeta_2,\ldots,\zeta_n$ for every $n$. If $P(x_1,x_2,\ldots,x_N)\in\C[\bs{x}]$  is of degree at most $2n-1$ in each $x_i$, then 
\begin{equation}\label{mqf}\EE[P(x_1,x_2,\ldots,x_N)]=\sum_{(i_1,i_2,\ldots,i_N)\atop 1\leq i_k\leq n}c_{i_1}c_{i_2}\cdots c_{i_N}\,P(\zeta_{i_1},\zeta_{i_2},\cdots,\zeta_{i_N}),\end{equation}
where $c_1,c_2,\ldots,c_n$ are given by \eqref{ci}.
\end{thm}

\begin{proof}
Write $P(x_1,\dots,x_N) = \sum\limits_{\substack{1 \leq i_j \leq 2n-1\\ j=1,\dots,N}}A_{i_1,\dots,i_N}x_1^{i_1}\cdots x_{N}^{i_N}$, with $A_{i_1,\dots,i_N} \in \C$. Then,
$$ \EE[P(x_1,\dots,x_N)] = \sum_{\substack{1 \leq i_j \leq 2n-1\\j=1,\dots,N}} A_{i_1,\dots,i_N} \EE[x_1^{i_1}]\cdots \EE[x_N^{i_N}].$$
Since $i_j\leq 2n-1$, one has $\EE[x_j^{i_j}]= a_{i_j} = \sum\limits_{l_j=1}^n c_{l_j,n}\zeta_{l_j}^{i_j}$, yielding:
\begin{align*}
\EE[P(x_1,\dots,x_N)] &= \sum_{\substack{1 \leq i_j \leq 2n-1\\j=1,\dots,N}} A_{i_1,\dots,i_N} \prod_{j=1}^N \bigg( \sum_{l_j=1}^n c_{l_j,n}\zeta_{l_j}^{i_j}\bigg) \\
&=  \sum_{\substack{1 \leq i_j \leq 2n-1\\j=1,\dots,N}} A_{i_1,\dots,i_N} \sum_{\substack{l_1,\dots,l_N\\ 1 \leq l_j \leq n}} c_{l_1,n}\cdots c_{l_N,n} \zeta_{l_1}^{i_1}\cdots \zeta_{l_N}^{i_N} \\
&=\sum_{\substack{l_1,\dots,l_N\\ 1 \leq l_j \leq n}} c_{l_1,n}\cdots c_{l_N,n}  \sum_{\substack{1 \leq i_j \leq 2n-1\\j=1,\dots,N}} A_{i_1,\dots,i_N}\zeta_{l_1}^{i_1}\cdots \zeta_{l_N}^{i_N}\\
&= \sum_{\substack{l_1,\dots,l_N\\ 1 \leq l_j \leq n}} c_{l_1,n}\cdots c_{l_N,n} P(\zeta_{l_1},\dots,\zeta_{l_N}).\qedhere
\end{align*}
\end{proof}

As a consequence, the following formula for the \textit{moments} of a Vandermonde polynomial can be stated. Since the
vanishing criterion ensures that $\EE[\Delta(x_1,x_2,\ldots,x_n)^{2k-1}]=0$, for all $k\in\mathbb{N}$ (recall that $a_k=a_{jk}$ for all $k\in\N$), it is sufficient to  focus only on even \textit{moments}.
\begin{cor}
\label{MomDiscriminant2}
For given $n,k,N\in\mathbb{N}$, assume that  $2k(N-1)\leq 2n-1$. Then, 
\begin{equation}\label{mqf2}\EE[\Delta(x_1,x_2,\ldots,x_N)^{2k}]=\sum_{(i_1,i_2,\ldots,i_N)\atop 1\leq i_j\leq n}\,c_{i_1}c_{i_2}\cdots c_{i_N}\,\Delta(\zeta_{i_1},\zeta_{i_2},\cdots,\zeta_{i_N})^{2k}, 
\end{equation}
where $c_1,c_2,\ldots,c_n,\zeta_1,\zeta_2,\ldots,\zeta_n\in\C$ are determined via Sylvester's Theorem as in \eqref{ci}.
\end{cor}
\begin{proof}
$\Delta(x_1,x_2,\ldots,x_N)^{2k}$ has maximum degree $2k(N-1)$ in $x_i$, for every $i$. Then, Theorem \ref{GenQuadFormula} applies  whenever $2k(N-1)\leq 2n-1$. 
\end{proof}

If $X_1,X_2,\ldots,X_N$ are independent and identically distributed random variables on a given probability space, the most natural choice for the functional $\EE$ is the the expectation $\E$. In this case, the statistics $\Delta(X_1,X_2,\ldots,X_N)^2$ is the \textit{random discriminant} \cite{Lu}. Then, identity \eqref{mqf2} can be seen an explicit formula for the moments  of a random discriminant in terms of the roots the orthogonal polynomials associated with the law of $X_1$. 
In particular, when $k=1$ and $N=n$, Corollary \ref{MomDiscriminant2} yields that:
\[\E[\Delta(X_1,X_2,\ldots,X_n)^{2}]=n!\,c_{1}c_{2}\cdots c_{n}\,\Delta(\zeta_{1},\zeta_{2},\cdots,\zeta_{n})^{2},\]
which shows that the expected value of the random discriminant reduces to the discriminant of the $n$-th orthogonal polynomial (up to multiplicative coefficients). \\

\begin{exm}
In \cite{Lu}, Selberg's integral is used to compute the exact distribution of the random discriminants $\Delta(X_1,\dots,X_n)^2$, where $X_1,\dots,X_n$ are i.i.d. random variables, Gaussian, Gamma or Beta distributed: in these cases, an explicit formula for $\E[\Delta(X_1,\dots,X_n)^{2k}]$ is provided.  For instance,  if $X_1,\dots,X_n$ are $\mathcal{N}(\mu,\sigma)$-distributed random variables, then for every $k\geq 1$,
$$ \E[\Delta(X_1,\dots,X_n)^{2k}] = \sigma^{n(n-1)k} \prod_{j=1}^n j^{jk} \prod_{1 \leq i <j \leq n }\dfrac{\Gamma(k + \frac{i}{j})}{\Gamma(\frac{i}{j})},$$
where $\Gamma(r)$ denotes the Gamma function (see \cite[Lemma 3.1]{Lu}). For instance, set $N=3, n=5$, $k=2$, $\mu=0$ and $\sigma=1$. Then $ \E[\Delta(X_1,X_2,X_3)^4] = 4320.$ 
Since the roots of $H_5(x) = x^5-10x^3 + 15x$ are $r_1= 0, r_{2}= \sqrt{5-\sqrt{10}}, r_3= - \sqrt{5-\sqrt{10}}, r_4 = \sqrt{5+\sqrt{10}}$ and $r_5 = - \sqrt{5+\sqrt{10}}$, the computation of the $c_i$'s as in \eqref{ci} gives\footnote{The computations have been run with Maple 13.}
\begin{enumerate}
\item $c_1= 0.5333333333$;
\item $c_2 = c_3=0.2220759228$;
\item $c_4=c_5 = 0.01125741133$,
\end{enumerate}
and therefore, via formula \eqref{mqf2}, $\E[\Delta(X_1,X_2,X_3)^4]=4320$.\\
\end{exm}

Another approach to compute the moments of the random discriminant for any random variable $X$ can be outlined by carrying the dual reasoning to the one that led to formula \eqref{mqf}. For the sake of clarity, the strategy  is first sketched with an example.
\begin{exm}
For the first Hermite polynomials $H_n(x)$ arising from \eqref{intorth}, for a standard Gaussian random variable $N$, 
\begin{equation}
H_2(x) = 2(x^2-1) = (x-\imath)^2 + (x+\imath)^2 \;,
\end{equation}
with $\mathbb{E}[(N-x)^2] = -(x^2+1) = -(x+\imath)(x-\imath)$. Similarly,
\begin{equation}
H_3(x) = -12(x^3 -3x)  = -8x^3 + 2(x-\imath \sqrt{3})^3 + -2(x+\imath \sqrt{3})^3
\end{equation}
and $\mathbb{E}[(N-x)^3] = -(x^3+3x) = x(\imath \sqrt{3}-x)(x+\imath \sqrt{3})$.
\end{exm}

Let $X$ be a centered random variable, and assume that $p_n(x)= \sum_{h=0}^n \binom{n}{h}(-1)^h b_{n-h} x^{h}$ is its OPs, as in \eqref{intorth}. As a consequence of the orthogonality, if  $A_n(x)=\E[(X-x)^n]$, then $\{p_n(x),A_n(x)\}=0$:  indeed,  if $a_k=\mathbb{E}[X^k]$, $\E[p_n(X)] = \sum_{h=0}^n \binom{n}{h}(-1)^h a_{h} b_{n-h} = 0$. Assume, further, that $A_{n}(x)$ has simple roots for every $n$, say $r_1,\dots,r_n$. Then, for every $j=1,\dots,n$, the polynomial $q_j(x)=(r_j-x)^n$ is apolar to $A_n(x)$ (see, for instance, \cite[Theorem 1]{Rota98} for the proof via umbral methods), and since $q_1(x),q_2(x),\dots,q_n(x)$ are linearly independent, Sylvester's Theorem implies the existence of unique coefficients $c_1,c_2,\dots,c_n$ such that:
$$ p_n(x) = \sum_{j=1}^n c_j (r_j-x)^n ,$$
and hence, by comparing the leading terms, $$ \E[\Delta(X_1,\dots,X_n)^2] = \sum_{j=1}^n c_j \,.$$
More generally, for any $k\geq 1$, consider the polynomials $p_{n,k}(x):= p_{n,k}(x,X)$ defined via
\begin{equation}
\label{p_nk}
p_{n,k}(x)=n!\E[q_{nk}(x,X_1,\ldots,X_n)],
\end{equation}
with 
$$q_{nk}(x,x_1,\ldots,x_n)=x_1^{n-1}x_2^{n-2}\cdots x_{n-1} \Delta(x,x_1,x_2,\dots,x_n)^{2k-1}\, , $$ 
and $X_1,\dots,X_n$ independent copies of $X$. By symmetrizing $q_{nk}(x,x_1,\dots,x_n)$ with respect to $x_1,x_2,\ldots,x_n$, one has
$$ p_{n,k}(x)= \E[(X_1-x)^{2k-1}(X_2-x)^{2k-1}\cdots(X_n-x)^{2k-1}\Delta(X_1,X_2,\ldots,X_n)^{2k}].$$
In the sequel, if $m=n(2k-1)$, write:
$$ p_{n,k}(x)= \sum_{h=0}^m \binom{m}{h}(-1)^{h}b_{m-h}x^h.$$
\begin{thm}\label{MomDiscriminant}
Let the previous notation prevail, and, for fixed $n,k\in \N$, if $m=n(2k-1)$, assume that $A_m(x)=\E[(X-x)^m]$ has $m$ simple roots, say $r_1(X),\dots,r_m(X)$. Assume that the linear system
$$ b_{m-k} = \sum_{j=1}^{m} c_j r_{j}^{m-k} \;\, \quad \forall k=0,\dots,m $$
admits a unique solution $(c_{m,1}(X),\dots,c_{m,m}(X))$. Then,
$$ (-1)^m \E[\Delta(X_1,\dots,X_n)^{2k}]= \sum_{j=1}^m c_{m,j}(X) .$$
\end{thm}

\begin{proof}
Let $\mathcal{I}(x,y)$ denote the covariant obtained by replacing $X_i-X_j$ with the bracket $[i\;j]$ in $p_{n,k}(x)$, so that $p_{n,k}(x)=\mathcal{I}(x,1)$. Then, the vanishing criterion with respect to the functional $\E$, applied to $x_{n+1}^h\;q_{nk}(x_{n+1},x_1,x_2,\ldots,x_n)$, for $h=1,\dots,n-1$, implies that:
$$\E[ q(X)\,p_{n,k}(X)]=0 \quad \forall  q(x)\in\C[x]_{\leq n-1},$$
implying in turn that the apolar covariant  $\mathcal{A}(f,\mathcal{I})$ vanishes whenever $f(x,y)=\E[(y X-x)^{2kn-1}]$. In terms of polynomials, the orthogonality conditions satisfied by $p_{n,k}(x)$ can be rewritten as:
$$ \mathcal{A}(A_{m}(x), p_{n,k}(x)) = 0. $$
Then, by virtue of Sylvester's Theorem, $p_{n,k}(x)$ can be decomposed as:
\begin{align*}
p_{n,k}(x) &= \sum_{j=1}^m c_{m,j}(X) (r_j(X)-x)^m \\
&= \sum_{j=1}^m c_{m,j}(X) \sum_{h=0}^m \binom{m}{h} (-1)^h x^h r_{j}(X)^{m-h} \\
&= \sum_{h=0}^m x^h \binom{m}{h}(-1)^h \sum_{j=1}^m r_j(X)^{m-h}c_{m,j}(X),
\end{align*}
and the conclusion follows by identifying the leading terms.
\end{proof}

\begin{exm}
Let $n=k=2$, so that $m=n(2k-1)=6$. If $N\sim \mathcal{N}(0,1)$, $A_6(x)=x^6+15x^4+45x^2+15$, whose zeros are given by $r=[r_1,\dots,r_6]$,
\begin{enumerate}
\item $r_1 = -1.349630430\cdot 10^{-9} -0.6167065905\,\imath$,
\item $r_2= 1.349630430\cdot 10^{-9} +0.6167065905\,\imath$,
\item $r_3= 2.340691285\cdot 10^{-10}-3.324257434\,\imath,$
\item $r_4= -2.340691285\cdot 10^{-10}+3.324257434\, \imath,$
\item $r_5= 6.504148675\cdot 10^{-11}-1.889175878\,\imath,$
\item $r_6= -6.504148675\cdot 10^{-11}+1.889175878\,\imath.$
\end{enumerate}
Moreover, 
$$ p_{2,2}(x) = \sum_{h=0}^6 \binom{6}{h}(-1)^h b_{m-h}x^{h} = 12x^6 +  90 x^2 -360$$
and, if $N_1,N_2$ are independent $\mathcal{N}(0,1)$-distributed, $\E[\Delta(N_1,N_2)^4]=12$. Then, the solution of the system $M\cdot c=b$, where $M= (r_{j}^{6-i})_{i,j}$, for $i,j=1,\dots,6$,  and $b=[b_5,b_4,\dots,b_0]^T= [0, 0, 0, 6, 0, 12]^T$, is given by:
\begin{enumerate}
\item $c_1 = 8.244073700-0.5810482414e^{-8} \, \imath$,
\item $c_2= 8.244073680-0.5810482378e^{-8}\, \imath$,
\item $ c_3 = 0.2504273076-0.6122705598e^{-9}\,\imath$,
\item $c_4 = 0.2504273069-0.6122705576e^{-9}\, \imath$,
\item $c_5= -2.494501001+0.6422752968e^{-8}\,\imath  $,
\item $c_6= -2.494500994+0.6422752942e^{-8}\,\imath$,
\end{enumerate}
yielding $c_1+c_2+\cdots +c_6 = \E[\Delta(N_1,N_2)^4]$.
\end{exm}

\section{Invariant theory in several variables}

The present section aims at extending the algebraic framework for orthogonality, based on invariant theory, for polynomials in $\C[x_0,x_1,\dots,x_d]$. The choice of dealing separately with the multivariable setting is made to highlight some non-trivial  aspects that arise here and that were missing in the univariate framework. \\

Recall that a $q$-ary form of degree $(k_0,k_1,\dots,k_q)\in \mathbb{N}_0^{q+1}$ is a polynomial in \linebreak$\C[x_0,x_1,\dots,x_q,y_0,y_1,\dots,y_q]$, homogeneous of degree $k_j$ in $(x_j,y_j)$. One of the major advantages in working with the symbolic method of invariant theory  is that the proofs  using umbral notation, in the setting of binary forms, are suitable to be easily generalized for $q$-ary forms. However, quoting Rota and Kung \cite{KungRota}, ``[...]\textit{The notion of a covariant ramifies in several variables into several kinds of concomitants\footnote{Concomitants are a class of invariants.}, and the various kinds of apolarity never seem to have been fully explored} [..]''. Hence, the framework set in the sequel is not meant to be exhaustive; on the other hand, it corresponds to what is, in the author's belief, the most natural extension of the univariate setting discussed in Chapters \ref{Chapter_OPS} and \ref{Chapter_WhatIsOrtho}.\\

Consider two sets of independent indeterminates $\{x_i:i \in \N_0\}$ and $\{y_i:i \in \N_0\}$. For a fixed $d\in\N$,  set $\bs{x}_{i}=\{x_{i(d+1)},x_{i(d+1)+1},\ldots,x_{i(d+1)+d}\}$, and $\bs{x}=\{\bs{x}_i\,|\,i\in\N_0\}$. Similarly, set
$\bs{y}_{i}=\{y_{i(d+1)},y_{i(d+1)+1},\ldots,y_{i(d+1)+d}\}$ and
$\bs{y}=\{\bs{y}_i\,|\,i\in\N_0\}$. To further shorten the notation, set
$x_{ij}=x_{i(d+1)+j}$ and $y_{ij}=y_{i(d+1)+j}$ and write
\begin{equation}\label{xyi}
\left(\begin{matrix}\bs{x}_i\\\bs{y}_i\end{matrix}\right)=\left(\begin{matrix}x_{i0}&x_{i1}&\ldots&x_{id}\\ y_{i0}&y_{i1}&\ldots&y_{id}\end{matrix}\right) \text{ for all }i\in\N_0.
\end{equation}
Let $GL_2(\C)^{d+1}$ denote the direct product of $d+1$ copies of
$GL_2(\C)$. Every element $\bs{g}=(g_0,g_1,\ldots,g_d)\in GL_2(\C)^{d+1}$ acts on  $\left(\begin{matrix}\bs{x}_i\\\bs{y}_i\end{matrix}\right)$ according to the rule:
\[g_k\cdot\left(\begin{matrix}x_{ik}\\ y_{ik}\end{matrix}\right)=
\left(\begin{matrix}g_{11}^{\sst(k)}&g_{12}^{\sst(k)}\\g_{21}^{\sst(k)}&g_{22}^{\sst(k)}\end{matrix}\right)
\left(\begin{matrix}x_{ik}\\y_{ik}\end{matrix}\right) \forall i\in\N_0, \, \forall  k = 0,\dots, d,
\]
if $g_k= (g_{ij}^{(k)})_{i,j=1,2}$. Then, for every polynomial $p \in \C[\bs{x},\bs{y}]$, $\bs{g}\cdot p$ is obtained from $p$ by letting $g_k$ act on  $(x_{i_k},y_{i_k})$.

\begin{defn}
A polynomial $p \in \C[\bs{x},\bs{y}]$ is said to be a $GL_2(\C)^{d+1}$-invariant of index
$\bs{n}=(n_0,n_1,\ldots,n_d) \in \N_0^{d+1}$ if, and only if,
\[\bs{g}\cdot p=(\det\,g_0)^{n_0}(\det\,g_1)^{n_1}\cdots(\det\,g_d)^{n_d}\,p \quad \forall \,\bs{g}\in GL_2(\C)^{d+1}.\]
\end{defn}

If $\bs{n}=(n_0,n_1,\ldots,n_d)\in\N_0^{d+1}$, and following the standard multi-index notation:
\[\bs{x}_i^{\bs{n}}=\prod_{j=0}^d\,x_{i,j}^{n_j} \text{ and }\bs{y}_i^{\bs{n}}=\prod_{j=0}^d\,y_{i,j}^{n_j} \text{ for all }i\in\N,\]
one example of $GL_2(\C)^{d+1}$-invariant of index
$\bs{n}$ is  the \textit{bracket polynomial}, defined, for $i,j\in\N$, as the $2(d+1)$-ary form given by:
\[\bs{[}i\,j\bs{]}^{\bs{n}}=(x_{i0}y_{j0}-y_{i0}x_{j0})^{n_0}\,(x_{i1}y_{j1}-y_{i1}x_{j1})^{n_1}\,\cdots\,(x_{id}y_{jd}-y_{id}x_{jd})^{n_d}.\]
Indeed,  $g_k\cdot (x_{ik}y_{jk}-y_{ik}x_{jk})^{n_k}=(\det\,g_k)^{n_k}\,(x_{ik}y_{jk}-y_{ik}x_{jk})^{n_k}$ for all $k\in\N$. More generally, any product of the type
\[\bs{[}i_1\,j_1\bs{]}^{\bs{n}_1}\bs{[}i_2\,j_2\bs{]}^{\bs{n}_2}\cdots \bs{[}i_l\,j_l\bs{]}^{\bs{n}_l}\]
is a $GL_2(\C)^{d+1}$-invariant of index $\bs{n}$ if and only if $\bs{n}_1+\bs{n}_2+\cdots+\bs{n}_l=\bs{n}$, where $\bs{n}_1+\bs{n}_2+\cdots+\bs{n}_l$ denotes the componentwise sum (by virtue of the First Fundamental Theorem, any invariant $p$ of index $\bs{n}$ is a linear combination of products of this type).\\

For all $\bs{k},\bs{n}\in\N_0^{d+1}$, consider the componentwise order: $\bs{k}\leq \bs{n}$ if and only if $\bs{k}=(k_0,k_1,\ldots,k_d)$, $\bs{n}=(n_0,n_1,\ldots,n_d)$ and $k_i\leq n_i$ for all $0\leq i\leq d$. As already underlined in Chapter \ref{Chapter_OPS}, this choice  guarantees that $(\mathbb{N}_0^{d+1},\leq)$ is a graded poset, with rank function $\rho\colon\N_0^{d+1}\to\N_0$ given by $\rho(\bs{n})=n_0+n_1+\cdots+n_d$, and corresponds to the most natural extension of the univariate setting on $(\mathbb{N}_0,\leq)$ (see \cite{Stanley}). 

\begin{defn}
For a fixed $\bs{n}\in\N_0^{d+1}$, a  \textit{generic $2(d+1)$-ary form of degree $\bs{n}$} is a polynomial in $\C[a_{\bs{0}},\dots,a_{\bs{n}};\bs{x}_0,\bs{y}_0]$ of the type
\[f(a_{\bs{0}},\ldots,a_{\bs{n}};\bs{x}_0,\bs{y}_0)=\sum_{\bs{0}\leq\bs{k}\leq \bs{n}}\binom{\bs{n}}{\bs{k}}\,(-1)^{\rho(\bs{n}-\bs{k})}\,a_{\bs{k}}\,\bs{x}_0^{\bs{n}-\bs{k}}\bs{y}_0^{\bs{k}},\]
where $\bs{0}=(0,0,\ldots,0)$, and we have set:
\[\binom{\bs{n}}{\bs{k}}=\binom{n_0}{k_0}\binom{n_1}{k_1}\cdots\binom{n_d}{k_d}.\]
\end{defn}

Hereafter, set 
\begin{equation}
\label{lenght}
s(\bs{n}):=\big|\,\{\bs{k}\,|\,\bs{0}\leq\bs{k}\leq\bs{n}\}\big|
\end{equation}
so that $s(\bs{n})$ equals the number of monomials in the generic form of degree $\bs{n}$. A \textit{$2(d+1)$-ary form of degree $\bs{n}$} is a polynomial $f(\bs{x}_0,\bs{y}_0)$ arising from the
generic form of degree $\bs{n}$ when $a_{\bs{0}},\ldots,a_{\bs{n}}$ specialize at given coefficients in $\C$. \\

Given an ordered sequence $\bs{\phi}=(\phi_0,\phi_1,\ldots,\phi_d)$ of linear changes of variables, acting on \linebreak $f(a_{\bs{0}},\ldots,a_{\bs{n}};\bs{x}_0,\bs{y}_0)$,  the generic form $f(\bar{a}_{\bs{0}},\ldots,\bar{a}_{\bs{n}};\bs{x}_0,\bs{y}_0)$ is defined by letting $\phi_j$ act on the pair $(x_{0j},y_{0j})$.

\begin{defn}
For $\bs{n},\bs{m} \in \N_0^{d+1}$, a \textit{$GL_2(\C)^{d+1}$-covariant of index $\bs{m}$ for $2(d+1)$-ary forms of degree $\bs{n}$} is a
polynomial $\mathcal{I}(a_{\bs{0}},\ldots,a_{\bs{n}};\bs{x}_0,\bs{y}_0)$ satisfying, for every ordered sequence $\bs{\phi}=(\phi_0,\dots,\phi_d)$ of linear changes of variables,  with $\phi_j=\phi_j(c_{11}^{(j)},c_{12}^{(j)},c_{21}^{(j)},c_{22}^{(j)})$,
$$\mathcal{I}(\bar{a}_{\bs{0}},\ldots,\bar{a}_{\bs{n}};\bs{x}_0,\bs{y}_0)=(det\,\phi_0)^{m_0}(det\,\phi_1)^{m_1}\cdots (det\,\phi_d)^{m_d}\mathcal{I}(a_{\bs{0}},\ldots,a_{\bs{n}};\bs{x}_0,\bs{y}_0),$$
where
$$\mathcal{I}(\bar{a}_{\bs{0}},\ldots,\bar{a}_{\bs{n}};\bs{x}_0,\bs{y}_0) := \mathcal{I}(a_{\bs{0}},\ldots,a_{\bs{n}};\bs{x}_0^{\prime},\bs{y}_0^{\prime}),$$ and
$$\left(\begin{matrix}x'_{0i}\\ y'_{0i}\end{matrix}\right)=\phi_i\left(\begin{matrix}x_{0i}\\ y_{0i}\end{matrix}\right)
=\left(\begin{matrix}c_{11}^{\sst(i)}x_{0i}+c_{12}^{\sst(i)}y_{0i}\\c_{21}^{\sst(i)}x_{0i}+c_{22}^{\sst(i)}y_{0i}\end{matrix}\right),\quad \forall \, i=0,\dots, d.$$ 

Similarly, if $g(\bs{x}_0,\bs{y}_0) \in \C[b_{\bs{0}},\dots,b_{\bs{m}},\bs{x}_0,\bs{y}_0]$ is a generic $2(d+1)$-ary form of degree $\bs{m}$, a joint-covariant of index $\bs{k} \in \N^{d+1}$ of $2(d+1)$-ary forms of degree $(\bs{n},\bs{m})$ is a polynomial\linebreak $\mathcal{I}(a_{\bs{0}},\dots,a_{\bs{n}},b_{\bs{0}},\dots,b_{\bs{m}};\bs{x}_0,\bs{y}_0)$ such that
\begin{align*}
\mathcal{I}(\bar{a}_{\bs{0}},\ldots,\bar{a}_{\bs{n}}, \bar{b}_{\bs{0}}&,\ldots,\bar{b}_{\bs{m}} ;\bs{x}_0,\bs{y}_0) \\ 
&=(det\,\phi_0)^{k_0}(det\,\phi_1)^{k_1}\cdots (det\,\phi_d)^{k_d}\mathcal{I}(a_{\bs{0}},\ldots,a_{\bs{n}}, b_{\bs{0}},\dots,b_{\bs{m}};\bs{x}_0,\bs{y}_0),
\end{align*}
for every ordered sequence of linear changes $\bs{\phi}=(\phi_0,\dots,\phi_d)$.
\end{defn}

Let $\mathbb{N}_1, \mathbb{N}_2$ be a partition of $\mathbb{N}$ into disjoint infinite subsets. For fixed $\bs{m}, \bs{n}$, with $\bs{m}\leq \bs{n}$, consider generic forms $f(a_{1\bs{0}},\ldots,a_{1\bs{n}};\bs{x}_0,\bs{y}_0)$ and $g(a_{2\bs{0}},\ldots,a_{2\bs{m}};\bs{x}_0,\bs{y}_0)$ of degree $\bs{n}$ and $\bs{m}$, respectively, and consider the linear operator
\[\U\colon\C[\bs{x},\bs{y}]\to\C[a_{1\bs{0}},\ldots,a_{1\bs{n}};a_{2\bs{0}},\ldots,a_{2\bs{m}};\bs{x}_0,\bs{y}_0]\]
defined by 
\[\U\,\bs{x}_i^{\bs{k}_1}\bs{y}_i^{\bs{k}_2}=\begin{cases}\bs{x}_0^{\bs{k}_1}\bs{y}_0^{\bs{k}_2}&\text{ if }i=0;\\ a_{1\bs{k}_1} & \text{ if }\bs{k}_1+\bs{k}_2=\bs{n}\text{ and }i\in\mathbb{N}_1;\\a_{2\bs{k}_1} & \text{ if }\bs{k}_1+\bs{k}_2=\bs{m}\text{ and }i\in\mathbb{N}_2;\\0& \text{ otherwise. }\end{cases}\]
If $f(\bs{x}_0,\bs{y}_0)$ and $g(\bs{x}_0,\bs{y}_0)$ are $2(d+1)$-ary forms of
degree $\bs{n}$ and $\bs{m}$, respectively, then $\U(f)\,p$ (respectively $\U(f,g)$) will denote the operator whose value
$\U(f)\,p$ (respectively $\U(f,g)\,p$) is obtained from $\U\,p$ by
replacing the variates $a_{i\bs{k}}$'s with the corresponding coefficients of $f(\bs{x}_0,\bs{y}_0)$ (respectively
of $f(\bs{x}_0,\bs{y}_0)$ and $g(\bs{x}_0,\bs{y}_0)$). For instance,
\[\U(f,g)\bs{[}1\,0\bs{]}^{\bs{n}}=f(\bs{x}_0,\bs{y}_0) \text{ and }\U(f,g)\bs{[}2\,0\bs{]}^{\bs{m}}=g(\bs{x}_0,\bs{y}_0),\]
provided that $1\in\mathbb{N}_1$ and $2\in\mathbb{N}_2$. Furthermore, 
\begin{multline*}\U(f,g)\,\bs{[}1\,0\bs{]}^{\bs{n}-\bs{m}}\bs{[}2\,1\bs{]}^{\bs{m}} \\=\U(f)\,\sum_{\bs{0}\leq \bs{k}\leq \bs{n}-\bs{m}}\binom{\bs{n}-\bs{m}}{\bs{k}}(-1)^{\rho(\bs{n}-\bs{m}-\bs{k})}\,\bs{x}_1^{\bs{k}}\bs{y}_1^{\bs{n}-\bs{m}-\bs{k}}\,g(\bs{x}_1,\bs{y}_1)\bs{x}_0^{\bs{n}-\bs{m}-\bs{k}}\bs{y}_0^{\bs{k}} \numberthis \label{Apolmulti}.\end{multline*}

\begin{defn}

If $\bs{m}\leq \bs{n}$, the \textit{apolar $GL_2(\C)^{d+1}$-covariant} is the joint-covariant of index $\bs{m}$ of $2(d+1)$-ary forms of degree $(\bs{n},\bs{m})$ defined by
\begin{equation}
\label{multiapolcov}\mathcal{A}(a_{1\bs{0}},\ldots,a_{1\bs{n}};a_{2\bs{0}},\ldots,a_{2\bs{m}};\bs{x}_0,\bs{y}_0)=\U\,\bs{[}1\,0\bs{]}^{\bs{n}-\bs{m}}\bs{[}2\,1\bs{]}^{\bs{m}}.
\end{equation}
\end{defn}

%
For $2(d+1)$-ary forms $f$ and $g$ of degrees $\bs{n}$ and $\bs{m}$ respectively, the associated \textit{apolar form} is obtained by setting
\begin{equation}\label{mutliapol}
\{f,g\}=\U(f,g)\,\bs{[}1\,0\bs{]}^{\bs{n}-\bs{m}}\bs{[}2\,1\bs{]}^{\bs{m}},
\end{equation}
so that, $f(\bs{x}_0,\bs{y}_0)$ and $g(\bs{x}_0,\bs{y}_0)$ are said to be \textit{apolar} if and only if $\{f,g\}=0$ or, equivalently, from \eqref{Apolmulti}, if and only if:
\begin{equation}\label{mutliapolbis}
\U(f)\,\bs{x}_1^{\bs{k}}\bs{y}_1^{\bs{n}-\bs{m}-\bs{k}}\,g(\bs{x}_1,\bs{y}_1)=0 \quad \forall \, \bs{k}: \bs{0}\leq\bs{k}\leq \bs{n}-\bs{m}.
\end{equation}
Note that, by virtue of \eqref{mutliapolbis} the $\C$-vector space of the forms of degree $\bs{m}$ that are apolar to a given form of degree $\bs{n}$ has, in general, dimension $s(\bs{m})-s(\bs{n}-\bs{m})$, with $s(\bs{m}), s(\bs{n}-\bs{m})$ as defined in \eqref{lenght}. 
Indeed, \eqref{Apolmulti} asks for the solution of the system in $s(\bs{n}-\bs{m})$ equations and $s(\bs{m})$ unknowns $a_{2\bs{h}_j}$'s
$$ \sum_{i=0}^{s} \binom{\bs{m}}{\bs{h}_i}(-1)^{\rho(\bs{m}-\bs{h}_i)}a_{2\bs{h}_i} a_{1\,\bs{h}_i + \bs{k}_j} = 0 \qquad \text{ for } j=0,\dots,r \,,$$
where $s=s(\bs{m})-1, \{\bs{k}:\bs{k}\leq \bs{m}\}=\{\bs{h}_0=\bs{0},\bs{h}_1,\dots,\bs{h}_s=\bs{m}\}$ and $r=s(\bs{n}-\bs{m})-1$, with 
$\{\bs{k}:\bs{k}\leq \bs{n}-\bs{m}\}=\{\bs{k}_0=\bs{0},\bs{k}_1,\dots,\bs{k}_r=\bs{n}-\bs{m}\}$.\\

In analogy with the case $d=0$, the apolar form leads to generalized orthogonal polynomial systems in $d+1$ indeterminates  with respect to a linear functional $\EE\colon\C[\bs{x}]\to\C$ satisfying \eqref{E} and \eqref{Ebis}, and with  $a_{\bs{k}}=a_{0\bs{k}}=a_{1\bs{k}}$. Trivially, $\EE$ can be determined either by the sequence of its moments or by the family $\{f_{\bs{n}}(\bs{x}_0,\bs{y}_0)\,|\,\bs{n}\in\N^{d+1}\}$ of $2(d+1)$-forms defined by
\begin{equation}\label{Lformmulti}
f_{\bs{n}}(\bs{x}_0,\bs{y}_0)=\sum_{\bs{0}\leq \bs{k}\leq \bs{n}}\binom{\bs{n}}{\bs{k}}\,a_{\bs{k}}(-1)^{\rho(\bs{n}-\bs{k})}\,\bs{x}_0^{\bs{n}-\bs{k}}\bs{y}_0^{\bs{k}}, \quad \forall \bs{n}\in\N^{d+1}.
\end{equation}

\begin{defn}
For the fixed $\leq $, consider a triangular array of polynomials in $\C[\bs{x}_0]$  as in \eqref{polm}.
Then $\{p_{\bs{n}\bs{m}}(\bs{x}_0): \bs{0}<\bs{m}\leq \bs{n}\}_{\bs{n} \in \mathbb{N}_0^{d+1}}$ is a \textit{generalized orthogonal polynomial system} for $\EE$ if it satisfies
\begin{equation}
\label{GMOPS2}\EE[\bs{x}_0^{\bs{k}}\,p_{\bs{n}\bs{m}}(\bs{x}_0)]=0 \quad \forall \, \bs{k}: \bs{0}\leq\bs{k}\leq\bs{n}-\bs{m},
\end{equation}
and $\EE[\bs{x}_0^{\bs{k}}\,p_{\bs{n}\bs{m}}(\bs{x}_0)] \neq 0$ for every multi-index $\bs{k} \leq \bs{n}$, covering $\bs{n}-\bs{m}$.
\end{defn}

With these tools, Theorem \ref{ApolarityOrth} can be extended to the multivariable setting.

\begin{thm}
A set of $2(d+1)$-ary forms  $\{g_{\bs{n}\bs{m}}(\bs{x}_0,\bs{y}_0)$:  $\bs{0}<
\bs{m}\leq \bs{n}\}$ of degree $\bs{n}$, satisfies 
\[\{f_{2\bs{n}-\bs{m}},g_{\bs{n}\bs{m}}\}=0,\, \text{ for }\, \bs{0}<\bs{m}\leq \bs{n}, \]
or, equivalently,
\[\U(f_{2\bs{n}-\bs{m}})\,\bs{x}_1^{\bs{k}}\bs{y}_1^{\bs{n}-\bs{m}-\bs{k}}\,g_{\bs{n}\bs{m}}(\bs{x}_1,\bs{y}_1)=0 \quad \forall \bs{k}: \bs{0}\leq\bs{k}\leq \bs{n}-\bs{m},  \, \bs{0} < \bs{m}\leq \bs{n}\]
 if, and only if, $g_{\bs{n}\bs{m}}(\bs{x_0},\bs{y}_0) = \bs{y}_0^{\bs{n}}p_{\bs{n}\bs{m}}(\bs{x}_0^{\bs{n}}\bs{y}_0^{-\bs{n}})$, with $p_{\bs{n}\bs{m}}(\bs{x_0})$  generalized orthogonal system for $\EE$ as in \eqref{GMOPS2}.\\
\end{thm}

Classical orthogonal polynomial systems are uniquely determined up to a multiplicative factor (see Theorem 2.2. and its corollary in \cite{Chihara}). This is due to the fact that the space of all binary forms of degree $n$, that are apolar to a given form of degree $2n-1$, has, in general, dimension $1$. Moreover, as shown in Chapter \ref{Chapter_OPS}, classical orthogonal polynomials arise by selecting those polynomials corresponding to $m=1$ from a generalized orthogonal polynomial system $\{p_{nm}(x_0)\}_{n,m\geq 1}$.  \\
This phenomenon is no longer true in the multivariable setting. In fact, if $\bs{\delta}$ has rank $1$ (i.e. $\rho(\bs{\delta})=1$), the space of the forms of degree $\bs{n}$ that are apolar to a given form of degree $2\bs{n}-\bs{\delta}$ does not have, in general, dimension $1$. For instance, consider $\bs{\delta}_k =(0,0,\dots,1,\dots,0)$, then $2\bs{n}-\bs{\delta}_k= ( 2n_0,  2n_1,\dots, 2n_k -1,   \dots, 2n_d )$,  $2\bs{n}-\bs{\delta}_k- \bs{n} = (n_0,  n_1,\dots, n_k -1,   \dots, n_d )$, hence for every $h_j  \leq n_j$ for $j=k+1, \dots, d$, the element $(n_0,  n_1,\dots, n_k , h_{k+1}  \dots, h_d ) \leq  \bs{n} $ but  $\nleq \bs{n}-\bs{\delta}_k$. As such, the polynomial sequence obtained by extracting all the polynomials $p_{\bs{n}\bs{m}}(\bs{x}_0)$, for $\bs{m}=\bs{\delta}_k$, is not uniquely determined, up to multiplicative factors.
 
To achieve the multivariable counterpart to this phenomenon, it is necessary to proceed as follows: let $\bs{\delta}_0=(1,0,\ldots,0)$, $\bs{\delta}_1=(0,1,\ldots,0)$, \ldots, $\bs{\delta}_d=(0,0,\ldots,1)$ be the only $d+1$ elements in $\N^{d+1}$ having rank equal to $1$, and consider the binary forms $\{f_i(\bs{x}_0,\bs{y}_0)\,|\,0\leq i\leq d\}$ defined by \eqref{Lformmulti} with $f_i(\bs{x}_0,\bs{y}_0):=f_{2\bs{n}-\bs{\delta}_i}(\bs{x}_0,\bs{y}_0)$. Then, consider the space of the $2(d+1)$-ary forms $g(\bs{x}_0,\bs{y}_0)$ of degree $\bs{n}$ such that:
\[\{f_0,g\}=\{f_1,g\}=\cdots=\{f_d,g\}=0.\]
Then, mindful of the notation introduced in the previous sections, a form $g(\bs{x}_0,\bs{y}_0)$ of degree $\bs{n}$ is apolar to each $f_i(\bs{x}_0,\bs{y}_0)$ if, and only if, the polynomial $p_{\bs{n}}(\bs{x}_0)$ obtained from $g(\bs{x}_0,\bs{y}_0)$ by setting $y_{i0}=1$ for all $i=0,\dots,d+1$, satisfies:
$$ \EE[\bs{x}_0^{\bs{k}} p_{\bs{n}}(\bs{x}_0)] = 0  \quad \forall \bs{k}:\bs{0}\leq \bs{k} \leq \bs{n}-\bs{\delta}_j, \, \forall j=0,\dots,d, \text{ and } \EE[\bs{x}_0^{\bs{n}} p_{\bs{n}}(\bs{x}_0)]\neq 0,$$
and, hence, if and only if
\begin{equation}\label{MultiOrtogonali}
\EE[\bs{x}_0^{\bs{k}}\,p_{\bs{n}}(\bs{x}_0)]=0 \quad \forall \bs{k}: \bs{0}\leq\bs{k}<\bs{n}, \text{ and } \EE[\bs{x}_0^{\bs{n}} p_{\bs{n}}(\bs{x}_0)]\neq 0.
\end{equation}
These relations imply that $p_{\bs{n}}(\bs{x}_0)$ belongs to a vector space whose dimension, in general, is $1$, and therefore the sequence $\{p_{\bs{n}}(\bs{x}_0)\}_{\bs{n} \in \mathbb{N}_0^{d+1}}$ is uniquely determined up to a multiplicative factor. In particular, the polynomials $p_{\bs{n}}(\bs{x}_0)$ satisfy:
\[\EE[p_{\bs{n}}(\bs{x}_0)p_{\bs{m}}(\bs{x}_0)] =0 \text{ for }\bs{m}<\bs{n} \text{ or }\bs{m}>\bs{n}\, , \text{ and } \EE[\bs{x}_0^{\bs{n}} p_{\bs{n}}(\bs{x}_0)]\neq 0.\]
In conclusion, explicit formulae for $p_{\bs{n}}(\bs{x}_0)$ can be obtained by applying a reasoning that closely parallels the proofs of Theorem \ref{detformulti} and Theorem \ref{morthpol}, as summarized in the next statement.

\begin{thm}
\label{MultiOrth}
For every fixed $\bs{n} \in \mathbb{N}_0^{d+1}$, let $\{\bs{k}\,|\,\bs{0}\leq\bs{k}\leq\bs{n}\}=\{\bs{k}_0:=\bs{0},\bs{k}_1,\ldots,\bs{k}_s=\bs{n}\}$. Then, the polynomials defined via:
\begin{equation}
\label{DetMutli}
p_{\bs{n}}(\bs{x}_0)=
\begin{vmatrix}
1&\bs{x}_0^{\bs{k}_1}&\bs{x}_0^{\bs{k}_2}&\ldots&\bs{x}_0^{\bs{k}_s}\\
a_{\bs{0}}&a_{\bs{k}_1}&a_{\bs{k}_2}&\ldots&a_{\bs{k}_s}\\
a_{\bs{k}_1}&a_{\bs{k}_1+\bs{k}_1}&a_{\bs{k}_2+\bs{k}_1}&\ldots&a_{\bs{k}_s+\bs{k}_1}\\
a_{\bs{k}_2}&a_{\bs{k}_1+\bs{k}_2}&a_{\bs{k}_2+\bs{k}_2}&\ldots&a_{\bs{k}_s+\bs{k}_2}\\
\vdots&\vdots&\vdots&&\vdots\\
a_{\bs{k}_{s-1}}&a_{\bs{k}_{s-1}+\bs{k}_{1}}&a_{\bs{k}_{s-1}+\bs{k}_2}&\ldots&a_{\bs{k}_{s-1}+\bs{k}_s}
\end{vmatrix},
\end{equation}
or, equivalently, via
\begin{equation}
\label{HeineMulti}
p_{\bs{n}}(\bs{x}_0)=\EE_0[\bs{\Delta}^*_{\bs{n}}(\bs{x}_1,\bs{x}_2,\ldots,\bs{x}_{s})\bs{\Delta}_{\bs{n}}(\bs{x}_0,\bs{x}_1,\ldots,\bs{x}_{s})],
\end{equation}
satisfy \eqref{MultiOrtogonali}, provided that $ \deg\,p_{\bs{n}}(\bs{x}_0) =\bs{n}$, where $\EE_0[\bs{x}_0^{\bs{s}_0}\bs{x}_1^{\bs{s}_1}\cdots \bs{x}_r^{\bs{s}_r}] := \bs{x}_0^{\bs{s}_0}\EE[\bs{x}_1^{\bs{s}_1}\cdots \bs{x}_r^{\bs{s}_r}]$.
\end{thm}

\chapter{Cumulants and diagonal measures}\label{Chapter_Diagonal}

In Chapter \ref{Chapter_WhatIsOrtho}, orthogonal polynomials are found to correspond to apolar covariants of binary forms, under the image of a suitable linear operator (see Theorem \ref{ApolarityOrth}). When considering only the group of affine transformations, one speaks about semi-invariants. A semi-invariant (of binary forms of degree $n$) is a polynomial $\mathcal{I}(a_0,a_1,\dots,a_n)$ in the coefficients of a generic binary form $f$ of degree $n$, such that there exist non-negative integers $\mu, w$ satisfying:
$$ \mathcal{I}(\bar{a}_0,\dots,\bar{a}_n) = \alpha^{\mu} \det(T)^{w} \mathcal{I}(a_0,\dots,a_n),$$
for every matrix $T \in GL_2(\C)$ of the type:
$$ T = \bigg(\,\begin{matrix}
\alpha & 0 \\
\gamma & \delta
\end{matrix}\,\bigg), 
$$
where the coefficients $\bar{a}_k$ are determined according to \eqref{tranlsate3}, for the umbral operator $U x_i^r y_i^{n-r} = a_r$ (remark that the subgroup of such matrices is isomorphic to the group of the affine transformations of the plane, determined by $x \mapsto \alpha x + \gamma/\delta$). When dealing with semi-invariants, it is sufficient to focus on matrices of the type:
$$ T = \bigg(\,\begin{matrix}
1 & 0 \\
-t & 1
\end{matrix}\,\bigg),
$$
for $t \in \mathbb{R}$, and in such cases, a semi-invariant $\mathcal{I}(a_0,\dots,a_n)$ satisfies:
$$ \mathcal{I}(\bar{a}_0,\bar{a}_1\dots,\bar{a}_n) = \mathcal{I}(a_0,a_1,\dots,a_n),$$
with $\bar{a}_k = \sum_{r=0}^k \binom{k}{r}t^{r} a_{k-r}$.

For probabilists and statisticians, the most important family of semi-invariants are the cumulants \cite{McCullagh,Speed1,Shiryaev}: setting $a_0=1$, the $n$-th cumulant of a random variable $X$, admitting finite moments $a_1,\dots,a_n$ up to the order $n$, is a polynomial $\chi_n(X)= \chi_n(a_1,\dots,a_n)$ in the first $n$ moments of $X$ (without constant term), which is semi-invariant under translation:
$$ \chi_n(X+t) = \chi_n(X) \quad \forall n \geq 2, \quad \chi_1(X+t)= \chi_1(X) + t,$$
for every $t\in \mathbb{R}$, or equivalently $\chi_n(\bar{a}_1,\dots,\bar{a}_n) = \chi_n(a_1,\dots,a_n)$.\\

In this chapter, using the language introduced by Rota and Wallstrom \cite{Rota1}, cumulants are presented as the stochastic counterpart to orthogonal polynomials, in the sense that cumulants are found to be the expectation of random variables that are invariant under translation: diagonal measures (see Theorem \ref{TEO}). In the wake, some combinatorial and statistical properties of diagonal measures will be highlighted. \\

The original setting in \cite{Rota1} gave the birth to the first systematic theory of stochastic integration in combinatorial terms. The basic idea is the representation of the product random measure as a sum of partition-depending measures, one of this being the exact random analogue of the classical one: the \textit{stochastic measure}. One speaks about exact random analogue because the product random measure does not vanish identically when integrating on the so called \textit{diagonal sets}, producing therefore an ``anomaly'' (see \cite{Engel,PeccatiTaqqu,Rota1}). As a consequence, several known identities concerning stochastic integrals were recovered in a more compact way as identities over the lattice of partitions, giving a unifying treatment of the subject. \\

In the sequel, the accent will be put, in particular, on the diagonal measures associated with L\'{e}vy processes: the process of variations. As pointed out in \cite{Sole}, the variations process of a semi-martingale is homogeneous, that is, for any real number $a$ and for all $n\geq 1$, $ (aX)^{(n)} = a^{n}X^{(n)}$. Moreover, for two semi-martingales $X$ and $Y$, with zero quadratic covariation $[X, Y]_{t} =0$ for all $t$, the additivity property turns into: $(X + Y)^{(n)} = X^{(n)} + Y^{(n)}$. When $Y$ is a constant process, say $Y_{t} =c$, for all $t\geq 0$, the additivity turns to be a semi-invariance property, namely $(X + c)^{(1)} = X^{(1)} + c$, while for all $n \geq 2$, $(X+c)^{(n)} = X^{(n)}$. Since these three properties characterize cumulants \cite{Rota3}, and  $\mathbb{E}[X_{t}^{(n)}] = \chi_{n}(X_{t})$, the process of variations (and, more generally, diagonal measures) can be referred to as \textit{functions of cumulant type}. \\

Here is a short outline of the chapter:
\begin{enumerate}
\item Section \ref{Preliminaries_RotaWallstrom} presents a brief summary on the combinatorial theory of stochastic integration, which is not meant to be exhaustive: any unspecified result can be traced to \cite{PeccatiTaqqu} and \cite{Rota1};
\item in Section \ref{CumDiag}, Theorem \ref{TEO} provides cumulants as (deterministic) measures, namely the expectation of the diagonal measures. As a consequence, an alternative simpler proof of the identity between cumulants of a L\'{e}vy processes and cumulants of its variation processes (see \eqref{Cum}) is achieved. The advantage of this approach is that the expression provided for the cumulants does not depend on the orders of the involved variations processes. 
\item In Section \ref{kstat}, Theorem \ref{TEO} is given a statistical interpretation: more specifically, it is  shown that diagonal measures naturally correspond to $\kappa$-statistics for positive random measures. At the end, the discussion is supplied with comparison with analogous results in the free probability setting \cite{Ans1,Ans2}. 
\end{enumerate}

\section{Preliminaries: Random measures}\label{Preliminaries_RotaWallstrom}

Let $(Z, \mathcal{Z})$ denote a Polish space (that is, a complete metrizable and separable topological space), where $\mathcal{Z}$ denotes the Borel $\sigma$-algebra of $Z$. For every $n\geq 1$, let $(Z^{\otimes n}, \mathcal{Z}^{\otimes n})$ denote the $n$-fold product space of $(Z, \mathcal{Z})$. Following Rota and Wallstrom \cite{Rota1}, a measurable set of the type $\mathcal{A} = A_{1}\times \cdots \times A_{n}$  in $\mathcal{Z}^{\otimes n}$  will be called a rectangle, and a rectangle with equal sides will be called a cube.\\
The combinatorial approach to the theory of stochastic integration is based on the concept of partitions of a set: for any partition $\pi \in \mathcal{P}([n])$, and any measurable set $\mathcal{A} \in \mathcal{Z}^{\otimes n}$, consider:
\begin{itemize}
\item[(i)] the \textit{diagonal} set associated with $\pi$:
$$ \mathcal{A}_{\pi} = \{(z_{1},\dots,z_{n})\in \mathcal{A}: z_{i} = z_{j} \text{ if and only if } i \sim_{\pi} j\};$$
\item[(ii)] the \textit{superdiagonal} set associated with $\pi$:
$$ \mathcal{A}_{\geq \pi} = \{(z_{1},\dots, z_{n}): i \sim_{\pi} j \Rightarrow z_{i} = z_{j}\},$$
where $i \sim_{\pi} j$ denotes the equivalence relation on $[n]$: $i \sim_{\pi} j$ if and only if $i$ and $j$ belong to the same block of $\pi$.
\end{itemize}
Since the product space is itself a Polish space (and therefore, it is second countable), every diagonal set can be written as the union of at most countably many rectangles whose kernel is $\pi$ (where a rectangle $A_{1} \times \cdots \times A_{n}$ is said to have kernel $\pi$ if $A_{i} \cap A_{j} = \emptyset$ for $i \nsim_{\pi} j$). Thanks to the $\sigma$-additivity property of the random measures defined later on, one can and will focus the attention only on such rectangles.\\
It is easy to check that the following properties hold:
\begin{enumerate}
\item $\mathcal{A}_{\geq \pi} = \bigcup\limits_{\sigma \geq \pi}\mathcal{A}_{\sigma}$ (in particular, $\mathcal{A}_{\geq \hat{0}} = \bigcup\limits_{\sigma \geq \hat{0}}\mathcal{A}_{\sigma} = \mathcal{A}$);
\item $ \mathcal{A}_{\pi} \cap \mathcal{A}_{\sigma} = \emptyset \; \text{ for } \sigma \neq \pi. $
\end{enumerate}

\begin{defn}
A \textbf{random measure} \index{Random measure}  $\Phi$ on $(Z, \mathcal{Z})$ is a finitely additive set function admitting a $\sigma$-additive extension, that maps each $A \in \mathcal{Z}$ to a random variable $\Phi(A)$ in some Banach space of random variables on a fixed probability space $(\Omega, \mathcal{F}, \mathbb{P})$, and such that $\Phi(\emptyset) = 0$ a.s.\,. $\Phi$ is said to be a \textbf{completely random measure}\index{Random measure!Completely random measure} (for short, CR-measure) if it has ``independent increments'', that is if it maps pairwise disjoint sets to a system of independent random variables.
\end{defn}

\begin{rmk}
Usually, random measures are defined over $\mathrm{L}^2(\Omega)$, or $\bigcap\limits_{p \geq 1}\mathrm{L}^p(\Omega)$ (see \cite{PeccatiTaqqu}). In the literature, completely random measures are also called \textit{independently scattered measures}. Unlike the deterministic product measure, the $\sigma$-additive extension of the product $\Phi^{\otimes n}$ need not be uniquely determined (see \cite{Engel}). When this is not the case, and accordingly with the definitions given in \cite{PeccatiTaqqu} and \cite{Rota1}, $\Phi$ will be referred to as a \textit{good random measure}.\\
\end{rmk}

On the space $(Z, \mathcal{Z})$, a $\sigma$-finite \textit{non-atomic measure} $\nu$ is given to control the random measure $\Phi$ as follows:
$$ \mathbb{E}[\Phi(B)^{2}] = \nu(B), \quad \mathbb{E}[\Phi(B)\Phi(C)] = \nu(B \cap C),$$
where, as usual, $\mathbb{E}$ denotes the expectation on the fixed probability space. The most important feature of a non-atomic measure $\nu$ is that, for every integer $N$, any measurable set $A$, with $0<\nu(A) < \infty$, can be partitioned into $N$ measurable subsets, pairwise disjoint and with the same measure:
\begin{equation}
\label{nonatomicProperty}
 A = \bigcup_{i = 1}^{N} A_{iN}, \quad \text{ with } \nu(A_{iN}) = \frac{\nu(A)}{N}.
 \end{equation}

\begin{rmk}
By definition, a measure $\nu$ is non-atomic if for every measurable set $\mathcal{A}$ with $0 <\nu(\mathcal{A}) < \infty$, for every $r \in (0, \nu(\mathcal{A}))$ there exists a measurable set $B \subset \mathcal{A}$ with $\nu(B) = r$.  The non-atomicity of the measure $\nu$ is needed to ensure that the random field $\Phi$ is composed of infinitely divisible distributions (see, for instance, \cite[Proposition 5.3.2]{PeccatiTaqqu}), and to ensure that the class of elementary (simple) functions of $n$ variables is dense in the space $\mathrm{L}^2(\nu^{\otimes n})$ of functions that are square-integrable with respect to $\nu^{\otimes n}$, for every $n$ (see, for instance, \cite[Lemma 5.5.2]{PeccatiTaqqu}). The most important example of non-atomic measure is the Lebesgue measure on the real line. \\
\end{rmk}

For good CR-measures, the main idea in \cite{Rota1} was to consider the restrictions of the product measure to diagonal and superdiagonal sets as measures themselves.  In the sequel, as  in \cite{Sole}, the setting will be that of the product of $n$ jointly good random measures, according to the following definition.

\begin{defn}
Let $\Phi_1,\dots,\Phi_n$ be completely random measures given on the same space $(Z, \mathcal{Z})$. $\Phi_1,\dots,\Phi_n$ are \textbf{jointly good} if the (finitely additive) product vector measure $\Phi_1 \otimes \cdots \otimes \Phi_n$ can be extended to a unique $\sigma$-additive random measure on the product space $(Z^{\otimes n}, \mathcal{Z}^{\otimes n})$.
\end{defn} 
 
If $\Phi_1,\dots,\Phi_n$ are assumed to be jointly good on a Polish space $(Z,\mathcal{Z})$, the following definitions are well-posed.
\begin{defn}
For every $\pi \in \mathcal{P}([n])$ and $\mathcal{A} \in \mathcal{Z}^{\otimes n}$, define:
\begin{equation}
 St_{\pi}^{(\Phi_1,\dots,\Phi_n)}(\mathcal{A}) := \Phi_1\otimes \cdots \otimes \Phi_n(\mathcal{A}_{\pi})
 \end{equation}
 and
 \begin{equation}
  St_{\geq \pi}^{(\Phi_1,\dots,\Phi_n)}(\mathcal{A}) := \Phi_1 \otimes \cdots \otimes\Phi_n (\mathcal{A}_{\geq \pi}).
  \end{equation}
\end{defn}

It is easy to verify, by using the additivity property of $\Phi_1 \otimes \cdots \otimes \Phi_n$, and the M\"obius inversion theory on the lattice $\mathcal{P}([n])$ \cite{Rota2,Stanley}, that the measures $ St_{\pi}^{(\Phi_1,\dots,\Phi_n)}$ and  $St_{\geq \pi}^{(\Phi_1,\dots,\Phi_n)}$ satisfy the combinatorial identities summarized in the next proposition (see \cite{Sole}).
\begin{prop}
\label{prop1.1} Let $\Phi_1,\dots,\Phi_n$ be jointly good CR-measures. Then:
\begin{itemize}
\item[(i)] $ \Phi_1 \otimes \cdots \otimes \Phi_n = \sum\limits_{\sigma \in \mathcal{P}([n])}St_{\sigma}^{(\Phi_1,\dots,\Phi_n)};$
\item[(ii)] $ St_{\geq \pi}^{(\Phi_1,\dots,\Phi_n)} = \sum\limits_{\substack{\sigma \in \mathcal{P}([n]) \\ \sigma \geq \pi}}St_{\sigma}^{(\Phi_1,\dots,\Phi_n)};$
\item[(iii)] $ St_{\pi}^{(\Phi_1,\dots,\Phi_n)}= \sum\limits_{\substack{\sigma \in \mathcal{P}([n]) \\ \sigma \geq \pi}}\mu(\pi,\sigma)St_{\geq \sigma}^{(\Phi_1,\dots,\Phi_n)};$
\item[(iv)] $ St_{\hat{0}}^{(\Phi_1,\dots,\Phi_n)} = \sum\limits_{\sigma \in \mathcal{P}([n])}\mu(\hat{0},\sigma)St_{\geq \sigma}^{(\Phi_1,\dots,\Phi_n)} ,$
\end{itemize}
where $\mu(\sigma, \pi)$ is the M\"obius function on the interval $[\sigma, \pi] = \{\tau \in \mathcal{P}([n]): \sigma \leq \tau \leq \pi\}$ (see \cite{Rota2}).\\
\end{prop}

The measure $St_{\hat{0}}^{(\Phi_1,\dots,\Phi_n)}$ is called the \textit{stochastic measure} \index{Random measure!Stochastic measure}of order $n$: it is concentrated on the so-called completely non-diagonal subset $\mathcal{A}_{\hat{0}}$ of $\mathcal{A}$, that is, on the subset of $\mathcal{A}$ whose elements are the $n$-tuples with coordinates all distinct among themselves. In this direction, rectangles with kernel $\hat{0}$ are called triangles, since they have only trivial diagonal subsets. Moreover, note that  the product measure $\Phi_1 \otimes \cdots \otimes \Phi_n$ can be recovered  as $St_{\geq \hat{0}}^{(\Phi_1,\dots,\Phi_n)}$. More generally, when $\Phi_j=\Phi$ for every $j=1,\dots,n$, $St_{\geq \pi}^{[n]}:= St_{\geq \pi}^{(\Phi,\dots, \Phi)}$ is always a product measure, in the sense specified in the following proposition \cite{Rota1}.
\begin{prop}
\label{prop1.2}
The measure $St_{\geq \pi}^{[n]}$ is the product of the completely non-diagonal measures related to the blocks of $\pi \in \mathcal{P}([n])$:
\begin{equation}
St_{\geq \pi}^{[n]} = \bigotimes_{b \in \pi}St_{\hat{1}}^{b},\\
\end{equation}
where $St_{\hat{1}}^{b}$ is a short for  $St_{\hat{1}}^{(\Phi^{\otimes |b|})} $.\\
\end{prop}

\begin{rmk}
Thanks to Proposition \ref{prop1.2} and identities $(ii)$-$(iii)$ of Proposition \ref{prop1.1}, in dealing with random measures, it is sufficient to focus on the measures $St_{\hat{0}}^{[n]}$ and $St_{\hat{1}}^{[n]}$. 
\end{rmk}

\begin{defn}
If $\Phi_1,\dots,\Phi_n$ are jointly good CR-measure, the $n$-th \textbf{diagonal measure} \index{Random measure!Diagonal measure} of $A \in \mathcal{Z}$ is defined by $ \Delta_{n}(A) := St_{\hat{1}}^{(\Phi_1,\dots,\Phi_n)}(A^{\otimes n}).$
\end{defn}

It is easy to check that $\Delta_{n}(A \cup B) = \Delta_{n}(A) + \Delta_{n}(B)$ whenever $A$ and $B$ are disjoint, and that diagonal measures satisfy the following intersection property \cite{Rota1}:
\begin{equation}
\label{St_Diag}
St_{\hat{1}}^{(\Phi_1,\dots,\Phi_n)}(A_{1}\times \cdots \times A_{n}) = St_{\hat{1}}^{(\Phi_1,\dots,\Phi_n)}\bigg(\big(\bigcap_{i \in [n]}A_{i}\big)^{\otimes n}\bigg) = \Delta_{n}\bigg(\bigcap_{i \in [n]}A_{i}\bigg).
\end{equation}

\section{Cumulants and diagonal measures}\label{CumDiag}
Jointly multiplicative random measures are defined accordingly  with the definition of multiplicative measure given in \cite{Rota1,PeccatiTaqqu}.
\begin{defn}
Let $\Phi_{1},\dots, \Phi_{n}$  be (jointly good) random measures over the same Polish space. $\Phi_{1}, \dots, \Phi_{n}$ are \textbf{jointly multiplicative} if, for every partition $\pi \in \mathcal{P}([n])$, the following factorization over the blocks of $\pi$ holds:
\begin{equation}
\mathbb{E}[St_{\pi}^{(\Phi_1,\dots, \Phi_n)}] = \bigotimes_{b \in \pi}\mathbb{E}[St_{\hat{1}}^{(\Phi_{j}:j \in b)}].
\end{equation}
\end{defn}
It can be easily proved, just rearranging the corresponding proof in \cite[Proposition 8]{Rota1}, that jointly multiplicative random measures can be characterized in the following way (it is a consequence of the fact that the product of a non-atomic measure does not charge diagonals).

\begin{prop}
Let $\Phi_{1},\dots, \Phi_{n}$ be (jointly good) CR-measures on a fixed Polish space. Then, the following statements are equivalent:
\begin{itemize}
\item[(i)]$\Phi_{1},\dots, \Phi_{n}$  are jointly multiplicative;
\item[(ii)] for every $b \subseteq [n]$, the deterministic measure $\mathbb{E}[St_{\hat{1}}^{(\Phi_j:j \in b)}(\cdot)]$ is non-atomic.
\end{itemize} 
\end{prop}
  
\begin{exm}
The main examples of multiplicative good CR-measures are listed below (see \cite{Rota1,PeccatiTaqqu}):
\begin{itemize}
\item[-] the \textit{Gaussian measure}, such that $\Phi(A) \sim \mathcal{N}(0,\nu(A))$ for every $A$ with $\nu(A) < \infty$, has diagonal measures  given by $\Delta_1(A)=0$, $\Delta_2(A) = \nu(A)$, $\Delta_n(A) = 0$, for all $n\geq 3$;
\item[-] the \textit{Poisson measure}, such that $\Phi(A)$ has a Poisson distribution of rate $\nu(A)$ for every $A$ with $\nu(A) < \infty$, has diagonal measures given by $\Delta_n(A) = \Phi(A)$, for all $n\geq 2$, while $\Delta_1(A) = \Phi(A) - \nu(A)$;
\item[-] the \textit{Compound Poisson measure} given by $\Phi(A) = X_1 + \cdots + X_{\Psi(A)}$, where $\Psi(\cdot)$ is a Poisson measure of control $\nu(A)$, has diagonal measure given by $\Delta_n(A) = X_{1}^n + \cdots + X_{\Psi(A)}^n$, for all $n\geq 1$;
\item[-] let $\Phi$ denote the random measure spanned by a c\'{a}dl\'{a}g\footnote{Within the theory of stochastic processes, c\'{a}dl\'{a}g is a short for \textit{continue à droite, limite à gauche}, indicating that the paths $X_t$ are right continuous and admit a left-limit.} L\'{e}vy process $\bs{X}=\{X_t\}_{t\geq 0}$, with control given by the Lebesgue measure on the positive half-line of the real numbers, namely $\Phi([s,t])=X_t-X_s$ for every $s<t$. Then, $\Delta_n([0,t]) = X_{t}^{(n)}$, where $\bs{X}^{(n)}=\{X_t^{(n)}\}_{t\geq 0}$ is the $n$-th process of variations of $\bs{X}$.
\end{itemize}
\end{exm}

The representation of cumulants as deterministic real measures is established via the next statement.
\begin{thm}
\label{TEO}
For $n\in \mathbb{N}$, let $\Phi_{1},\dots, \Phi_{n}$ be jointly good and multiplicative completely random measures. Then, for every $A_1,\dots, A_n \in \mathcal{Z}$ of finite measure,
\begin{equation}
\label{MainProof2}
\chi(\Phi_{1}(A_1), \dots, \Phi_{n}(A_n)) = \mathbb{E}[St_{\hat{1}}^{(\Phi_1,\dots, \Phi_{n})}\big(A_1 \times \cdots \times A_n\big)],
\end{equation}
where $\chi(\Phi_{1}(A_1), \dots, \Phi_{n}(A_n))$ denotes the multidimensional cumulant of the random vector $(\Phi_{1}(A_1), \dots, \Phi_{n}(A_n))$ (see \eqref{MomCum2}).
\end{thm}

\begin{proof}
By virtue of Proposition \ref{prop1.1},
\begin{equation}
\Phi_1 \otimes \cdots \otimes \Phi_n \big(A_1 \times \cdots \times A_n\big) = \sum_{ \pi \in \mathcal{P}([n])} \mathrm{St}_{\pi}^{(\Phi_1,\dots,\Phi_n)}\big(A_1 \times \cdots \times A_n\big).
\end{equation}
Since the random measures are jointly multiplicative, taking the expectation on both sides of the above equation yields:
\begin{equation}
\mathbb{E}[\Phi_1 \otimes \cdots \otimes \Phi_n \big(A_1 \times \cdots \times A_n\big)] = \sum_{\pi \in \mathcal{P}([n])}\prod_{b \in \pi} \mathbb{E}[\mathrm{St}_{\hat{1}}^{(\Phi_j:j \in b)}\big(\bigotimes_{j \in b}A_j\big)],
\end{equation}
from which trivially follows that:
\begin{equation}
\chi(\Phi_j(A_j):j \in b) = \mathbb{E}[\mathrm{St}_{\hat{1}}^{(\Phi_j :j \in b)}\big(\bigotimes_{j \in b}A_j\big)].
\end{equation}
In particular, when $\pi = \hat{1}$,  the desired relation follows.
\end{proof}

Applying \eqref{MainProof2} for $\Phi = \Phi_j$ for all $j=1,\dots,n$, and the diagonal measures $\Delta_{n}$ associated with its product $\Phi^{ \otimes n}$,  the following statements hold.

\begin{thm}
\label{MioTeo}
Let $\Phi$ be a multiplicative good CR-measure on the non-atomic Polish space $(Z, \mathcal{Z}, \nu)$. For every measurable set $A \in \mathcal{Z}$, with $\nu(A) < \infty$:
\begin{equation}
\label{miaformula}
\chi_{n}(\Phi(A)) = \mathbb{E}[ \Delta_{n}(A) ].
\end{equation}
\end{thm}

The following corollary points out the relationship between multivariate and univariate cumulants for any random measure, obtained straightforwardly from the intersection property of diagonal measures.
\begin{cor}\label{MultiCumVsUnivbis}
Let $\Phi$ be a multiplicative good CR-measure on the Polish space $(Z,\mathcal{Z},\nu)$. Then, for every $A_1,\dots, A_n \in \mathcal{Z}$, $\nu(A_j) < \infty$,  
\begin{equation}
\label{MultiCumVsUniv2}
\chi_{n}(\Phi(\bigcap_{j \in [n]} A_{j})) = \chi(\Phi(A_{1}),\dots, \Phi(A_{n})).
\end{equation}
\end{cor}

\begin{exm}
Let $\Phi$ denote the random measure spanning a L\'{e}vy process $\bs{X}=\{X_t\}_{t \geq 0}$, with $X_0 = 0$ a.s., namely $\Phi([s,t]) = X_t - X_s$ if $s < t$. Then, for every $n\in \mathbb{N}$ and every choice of real numbers $0 \leq t_1 < t_2 < \cdots < t_n$, identity \eqref{MultiCumVsUniv2} gives:
$$ \chi(X_{t_1},\dots,X_{t_n}) = \chi_n(X_{t_1}).$$
\end{exm}

Theorem \ref{MioTeo} entails, as a consequence, that the additivity of cumulants is inherited by the additivity of measures on disjoint measurable sets. Likewise, Corollary \ref{MultiCumVsUnivbis} provides the vanishing of cumulants of independent entries ($\Phi(A_i), \Phi(A_j)$, if $A_i \cap A_j = \emptyset$) as a consequence of  $\Phi(\emptyset) = 0$ a.s.. 

\subsection{An application to the process of variation of a L\'{e}vy process}

Let $\mathbf{X}=\{X_t\}_{t \geq 0}$ denote a c\'{a}dl\'{a}g L\'{e}vy process on $\mathbb{R}_+$, admitting moments of all orders (see \cite{Applebaum, Sato} for any unspecified definition). Recall that for every integer $n\geq 1$, the $n$-th process of variations $\mathbf{X}^{(n)}= \{X_t^{(n)}\}_{t \geq 0}$ associated with $\mathbf{X}$ is the L\'{e}vy process defined by:
\begin{itemize}
\item[(i)] $X_{t}^{(1)} = X_{t}$;
\item[(ii)] $X_{t}^{(2)} = \sigma^{2}t + \sum\limits_{0 < s \leq t} (\Delta X_s)^{2}$;
\item[(iii)] for every integer $n \geq 3$, $X_t^{(n)} = \sum\limits_{0 < s \leq t} (\Delta X_s)^{n}$,
\end{itemize}
where $\sigma^{2}$ denotes the variance of the Gaussian component of $\mathbf{X}$ (as in the Khintchine formula for its characteristic function), and $\Delta X_s$ is the jump in $s$, namely $\Delta X_s = X_s - X_{s^{-}}$, with $X_{s^{-}} = \lim\limits_{t \rightarrow s^{-}}X_t$. Cumulants of the variations process are related to the process $\bs{X}$ via:
\begin{equation}
\label{Cum}
\chi(X_{t}^{(c_{1})},\dots, X_{t}^{(c_{n})}) = \chi_{c_{1}+\cdots + c_{n}}(X_{t}),
\end{equation}
where $c_{j} \in \mathbb{N}$, and with a slight abuse of notation,  $\chi$ denotes both a multidimensional and a unidimensional cumulant.\\

Consider the $n$-dimensional variations process $(X^{(1)}, \dots, X^{(n)})$ associated  $\mathbf{X} = \{X_{t}\}_{t \geq 0}$. For every $t\geq 0$, the characteristic function of the random vector $(X_{t}^{(1)}, \dots, X_{t}^{(n)})$ is given by:
\begin{equation}
\varphi_{t}(z_1,\dots,z_n) = \exp\left\lbrace-\frac{1}{2}t z_{1}^{2}\sigma^{2} + \imath t z_{2}\sigma^{2} + t\int_{\mathbb{R}}(e^{\imath \sum_{j=1}^{n}z_{j}x^{j}} - 1- \imath z_{1}x ) \mu(dx)\right\rbrace,
\end{equation}
where $\mu$ denotes the L\'{e}vy measure of $\mathbf{X}$ and $\sigma^{2}$ the variance of its Gaussian component (see, for instance, \cite{Sole}).
Similarly, for every $t \geq 0$ and every choice of non-negative integers $c_{j}$'s, the characteristic function of the $n$-dimensional variations process $(X_t^{(c_{1})}, \dots, X_t^{(c_{n})})$ is given by:
\begin{equation}\label{CumFunVarProc}
\varphi_{t}(z_1,\dots,z_n) = \exp\left\lbrace-\frac{1}{2} t\sigma^2 \sum_{\substack{h,j=1,\dots,n \\ c_h=c_j =1}} z_{h}z_j + \imath t \sigma^{2}\sum_{\substack{h=1,\dots,n\\ c_h=2}}z_h + t\int_{\mathbb{R}}( e^{\imath \sum\limits_{c_j \geq 2}z_{j}x^{c_{j}}} - 1- \imath z_{1}x) \mu(dx)\right\rbrace.
\end{equation}
Thanks to \eqref{CumFunVarProc}, it follows that, if $c_1 +\cdots + c_n \geq 3$:
\begin{equation}
\chi(X_t^{(c_{1})}, \dots, X_t^{(c_{n})}) = t \int_{\mathbb{R}}x^{c_{1}+\dots + c_{n}}\mu(dx) = t L_{c_{1}+\dots + c_{n}},
\end{equation}
while, if $c_1 + \cdots + c_n = 2$ (in the case $n=2$), then one has:
\begin{equation}
\chi(X_t^{(c_{1})}, \dots, X_t^{(c_{n})}) = t\sigma^2 + t L_2,
\end{equation}
where $L_{m}$ denotes the $m$-th L\'{e}vy moment of $\mathbf{X}$ (that is, the $m$-th moment of the L\'{e}vy measure $\mu$). From these identities, (\ref{Cum}) follows trivially remarking that, if $c_1 +\cdots + c_n \geq 3$, then:
\begin{equation}
\chi_{c_{1}+\dots + c_{n}}(X_{t}) = t L_{c_{1}+\dots + c_{n}},
\end{equation}
while if $c_1 + \cdots + c_n = 2$ (in the case $n=2$):
\begin{equation}
\chi_{2}(X_{t}) = t\sigma^2 + t L_2,
\end{equation}

In general, \eqref{Cum} is proved by differentiating the respective cumulant generating functions (as for the computation of any cumulant), and by checking that the two quantities are equal. One of the advantages of the random measure approach is that \eqref{Cum} follows simply by taking an expectation. For any integer $c \geq 1$, consider the random measure spanning the variation process of the L\'{e}vy process $\mathbf{X}$, namely $\Phi_{c}([s,t]) = X_{t}^{(c)}$ - $X_{s}^{(c)}$, as introduced in \cite{Sole2}. In particular, $\Phi_{c}([0,t]) = X_{t}^{(c)}$ if $X_{0} = 0$ a.s.\,. For every choice of non-negative integers $c_{1},\dots, c_{n}$, the authors in \cite{Sole2} showed that the diagonal measure associated with the product random measure $\Phi_{c_{1}}\otimes \cdots \otimes \Phi_{c_{n}}$ corresponds to the variation of order $c_{1} + \cdots + c_{n}$. More precisely, if $St_{\hat{1}}^{(c_{1},\dots, c_{n})}$ stands for $St_{\hat{1}}^{(\Phi_{c_1},\dots,\Phi_{c_n})}$, and  following the notation introduced in \cite{Sole2}, one has:
\begin{equation}
\label{MultiDiagMeas}
St_{\hat{1}}^{(c_{1},\dots, c_{n})}([0,t]^{\otimes n}) = \Phi_{c_{1}}\otimes \cdots \otimes \Phi_{c_{n}}\big(([0,t]^{\otimes n})_{\hat{1}}\big) = \Phi_{c_{1}+\cdots + c_{n}}([0,t]) = X_{t}^{(c_{1}+\cdots + c_{n})},
\end{equation}
and therefore, if $m=c_1 + \cdots +c_n$,
\begin{equation}
St_{\hat{1}}^{(c_{1},\dots, c_{n})}([0,t]^{\otimes n}) = St_{\hat{1}}^{(\Phi_1,\dots,\Phi_1)}([0,t]^{\otimes m}),
\end{equation}
namely, the diagonal measure associated with $\Phi_1 = \Phi$. 
\begin{rmk}
For the random measures $\Phi_{c_1},\dots, \Phi_{c_n}$ spanning the processes of variations of orders $c_1,\dots,c_n$, Proposition \ref{prop1.2} has been generalized to yield (see \cite[Lemma 4.4]{Sole}): 
$$ \Phi_{c_1}\otimes \cdots \otimes \Phi_{c_n}\big(  (A^{\otimes n})_{\geq \sigma} )    \big) = \prod_{j=1}^m \Phi_{\sum\limits_{i \in B_j}c_j}(A),$$
for $\sigma \in \mathcal{P}([n])$, where $B_1,\dots,B_m$ denote the blocks of $\sigma$.\\
\end{rmk}

Most importantly, the random measures spanning the processes of variation of a Lévy process are jointly-multiplicative.
\begin{prop}
The random measures $\Phi_{c_{1}}, \dots, \Phi_{c_{n}}$ generating the variations processes of orders $c_{1},\dots, c_{n}$ of a (c\'{a}dl\'{a}g) L\'{e}vy process $\mathbf{X}$, with moments of all orders, are jointly multiplicative.
\end{prop}
\begin{proof}
From \eqref{MultiDiagMeas}, it follows that for every $b \subseteq [n]$,
$$ \mathbb{E}[St_{\hat{1}}^{(c_j:j \in b)}\big([0,t]^{\otimes |b|}\big)] = 	\mathbb{E}[X_t^{(a_b)}] = L_{a_b} \, t \, ,$$
where $a_b:= \sum\limits_{j \in b} c_j$. 
\end{proof}

For the sake of convenience, Theorem \ref{TEO} will be reformulated explicitly for random measures spanning the processes of variation of a L\'{e}vy process.
\begin{prop}
Let $\mathbf{X}$ be a (c\'{a}dl\'{a}g) L\'{e}vy process, with finite moments of all order. For every $n\in \mathbb{N}$, and every choice of non-negative integers $c_1,\dots,c_n$:
\begin{equation}
\label{rel1}
\chi(X_{t}^{(c_{1})},\dots,X_{t}^{(c_{n})}) = \mathbb{E}[St_{\hat{1}}^{(c_{1},\dots, c_{n})}\big([0,t]^{\otimes n}\big)].
\end{equation}
\end{prop}

As a consequence, (\ref{Cum}) can be proved in few lines. 
\begin{proof}
As already recalled, for a random measure spanning a L\'{e}vy process, the diagonal measures correspond to the processes of variations: $\Delta_{n}([0,t]) = X_{t}^{(n)}$. Then, from (\ref{miaformula}) and by virtue of Theorem \ref{TEO},  the identity (\ref{Cum}) follows simply by taking the expectation in (\ref{MultiDiagMeas}).
\end{proof}

\section{Diagonal measures and $\kappa$-statistics}\label{kstat}
Throughout this section, assume that $\Phi$ is a positive good CR-measure on the fixed non-atomic Polish space $(Z,\mathcal{Z},\nu)$, that is $\Phi(A) \geq 0$ a.s. for every measurable set $A$ in $\mathcal{Z}$, with $\nu(A)$ finite. 

\begin{exm}
This is the case for the random measure on $\mathbb{R}_+$ spanning a L\'{e}vy process that is a \textit{subordinator} (namely, a L\'{e}vy processes that is a.s. increasing in time), as for the Poisson and the Gamma processes (see \cite{Applebaum}). \\
\end{exm}
Few statistical definitions, in the setting of simple random sampling, will be needed.

\begin{defn}$ $
\begin{itemize}
\item[-] An estimator $V$ of a population characteristic $\theta$ is called \textbf{unbiased} if $\mathbb{E}[V] = \theta$. 
\item[-] Given a finite population $\mathbf{x}=\{x_1,\dots,x_N\}$ of size $N$, consider a sequence of statistics $T=\{T_1,\dots,T_N\}$, with $T_n$ function of $n$ variables for all $n\geq 1$ (usually, $\theta:= T_N(\mathbf{x})$ is a parameter to be estimate). $T$ is said to be \textbf{inherited on the average} if, for every $n \leq N$, the average over all possible samples  $\mathbf{y}=(y_1,\dots,y_n)$ of size $n$, drawn from the population $\mathbf{x}$, equals $T_N(x_1,\dots,x_N)$, in symbol:
$$ \mathbb{E}[T_n(\mathbf{y})|\mathbf{x}] = T_N(\mathbf{x}) .$$
\item[-] For an infinite population, the inheritance on the average is  satisfied in the limit, as:
$$ \lim_{n\rightarrow \infty} \mathbb{E}[T_n(\bs{y})] = \theta,$$
where $\mathbf{y}=(y_1,\dots,y_n)$ denotes a sample of size $n$ drawn from the population.
\end{itemize}
\end{defn}

Remark that unbiasedness and inheritance on the average are two structurally different concepts: indeed, unbiasedness refers to a single estimator, while inheritance is concerned with sequences of statistics. Sometimes, statistics being inherited on the average for finite populations are called \textit{natural statistics} (see \cite{McCullaghSenatoDiNardo} for a new approach to natural statistics for spectral samples, via symbolic methods).

\begin{defn}
Given any statistical distribution, or equivalently any random variable $X$, the $\kappa$-\textbf{statistics} \index{$\kappa$-statistics} $c_{n}$ is the unique symmetric unbiased estimator of its cumulant $\chi_{n}(X)$, that is, $\mathbb{E}[c_{n}] = \chi_{n}(X)$. A \textbf{polykay} \index{Polykay}$c_{r,\dots ,s}$ is an unbiased estimator for a product of cumulants: $ \mathbb{E}[c_{r,\dots, s}]= \chi_{r}(X)\cdots \chi_{s}(X)$, for $r,\dots,s \in \mathbb{N}$.
\end{defn}

Moreover, sequences of $\kappa$-statistics and polykays satisfy the inheritance on the average property.\\

By virtue of Theorem \ref{MioTeo}, the $n$-th diagonal measure $\Delta_n(A)$ is an unbiased estimator for the $n$-th cumulant of $\Phi(A)$ whenever $\Phi$ is a good CR-measure. Similarly, if $\Phi$ is multiplicative,  the sequence $St_{\pi}^{[n]}(A^{\otimes n})$ is a sequence of polykays for the parent distribution $\Phi(A)$.

More generally, the decomposition (\ref{nonatomicProperty}) yields an explicit description of   the $\kappa$-statistics for $\Phi(A)$, in the sense of the following proposition.

\begin{prop}\label{Kstat}
Consider the population described by $\Phi(A)$, and a sample of size $N$, whose elements are described by the i.i.d. random variables $ \Phi(A_{1N}),\dots,\Phi(A_{NN})$. For every integer $n \geq 1$, the sequence with general term  $\sum\limits_{i=1}^{N}\Phi(A_{iN})^{n}$ satisfies the inheritance on the average property as to estimation of $\chi_n(\Phi(A))$, namely:
\begin{equation}
\lim_{N \rightarrow \infty} \mathbb{E}\bigg[ \sum\limits_{i=1}^{N}\Phi(A_{iN})^{n} \bigg] = \chi_{n}(\Phi(A)).
\end{equation}
\end{prop}

\begin{proof}
Last statement follows straightforwardly starting from the $\mathrm{L}^{2}$-limit relation for diagonal measures provided in \cite[Proposition 12]{Rota1}: for a fixed integer $n\geq 1$, and a multiplicative good CR-measure $\Phi$, if for every $m < n$ there exist positive constants $c_{m}$ and $d_{m}$ such that:
\begin{enumerate}
\item $\mathbb{E}[ (\Delta_{m} - \langle \Delta_{m}\rangle)^{2}] \leq c_{m}\nu$,
\item $\vert \mathbb{E}[ \Delta_{m}]\vert \leq d_{m}\nu$;
\end{enumerate}
(where $\nu$ denotes the non-atomic control of $\Phi$ and $\Delta_n$ the $n$-th diagonal measure associated with $\Phi^{\otimes n}$), then it holds true that:
\begin{equation}
\label{limitRel}
\Delta_{n}(A) = \lim_{N \rightarrow \infty}\sum_{i=1}^{N}\Phi(A_{iN})^{n} \,,
\end{equation}
where the limit is to be intended in $\mathrm{L}^{2}(\Omega)$.

Since, for every L\'{e}vy process, the non-atomic control is given by the Lebesgue measure times its second L\'{e}vy moment, from the identity:
$$\mathbb{E}[X_{t}^{(n)}X_{s}^{(n)}] = L_{2n}min\{s,t\} + L_{n}^{2}t^{2}$$
(see \cite{Sole2}), it is easy to check that the diagonal measures associated with $\mathbf{X}$ satisfy these requirements. Indeed,  for $c_{m}:= \frac{L_{2m}}{L_2}$, $d_{m} = \frac{\vert L_{m}\vert}{L_2}$, one has:
\begin{enumerate}
\item $\mathbb{E}[ (X_{t}^{(m)} - L_{m}t)^{2} ] = \mathbb{E}[ (X_{t}^{(m)})^{2}] - L_{m}^{2}t^{2} = L_{2m}t = c_m L_{2}t$;  
\item $\vert \mathbb{E}[ X_{t}^{(m)}] \vert = \vert t\,L_{m}\vert = \vert L_{m}\vert t =d_m \vert L_{2}t$,
\end{enumerate}
where $L_{n}$'s are the L\'{e}vy moments of $\bs{X}$ (note that, by virtue of Wolfe's Theorem, the L\'{e}vy process $\bs{X}$ has finite moments of every order if and only if its L\'{e}vy measure has, see for instance \cite[Lemma 5.3.4]{PeccatiTaqqu}). Since $\mathrm{L}^{2}$-convergence implies the convergence in $\mathrm{L}^{1}$-norm for finite measure spaces, the desired conclusion follows by applying Theorem \ref{MioTeo}:
\begin{equation}
\label{limitDiagonal2}
 \lim_{N \rightarrow \infty}||\sum_{i=1}^{N}\Phi(A_{iN})^{n} ||_{L^{1}} = || \Delta_{n}(A)||_{L^{1}} = \mathbb{E}[\Delta_{n}(A)].\qedhere
\end{equation}
\end{proof}

\paragraph{Free $\kappa$-statistics}

In \cite{Ans1,Ans2}, the author extended the combinatorial approach to stochastic integration to the free probability setting, namely for integration with respect to processes with freely independent increments.

\begin{defn}
A stationary stochastic process $\mathbf{X}$ with freely independent increments on a $W^{\star}$-probability space $(\mathcal{A},\varphi)$ is an additive function that maps every real interval $I=[a,b]$ to a random variable $X(I) \in \mathcal{A}$ such that:
\begin{itemize}
\item[-] $X(I)$ is centered, and $\varphi(X(I)^{2}) = m(I)$, with $m(\cdot)$ denoting the Lebesgue measure on $\mathbb{R}$;
\item[-] $X(A)$ and $X(B)$ are freely independent whenever $A \cap B = \emptyset$;
\item[-] if $0 \leq t_1 < t_2 < \cdots < t_n$,  the increments $X(t_1), X(t_2) - X(t_1), \cdots, X(t_n) - X(t_{n-1})$ are freely independent (where $X(t)$ is a short for $X([0,t])$);
\item[-] (stationarity) if $ s < t$, the distribution of $X(t) - X(s)$ equals that of $X(t-s)$.\\
\end{itemize} 
\end{defn}

Let $\mathbf{X}$ be a fixed stationary stochastic process with freely independent increments. The non-atomicity of the Lebesgue measure guarantees that every measurable set $A$, with $m(A) < \infty$,  can be decomposed as the union of $N$ half-open finite intervals of the real line, for every integer $N$: let $X_{1,N}, X_{2,N},\dots, X_{N,N}$ be the freely independent, identically distributed, and adding up to $X(A)$, random variables corresponding to such intervals, namely $X(A) = X_{1,N} + \cdots X_{N,N}$. \\
The \textit{diagonal measure} of order $k$ of $A$ is defined to be the limit (in the operator norm) of the $k$-th power sum polynomial:
\begin{equation}
\label{freeDiag}
\Delta_{k}(A) = \lim_{N \rightarrow \infty} \sum_{j=1}^{N}X_{j,N}^{k}.
\end{equation}
Moreover, in \cite{Ans1}, the author defined $\bs{X}$ to be a \textit{free multiplicative} random measure on a $W^{\star}$-probability space $(\mathcal{A},\varphi)$, if $\bs{X}$ fulfills the requirement:
\begin{equation}
\label{FreeMultiplicative}
\varphi(St_\pi) = \prod_{B \in \pi} \varphi(\Delta_{|B|}),
\end{equation}
where, with the same notation as in \cite{Ans1},
$$St_\pi(A^{\otimes k}) = \lim_{N\rightarrow \infty}\sum_{\substack{(i_1,\dots,i_k) \in [N]^{k}\\ \mathrm{Ker}(i_1,\dots,i_k)=\pi}}X_{i_1,N}\cdots X_{i_k,N}$$
for $\pi \in \mathcal{P}([k])$. It is important to stress that the process $\bs{X}$ is multiplicative in the above sense only with respect to the lattice of non-crossing partitions, that is, for such a process, (\ref{FreeMultiplicative}) holds if and only if $\pi$ is non-crossing (see \cite[Corollary 2]{Ans1}). \\

Since for any element $a \in (\mathcal{A},\varphi)$, the Cauchy-Schwarz inequality implies $|\varphi(a)| \leq \|a \|$, where $\|a \|=\varphi(aa^{\star})^{\frac{1}{2}}$,  for every non-negative integers $N$ and $r$, setting $a_N = \Delta_r(A) - \sum\limits_{j=1}^{N}X_{j,N}^{r}$, the inequality $|\varphi(a_N)| \leq \|a_N \|$ implies, in its limit, that:
$$\varphi(\Delta_r(A)) = \lim_{N\rightarrow \infty}\varphi \bigg(\sum_{j=1}^{N}X_{j,N}^r\bigg) ,$$
yielding in turn, together with $ \varphi(\Delta_r(A)) = \kappa_r(X(A))$ (see \cite[Theorem 2]{Ans1}), that the sequence $\{\{\sum\limits_{j=1}^{N}X_{j,N}^{r}\}_{N\geq 1}\}_{r\geq 1}$ can be said to be a sequence of \textit{ free $\kappa$-statistics} for $X(A)$. Similarly, if $\bs{X}$  is multiplicative, and $\pi \in \mathcal{NC}([k])$, the sequence 
$$ \sum_{\substack{(i_1,\dots,i_k) \in [N]^{k}\\ \mathrm{Ker}(i_1,\dots,i_k)=\pi}}X_{i_1,N}\cdots X_{i_k,N}$$
can be referred to as  \textit{free polykays} since it verifies:
$$\varphi(St_{\pi}(A)) =  \prod\limits_{B \in \pi}\kappa_{|B|}(X(A)).$$ 
In this sense, the theory of partition-depending stochastic measures supplies the free probability setting with the concept of unbiased estimators.

\backmatter
\appendix


\newpage

\printindex

\small
\bibliographystyle{plain}
\bibliography{biblio}

\end{document}